\documentclass[10pt,a4paper]{article}
\usepackage[utf8]{inputenc}
\usepackage[T1]{fontenc}
\usepackage{geometry}
\usepackage{amsthm}
\usepackage[francais,english]{babel}
\usepackage{amssymb}
\usepackage{lmodern}
\usepackage{calligra,amsmath,amsfonts}
\usepackage{mathrsfs} 
\usepackage{dsfont}
\usepackage{stmaryrd}
\DeclareMathOperator{\dive}{div}
\usepackage{graphicx}
\usepackage{fullpage}
\usepackage[nottoc,notlot,notlof]{tocbibind}
\usepackage[colorlinks=true,urlcolor=black,linkcolor=black,citecolor=black]{hyperref}
\numberwithin{equation}{section}
\usepackage{tikz-cd}
\usepackage{pst-node}
\usetikzlibrary{matrix}
\usepackage{mathtools}
\usepackage{verbatimbox}
\usepackage{diagbox}
\usepackage{array}
\newcolumntype{C}{>{$\displaystyle} c <{$}}
\usepackage{makecell}
\usepackage{float}
\usepackage{tabu}
\usepackage[babel=true]{csquotes}
\usepackage{cancel}
\usepackage{xcolor}

\makeatletter
\def\env@dmatrix{\hskip -\arraycolsep
	\let\@ifnextchar\new@ifnextchar
	\def\arraystretch{2}%
	\array{*{\c@MaxMatrixCols}{>{\displaystyle}c}}%
}

\newenvironment{dmatrix}{\left(\env@dmatrix}{\endmatrix\right)}

\begin{document}
	
	\title{The Classification of Branched Willmore Spheres in the $3$-Sphere and the $4$-Sphere}
	\author{Alexis Michelat\footnote{Department of Mathematics, ETH Zentrum, CH-8093 Zürich, Switzerland.}\; and Tristan \selectlanguage{french}Rivière$^{*}$\selectlanguage{english}\setcounter{footnote}{0}}
	\date{\today}
	
	\maketitle
	
	\vspace{1em}

	\begin{abstract}
		We extend the classification of Robert Bryant of Willmore spheres in $S^3$ to \emph{variational} branched Willmore spheres $S^3$ and show that they are inverse stereographic projections of complete minimal surfaces with finite total curvature in $\mathbb{R}^3$ and vanishing flux. We also obtain a classification of \emph{variational} branched Willmore spheres in $S^4$, generalising a theorem of Seb\'{a}stian Montiel.  As a result of our asymptotic analysis at branch points, we obtain an improved $C^{1,1}$ regularity of the unit normal of \emph{variational} branched Willmore surfaces in arbitrary codimension. We also prove that the width of Willmore sphere min-max procedures in dimension $3$ and $4$, such as the sphere eversion, is an integer multiple of $4\pi$.
	\end{abstract}

	\tableofcontents
	\vspace{0.5cm}
	\begin{center}
	\end{center}
	\theoremstyle{plain}
	\newtheorem*{theorem*}{Theorem}
	\newtheorem{theorem}{Theorem}[section]
	\newenvironment{theorembis}[1]
	{\renewcommand{\thetheorem}{\ref{#1}$'$}%
		\addtocounter{theorem}{-1}%
		\begin{theorem}}
		{\end{theorem}}
    \renewcommand*{\thetheorem}{\Alph{theorem}}
	\newtheorem{theoremdef}{Théorème-Définition}[section]
	\newtheorem{lemme}[theorem]{Lemma}
	\newtheorem{propdef}[theorem]{Proposition-Definition}
	\newtheorem{defi2}{Definition}
	\newtheorem{propdef2}[defi2]{Proposition-Definition}
	\newtheorem{prop}[theorem]{Proposition}
	\newtheorem{cor}[theorem]{Corollary}
	\theoremstyle{definition}
	\newtheorem*{definition}{Definition}
	\newtheorem*{definitions}{Definitions}
	\newtheorem{defi}[theorem]{Definition}
	\newtheorem{rem}[theorem]{Remark}
    \newtheorem*{remast}{Remark}
	\newtheorem{remimp}[theorem]{Important remark}
	\newtheorem{rems}[theorem]{Remarks}
	\newtheorem{remintro}[defi2]{Remark}
	\newtheorem{remsintro}[defi2]{Remarks}
	\newtheorem{exemple}[theorem]{Example}
	\theoremstyle{plain}
	\newtheorem{par2}{Parenthesis}
	\newtheorem{obs}{Important observations}
	\newcommand{\N}{\ensuremath{\mathbb{N}}}
	\renewcommand\hat[1]{%
		\savestack{\tmpbox}{\stretchto{%
				\scaleto{%
					\scalerel*[\widthof{\ensuremath{#1}}]{\kern-.6pt\bigwedge\kern-.6pt}%
					{\rule[-\textheight/2]{1ex}{\textheight}}
				}{\textheight}%
			}{0.5ex}}%
		\stackon[1pt]{#1}{\tmpbox}
	}
	\parskip 1ex
	\newcommand{\vc}[3]{\overset{#2}{\underset{#3}{#1}}}
	\newcommand{\conv}[1]{\ensuremath{\underset{#1}{\longrightarrow}}}
	\newcommand{\A}{\ensuremath{\mathscr{A}}}
	\newcommand{\D}{\ensuremath{\nabla}}
	\renewcommand{\N}{\ensuremath{\mathbb{N}}}
	\newcommand{\Z}{\ensuremath{\mathbb{Z}}}
	\newcommand{\I}{\ensuremath{\mathbb{I}}}
	\newcommand{\R}{\ensuremath{\mathbb{R}}}
	\newcommand{\W}{\ensuremath{\mathscr{W}}}
	\newcommand{\Q}{\ensuremath{\mathscr{Q}}}
	\newcommand{\C}{\ensuremath{\mathbb{C}}}
	\newcommand{\totimes}{\ensuremath{\,\dot{\otimes}\,}}
	\newcommand{\z}{\ensuremath{\bar{z}}}
	\newcommand{\p}[1]{\ensuremath{\partial_{#1}}}
	\newcommand{\Res}{\ensuremath{\mathrm{Res}}}
	\newcommand{\pwedge}[2]{\ensuremath{\,#1\wedge#2\,}}
	\newcommand{\lp}[2]{\ensuremath{\mathrm{L}^{#1}(#2)}}
	\renewcommand{\wp}[3]{\ensuremath{\left\Vert #1\right\Vert_{\mathrm{W}^{#2}(#3)}}}
	\newcommand{\np}[3]{\ensuremath{\left\Vert #1\right\Vert_{\mathrm{L}^{#2}(#3)}}}
	\newcommand{\h}{\ensuremath{\vec{h}}}
	\renewcommand{\Re}{\ensuremath{\mathrm{Re}\,}}
	\renewcommand{\Im}{\ensuremath{\mathrm{Im}\,}}
	\newcommand{\diam}{\ensuremath{\mathrm{diam}\,}}
	\newcommand{\leb}{\ensuremath{\mathscr{L}}}
	\newcommand{\supp}{\ensuremath{\mathrm{supp}\,}}
	\renewcommand{\phi}{\ensuremath{\vec{\Phi}}}
	\newcommand{\Perp}{\ensuremath{\perp}}
    \newcommand{\nperp}{\ensuremath{N}}
	\renewcommand{\H}{\ensuremath{\vec{H}}}
	\newcommand{\e}{\ensuremath{\vec{e}}}
	\newcommand{\f}{\ensuremath{\vec{f}}}
	\renewcommand{\epsilon}{\ensuremath{\varepsilon}}
	\renewcommand{\bar}{\ensuremath{\overline}}
	\newcommand{\s}[2]{\ensuremath{\langle #1,#2\rangle}}
	\newcommand{\bs}[2]{\ensuremath{\left\langle #1,#2\right\rangle}}
	\newcommand{\n}{\ensuremath{\vec{n}}}
	\newcommand{\ens}[1]{\ensuremath{\left\{ #1\right\}}}
	\newcommand{\w}{\ensuremath{\vec{w}}}
	\newcommand{\vg}{\ensuremath{\mathrm{vol}_g}}
	\renewcommand{\d}[1]{\ensuremath{\partial_{x_{#1}}}}
	\newcommand{\dg}{\ensuremath{\mathrm{div}_{g}}}
	\renewcommand{\Res}{\ensuremath{\mathrm{Res}}}
	\newcommand{\un}[2]{\ensuremath{\bigcup\limits_{#1}^{#2}}}
	\newcommand{\res}{\mathbin{\vrule height 1.6ex depth 0pt width
			0.13ex\vrule height 0.13ex depth 0pt width 1.3ex}}
	\setlength\boxtopsep{1pt}
	\setlength\boxbottomsep{1pt}
	\newcommand\norm[1]{%
		\setbox1\hbox{$#1$}%
		\setbox2\hbox{\addvbuffer{\usebox1}}%
		\stretchrel{\lvert}{\usebox2}\stretchrel*{\lvert}{\usebox2}%
	}
	\allowdisplaybreaks
	\newcommand*\mcup{\mathbin{\mathpalette\mcapinn\relax}}
	\newcommand*\mcapinn[2]{\vcenter{\hbox{$\mathsurround=0pt
				\ifx\displaystyle#1\textstyle\else#1\fi\bigcup$}}}
	\def\Xint#1{\mathchoice
		{\XXint\displaystyle\textstyle{#1}}%
		{\XXint\textstyle\scriptstyle{#1}}%
		{\XXint\scriptstyle\scriptscriptstyle{#1}}%
		{\XXint\scriptscriptstyle\scriptscriptstyle{#1}}%
		\!\int}
	\def\XXint#1#2#3{{\setbox0=\hbox{$#1{#2#3}{\int}$ }
			\vcenter{\hbox{$#2#3$ }}\kern-.58\wd0}}
	\def\ddashint{\Xint=}
	\newcommand{\dashint}[1]{\ensuremath{{\Xint-}_{\mkern-10mu #1}}}
	\newcommand{\vv}[1]{\vec{\mkern0mu#1}}
	\newcommand\ccancel[1]{\renewcommand\CancelColor{\color{red}}\cancel{#1}}
	\newcommand\colorcancel[2]{\renewcommand\CancelColor{\color{#2}}\cancel{#1}}
	\newpage
	
	\section{Introduction}
	
	\subsection{Willmore functional and quantization of energy}
	
	This article primarily addresses the generalisation of Bryant's classification of smooth Willmore immersions of the sphere $S^2$ into $S^3$ to branched immersions. Let $\Sigma^2$ be a closed Riemann surface, and $n\geq 3$ a fixed integer. The Willmore energy on a smooth Riemannian manifold $(M^n,h)$ with sectional curvature $\widetilde{K}_h$ is defined on any smooth immersion $\phi:\Sigma^2\rightarrow M^n$ by
	\begin{align*}
		W_{M^n}(\phi)=\int_{\Sigma^2}\left(|\H_g|^2+\widetilde{K}_h\right)d\vg
	\end{align*}
	where $g=\phi^\ast h$ is the pull-back metric of $(M^n,h)$ by $\phi$, and $\H_{\phi}$ is the mean-curvature, that is the half-trace of the second fundamental form $\vec{\I}$ the immersion, given by
	\begin{align*}
		\H_g=\frac{1}{2}\sum_{i,j=1}^2g^{i,j}\,\vec{\I}_{i,j}.
	\end{align*}
	The most basic property of the Willmore energy is its conformal invariance which can be stated as follows. For all conformal diffeomorphism $\varphi:(M^n,h)\rightarrow (\widetilde{M}^n,\widetilde{h})$, we have
	\begin{align*}
		W_{\widetilde{M}^n}(\varphi\circ\phi)=W_{M^n}(\phi).
	\end{align*}
	However, in the special case of $\R^n$, if $\iota_a:\R^n\setminus\ens{a}\rightarrow\R^n\setminus\ens{a}$ is the inversion centred at $a\in \phi(\Sigma^2)$, we do not have in general
	\begin{align*}
		W_{\R^n}(\iota_a\circ \phi)=W_{\R^n}(\phi),
	\end{align*}
	while we have equality for inversions with centre outside of $\phi(\Sigma^2)$. 
	Nevertheless, the quantity
	\begin{align*}
		\mathscr{W}(\phi)=\int_{\Sigma^2}\left(|\H_g|^2-K_g\right)d\vg
	\end{align*}
	where $K_g$ is the intrinsic Gauss curvature of $\phi$, is invariant under every conformal transformation. Indeed, the $2$-form
	\begin{align*}
		\left(|\H_g|^2-K_g\right)d\vg=|\h_0|_{WP}^2d\vg,
	\end{align*}
	where $\h_ 0$ is the Weingarten tensor and $|\,\cdot\,|_{WP}$ designs the Weil-Petersson metric, is a\emph{pointwise} invariant (see for example $7.3.1$ \cite{willmorebook}). We shall come back to this point when we will state Noether's theorem for the Willmore energy (see for example \eqref{inversionWeingarten} in the proof of Theorem \ref{galois}).
	
	We now come to the critical points of the Willmore energy. Classically, they are the smooth immersions satisfying the equation
	\begin{align}\label{eq1W}
		\Delta_g^N\H-2|\H|^2\H+\mathscr{A}(\H)+\mathscr{R}(\H)=0
	\end{align} 
	where $\Delta_g^N$ is the Laplacian on the normal bundle, $\mathscr{A}$ the Simons operator and $\mathscr{R}$ a curvature operator, given by
	\begin{align*}
		\mathscr{A}(\H)&=\sum_{i,j=1}^{2}\s{\vec{\I}(\vec{\epsilon}_i,\vec{\epsilon}_j)}{\H}\vec{\I}(\vec{\epsilon}_i,\vec{\epsilon}_j),\qquad
		\mathscr{R}(\H)=\left(\sum_{i=1}^2R(\H,\vec{\epsilon}_i)\vec{\epsilon}_i\right)^N
	\end{align*}
	where $(\vec{\epsilon}_1,\vec{\epsilon}_2)$ is any local orthonormal moving frame, and $R$ is the Riemann curvature tensor of $(M^n,h)$. However, for the natural regularity $\phi\in W^{2,2}(\Sigma^2,M^n)$ this equation does not even have a distributional meaning, as it would require $\H\in L^{3}_{loc}(\Sigma^2,TM^n)$. The weakest possible setting to work with is the space of \emph{weak immersions} (introduced in \cite{riviere1}, \cite{rivierecrelle})
	\begin{align*}
	\mathscr{E}(\Sigma^2,M^n)=W^{2,2}\cap W^{1,\infty}(\Sigma^2,M^n)\cap
	\left\{\begin{alignedat}{1}
	\phi :\; & d\phi(x)\;\text{has rank}\; 2\;\text{for almost all}\; x\in \Sigma^2\\
	& \text{and}\; \inf_{\Sigma^2} |d\phi\wedge d\phi|_{g_0}>0
	\end{alignedat}\right\}. 
	\end{align*}
	for any fixed Riemannian metric $g_0$ on $\Sigma^2$. In the rest of the introduction, we suppose that $M^n=\R^n$ and that $h$ is the standard flat Euclidean metric. 	The second author showed in the Willmore equation can be written in a conservative weak formulation.
	
	\begin{theorem*}[\cite{riviere1}]
		Let $\Sigma^2$ be a closed Riemann surface, and $\phi:\Sigma^2\rightarrow\R^n$ be a smooth immersion. Then, (identifying $2$-vectors and functions on $\Sigma^2$)
		\begin{align}\label{1}
		\Delta_g^N\H_g-2|\H_g|^2\H_g+\mathscr{A}(\H_g)=d\left(\ast_g d\H_g-3\ast_g (d\H_g)^N+\star (\H_g\wedge d\n)\right)
		\end{align}
		where $\H_g$ is the mean curvature of $\phi$, where $\ast_g$ is the Hodge star operator on $\Sigma^2$ for the metric $g$, and $\star$ the Hodge star operator on $\R^n$ for the flat metric.
	\end{theorem*}

	 As the $1$-form under the exterior derivative in \eqref{1} is in $W^{-1,2}+L^1$, the right-hand side is well-defined in a distributional sense as a element of $\mathscr{D}'(\Sigma^2)$. If the left-hand side is not defined in general for $\phi\in \mathscr{E}(\Sigma^2,\R^n)$ , this comes from the fact that to write it, one has to make a projection on the normal bundle, while the normal is only in $W^{1,2}(\Sigma^2,\mathscr{G}_{n-2}(\R^n))$, where $\mathscr{G}_{n-2}(\R^n)$ denotes the oriented Grassmannian of $(n-2)$-plans in $\R^n$. Computing the Euler-Lagrange equation for arbitrary variations allows one to recover the weak formulation of the right-hand side (see \cite{indexS3}). Furthermore, as we shall see, the conservative form of the Euler-Lagrange equation of Willmore energy is a consequence of Noether's theorem (this last fact already appears in a paper by Yann Bernard (\cite{bernard})).

	Furthermore, writing the Willmore equation as the closeness of a $1$-form allows one to introduce the concept of \emph{variational} Willmore surfaces. In general, a critical point of $W$ is smooth outside a finite number of points (called branch points, where $\phi$ fails to be an immersion), but globally only in $W^{2,p}(\Sigma^2,\R^n)$ for all $p<\infty$. In particular, if the branch points are $p_1,\cdots,p_m\in \Sigma^2$, we could have
	\begin{align}\label{dirac}
		d\left(\ast_g d\H_g-3\ast_g (d\H_g)^N+\star (\H_g\wedge d\n)\right)=\sum_{i=1}^{m}\vec{\alpha}_i \delta_{p_i}
	\end{align}
	for some $\vec{\alpha}_1,\cdots,\vec{\alpha}_m\in \R^n$, or more generally derivatives of Dirac masses.
		
	\begin{defi2}
		\emph{We say that a branched Willmore immersion is \emph{variational} if it is obtained as a weak limit or as a bubble of a sequence of Willmore immersions of uniformly bounded energy.}
	\end{defi2}

    We will see thanks to Theorem \ref{A} that variational Willmore surfaces are also \emph{true} Willmore surfaces in the sense introduced by \cite{eversion}. 
    By definition, a branched Willmore immersion is true if no Dirac mass appears in \eqref{dirac}. However, there might still be derivatives of Dirac masses, related to the lack of smoothness of a branched immersion, measured by the second residue. On conditions under which one can remove this second residue, we refer to \cite{blow-up}, \cite{blow-up2}. 

    A common example of non-variational Willmore spheres are the inversions of the family of catenoids in $\R^3$, which have non-zero flux around their two embedded ends. We refer to the discussion after Theorem \ref{C} for more details on this example.
	
	The equation \eqref{dirac} permits to introduce the first residue defined in \cite{beriviere} as
	\begin{align}\label{resber}
		\widetilde{\vec{\gamma_0}}(p_i)=\frac{1}{4\pi}\int_{\gamma}\ast_g d\H_g-3\ast_g (d\H_g)^N+\star (\H_g\wedge d\n)=\frac{1}{2}\vec{\alpha}_i
	\end{align}
	for any smooth closed curved $\gamma$ around $p_i$, for $i=1,\cdots,m$.  This quantity was first defined for immersions in codimension $1$ by Kuwert and Sch\"{a}tzle in \cite{kuwert}, and in any codimension in \cite{beriviere}. We will see that $\widetilde{\vec{\gamma_0}}(p_i)$ measures on the basic first obstruction to the regularity of Willmore surfaces through the branch points. It appears in particular in the quantization of Willmore energy. Furthermore, as will appear clear in the introduction, the following theorem shows that branched immersions naturally appear and justify much of the work here, outside of the theoretical interest to determine when branched immersions from the sphere are conformally minimal in some space form geometry. 
	
	\begin{theorem*}[\cite{quanta}]
		Let $\{\phi_k\}_{k\in\N}$ be a sequence of Willmore immersions from a closed Riemann surface $\Sigma^2$ into $\R^n$. Assume that
		\begin{align*}
			\limsup_{k\rightarrow\infty}W(\phi_k)<\infty
		\end{align*}
		and that the conformal class of $\{\phi_k^\ast g_{\,\R^n}\}_{k\in\N}$ remains within a compact sub-domain of the moduli space of $\Sigma^2$. Then, modulo extraction of a subsequence, the following energy identity holds:
		\begin{align}\label{quantization}
			\lim\limits_{k\rightarrow\infty}W(\phi_k)=W(\phi_{\infty})+\sum_{i=1}^{p}W(\vec{\Psi}_i)+\sum_{j=1}^{q}\left(W(\vec{\xi}_j)-4\pi\theta_j
			\right)
		\end{align}
		where $\phi_{\infty}:\Sigma^2\rightarrow\R^n$ is a \emph{true} branched Willmore and the bubbles $\vec{\Psi}_i:S^2\rightarrow\R^n$ and $\vec{\xi}_j:S^2\rightarrow \R^n$ are compact branched Willmore spheres, while the integer $\theta_j=\theta_0(\vec{\xi}_j,p_j)\geq 1$ is the multiplicity of the branched immersion $\vec{\xi}_j$ at some point $p_j\in \xi_j(S^2)\subset \R^n$.
	\end{theorem*} 

    Recall that for all $p\in \R^n$, the multiplicity of a branched immersion is defined by
    \begin{align*}
    	\theta_0(\phi,p)=\lim_{r\rightarrow 0}\frac{\mathrm{Area}(\phi(\Sigma)\cap B_r(p))}{\pi r^2}\in \N.
    \end{align*}
    
    We finally introduce the definition of branch points of Willmore immersions.
    
    \begin{propdef2}[\cite{riviere1}, \cite{beriviere}]
    	Let $\phi\in W^{2,2}\cap W^{1,\infty}(D^2)\cap C^{\infty}(D^2\setminus\ens{0})$ be a conformal Willmore immersion of finite total curvature on $D^2$. Then there exists an integer $\theta_0\geq 1$ and $\vec{A}_0\in \C^n\setminus\ens{0}$ such that 
    	\begin{align}\label{notalgebraic}
    	\left\{\begin{alignedat}{1}
    	&\phi(z)=\Re\left(\vec{A}_0z^{\theta_0}\right)+O\left(|z|^{\theta_0+1}\log|z|\right)\\
    	&\p{z}\phi(z)=\frac{\theta_0}{2}\vec{A}_0z^{\theta_0-1}+O\left(|z|^{\theta_0}\log|z|\right),
    	\end{alignedat}\right.
    	\end{align}
    	and we say that $\phi$ has a branch point of order $\theta_0\geq 1$ at $z=0$.
    	Furthermore, provided the mean curvature $\H$ be not identically zero, there exists an integer $m\leq \theta_0-1$ and $\vec{C}_0\in \C^{n}\setminus\ens{0}$ such that for $\theta_0\geq 2$
    	\begin{align}\label{theta2}
    	\H=\Re\left(\frac{\vec{C}_0}{z^{m}}\right)+O\left(|z|^{1-m}\log|z|\right),
    	\end{align}
    	while for $\theta_0=1$, there exists $\vec{\gamma}_0\in\R^n$ such that 
    	\begin{align}\label{theta1}
    	\H=\vec{\gamma}_0\log|z|+O(|z|\log|z|).
    	\end{align}
    	We call $r=\max\ens{m,0}\in \ens{0,\cdots,\theta_0-1}$ the second residue of $\phi$ at the branch point $z=0$. More generally, if $\Sigma$ is a closed Riemann surface, $p_1,\cdots,p_d\in \Sigma$ are fixed distinct points and $\phi:\Sigma\setminus\ens{p_1,\cdots,p_d}\rightarrow \R^n$ is a conformal Willmore immersion of finite total curvature, then we define for all $1\leq j\leq d$ the integers $\theta_0(p_j)\in \N$ to be the order of branch point and $0\leq r(p_j)\leq \theta_0(p_j)-1$ to be the associated residue at $z=0$ of the composition $\phi\circ \psi:D^2\rightarrow \R^n$, for any complex chart $\psi:D^2\rightarrow \Sigma$ such that $\psi(0)=p_i$. This definition does not depend on the chart.
    \end{propdef2}

    We fix some terminology. Let $\phi:\Sigma^2\rightarrow\R^n$ be a smooth immersion, and $\D$ the pull-back connection of the flat connection on $\R^n$ by $\phi$. We let 
    \begin{align*}
    \phi^\ast_{\C}T\R^n=\phi^{\ast}T\R^n\otimes_{\R}\C
    \end{align*}
    be the complexified pull-back bundle of the tangent bundle of $\R^n$ by $\phi$. Then we have the decomposition of the Levi-Civita into tangent and normal parts $\D=\D^\top+\D^N$.  Furthermore, if we define two differential operators $\partial$ and $\bar{\partial}$ of order $1$,  
    \begin{align*}
    	\partial =\D_{\p{z}}\left(\,\cdot\,\right)\otimes dz,\quad \bar{\partial}=\D_{\p{\z}}\left(\,\cdot\,\right)\otimes d\z,
    \end{align*}
    then we also have a decomposition
    \begin{align}\label{defnormal}
    	 \partial=\partial^\top+\partial^N,\quad \bar{\partial}=\bar{\partial}^\top+\bar{\partial}^N.
    \end{align}    
    To be able to rule out the existence of the first residue in a limit of of Willmore immersions, it suffices to understand how bubbles form. Up to diffeomorphisms and rescaling, they are obtained by taking conformal transformations of $\R^n$. The first residue is invariant by translations, rotations, but not by inversions (as for example, it vanishes for minimal surfaces). We are nevertheless able to define thanks to Noether's theorem three others residues, which are transformed one into each other under a simple rule.    
    The invariance by rotations, dilatations, and composition of translations with inversions give the four residues
    \begin{align}\label{res00}
    \left\{\begin{alignedat}{1}
        	\vec{\gamma}_0(\phi,p)&=\frac{1}{4\pi}\,\Im\int_{\gamma}g^{-1}\otimes \left(\bar{\partial}^N-\bar{\partial}^\top\right)\h_0-|\h_0|_{WP}^2\,\partial\phi\\
            \vec{\gamma}_1(\phi,p)&=\frac{1}{4\pi}\,\Im\int_{\gamma}\phi\wedge \left(g^{-1}\otimes\left(\bar{\partial}^N-\bar{\partial}^\top\right)\h_0-|\h_0|^{2}_{WP}\,\partial\phi\right)+g^{-1}\otimes \h_0\wedge \bar{\partial}\phi\\
            \vec{\gamma}_2(\phi,p)&=\frac{1}{4\pi}\,\Im\int_{\gamma}\phi\cdot\left(g^{-1}\otimes\left(\bar{\partial}^N-\bar{\partial}^\top\right)\h_0-|\h_0|^{2}_{WP}\,\partial\phi\right)\\
            \vec{\gamma}_3(\phi,p)&=\frac{1}{4\pi}\,\Im\int_{\gamma}\mathscr{I}_{\phi}\left(g^{-1}\otimes\left(\bar{\partial}^N-\bar{\partial}^\top\right)\h_0-|\h_0|_{WP}^2\,\partial\phi\right)-g^{-1}\otimes\left(\bar{\partial}|\phi|^2\otimes \h_0-2\,\s{\phi}{\h_0}\otimes\bar{\partial}\phi\right)
    \end{alignedat}\right.
    \end{align}
    where for all vector $\vec{X}\in\C^n$, 
    \begin{align*}
    	\mathscr{I}_{\phi}(\vec{X})=|\phi|^2\vec{X}-2\s{\phi}{\vec{X}}\phi.
    \end{align*}
    \begin{remintro}
    	If one prefers an expression without normal derivatives, something which will actually prove crucial in the proof of the main Theorem \ref{devh0}, let us mention that by Codazzi identity, we have
    	\begin{align*}
    		g^{-1}\otimes \left(\bar{\partial}^N-\bar{\partial}^\top\right)\h_0-|\h_0|_{WP}^2\,\partial\phi=\partial\H+|\H|^2\partial\phi+2\,g^{-1}\otimes\s{\H}{\h_0}\otimes\bar{\partial}\phi
    	\end{align*}
    \end{remintro}
    \begin{remintro}
    	In codimension $1$, we have the alternative formulae corresponding  to the four residues
    	\begin{align}\label{constantesnulles}
    		\left\{\begin{alignedat}{1}
    		{\vec{\gamma}_0}(\phi,p)&=-\frac{1}{\pi}\int_{\gamma}\dive\left(\D H\,\n-H\D\n-H^2\D\phi\right),\\
    		{\vec{\gamma}_1}(\phi,p)&=-\frac{1}{\pi}\int_{\gamma}\dive\left(\D H\left(\phi\wedge \n\right)-H\,\D\left(\phi\wedge\n\right)-H^2\,\left(\phi\wedge \D\phi\right)+2H\,\D^{\perp
    		}\phi\right),\\
    		{\vec{\gamma}_2}(\phi,p)&=-\frac{1}{\pi}\int_{\gamma}\dive\left(\D H\left(\phi\cdot \n\right)-H\,\D\left(\phi\cdot \n\right)-\frac{1}{2}H^2\D|\phi|^2\right),\\
    		{\vec{\gamma}_3}(\phi,p)&=-\frac{1}{\pi}\int_{\gamma}\dive\bigg(2\D\phi+2\phi\left(\D H\left(\phi\cdot\n\right)-H\,\D\left(\phi\cdot\n\right)\right)-|\phi|^2\left(\D H\,\n-H\,\D\n\right)\\
    		&+H^2\left(|\phi|^2\D\phi-\D|\phi|^2\phi\right)\bigg).
    		\end{alignedat}\right.
    	\end{align}
    	In particular, comparing \eqref{resber} and \eqref{res00}, we have
    	\begin{align*}
    		\widetilde{\vec{\gamma}_0}(\phi,p)=-4\,\vec{\gamma}_0(\phi,p).
    	\end{align*}
    \end{remintro}
    One can recognize in these formulae the infinitesimal generators of the afore cited symmetries. We have the following correspondence.
    \begin{theorem}\label{A}
    	Let $\phi:\Sigma^2\rightarrow\R^n$ be a branched Willmore surface and let $\iota:\R^n\setminus\ens{0}\rightarrow\R^n\setminus\ens{0}$ be the inversion centred at zero. If $\vec{\Psi}=\iota\circ \phi:\Sigma^2\setminus\phi^{-1}(\ens{0})\rightarrow\R^n$ is the inverted Willmore surface, for all $p\in \Sigma^2$, we have
    	\begin{align}
    	\left\{\begin{alignedat}{1}
    	\vec{\gamma}_0(\phi,p)&=\vec{\gamma}_3(\vec{\Psi},p)\\
    	\vec{\gamma}_1(\phi,p)&=\vec{\gamma}_1(\vec{\Psi},p)\\
    	\vec{\gamma}_2(\phi,p)&=-\vec{\gamma}_2(\vec{\Psi},p)\\
    	\vec{\gamma}_3(\phi,p)&=\vec{\gamma}_0(\vec{\Psi},p).
    	\end{alignedat}\right.
    	\end{align}
    	where the residues $\vec{\gamma}_0,\vec{\gamma}_1,\vec{\gamma}_2,\vec{\gamma}_3$ are given by \eqref{res00}. Furthermore, if $p_1,\cdots,p_m\in \Sigma^2$ are fixed points and $\vec{\Psi}:\Sigma^2\setminus\ens{p_1,\cdots,p_m}\rightarrow\R^n$ is a complete minimal surface with finite total curvature, then for all $j=1,\cdots,m$
    	\begin{align*}
    		&\vec{\gamma}_0(\vec{\Psi},p_j)=\vec{\gamma}_1(\vec{\Psi},p_j)=\vec{\gamma}_2(\vec{\Psi},p_j)=0,
    	\end{align*}
    	and the fourth residue corresponds to the \emph{flux}, that is
    	\begin{align*}
    		\vec{\gamma}_3(\vec{\Psi},p_j)=-\frac{1}{4\pi}\Im\int_{\gamma}g^{-1}\left(\bar{\partial}|\vec{\Psi}|^2\otimes\h_0-2\s{\vec{\Psi}}{\h_0}\otimes\bar{\partial}\vec{\Psi}\right)=\frac{1}{4\pi}\Im\int_{\gamma}\partial\vec{\Psi},\quad \text{for}\;\, j=1,\cdots,m.
    	\end{align*}
    	In particular, if $\phi:\Sigma^2\rightarrow\R^n$ is the inversion of a complete minimal surface $\vec{\Psi}:\Sigma^2\setminus\ens{p_1,\cdots,p_m}\rightarrow\R^n$ with finite total curvature, for all $j=1,\cdots,m$, we have
    	\begin{align}
    	\left\{\begin{alignedat}{1}
    	    		&\vec{\gamma}_1(\phi,p_j)=\vec{\gamma}_2(\phi,p_j)=\vec{\gamma}_3(\phi,p_j)=0\nonumber\\
                 	&\vec{\gamma}_0(\phi,p_j)=\frac{1}{4\pi}\Im\int_{\gamma}g^{-1}\otimes\left(\bar{\partial}^N-\bar{\partial}^\top\right)\h_0-|\h_0|_{WP}^2\partial\phi=\frac{1}{4\pi}\Im\int_{\gamma}\partial\vec{\Psi}.
    	\end{alignedat}\right.
    	\end{align}
    \end{theorem}
    From this, the $\epsilon$-regularity of the bubble tree convergence in \eqref{quantization} allows one to pass to the limit in the first residue of inverted surfaces, thanks to the strong convergence in $C^l_{\mathrm{loc}}$ for all 
    $l\in\N$ outside of a number finite set of points. As we can further improve the quantization theorem thanks to the classification of branched Willmore spheres, anticipating on the next subsections, we have the following.
    
    \begin{theorem}\label{B}
    	Let $\{\phi_k\}_{k\in\N}$ be a sequence of Willmore immersions of a closed surface $\Sigma^2$ to $\R^n$. Assume that
    	\begin{align*}
    	\limsup_{k\rightarrow\infty}W(\phi_k)<\infty
    	\end{align*}
    	and that the conformal class of $\{\phi_k^\ast g_{\R^n}\}_{k\in\N}$ remains within a compact sub-domain of the moduli space of $\Sigma^2$. Then, modulo extraction of a subsequence, the following energy identity holds
    	\begin{align}\label{quantizationidentity}
    	\lim\limits_{k\rightarrow\infty}W(\phi_k)=W(\phi_{\infty})+\sum_{i=1}^{p}W(\vec{\Psi}_i)+\sum_{j=1}^{q}\left(W(\vec{\xi}_j)-4\pi\theta_j\right)
    	\end{align}
    	where $\phi_\infty:\Sigma^2\rightarrow\R^n$ is a \emph{true} Willmore immersion, and $\vec{\Psi}_i:S^2\rightarrow\R^n$ and $\vec{\xi}_j:S^2\rightarrow\R^n$ are compact \emph{true} Willmore spheres, and the integer $\theta_j=\theta_0(\vec{\xi}_j,p_j)\geq 1$ is the multiplicity of the branched immersion $\vec{\xi}_j$ at some point $p_j\in \vec{\xi}_j(S^2)\subset \R^n$
    \end{theorem}

\begin{remintro}
	More generally, the Willmore spheres arising in more general formulations of the quantizations, such as \cite{quantamoduli}, are also \emph{true} Willmore immersions.
\end{remintro}

    In particular, we deduce the following theorem, interesting in itself.
    
    \begin{theorembis}{B}\label{B'}
    	Let $\{\phi_k\}_{k\in\N}$ be a sequence of smooth Willmore immersions and $\phi_{\infty}:\Sigma^2\rightarrow\R^n$ (where $n=3$ or $n=4$) be a branched Willmore surface such that $\{\phi_k\}_{k\in\N}$ converges weakly in $W^{2,2}$ to $\phi_{\infty}$ as $k\rightarrow\infty$. Then there exists an \emph{integer} $m\in \N$ such that 
    	\begin{align}\label{decrease}
    		W(\phi_{\infty})=\lim\limits_{k\rightarrow\infty}W(\phi_k)-4\pi m.
    	\end{align}
    	Furthermore, we have $m=0$ in \eqref{decrease} if and only if 
    	\begin{align*}
    		\phi_k\conv{k\rightarrow\infty}\phi_{\infty}\quad \text{in}\;\, C^l(\Sigma^2),\quad \text{for all}\;\, l\in \N.
    	\end{align*}
    \end{theorembis}

    \begin{remintro}
    	We may have $m=1$ in $\R^3$. For example, if the limiting branched immersion has a unique branched point of order $\theta_0=3$ (and no other branched), one may glue the non-compact end of multiplicity $3$ of the L\'{o}pez minimal surface and a sphere to its planar end (of multiplicity $1$). Denote by $\vec{\xi}:S^2\setminus\ens{0,\infty}\rightarrow \R^3$ the L\'{o}pez surface, $\phi_{\infty}$ the limiting immersion and $\vec{\Psi}:S^2\rightarrow \R^3$ an immersion of a round sphere. Then we have by the Gauss-Bonnet theorem
    	\begin{align*}
    		&\int_{\Sigma}K_{g_{\phi_{\infty}}}d\mathrm{vol}_{g_{\phi_{\infty}}}=2\pi\chi(\Sigma)+2\pi(3-1)=2\pi\chi(\Sigma)+4\pi\\
    		&\int_{S^2}K_{g_{\vec{\chi}}}d\mathrm{vol}_{g_{\vec{\chi}}}=-8\pi\\
    		&\int_{S^2}K_{g_{\vec{\Psi}}}d\mathrm{vol}_{g_{\vec{\Psi}}}=4\pi,
    	\end{align*}
    	so this possible bubbling is consistent with the quantization of the Gauss curvature, \textit{i.e.}
    	\begin{align*}
    		2\pi\chi(\Sigma)=\int_{\Sigma}K_{g_{\phi_{\infty}}}d\mathrm{vol}_{g_{\phi_{\infty}}}+\int_{S^2}K_{g_{\vec{\chi}}}d\mathrm{vol}_{g_{\vec{\chi}}}+\int_{S^2}K_{g_{\vec{\Psi}}}d\mathrm{vol}_{g_{\vec{\Psi}}}.
    	\end{align*}
    \end{remintro}

    Furthermore, we note here for the convenience of the reader one of the by-products of Theorem \ref{F} and \cite{blow-up}, \cite{blow-up2}, which is interesting in itself as it gives an improved regularity for Willmore surfaces at branch points (see \cite{kusnerpacific} and \cite{beriviere} for the first results in this direction).
    
    \begin{theorem}\label{C}
    	Let $\Sigma^2$ be a closed Riemann surface, $n\geq 3$, $\mathscr{G}_{n-2}(\R^n)$ be the oriented Grassmannian of $(n-2)$-plans in $\R^n$, and $\phi: \Sigma^2\rightarrow\R^n$ be a \emph{variational} branched Willmore surface. Then $\n\in C^{1,1}(\Sigma^2)$, and $\phi\in C^{5,\alpha}(\Sigma^2)$ for all $\alpha<1$.
    \end{theorem}
    
    This theorem also permits to  improve the main result of \cite{eversion}, anticipating on the next section, and referring to \cite{geodesics} for definitions related to admissible families).

    \begin{theorem}\label{D}
    	Let $n\geq 3$ and $\mathscr{A}$ be an admissible
    	 family of $W^{2,4}$ immersions of the sphere $S^2$ into $\R^n$. Assume that
    	\begin{align*}
    		\beta_0=\inf_{A\in \mathscr{A}}\sup W(A)>0.
    	\end{align*}
    	Then there exists finitely many true branched compact Willmore spheres $\phi_1
    	,\cdots,\phi_p:S^2\rightarrow\R^n$, and true branched compact Willmore spheres $\vec{\Psi}_1,\cdots,\vec{\Psi}_q:S^2\rightarrow\R^n$ such that
    	\begin{align}\label{quantization2}
    		\beta_0=&\sum_{i=1}^{p}W(\phi_i)+\sum_{j=1}^{q}\left(W(\vec{\Psi}_j)-4\pi\theta_j\right)\in 4\pi\N,
    	\end{align}
    	 where the integer $\theta_1,\cdots,\theta_q$ correspond respectively to the highest multiplicities of $\vec{\Psi}_1,\cdots,\vec{\Psi}_q$, and the integer $\theta_j=\theta_0(\vec{\Psi}_j,p_j)\geq 1$ is the multiplicity of the branched immersion $\vec{\Psi}_j$ at some point $p_j\in \vec{\Psi}_j(S^2)\subset \R^n$
    \end{theorem}

    \subsection{Bryant's duality theory and the cost of the sphere eversion}

    We briefly describe Bryant's theory of the geometrical aspects of Willmore surfaces in $S^3$. Its most basic ingredient is the introduction of a holomorphic quartic form associated to any Willmore sphere. In particular, in the case of genus $0$ surfaces, this quartic form must vanish thanks to Riemann-Roch theorem, and this information furnishes rich consequences. Indeed, the idea of introducing holomorphic quartic forms originated first in a paper of Calabi (\cite{calabi}) in the context of minimal surfaces in spheres, then in the subsequent work of Chern (\cite{chern}) for the same objects, and of Bryant for conformal immersions into $S^4$ - and so \textit{before} his paper on Willmore surfaces (see \cite{bryant2} for references on this subject) - and is the basis for example of the fairly complete description of minimal two-sphere in $S^n$ for $n\geq 3$ by Calabi.
    
    The other remarkable feature of the theory is the introduction of a pseudo Gauss map with values into a Lorentzian manifold, associated to any surface immersion in $S^3$, which is harmonic if and only if the immersion is a Willmore immersion. A holomorphic quartic form corresponding to any Willmore surface is then defined thanks to this pseudo Gauss map as follows.
    
    Let $h$ be the Lorentzian metric of signature $(1,4)$ on $\R^5$
    \begin{align*}
    	h=-dx_0^2+dx_1^2+dx_2^2+dx_3^2+dx_4^2
    \end{align*}
    and $S^{3,1}$ be the Lorentzian sphere in $(\R^5,h)$, defined by
    \begin{align*}
    	S^{3,1}=\R^5\cap\ens{x:|x|_h^2=-x_0^2+x_1^2+x_2^2+x_3^2+x_4^2=1}.
    \end{align*}
    For all smooth immersion $\phi:\Sigma^2\rightarrow S^3$, if $\n:\Sigma^2\rightarrow S^3$ is the Gauss map of $\phi$, we define a map $\psi_{\phi}:\Sigma^2\rightarrow S^{3,1}$ by
    \begin{align*}
    	\psi_{\phi}=(H,\phi H+\n)
    \end{align*}
    which is called the pseudo Gauss map of $\phi$. The first step in Bryant's theory is the following.
    
    \begin{theorem*}[Bryant, \cite{bryant}]
    	Let $\Sigma^2$ be a closed Riemann surface and $\phi:\Sigma^2\rightarrow S^3$ be a smooth immersion. Then the pseudo Gauss map $\psi_{\phi}:\Sigma^2\rightarrow S^3$ is weakly conformal, is an immersion outside of the umbilic locus of $\phi$, and if $\phi$ is a Willmore immersion, then the quartic form
    	\begin{align*}
    		\mathscr{Q}_{\phi}=\s{\partial^2\psi_{\phi}}{\partial^2\psi_{\phi}}_h
    	\end{align*}
    	is holomorphic. Furthermore, $\phi:\Sigma^2\rightarrow S^3$ is a Willmore surface if and only if $\psi_{\phi}:\Sigma^2\rightarrow S^{3,1}$ is harmonic with values into the Lorentzian manifold $(S^{3,1},h)$.
    \end{theorem*}

   To describe the second ingredient of the theory, we need to make some additional definitions on the umbilic locus and on the Willmore adjoint.
      
   Let $\phi:\Sigma^2\rightarrow S^3$ be a smooth immersion. The umbilic locus of $\phi$ is equal to the subset of $\Sigma^2$ where the two principal curvatures coincide. If $\phi$ is completely umbilic, then $\Sigma^2=S^2$ and $\phi$ is a diffeomorphism, so we assume that $\phi$ is not completely umbilic, and we note $\mathscr{U}_{\phi}$ the closed set
    \begin{align}\label{umbiliclocus}
    	\mathscr{U}_{\phi}=\Sigma^2\cap\ens{z:|\h_0(z)|^2_{WP}\,d\vg(z)=0}.
    \end{align}
    Then it is possible to define a Willmore adjoint of any Willmore surface $\phi:\Sigma^2\rightarrow S^3$, that is a branched immersion $\vec{\Psi}:\Sigma^2\setminus\mathscr{U}_{\phi}\rightarrow S^3$ such that for all 
    $z_0\in \Sigma^2$, the point $p=\vec{\Psi}(z_0)\in S^3$ is the unique element in $S^3$ such that after a stereographic projection based on $p$, the mean curvature vanishes at order two at $z_0$; \textit{i.e.} if $\pi_{p}:S^3\setminus\ens{p}\rightarrow \R^3$ is the stereographic projection, then
    \begin{align}\label{adjoint}
    	\H_{\pi_{p(z_0)}\circ \phi}(z)=O(|z-z_0|^2).
    \end{align}
    One of the main results of Bryant's paper is the following.
    
    \begin{theorem*}[Bryant, \cite{bryant}]
    	Let $\phi:\Sigma^2\rightarrow S^3$ be a Willmore surface. Then the set of umbilic points $\mathscr{U}_{\phi}$ is equal to $\Sigma^2$ or is a closed set with empty interior. In the first alternative, $\Sigma^2=S^2$ and $\phi$ is a diffeomorphism. In the second alternative, there exists an immersion $\vec{\Psi}:\Sigma^2\setminus\mathscr{U}_{\phi}\rightarrow S^3$ satisfying \eqref{adjoint}, and a holomorphic quartic differential 
    	\begin{align*}
          	\mathscr{Q}_{\phi}=\s{\partial^2\psi_{\phi}}{\partial^2\psi_{\phi}}_h\in H^0(\Sigma^2,K_{\Sigma^2}^4)
    	\end{align*}
    	with the following property : if $\mathscr{Q}_{\phi}=0$, then $\vec{\Psi}$ is constant. Whenever $\vec{\Psi}=p\in S^3$ is constant, the set $\phi^{-1}(\ens{p})$ is a non-empty discrete set in $\Sigma^2$, the stereographic projection
    	\begin{align*}
    		\pi :S^3\setminus\ens{p}\rightarrow\R^3
    	\end{align*}
    	makes the mean curvature of $\pi\circ\phi$ vanish identically, and the Willmore surface
    	\begin{align*}
    	\pi\circ \phi:\Sigma^2\setminus\phi^{-1}(\ens{p})\rightarrow\R^3 
    	\end{align*}
    	is a complete minimal surface with finite total curvature and embedded planar ends. In particular, if $\Sigma^2=S^2$, then $K_{S^2}^4$ is a negative holomorphic line bundle, so $\mathscr{Q}_{\phi}=0$, and every non-completely umbilic Willmore sphere in $S^3$ is the inverse stereographic projection of a complete minimal surface in $\R^3$ with embedded planar ends.
    \end{theorem*}
    
    \begin{defi2}
    Whenever a compact Willmore surface in $S^3$ is the inverse stereographic projection of a complete minimal surface in $\R^3$ with finite total curvature, we call this underlying object the dual minimal surface.
    \end{defi2}

    By a result of Kusner (\cite{kusnerpacific}), the dual minimal surface is obtained by inverting the compact branched immersion at a point of highest multiplicity (this is also a direct consequence of a finer version of Li-Yau inequality \cite{lieyau}).
       
    The first ingredient of the generalisation of Bryant's theorem is the special algebraic structure of Bryant's quartic form, which did not appear in the previous literature on the subject.
    
    \begin{theorem}\label{E}
    	Let $\Sigma^2$ be a closed Riemann surface, and $\phi:\Sigma^2\rightarrow S^3$ be a smooth immersion. Then Bryant's quartic admits the following representation
    	\begin{align}\label{algebra}
    		\mathscr{Q}_{\phi}&=\s{\partial^2\psi_{\phi}}{\partial^2\psi_{\phi}}_h=
    		g^{-1}\otimes\left(\partial^N\bar{\partial}^N\h_0\totimes\h_0-\partial^N\h_0\totimes\bar{\partial}^N\h_0\right)+\frac{1}{4}\left(1+|\H|^2\right)\h_0\totimes\h_0\nonumber\\
    		&=g^{-1}\otimes\left(\partial\bar{\partial}\h_0\totimes\h_0-\partial\h_0\totimes\bar{\partial}\h_0\right)+\left(\frac{1}{4}\left(1+|\H|^2\right)+|\h_0|^2_{WP}\right)\h_0\totimes\h_0.
    	\end{align}
    \end{theorem}

	The second main result of this paper is a generalisation of Bryant's theorem, based on the algebraic structure of the quartic form and a refined analysis of its singularities at branch points, which prove to be removable under natural assumptions. We first have the following.
	
	\begin{theorem}\label{F}
		Let $\Sigma^2$ be a closed Riemann surface, and $\phi:\Sigma^2\rightarrow\R^3$ be a non-completely umbilic branched Willmore surface. Then $\phi$ is conformally minimal in $\R^3$ if and only if $\mathscr{Q}_{\phi}=0$.
	\end{theorem}

    This theorem can be deduced quite easily from the Weiertrass parametrisation and the observation that the quadratic form $Q$ defined on quadratic differentials by
    \begin{align*}
    	Q(\alpha)=\partial\bar{\partial}\alpha\otimes\alpha-\partial\alpha\otimes\bar{\partial}\alpha=\alpha^2\otimes \partial\bar{\partial}\log(\alpha).
    \end{align*}
    vanishes if $\alpha=f_1(z)\bar{f_2(z)}dz^2$, and $f_1$ and $f_2$ are holomorphic. Here, the last equality is formal but shows the special structure of $\partial\bar{\partial}$ of a logarithm.
    
    The following theorem extends a preceding one of Lamm and Nguyen in the case of Willmore spheres whose sum of multiplicities of branch points is less than three \cite{lamm}. Motivated by the generalisation in higher codimension, we remark      that the quartic form $\mathscr{Q}_{\phi}$ is a well-defined tensor for any immersion, but need not be meromorphic when $\phi$ is Willmore in $\R^n$ and $n\geq 4$. It is a particular case of the more general Theorem \ref{devh0}.

	\begin{theorem}[Meromorphy implies holomorphy]\label{G} 
			Let $\Sigma^2$ be a closed Riemann surface, $n\geq 3$ and $\phi:\Sigma^2\rightarrow S^n$ be a \emph{variational} branched Willmore surface. Furthermore, suppose that the quartic differential $\mathscr{Q}_{\phi}$ is meromorphic. Then
			\begin{align}\label{bounded}
				\mathscr{Q}_{\phi}\;\,\text{is holomorphic}.
			\end{align}
			 In particular, \emph{variational} branched Willmore spheres $\phi:S^2\rightarrow S^3$ are inverse stereographic projections of complete \emph{branched} minimal surfaces in $\R^3$ with finite total curvature and \emph{vanishing flux}.
	\end{theorem}

   We remark that the assertion on the flux finally justifies the last sentence of \cite{bryant}.
   
   \begin{remintro}
   	We stress out that the dual minimal surface might have interior branch points : the famous example of the \emph{hérissons} (hedgehogs in English) of Rosenberg and Toubiana (\cite{herisson}) shows that there even exist \emph{true} Willmore spheres whose dual minimal surface have interior branch points (and can even have total curvature equal to $-4\pi$).
   	An explicit example is given by the two-sheeted  covering of the Henneberg's surface, a non-orientable minmal surface with total curvature $-2\pi$ which is conformally equivalent to a once-punctured real projective plan $\R \mathbb{P}^2$. Its inversion is a \emph{true} Willmore sphere of energy $24\pi$. 
   \end{remintro}
   
   We can summarize the analogies between the theories of minimal and Willmore surfaces in the following table.
   
   {\tabulinesep=0.6mm
   	\begin{figure}[H]
   		\begin{center}
   			\begin{tabu}{|c|c|c|}
   				\hline
   				& Minimal surfaces in $\displaystyle\R^3$ & Willmore surfaces in $\displaystyle S^3$ \\
   				\hline
   				Conformal immersion & $\displaystyle\displaystyle\phi:\Sigma^2\rightarrow \R^3$ & $\displaystyle\phi:\Sigma^2\rightarrow S^3$\\
   				\hline
   				Euler-Lagrange equation & $\displaystyle 2\H=\Delta_g\phi=0$ &$\displaystyle\Delta_gH+2H(H^2-K+1)=0$\\
   				\hline
   				Harmonic Gauss map & $\displaystyle\n:\Sigma^2\rightarrow S^2\subset\R^3$ & $\displaystyle\psi_{\phi}:\Sigma^2\rightarrow S^{3,1}\subset \R^{4,1}$ \\
   				\hline
   				\makecell{Holomorphic quadratic and \\quartic differentials} & \makecell{Weingarten tensor\\$h_0=\s{2\,\partial^N\partial\phi}{\n}$} & \makecell{$\displaystyle\mathscr{Q}_{\phi}=g^{-1}\otimes\left(\partial{\bar{\partial}}h_0\totimes h_0-\partial h_0\totimes\bar{\partial}h_0\right)$\\ $\displaystyle\mkern-9mu+\frac{1}{4}\left(1+H^2\right)h_0\totimes h_0$}\\
   				\hline
   			\end{tabu}
   		\end{center}
   		\caption{Comparison between Willmore and minimal surfaces.}\label{comparison}
   	\end{figure}

    Combining Theorem \ref{G} with the improvement of the main result of \cite{eversion} of Theorem \ref{D}, we obtain the following information on the so-called min-max sphere eversion.
    \begin{theorem}\label{H0}
    	Let $S_+^2\subset\R^3$ (resp $S^2_-\subset \R^3$) be the standard Euclidean sphere with positive (resp. negative) orientation, $\mathrm{Imm}(S^2,\R^3)$ be the space of smooth immersions from $S^2$ to $\R^3$ and denote by $\Omega$ the set of paths between the two spheres, defined by
    	\begin{align*}
    	\Omega=C^0\left([0,1],\mathrm{Imm}(S^2,\R^3)\right)\cap\ens{\{\phi_t\}_{t\in [0,1]}, \phi_0(S^2)=S^2_+, \phi_1(S^2)=S^2_-}.
    	\end{align*}
    	If the cost of the sphere eversion is defined by
    	\begin{align*}
    	\beta_0=\min_{\phi\in \Omega}\max_{t\in [0,1]}W(\phi_t),
    	\end{align*}
    	then there exists finitely many \emph{true} branched Willmore spheres $\phi_1,\cdots,\phi_p,\vec{\Psi}_1,\cdots,\vec{\Psi}_q:S^2\rightarrow \R^3$ such that
    	\begin{align}\label{cost}
    		\beta_0=\sum_{i=1}^{p}W(\phi_i)+\sum_{j=1}^q\left(W(\vec{\Psi}_j)-4\pi\theta_j\right)\in 4\pi\N,
    	\end{align}
    	where the integer $\theta_j=\theta_0(\vec{\Psi}_j,p_j)\geq 1$ is the multiplicity of the branched immersion $\vec{\Psi}_j$ at some point $p_j\in \vec{\Psi}_j(S^2)\subset \R^n$
    \end{theorem}
    We also know thanks to a topological result of Max and Banchoff that every sphere eversion has a quadruple point (\cite{max}, actually, this reference is not easily accessible, and for an alternative proof, see \cite{another}), so by Li-Yau inequality, we deduce that $\beta_0\geq 16\pi$. It would be interesting to determine if this value is effectively attained, and by which surfaces.\
    
    \textbf{The Willmore energy as a quasi-Morse function.}
    
    More generally, one can consider the following problem.
    Let $k>0$ and $\Gamma\in \pi_k(\mbox{Imm}(S^2,{\mathbb R}^3))$ a non-zero element (provided $\pi_k(\mathrm{Imm}(S^2,\R^3))$ is not trivial). Thanks to Theorem \ref{D} we have
    \[
    \beta_\Gamma=\inf_{\vec{\Phi}(t,\cdot)\simeq\Gamma}\ \max_{t\in S^k}\ W(\vec{\Phi}(t,\cdot))\in 4\pi \, {\mathbb N}^\ast
    \]
    and this furnishes a map
    \begin{align}\label{map}
    \Gamma\in  \pi_k(\mbox{Imm}(S^2,{\mathbb R}^3))\ \longrightarrow \ \frac{\beta_\Gamma}{4\pi}\in {\mathbb N}^\ast.
    \end{align}
    It would be interesting to study the map \eqref{map} giving the Willmore energy of the optimal representative of a non-zero element of these groups.
    
    \subsection{Willmore immersions into $S^4$}
    
    The generalisation of Bryant's theorem relies on the specific algebraic structure of the quartic form and on the four-term asymptotics at branch points of the immersion of the Weingarten tensor the quartic form is a function of. The classification of Willmore spheres in $S^4$ of Montiel (see \cite{montiels4}) also relies on the holomorphy of certain $3$, $4$, and $8$-differentials (here $\partial$ and $\bar{\partial}$ are the normal operators $\partial^N$ and $\bar{\partial}^N$ as in \eqref{defnormal}). 
    
    \begin{theorem}\label{I}
    Let $\Sigma^2$ a closed Riemann surface and $\phi:\Sigma^2\rightarrow S^4$ be a smooth immersion, and $\partial$ and $\bar{\partial}$ the complex operators acting of the normal bundle induced by the immersion $\phi$. Then Montiel's forms of degree $3$, $4$ and $8$ have the following expressions
    \begin{align*}
    \left\{\begin{alignedat}{1}
    &\mathscr{T}_{\phi}=g^{-1}\otimes (\bar{\partial}\h_0\totimes J\h_0)\\
    &\mathscr{Q}_{\phi}=\,g^{-1}\otimes \left(\partial\bar{\partial}\h_0\,\dot{\otimes}\, \h_0-\partial\h_0\otimes\bar{\partial}\h_0
    \right)+\frac{1}{4}(1+|\H|^2)\h_0\,\dot{\otimes}\, \h_0\\
    &\mathscr{O}_{\phi}=g^{-2}\otimes\bigg\{\frac{1}{4}(\partial\bar{\partial}\h_0\,\dot{\otimes}\,\partial\bar{\partial}\h_0)\otimes (\h_0\,\dot{\otimes}\,\h_0)+\frac{1}{4}(\partial\h_0\,\dot{\otimes}\, \partial\h_0)\otimes (\bar{\partial}\h_0\,\dot{\otimes}\, \bar{\partial}\h_0)\\
    &-\frac{1}{2}(\partial\bar{\partial}\h_0\totimes\partial\h_0)\otimes(\bar{\partial}\h_0\totimes\h_0)-\frac{1}{2}(\partial\bar{\partial}\h_0\totimes\bar{\partial}\h_0)\otimes(\partial\h_0\totimes\h_0)+\frac{1}{2}(\partial\bar{\partial}\h_0\totimes\h_0)\otimes(\partial\h_0\totimes\bar{\partial}\h_0)\bigg\}\\
    &+\frac{1}{4}(1+|\H|^2)\,g^{-1}\otimes\left\{\frac{1}{2}(\partial\bar{\partial}\h_0\totimes \h_0)\otimes (\h_0\totimes\h_0)-(\partial\h_0\totimes\h_0)\otimes (\bar{\partial}\h_0\totimes \h_0)+\frac{1}{2}(\partial\h_0\totimes\bar{\partial}\h_0)\otimes (\h_0\totimes\h_0)\right\}\\
    &+\frac{1}{64}(1+|\H|^2)^2\,(\h_0\totimes\h_0)\otimes (\h_0\totimes\h_0),
    \end{alignedat}\right.
    \end{align*}
    where $J$ is the natural almost complex structure on the holomorphic normal bundle. Furthermore, if $\phi$ is a Willmore surface then $\mathscr{T}_{\phi}$ is holomorphic, and if $\mathscr{T}_{\phi}=0$, then $\mathscr{Q}_{\phi}$ and $\mathscr{O}_{\phi}$ are holomorphic. 
    \end{theorem}

    As the analysis of the singularities of the quartic form $\mathscr{Q}_{\phi}$ in Theorem \ref{F} does not depend on codimension, 
    we can prove the following. See Theorem \ref{s4vanish} for a more refined hypothesis. 

    \begin{theorem}\label{J}
    	Let $\Sigma^2$ be a closed Riemann surface and $\phi:\Sigma^2\rightarrow S^4$ be a \emph{variational} Willmore immersion. Then $\mathscr{T}_{\phi}$ is holomorphic, and if $\mathscr{T}_{\phi}=0$, then the meromorphic $4$ and $8$-forms $\mathscr{Q}_{\phi}$ and $\mathscr{O}_{\phi}$ are holomorphic. If 
    	$\mathscr{T}_{\phi}=\mathscr{Q}_{\phi}=\mathscr{O}_{\phi}=0$, the pseudo Gauss map $\mathscr{G}:\Sigma^2\rightarrow \mathbb{C}\mathbb{P}^{4,1}$  of $\phi$ is either holomorphic or anti-holomorphic, or lies lying in a null totally geodesic hypersurface of the null quadric $Q^{3,1}$ of the $5$-dimensional indefinite complex projective plan $\mathbb{C}\mathbb{P}^{4,1}$. In the first case, $\phi$ is the image by the Penrose twistor fibration of a (singular) algebraic curve $C\subset \mathbb{C}\mathbb{P}^3$, and in the other case, $\phi$ is the inverse stereographic projection of a complete (branched) minimal surface with finite total curvature in $\R^4$ and zero flux. Furthermore, the two possibilities coincide if and only if the algebraic curve $C\subset \mathbb{C}\mathbb{P}^3$ lies in some hypersurface $H\simeq  \mathbb{C}\mathbb{P}^2\subset\mathbb{C}\mathbb{P}^3$. In particular, the hypothesis are always satisfied for a \emph{variational} Willmore sphere.
    \end{theorem}

    We remark that we cannot rule out the existence of interior branch points of the dual minimal surface in $\R^4$ when it exists. 

    In particular, the Willmore energy of \emph{true} branched Willmore spheres in $S^4$ is quantized by $4\pi$, and the integer multiple depends only on the degree of the dual algebraic curve or some topological invariants of the dual minimal surface.
    
    Finally, we note as the expansion of $\h_0$ at branch points of Theorem \ref{F} is valid in any codimension, and as we can express any holomorphic form constructed on a Willmore surface only with respect to $\h_0$ for possibly singular terms which enjoy nice compensations as in \eqref{algebra}, this strongly suggests that any generalisation of Bryant and Montiel's classification of Willmore surfaces in $S^n$ for $n\geq 5$ and smooth unbranched immersions shall generalise immediately to branched immersions.
    
    \begin{remintro}
    	After the first version of this work, we saw that there was a version under press of a generalisation to $S^5$ of Bryant's classification (\cite{willmores5}). As there are also papers under review in the cases $S^n$ (with $n\geq 6$), and for obvious size limitation, we will not discuss these cases.
    \end{remintro}

    \renewcommand{\thetheorem}{\thesection.\arabic{theorem}}
   
   \section{Outline of the proofs of the main results}
   
   \textbf{\ref{A}.} The proof of Theorem \ref{A} is given in Section \ref{4}, and is composed of the reunion of the corollary \ref{residues1234} of Noether's Theorem \ref{noether2} for the definition of the four residues, of Theorem \ref{galois} for the correspondence, and finally of Proposition \ref{fluxresidue} and corollary \ref{inversionmin} for the link with minimal surfaces.
   
   \textbf{\ref{B}.} By the main theorem of \cite{quanta}, there exists finitely many points $\ens{p_1,\cdots,p_m}\subset \Sigma^2$, and a sequence of diffeomorphisms $\ens{f_k}_{k\in\N}\subset \mathrm{Diff}_+(\Sigma^2)$, a sequence of conformal transformations $\ens{F_k}_{k\in\N}$ of $\R^n$, such that $F_k\circ \phi_k\circ f_k$ is conformal for all $k\in \N$, and 
   \begin{align}\label{composition}
   	F_k\circ\phi_k\circ f_k\conv{k\rightarrow\infty}\phi_{\infty},\quad \text{in}\;\, C^l_{\mathrm{loc}}(\Sigma^2\setminus\ens{p_1,\cdots,p_m})\quad \text{for all}\;\, l\in \N,
   \end{align}
   and for all $i=1,\cdots,p$, and all $j=1,\cdots,q$, up to rescaling, the limiting Willmore spheres $\vec{\Psi}_i$ and $\vec{\xi}_j$ are obtained by taking compositions of diffeomorphism in the domain and conformal transformations in the target, the strong convergence in $C^l_{\mathrm{loc}}$ minus finitely many points allows one to pass to the limit in the first and fourth residue in \eqref{res00} of these compositions to deduce that $\vec{\Psi}_i$ and $\vec{\xi}_j$ have vanishing first residue by Theorem \ref{A}.
   
   \textbf{\ref{C}.} It follows from proposition \ref{parenthesis} and \cite{blow-up}, \cite{blow-up2}. 
   
   \textbf{\ref{D}.} This theorem is an easy consequence of lemma VI.2 of \cite{eversion} when we replace the first by the fourth residue.
   
   \textbf{\ref{E}.} This is Theorem \ref{int}.
   
   \textbf{\ref{F}.} This is Theorem \ref{devinv}.
   
   \textbf{\ref{G}.} This is Theorem \ref{devh0}.
   
   \textbf{\ref{H0}.} This is a trivial corollary of the main result of \cite{eversion} combined with Theorem \ref{G}.
   
   \textbf{\ref{I}.} This is Theorem \ref{threedev}
   
   \textbf{\ref{J}.} This is Theorem \ref{s4generalfinal}.

    		\section{Noether's theorem, residues and conformal invariance}\label{4}

   In the sequel we always assume that the ambient dimension $n$ satisfies the inequality $n\geq 3$.
  
   Let $\Sigma^2$ be a compact Riemann surface, $K_{\Sigma^2}$ be its canonical line bundle, and $\phi:\Sigma^2\rightarrow S^n$ be a $C^{1,\alpha}$ (for some $\alpha<1$) conformal immersion (as this is the minimal regularity for Willmore surfaces, this assumption is not restrictive), and $g$ be the induced metric on $\Sigma^2$ by $\phi$ of the Euclidean metric on $S^n$. Let us write $\D$ the Levi-Civita connection on the pull-back bundle $\phi^\ast TS^n$, and the splitting
    \begin{align*}
    \D=\D^\top+\D^N=\bar{\D}+\D^N
    \end{align*}
    where $\bar{\D}=\D^\top$ and $\D^N$ are the tangent and normal parts respectively. In particular, for all tangent vectors $X,Y,Z$, one has
    \begin{align*}
    \s{\bar{\D}_{Z}X}{Y}=\s{\D_ZX}{Y}.
    \end{align*}
    Consider on $S^n$ the complexified tangent bundle 
    \begin{align*}
    T_{\C}S^n=TS^n\otimes_{\R}\C
    \end{align*}
    and the following splitting of the complex pull-back bundle $\phi^\ast T_{\C}S^n$
    \begin{align*}
    \phi^\ast T_{\C}S^n=T_{\C}\Sigma^2\oplus T_{\C}^\nperp \Sigma^2.
    \end{align*}
    We still write $\D=\bar{\D}+\D^N$ the extension by linearity of the Levi-Civita connection $\D$ on $\phi^\ast T_{\C}S^n$. There exists a unique complex structure on $T_{\C}^N\Sigma^2$, see \cite{ejiri}. Let us see how to define it in low dimensions.
    
     If $n=3$, then the unit normal $\n:S^2\rightarrow S^3$ of $\phi$ furnishes a global non-vanishing section of $T_{\C}^\nperp \Sigma^2$. Therefore, if $s\in \Gamma(T_{\C}^\nperp \Sigma^2)$ is a $C^1$ section, there exists $f\in C^1(\Sigma^2,\C)$ such that $s=f\n$, and this never vanishing section  of $T_{\C}^N\Sigma^2$ furnishes a unique structure of holomorphic line bundle on $T_{\C}^\nperp\Sigma^2$, which makes it a trivial bundle. We can also give a more abstract proof : if $J$ the almost complex structure defined by
    \begin{align*}
    	J\n=i\n.
    \end{align*}
    Then $\D^\nperp J=0$, as
    \begin{align*}
    	\D_{\p{z}}^\nperp J(\n)=\D_{\p{z}}^\nperp (J\n)-J(\D_{\p{z}}^\nperp \n)=i\D_{\p{z}}^\nperp \n=0
    \end{align*}
    as $\D_{\p{z}}^\nperp\n=0$. As $\n$ is real, we also readily have $\D_{\p{\z}}^\nperp J=0$. Therefore, this almost complex structure is integrable, and by the Newlander-Nirenberg theorem (which we can apply as the normal bundle is $C^{1,\alpha}$, see \cite{newlander}, or chapter V of \cite{hormandercomplex}) there exists a unique complex structure on the normal bundle $T_{\C}^N\Sigma^2$ such that $T_\C^\nperp \Sigma^2\rightarrow\Sigma^2$ is a holomorphic line bundle, which is in particular the same as the one previously defined.
    
    If $n=2$, and $\n_1,\n_2$ is a local orthonormal base of $T_{\C}^\nperp\Sigma^2$, we define an almost complex structure $J$ by $J\n_1=-n_2$. Then we compute by definition of $\D$
    \begin{align*}
    	\D_{\p{z}}^\nperp J(\n_1)&=\D_{\p{z}}^\nperp(J\n_1)-J(\D_{\p{z}}^\nperp \n_1)\\
    	&=-\D_{\p{z}}^\nperp\n_2-J\left(\s{\D_{\p{z}}^\nperp\n_1}{\n_1}\n_1+\s{\D_{\p{z}}^\nperp\n_1}{\n_2}\n_2\right)\\
    	&=-\D_{\p{z}}^N\n_2-\s{\D_{\p{z}}^\nperp\n_1}{\n_2}J(\n_2)\\
    	&=-\left(\s{\D_{\p{z}}^N\n_2}{\n_1}+\s{\D_{\p{z}}^\nperp\n_1}{\n_2}\right)\n_1\\
    	&=-\p{z}(\s{\n_1}{\n_2})\,\n_1=0
    \end{align*}
    so we also have $\D^\nperp J=0$, and the previous argument applies. In dimension $4$, we note that one could directly define a complex structure on the real normal bundle (see \cite{montiels4}) ; however, this is not true in general and in the codimension $1$ case in particular. 
    \begin{prop}
    	Let $\phi:\Sigma^2\rightarrow S^n$ be a weak Willmore immersion. Then the complexified pull-back bundle $\phi^\ast T_{\C}S^n$ splits into tangent and normal parts as
    	\begin{align*}
    		\phi^\ast T_{\C}S^n=T_{\C}\Sigma^2\oplus T_{\C}^N\Sigma^2,
    	\end{align*}
    	and there exists a unique complex structure on $T_{\C}^N\Sigma^2$ which is compatible with the decomposition $\D=\D^\top+\D^\nperp$ of the Levi-Civita connection induced by $g=\phi^\ast g_{S^n}$ on $\phi^\ast TS^n$ and makes it a holomorphic line bundle.
    \end{prop}
    
    We finally introduce the operators
    \begin{align}\label{delbar}
    \partial^N=\D_{\partial_z}^\nperp\left(\,\cdot\,\right)\otimes dz\quad \bar{\partial}^N=\D_{\partial_{\z}}^\nperp\left(\,\cdot\,\right)\otimes d\z
    \end{align}
    acting on $\Gamma(T^\nperp_{\C}\Sigma^2)$, the space of sections of the complexified normal bundle. In particular, a section $s\in\Gamma(T^\nperp_{\C}\Sigma)$ is holomorphic is and only if $\bar{\partial}^Ns=0$. If we write $\mathscr{L}=T_{\C}^\nperp \Sigma^2$, then with the notations of  the appendix \ref{line}, we have $\bar{\partial}^N=\bar{\partial}_{\mathscr{L}}$. Furthermore, we have a decomposition
    \begin{align*}
    &\partial=\partial^N+\partial^\top,\quad \bar{\partial}=\bar{\partial}^{N}+\bar{\partial}^\top
    \end{align*}
    acting on sections of the total bundle $\phi^\ast T_{\C}S^n$. If $p,q\geq 1$ are fixed integers, and $g$ a smooth metric on $\Sigma^2$ let
    \begin{align*}
    	g^{-1}:\Gamma(K_{\Sigma^2}^p\otimes \bar{K}_{\Sigma^2}^q)\rightarrow \Gamma(K_{\Sigma^2}^{p-1}\otimes \bar{K}_{\Sigma^2}^{q-1})
    \end{align*}
    defined in the space of \emph{continuous} sections of $K_{\Sigma^2}^p\otimes \bar{K}_{\Sigma^2}^{q-1}$ as follows : in a local complex chart $z$, write $g=e^{2\lambda}dz\otimes d\z$, for some smooth positive function $e^{2\lambda}$. Then for all continuous section $\xi$, there exists locally, there exists $f$ such that
    \begin{align*}
    \xi=f(z)dz^p\otimes d\z^{\,q}
    \end{align*}
    and
    \begin{align*}
    	g^{-1}\otimes \xi=e^{-2\lambda(z)}f(z)dz^{p-1}\otimes d\z^{\,q-1}.
    \end{align*}
    This is easy to see that such definition defines a section of $K_{\Sigma^2}^{p-1}\otimes \bar{K}_{\Sigma^2}^{\,q-1}$.
    
    We also remark that one could even remove the positively condition on $p$ and $q$, as negative sections also occur naturally; for example with the Beltrami differentials of in the definition of the Weil-Petersson metric (see \cite{tromba}, \cite{weil}), as for any quadratic differential $\alpha\in \Gamma(K_{\Sigma^2}^2)$ given locally by $\alpha=f(z)dz^2$, if $g=e^{2\lambda}dz\otimes d\z$ is any smooth metric, we have
    \begin{align*}
    |\alpha|_{WP}^2=g^{-2}\otimes \alpha\otimes \bar{\alpha}=e^{-4\lambda}|f(z)|^2,
    \end{align*}
    and in the following, the reference to the metric $g$ in the Weil-Petersson norm will always be implicit.
    
    Let us come back to a slightly more general context, where we consider immersions $\phi:\Sigma^2\rightarrow (M^n,h)$, where $(M^n,h)$ is a $C^3$ Riemannian manifold of constant sectional curvature. In a local coordinates system $(x_1,x_2)$, we introduce the complex coordinate $z=x_1+ix_2$ and notations
    \begin{align*}
    &\partial_z=\frac{1}{2}(\d{1}-i\d{2}),\quad \quad
    \partial_{\bar{z}}=\frac{1}{2}(\d{1}+i\d{2}),\qquad
    \e_z=\partial_{z}\phi\quad\quad \e_{\bar{z}}=\partial_{\bar{z}}\phi
    \end{align*}
    We note $\s{\,\cdot\,}{\,\cdot\,}$ the metric $h$, $\D$ its Levi-Civita connexion, and we decompose $\D$ as
    \begin{align*}
    \D=\D^\top+\D^\nperp=\bar{\D}+\D^\nperp
    \end{align*}
    where $\bar{\D}=\D^\top$ and $\D^\nperp$ are the tangent and normal parts respectively. In particular, for all tangent vectors $X,Y,Z$, one has
    \begin{align*}
    \s{\bar{\D}_{Z}X}{Y}=\s{\D_ZX}{Y}.
    \end{align*}    
    Then by conformality of $\phi$, one has
    \begin{align}\label{orthogonality}
    \bar{\s{\e_{\bar{z}}}{\e_{\bar{z}}}}=\s{\e_z}{\e_z}&=\frac{1}{4}\left(|\d{1}\phi|^2-|\d{2}\phi|^2-2i\s{\d{1}\phi}{\d{2}\phi}\right)=0\nonumber\\
    \s{\e_z}{\e_{\bar{z}}}&=\frac{1}{4}\left(|\d{1}\phi|^2+|\d{2}\phi|^2\right)=\frac{e^{2\lambda}}{2}.
    \end{align}
    Therefore, we have
    \begin{align}\label{H}
    \vec{H}_0&=\frac{e^{-2\lambda}}{2}\left(\vec{\I}(\e_1,\e_1)-\vec{\I}(\e_2,\e_2)-2i\,\vec{\I}(\e_1,\e_2)\right)=2\,e^{-2\lambda}\vec{\I}(\e_z,\e_z)\nonumber\\
    \vec{H}&=\frac{e^{-2\lambda}}{2}\left(\vec{\I}(\e_1,\e_1)+\vec{\I}(\e_2,\e_2)\right)=2\,e^{-2\lambda}\vec{\I}(\e_z,\e_{\bar{z}})
    \end{align}
    Furthermore, as $\D_{\partial_z}\e_{\z}=\D_{\partial_{\z}}\e_z=4^{-1}\Delta\phi$ has no tangential component, by \eqref{orthogonality} we deduce the additional following properties
    \begin{align}\label{conforme}
    \s{\D_{\partial_z}\e_z}{\e_z}&=\s{\D_{\partial_z}\e_{\z}}{\e_{\bar{z}}}=0\nonumber\\
    \s{\D_{\partial_z}\e_z}{\e_{\z}}&=\frac{1}{2}\partial_z (e^{2\lambda}),\qquad
    \s{\D_{\partial_z}\e_{\z}}{\e_z}=\frac{1}{2}\partial_{\z} (e^{2\lambda})
    \end{align}
    If $W$ is defined by
    \begin{align*}
    W(\phi)=\int_{\Sigma^2}(|\H_g|^2-K_g+K_h)\,d\vg
    \end{align*}
    we have if $(\vec{\epsilon}_1,\vec{\epsilon}_2)=(e^{-\lambda}\e_1,e^{-\lambda}\e_2)$ is an orthonormal frame, the mean and Gauss curvature are respectively defined by
    \begin{align*}
    \vec{H}_g=\frac{1}{2}\left(\vec{\I}(\vec{\epsilon}_1,\vec{\epsilon}_1)+\vec{\I}(\vec{\epsilon}_2,\vec{\epsilon}_2)\right)\quad K_g=K_h+\s{\vec{\I}(\vec{\epsilon}_1,\vec{\epsilon}_1)}{\vec{\I}(\vec{\epsilon}_2,\vec{\epsilon}_2)}-|\vec{\I}(\vec{\epsilon}_1,\vec{\epsilon}_2)|^2.
    \end{align*}
    Therefore, we have
    \begin{align*}
    |\vec{H}_g|^2-K_g=\frac{1}{4}|\vec{\I}(\vec{\epsilon}_1,\vec{\epsilon}_1)-\vec{\I}(\vec{\epsilon}_2,\vec{\epsilon}_2)|^2+|\vec{\I}(\vec{\epsilon}_1,\vec{\epsilon}_2)|^2.
    \end{align*}
    Recall that the Weingarten operator is defined by
    \begin{align*}
    \vec{H}_0=\frac{1}{2}\left(\vec{\I}(\vec{\epsilon}_1,\vec{\epsilon}_1)-\vec{\I}(\vec{\epsilon}_2,\vec{\epsilon}_2)-2i\,\vec{\I}(\vec{\epsilon}_1,\vec{\epsilon}_2)\right).
    \end{align*}
    This implies that
    \begin{align*}
    |\vec{H}_0|_g^2&=\frac{1}{4}\s{\vec{\I}(\vec{\epsilon}_1,\vec{\epsilon}_1)-\vec{\I}(\vec{\epsilon}_2,\vec{\epsilon}_2)-2i\vec{\I}(\vec{\epsilon}_1,\vec{\epsilon}_2)}{\vec{\I}(\vec{\epsilon}_1,\vec{\epsilon}_1)-\vec{\I}(\vec{\epsilon}_2,\vec{\epsilon}_2)+2i\vec{\I}(\vec{\epsilon}_1,\vec{\epsilon}_2)}\\
    &=\frac{1}{4}\left|\vec{\I}(\vec{\epsilon}_1,\vec{\epsilon}_1)-\vec{\I}(\vec{\epsilon}_2,\vec{\epsilon}_2)\right|+|\vec{\I}(\vec{\epsilon}_1,\vec{\epsilon}_2)|^2
    =|\H_g|^2-K_g+K_h.
    \end{align*}
    As in a conformal chart, we have the following expression of the Weingarten tensor
    \begin{align*}
    \h_0=(e^{2\lambda}\H_0)dz^2
    \end{align*}
    and the Weil-Petersson norm of $\h_0$ reads
    \begin{align*}
    |\h_0|^2_{WP}=e^{-4\lambda}|e^{2\lambda}\H_0|^2=|H_0|^2=|\H_g|^2-K_g+K_h,
    \end{align*}
    we obtain
    \begin{align*}
    W(\phi)=\int_{\Sigma^2}^{}|\h_0|^2_{WP}\,d\vg.
    \end{align*}

    \begin{lemme}\label{complexcons}
    	Let $\Sigma^2$ be a closed Riemann surface and $(M^n,h)$ be a smooth Riemannian manifold with constant sectional curvature. Then for all smooth immersion $\phi:\Sigma^2\rightarrow (M^n,h)$, we have
    	\begin{align}\label{divh0}
    	\ast_g d\H-3\ast_g (d\H)^{N}+\star (\vec{H}\wedge d\n)=-4\,\mathrm{Im}\left(g^{-1}\otimes \left(\bar{\partial}^N-\bar{\partial}^\top\right)\h_0-|\h_0|^2_{WP}\partial\phi\right)
    	\end{align}
    	where $\H$ is the mean curvature and $\h_0=2\,\bar{\partial}^N\partial\phi=2\,\vec{\I}(\e_z,\e_z)dz^2$ is the Weingarten tensor.
    \end{lemme}
    \begin{rem}
    	We could make a statement for arbitrary target; however, curvature terms will prevent to write the equation in divergence form, and the notion of residue does not make sense as the integration of an exact form. We would obtain
    	\begin{align*}
    	\Im\,d\left(g^{-1}\otimes \left(\bar{\partial}^N-\bar{\partial}^\top\right)\h_0-|h_0|_{WP}^2\,\partial\vec{\Psi}+(R(\e_z,\e_{\z})\e_z)^N\right)=\mathscr{R}(\H)
    	\end{align*}
    	for some curvature tensor (of class $C^{\nu-2}$ is $(M^n,h)$ is $C^\nu$) depending only on $\H$ and $d\vec{\Psi}$, see \cite{riviere1} for more details. It makes little doubt that the results of \cite{beriviere} could be generalised to this setting, however, for our immediate goal, this seems of little interest. For the first complex formulation of Willmore equation, see the paper of Mondino in collaboration with the second author (\cite{mondino}).
    \end{rem}
    \begin{proof}
    	We recall (see \cite{riviere1}) that the Willmore equation in a space of constant sectional curvature is equivalent to
    	\begin{align*}
    	d\left(\ast_g d\H-3\ast_g(d\H)^\nperp-\star\left(d\n\wedge\H\right)\right)=0
    	\end{align*}
    	We first compute
    	\begin{align*}
    	\ast_g(d\H)^\nperp&=\D_{\e_1}^\nperp\H\,dx_2-\D_{\e_2}^\nperp \H dx_1
    	=\D_{\e_z+\e_{\z}}^\nperp\H \frac{dz-d\z}{2i}-\D_{i(\e_z-\e_{\z})}^\nperp \H\frac{dz+d\z}{2}\\
    	&=\frac{1}{i}\left(\D_{\e_z}^\nperp\H-\D_{\e_{\z}}^\nperp\H\right)=2\,\Im(\D_{\e_z}^\nperp \H dz)
    	=2\,\Im(\partial^N \H).
    	\end{align*}
    	We compute by definition of $\D^\nperp$ as $\bar{\D}_{\partial_{\z}}\e_z=0$
    	\begin{align*}
    	\D_{\partial_z}^\nperp\H&=\D_{\partial_z}^\nperp\left(2e^{-2\lambda}\vec{\I}(\e_z,\e_{\z})\right)=2\,\partial_z(e^{-2\lambda})\vec{\I}(\e_z,\e_{\bar{z}})+2e^{-2\lambda}\left(\D_{\partial_z}^\nperp\vec{\I}(\e_z,\e_{\z})+\vec{\I}(\bar{\D}_{\partial_z}\e_z,\e_{\z})+\vec{\I}(\e_z,\bar{\D}_{\partial_z}\e_{\z})\right)\\
    	&=-2\,e^{-4\lambda}\partial_z(e^{2\lambda})\vec{\I}(\e_z,\e_{\z})+2\,e^{-2\lambda}\left(\D_{\partial_z}^\nperp\vec{\I}(\e_z,\e_{\z})+\vec{\I}(\bar{\D}_{\partial_z}\e_z,\e_{\z})\right)\\
    	&=-e^{-2\lambda}\partial_z(e^{2\lambda})\H+2\,e^{-2\lambda}\left(\D_{\partial_z}^\nperp\vec{\I}(\e_z,\e_{\z})+\vec{\I}(\bar{\D}_{\partial_z}\e_z,\e_{\z})\right)
    	\end{align*}
    	Then by Codazzi-Mainardi formula and 
    	as $\bar{\D}_{\partial_z}\e_{\z}=0$, we have
    	\begin{align*}
    	\D_{\partial_z}^\nperp\vec{\I}(\e_z,\e_{\z})&=\D_{\partial_{\z}}^\nperp\vec{\I}(\e_z,\e_z)=\D_{\partial_{\z}}^\nperp\left(\vec{\I}(\e_z,\e_z)\right)-2\,\vec{\I}(\bar{\D}_{\partial_{\z}}\e_z,\e_z)
    	=\D_{\partial_{\z}}^\nperp\left(\vec{\I}(\e_z,\e_z)\right)
    	\end{align*}
    	and
    	\begin{align*}
    	\vec{\I}(\bar{\D}_{\partial_z}\e_z,\e_{\z})&=2e^{-2\lambda}\s{\D_{\partial_z}\e_z}{\e_{\z}}\vec{\I}(\e_z,\e_{\z})+2e^{-2\lambda}\s{\D_{\partial_z}\e_z}{\e_z}\vec{\I}(\e_{\z},\e_z)
    	=e^{-2\lambda}\partial_z(e^{2\lambda}\vec{\I}(\e_z,\e_{\z})
    	=\frac{1}{2}\partial_z(e^{2\lambda})\H.
    	\end{align*}
    	Therefore
    	\begin{align*}
    	\D_{\partial_z}^\nperp\H&=-e^{-2\lambda}\partial_z(e^{2\lambda})\H+2e^{-2\lambda}\left(\D_{\partial_{\z}}^\nperp\vec{\I}(\e_z,\e_z)+\frac{1}{2}\partial_z(e^{2\lambda})\vec{H}\right)
    	=2e^{-2\lambda}\D_{\partial_{\z}}^\nperp\left(\vec{\I}(\e_z,\e_z)\right)
    	=e^{-2\lambda}\D_{\partial_{\z}}^\nperp(e^{2\lambda}\H_0).
    	\end{align*}
    	and we deduce that
    	\begin{align}\label{codazzi}
    	\partial^N\H=g^{-1}\otimes \bar{\partial}^N\h_0
    	\end{align}
    	Then we have
    	\begin{align*}
    	\ast_g d\H-3\ast_g(d\H)^N=\ast_g (d\H)^\top-4\Im\left(g^{-1}\otimes \bar{\partial}^N\h_0\right)
    	\end{align*}
    	while
    	\begin{align*}
    	\D_{\e_z}^\top\H=-|\H|^2\e_z-\s{\H}{\H_0}\e_{\z}
    	\end{align*}
    	and
    	\begin{align}\label{h01}
    	\ast_g (d\H)-3\ast_g (d\H)^\nperp=-4\,\Im(g^{-1}\otimes\bar{\partial}^N\h_0)-2\,\Im\left(|\H|^2\e_z+\s{\H}{\H_0}\e_{\z}\right).
    	\end{align}
    	Now recall
    	that the unit normal $\n$ is defined by $\n=e^{-2\lambda}\,\star(\e_1\wedge \e_2)$, so that 
    	\begin{align*}
    	\star(\n\wedge\e_1)=\e_2,\quad \star(\n\wedge\e_2)=-\e_1.
    	\end{align*}
    	Furthermore, an immediate computation shows that 
    	\begin{align*}
    	\D_{\e_z}\n=-H\e_z-\H_0\e_{\z}
    	\end{align*}
    	and
    	\begin{align}\label{step0}
    	\star (\n\wedge \D_{\e_z}\n)=-iH\e_z+i\H_0\e_{\z}.
    	\end{align}
    	As 
    	$
    	d\n=2\Re(\partial\n)
    	$, 
    	we deduce from \eqref{step0} that 
    	\begin{align}\label{h02}
    	\star( \H\wedge d\n)=2\Re\left(-i|H|^2\e_z\,dz+i\s{\H}{\H_0}\e_{\z}d\z\right)
    	=2\,\Im\left(|\H|^2\e_z dz-\s{\H}{\H_0}\e_{\z}d\z\right).
    	\end{align}
    	Finally by \eqref{h01} and \eqref{h02}
    	\begin{align*}
    	&\ast_g(d\H)-3\ast_g(d\H)^\nperp+\ast(\H\wedge d\n)=-4\,\Im\left(g^{-1}\otimes\bar{\partial}^N\h_0+\s{\H}{\H_0}\e_{\z} d\z\right)\\
    	&=-4\,\Im\left(g^{-1}\otimes \left(\bar{\partial}^N\h_0+\s{\H}{\h_0}\otimes \bar{\partial}\phi\right)\right)
    	=-4\,\Im\left(g^{-1}\otimes\left(\bar{\partial}^N-\bar{\partial}^\top\right)\h_0-|\h_0|^2_{WP}\,\partial\phi\right).
       	\end{align*}
    	As
    	\begin{align*}
    	\bar{\D}_{\e_{\z}}\h_0=-|\vec{h}_0|_{WP}^2\e_z-\s{\H}{\H_0}\e_{\z},
    	\end{align*}
    	this concludes the proof.
    \end{proof}
    Proceeding directly in the general case gives the following.
    \begin{prop}
    	Let $(M^n,h)$ be a smooth Riemannian manifold, and $\phi:\Sigma^2\rightarrow (M^n,h)$ be a smooth immersion, then we have
    	\begin{align*}
    	\Delta^\nperp \vec{H}-2|\vec{H}|^2\H +\mathscr{A}(\vec{H})&=4\,\Re\left\{g^{-1}\otimes\bar{\partial}^\nperp\left(g^{-1}\otimes\left(\bar{\partial}^\nperp\vec{h}_0+\s{\H}{\h_0}\otimes\bar{\partial}\phi+2(R(\e_z,\e_{\z})\e_z)^\nperp  dz^2\otimes d\z\right)\right)\right\}\\
    	&=-4\,g^{-1}\otimes d\, \Im \left\{g^{-1}\otimes\left( \left(\bar{\partial}^N-\bar{\partial}^\top\right)\h_0+2\left(R(\partial\phi,\bar{\partial}\phi)\partial\phi\right)^N\right)-|\h_0|^2_{WP}\,\partial\phi\right\}.
    	\end{align*}
    	which reduces if $M^n$ has constant sectional curvature to 
    	\begin{align*}
    	\Delta^\nperp \vec{H}-2|\vec{H}|^2\H+\mathscr{A}(\vec{H})&=-4\,g^{-1}\otimes \Im d\left(g^{-1}\otimes \left(\bar{\partial}^N-\bar{\partial}^\top\right)\h_0+|\h_0|^2_{WP}\,\partial\phi\right).
    	\end{align*}
    \end{prop}
    \begin{proof}
    	We recall that in a conformal chart, we have if $\e_i=\d{i}\phi$ ($1\leq i\leq 2$)
    	\begin{align*}
    	\Delta^\nperp=e^{-2\lambda}\sum_{i=1}^{2}\left(\D_{\e_i}^\nperp\D_{\e_i}^\nperp-\D^\nperp_{\bar{\D}_{\e_i}\e_i}\right)
    	\end{align*}
    	and $\mathscr{A}$ is the Simon's operator, given by
    	\begin{align*}
    	\mathscr{A}(\,\cdot\,)=e^{-4\lambda}\sum_{i,j=1}^{2}\s{\vec{\I}(\e_i,\e_j)}{\,\cdot\,}\vec{\I}(\e_i,\e_j).
    	\end{align*}
    	We have in a local complex coordinate $z$ the identity
    	\begin{align*}
    	\sum_{i=1}^{2}\D_{\e_i}^\nperp\D^\nperp_{\e_i}=\D_{\e_z+\e_{\z}}^\nperp\D_{\e_z+\e_{\z}}^\nperp+\D_{i(\e_z-\e_{\z})}^\nperp\D_{i(\e_z-\e_{\z})}^\nperp&=2\D_{\e_{z}}^\nperp\D_{\e_{\z}}^\nperp+2\D_{\e_{\z}}^\nperp\D_{\e_z}^\nperp
    	\end{align*}
    	and
    	\begin{align*}
    	\sum_{i=1}^{2}\bar{\D}_{\e_i}\e_{i}&=\bar{\D}_{\e_z+\e_{\z}}(\e_z+\e_{\z})+\bar{\D}_{i(\e_z-\e_{\z})}i(\e_z-\e_{\z})
    	=2\bar{\D}_{\e_z}\e_{\z}+2\bar{\D}_{\e_{\z}}\e_z
    	=0.
    	\end{align*}
    	As 
    	\begin{align*}
    	\D_{\e_z}\e_{\z}=\D_{\e_{\z}}\e_z=\frac{e^{2\lambda}}{4}\Delta_g \phi=\frac{\vec{H}}{2}
    	\end{align*}
    	has only \textit{normal} components, \textit{i.e.} $\bar{\D}_{\e_z}\e_{\z}=\bar{\D}_{\e_{\z}}\e_z=0$, we deduce that 
    	\begin{align*}
    	\frac{1}{2}\D_{\e_z}^\nperp\vec{H}&=\D_{\e_z}^\nperp\left(e^{-2\lambda}\vec{\I}(\e_z,\e_{\z})\right)=\partial_{z}(e^{-2\lambda})\vec{\I}(\e_z,\e_{\z})+e^{-2\lambda}\left(\D_{\e_z}^\nperp \vec{\I}(\e_z,\e_{\z})+\vec{\I}(\bar{\D}_{\e_z}\e_{z},\e_{\z})+\vec{\I}(\e_z,\bar{\D}_{\e_z}\e_{\z})\right)\\
    	&=\partial_{z}(e^{-2\lambda})\vec{\I}(\e_z,\e_{\z})+e^{-2\lambda}\left(\D_{\e_z}^\nperp \vec{\I}(\e_z,\e_{\z})+\vec{\I}(\bar{\D}_{\e_z}\e_{z},\e_{\z})\right).
    	\end{align*}
    	Noting that
    	\begin{align*}
    	\bar{\D}_{\e_z}\e_z=a\e_z+b\e_{\z}
    	\end{align*}
    	we obtain
    	\begin{align*}
    	\frac{e^{2\lambda}}{2}a&=\s{\bar{\D}_{\e_z}\e_z}{\e_{\z}}=\partial_z\s{\e_z}{\e_{\z}}=\frac{1}{2}\partial_z(e^{2\lambda})\\
    	\frac{e^{2\lambda}}{2}b&=\s{\D_{\e_z}\e_z}{\e_z}=\frac{1}{2}\partial_z\s{\e_z}{\e_z}=0,
    	\end{align*}
    	while the Codazzi-Mainardi implies that
    	\begin{align*}
    	\D_{\e_z}^\nperp\vec{\I}(\e_z,\e_{\z})&=\D_{\e_{\z}}^\nperp\vec{\I}(\e_z,\e_{z})+\left(R(\e_z,\e_{\z})\e_z\right)
    	=\D_{\e_{\z}}^\nperp\left(\vec{\I}(\e_z,\e_{z})\right)-2\vec{\I}(\bar{\D}\e_{\z}\e_z,\e_z)+\left(R(\e_z,\e_{\z})\e_z\right)^\nperp\\
    	&=\D_{\e_{\z}}^\nperp\left(\vec{\I}(\e_z,\e_{z})\right)+\left(R(\e_z,\e_{\z})\e_z\right)^\nperp.
    	\end{align*}
    	In particular, if $(M^m,h)$ has constant sectional curvature $\kappa\in\R$, we have for all vector-fields $X,Y,Z$,
    	\begin{align*}
    	R(X,Y)Z=\kappa\left(\s{Y}{Z}X-\s{X}{Z}Y\right)
    	\end{align*}
    	so $(R(\e_z,\e_{\z})\e_z)^\nperp=0$. Then, we obtain
    	\begin{align*}
    	\frac{1}{2}\D_{\p{z}}^\nperp\H&=\left(\p{z}(e^{-2\lambda})+e^{-4\lambda}\p{z}(e^{2\lambda})\right)\vec{\I}(\e_z,\e_{\z})+\D_{\e_{z}}^\nperp\left(\vec{\I}(\e_z,\e_z)\right)+(R(\e_z,\e_{\z})\e_z)^\nperp\\
    	&=e^{-2\lambda}\D_{\e_{z}}^\nperp\left(\vec{\I}(\e_z,\e_z)\right)+e^{-2\lambda}(R(\e_z,\e_{\z})\e_z)^\nperp\\
    	&=\frac{1}{2}e^{-2\lambda}\D_{\e_{\z}}^\nperp\vec{h}_0+e^{-2\lambda}(R(\e_z,\e_{\z})\e_z)^\nperp
    	\end{align*}
    	and as $\H$ is real, we have
    	\begin{align*}
    	\D_{\e_{\z}}^\nperp\vec{H}&=\bar{\D_{\e_{z}}^\nperp\vec{H}}
    	=e^{-2\lambda}\bar{\D_{\e_{\z}}^\nperp\vec{h}_0}+2e^{-2\lambda}(R(\e_{\z},\e_z)\e_{\z})^\nperp
    	\end{align*}
    	and we deduce that
    	\begin{align*}
    	\Delta^\nperp \vec{H}&=2e^{-2\lambda}\left\{\D_{\e_{\z}}^\nperp\left(e^{-2\lambda}\D_{\e_{\z}}^\nperp\h_0\right)+\D_{\e_{\z}}^\nperp\left(e^{-2\lambda}\bar{\D_{\e_{\z}}^\nperp\h_0}\right)\right\}+8e^{-2\lambda}\,\Re\left\{e^{-2\lambda}\D_{\e_{\z}}^\nperp(R(\e_z,\e_{\z})\e_z)^\nperp\right\}\\
    	&=4\,\Re\left\{e^{-2\lambda}\D_{\p{\z}}^\nperp\left(e^{-2\lambda}\left(\D_{\p{\z}}^\nperp\vec{h}_0+2(R(\e_z,\e_{\z})\e_z)^\nperp\right)\right)\right\}
    	\end{align*}
    	We now want to express the Simon's operator only with respect of $\H_0$ and $\H$, but this is easy as
    	\begin{align*}
    	\vec{\I}(\e_1,\e_1)&=e^{-2\lambda}\vec{\I}(\e_z+e_{\z},\e_z+\e_{\z})=\frac{1}{2}\left(2e^{-2\lambda}\vec{\I}(\e_z,\e_z)+2e^{-2\lambda}\vec{\I}(\e_{\z},\e_{\z})\right)+2e^{-2\lambda}\vec{\I}(\e_z,e_{\z})\\
    	&=\Re\H_0+\H\\
    	\vec{\I}(\e_2,\e_2)&=-e^{-2\lambda}\vec{\I}(\e_z-\e_{\z},\e_z-\e_{\z})=-\Re\H_0+\H\\
    	\vec{\I}(\e_1,\e_2)&=\frac{1}{i}e^{-2\lambda}\I(\e_z+\e_{\z},\e_z-\e_{\z})=\Im\H_0
    	\end{align*}
    	therefore
    	\begin{align*}
    	\mathscr{A}(\H)&=\sum_{i,j=1}^{2}\s{\vec{\I}(\e_i,\e_j)}{\H}\vec{\I}(\e_i,\e_j)\\
    	&=\s{\Re\H_0+\H}{\H}(\Re\H_0+\H)+2\s{\Im\H_0}{\H}\Im\H_0+\s{-\Re\H_0+\H}{\H}(-\Re\H_0+\H)\\
    	&=2\s{\Re \H_0}{\H}\Re\H_0+2\s{\Im\H_0}{\H}\Im\H_0+2|\H|^2\H\\
    	&=2\Re \left(\s{\H_0}{\H}\bar{\H}_0\right)+2|\H|^2\H
    	\end{align*}    	
    	and finally, we obtain
    	\begin{align*}
    	\Delta^\nperp\H-2|\vec{H}|^2\H+\mathscr{A}(\H)&=4\,\Re\left\{e^{-2\lambda}\D_{\p{\z}}^\nperp\left(e^{-2\lambda}\left(\D_{\p{\z}}^\nperp\vec{h}_0+\s{\h_0}{\H}\e_{\z}+2(R(\e_z,\e_{\z})\e_z)^\nperp\right)\right)\right\}
    	\end{align*}
    	and the last equality goes like Lemma \ref{complexcons}.
    \end{proof}

	\subsection{Noether's theorem for second order functionals}
	
	    Noether's theorem is the mathematical formulation of the physical phenomenon that infinitesimal symmetries correspond to conserved quantities, \textit{i.e.} closed differential forms (see \cite{noether}).
	
		\begin{definitions}
		(1) Let $\Sigma^k$ be a $C^2$ manifold and $(M^n,h)$ be a $C^2$ Riemannian manifold.
		For all $p\in\N$, we define the $p$-differentiation bundle $B^p(\Sigma^k,M^n)$ of the couple $(\Sigma^k,M^n)$ as the product
		\begin{align*}
		B^p(\Sigma^k,M^n)=\prod_{j=1}^{k}(T^\ast\Sigma^k)^{\otimes j}\otimes TM^n.
		\end{align*}
		(2) If $\mathscr{U}\subset B^2(\Sigma^k,M^n)$ and $L\in C^1(M^n\times \mathscr{U})$, we say that a vector field $\vec{X}\in \Gamma(TM)$ is an infinitesimal symmetry of $L$ if for all $\phi\in C^2(\Sigma^k,M^n)$ such that $\Im(d\phi,\D d\phi,\cdots \D^{k-1} d\phi)\subset \mathscr{U}$,
		\begin{align*}
			L(\exp_{\phi}(t\vec{X}),d(\exp_{\phi}(t\vec{X})),\cdots,\D^{k-1} d(\exp_{\phi}(t\vec{X})))=L(\phi,d\phi,\cdots \D^{k-1}d\phi).
		\end{align*}
		for all $t\in\R$ in some small interval around $0$.
	\end{definitions}

		\begin{theorem}\label{noether1}
		Let $m\geq 2$, $1\leq p\leq\infty$, $\Sigma^k$ be a $C^2$ manifold and $(M^n,h)$ be a $C^2$ Riemannian manifold, $U$ be an open subset of $\Sigma^k$, $\mathscr{U}$ be an open subset of $B^2(\Sigma^k,M^n)$, $L=L(y,p,q)\in C^1(M^n\times \mathscr{U},\R)$, $\mathscr{V}$ be an open subset of $W^{k,p}(\Sigma^k,M^n)$, and $\mathscr{L}\in C^1(\mathscr{V},\R)$, such that for all $\phi\in \mathscr{V}$, we have
		\begin{align*}
		\mathscr{L}(\phi)=\int_{U}L(\phi,d\phi,\D d\phi)d\mathscr{H}^2.
		\end{align*}
		For all infinitesimal symmetry $\vec{X}\in\Gamma(TM)$, and for all critical point $\phi\in \mathscr{V}$, we have
		\begin{align}\label{noether}
		\sum_{i,j=1}^{2}\d{i}\left(\p{p_i}L(\phi,d\phi, \D d\phi)\cdot \vec{X}(\phi)-2\d{j}(\p{q_{ij}}L(\phi,d\phi,\D d\phi))\cdot \vec{X}(\phi)+2\p{q_{ij}}L(\phi,d\phi, \D d\phi)\cdot \d{j}\vec{X}(\phi)\right)=0.
		\end{align}
	\end{theorem}
\begin{rem}
	In particular, Noether's theorem does not depend on the derivatives in the space variable $y$. This should be useful in the correspondence of Section \ref{correspondance}.
\end{rem}
	\begin{proof}
		Following \cite{helein}, we can suppose that $M^n$ is a submanifold of an Euclidean space. We fix a critical point  $\phi$ of $\mathscr{L}$, and if $\exp$ is the exponential application on $(M^n,h)$, for all test function $\varphi\in C^{\infty}_{c}(U)$, we have
		\begin{align}\label{dev1}
		\mathscr{L}(\exp_{\phi}(t\varphi\vec{X}))=\mathscr{L}(\phi+t\varphi \vec{X}+o(t))=\mathscr{L}(\phi)+o(t).
		\end{align}
		Therefore, we obtain, abbreviating $\vec{X}=\vec{X}(\phi)$
		\begin{align*}
		\mathscr{L}(\phi+t\varphi\vec{X}+o(t))&=\int_{U}L(\phi+t\varphi\vec{X},d\phi+t\varphi d\vec{X}+td\varphi\cdot \vec{X},
		\D d\phi+t\varphi \D d\vec{X}+2d\varphi\cdot d\vec{X}+t\D d\varphi\cdot \vec{X})d\mathscr{H}^2\\
		&+o(t)\\
		&=\int_{U}L(\phi+t\varphi d\vec{X},d\phi+t\varphi d\vec{X},\D d\phi+t\varphi \D d\vec{X})d\mathscr{H}^2\\
		&+t\int_{U}\p{p_i}L(\phi,d\phi,\D d\phi)\cdot X\p{x_i}\varphi d\mathscr{H}^2+2t\int_{U}\partial_{q_{ij}}L(\phi,d\phi,\D d\phi)\cdot \p{x_j}\vec{X}\p{x_i}\varphi d\mathscr{H}^2\\
		&+t\int_{U}\partial_{q_{ij}}L(\phi,d\phi,\D d\phi)\cdot \vec{X}\p{x_ix_j}^2\varphi d\mathscr{H}^2+o(t)\\
		&=\mathscr{L}(\phi)+o(t)
		\end{align*}
		therfore, comparing this equation to \eqref{dev1} we deduce that
		\begin{align*}
		\int_{U}\p{p_i}L(\phi,d\phi,\D d\phi)\cdot X\p{x_i}\varphi +2\partial_{q_{ij}}L(\phi,d\phi,\D d\phi)\cdot \p{x_j}\vec{X}\p{x_i}\varphi
		&+\partial_{q_{ij}}L(\phi,d\phi,\D d\phi)\cdot \vec{X}\p{x_ix_j}^2\varphi\, d\mathscr{H}^2=0
		\end{align*}
		so integrating by parts, this gives
		\begin{align*}
		\d{i}\left(\p{p_i}L\cdot \vec{X}+2\p{q_{ij}}L\cdot \p{x_j}\vec{X}-\p{x_j}\left(\p{q_{ij}}L\cdot \vec{X}\right)\right)=0
		\end{align*}
		which is equivalent to
		\begin{align*}
		\d{i}\left(\p{p_i}L\cdot \vec{X}+\p{q_{ij}}L\cdot \p{x_j}\vec{X}-\p{x_j}(\p{q_{ij}}L)\cdot \vec{X}\right)=0.
		\end{align*}
		which is the expected result, as the sums in $j$ are performed \textit{inside} the parenthesis, contrary to the formula announced in the theorem. This concludes the proof.
	\end{proof}
	
	As the equation does not involve derivatives in $y$ of $L$, we write $L=L(p_1,p_2,q_{11},q_{12},q_{21},q_{22})$, where the index stands for the corresponding partial derivative with respect to any local frame, and we let
	\begin{align*}
	\left\{\begin{aligned}
	z_1&=\frac{1}{2}(p_1-ip_2)\\
	z_2&=\frac{1}{2}(p_1+ip_2)\\
	w_1&=\frac{1}{4}\left(q_{11}-q_{22}-i(q_{12}+q_{21})\right)\\
	w_2&=\frac{1}{4}\left(q_{11}-q_{22}+i(q_{12}+q_{21})\right)\\
	w_3&=\frac{1}{4}(q_{11}+q_{22}+i(q_{12}-q_{21}))\\
	w_4&=\frac{1}{4}(q_{11}+q_{22}-i(q_{12}-q_{21}))
	\end{aligned}\right.
	\end{align*}
	such that $L_0(z_1,z_2,w_1,w_2,w_3,w_4)=L(p_1,p_2,q_{11},q_{12},q_{21},q_{22})$.
	We deduce that
	\begin{align*}
	\left\{\begin{aligned}
	\frac{\partial L}{\partial p_1}&=\frac{1}{2}\left(\frac{\partial L_0}{\partial z_1}+\frac{\partial L_0}{\partial z_2}\right)\\
	\frac{\partial L}{\partial p_2}&=\frac{1}{2i}\left(\frac{\partial L_0}{\partial z_1}-\frac{\partial L_0}{\partial z_2}\right)\\
	\frac{\partial L}{\partial q_{11}}&=\frac{1}{4}\left(\frac{\partial L_0}{\partial w_1}+\frac{\partial L_0}{\partial w_2}+\frac{\partial L_0}{\partial w_3}+\frac{\partial L_0}{\partial w_4}\right)\\
	\frac{\partial L}{\partial q_{12}}&=\frac{1}{4i}\left(\frac{\partial L_0}{\partial w_1}-\frac{\partial L_0}{\partial w_2}-\frac{\partial L_0}{\partial w_3}+\frac{\partial L_0}{\partial w_4}\right)\\
	\frac{\partial L}{\partial q_{21}}&=\frac{1}{4i}\left(\frac{\partial L_0}{\partial w_1}-\frac{\partial L_0}{\partial w_2}+\frac{\partial L_0}{\partial w_3}-\frac{\partial L_0}{\partial w_4}\right)\\
	\frac{\partial L}{\partial q_{22}}&=\frac{1}{4}\left(-\frac{\partial L_0}{\partial w_1}-\frac{\partial L_0}{\partial w_2}+\frac{\partial L_0}{\partial w_3}+\frac{\partial L_0}{\partial w_4}\right)
	\end{aligned}\right.
	\end{align*}
	Now as we are mostly interested in deriving conservations laws for the Willmore energy in spaces with known conformal transformations, \textit{i.e.} space forms, as in this case no curvature terms can arise we can suppose that $q_{12}=q_{21}$ (implying that $w_3=w_4$).
	As $L_0$ is real, we deduce that
	\begin{align*}
	\frac{\partial L_0}{\partial z_2}=\bar{\frac{\partial L_0}{\partial z_1}},\quad \frac{\partial L_0}{\partial w_2}=\bar{\frac{\partial L_0}{\partial w_1}}
	\end{align*}
	so the system reduces to 
	\begin{align}\label{div2}
	\left\{\begin{aligned}
	\frac{\partial L}{\partial p_1}&=\Re\left(\frac{\partial L_0}{\partial z_1}\right)\\
	\frac{\partial L}{\partial p_2}&=\Im\left(\frac{\partial L_0}{\partial z_1}\right)\\
	\frac{\partial L}{\partial q_{11}}&=\frac{1}{2}\Re\left(\frac{\partial L_0}{\partial w_1}\right)+\frac{1}{4}\frac{\partial L_0}{\partial w_3}\\
	\frac{\partial L}{\partial q_{12}}&=\frac{1}{2}\Im\left(\frac{\partial L_0}{\partial w_1}\right)\\
	\frac{\partial L}{\partial q_{22}}&=-\frac{1}{2}\Re\left(\frac{\partial L_0}{\partial w_1}\right)+\frac{1}{4}\frac{\partial L_0}{\partial w_3}
	\end{aligned}\right.
	\end{align}	
	If $L_0=L(\zeta,\omega,\chi)=L(\zeta,\bar{\zeta},\omega,\bar{\omega},\chi)=L_0(z_1,z_2,w_1,w_2,w_3)$, we obtain the following.
	\begin{cor}\label{noether2}
		Under the hypothesis of Theorem \ref{noether1}, we have
		\begin{align}\label{noethercomplexe}
		\Re\left(\partial_{z}\left(\frac{\partial L_0}{\partial\zeta}\cdot\vec{X}-\partial_{z}\left(\frac{\partial L_0}{\partial \omega}\right)\cdot\vec{X}+\frac{\partial L_0}{\partial \omega}\cdot \partial_{z}\vec{X}-\frac{1}{2}\p{\z}\left(\frac{\partial L_0}{\partial \chi}\right)\cdot\vec{X}+\frac{1}{2}\frac{\partial L_0}{\partial\chi}\cdot\p{\z}\vec{X}\right)\right)=0.
		\end{align}
	\end{cor}
	\begin{proof}
		Using $\p{z}=\dfrac{1}{2}\left(\d{1}-i\d{2}\right),\,\p{\z}=\dfrac{1}{2}\left(\d{1}+i\d{2}\right)$, we obtain by \eqref{noether} and \eqref{div2}
		\begin{align*}
		&(\p{z}+\p{\z})\bigg\{\Re\left(\frac{\partial L_0}{\partial \zeta}\right)\cdot\vec{X}-(\p{z}+\p{\z})\left(\frac{1}{2}\Re\left(\frac{\partial L_0}{\partial \omega}\right)+\frac{1}{4}\frac{\partial L_0}{\partial\chi}\right)\cdot\vec{X}+\left(\frac{1}{2}\Re\left(\frac{\partial L_0}{\partial \omega}\right)\cdot \vec{X}+\frac{1}{4}\frac{\partial L_0}{\partial\chi}\right)\cdot (\p{z}+\p{\z})\vec{X}\\
		&-i(\p{z}-\p{\z})\left(\frac{1}{2}\Im\left(\frac{\partial L_0}{\partial \omega}\right)\right)\cdot\vec{X}+\frac{1}{2}\Im\left(\frac{\partial L_0}{\partial \omega}\right)\cdot i(\p{z}-\p{\z})\vec{X}\bigg\}\\
		&+i(\p{z}-\p{\z})\bigg\{\Im\left(\frac{\partial L_0}{\partial\zeta}\right)\cdot\vec{X}-(\p{z}+\p{\z})\left(\frac{1}{2}\Im\left(\frac{\partial L_0}{\partial \omega}\right)\right)\cdot\vec{X}+\frac{1}{2}\Im\left(\frac{\partial L_0}{\partial\omega}\right)\cdot (\p{z}+\p{\z})\vec{X}\\
		&-i(\p{z}-\p{\z})\left(-\frac{1}{2}\Re\left(\frac{\partial L_0}{\partial \omega}\right)+\frac{1}{4}\frac{\partial L_0}{\partial \chi}\right)\cdot\vec{X}+\left(-\frac{1}{2}\Re\left(\frac{\partial L_0}{\partial \omega}\right)+\frac{1}{4}\frac{\partial L_0}{\partial \chi}\right)\cdot i(\p{z}-\p{\z})\vec{X}\bigg\}=0
		\end{align*}
		and after rearranging, we have
		\begin{align*}
		\Re\left(\partial_{z}\left(\frac{\partial L_0}{\partial\zeta}\cdot\vec{X}-\partial_{z}\left(\frac{\partial L_0}{\partial \omega}\right)\cdot\vec{X}+\frac{\partial L_0}{\partial \omega}\cdot \partial_{z}\vec{X}-\frac{1}{2}\p{\z}\left(\frac{\partial L_0}{\partial \chi}\right)\cdot\vec{X}+\frac{1}{2}\frac{\partial L_0}{\partial\chi}\cdot\p{\z}\vec{X}\right)\right)=0
		\end{align*}
		which concludes the proof.
	\end{proof}

\subsection{Residues of Willmore and minimal surfaces}

	In this section, we want to derive the four conservation laws for the Willmore energy with respect to tensors only depending on the immersion (for such formulation, see \cite{riviere1}, and for a derivation of the first three conservation with Noether's theorem, \cite{bernard}). 
	
	We recall that the mean curvature $\H$ of an immersion $\phi:\Sigma^2\rightarrow\R^n$ is the tensor
	\begin{align}\label{meandef}
		\H=\frac{1}{2}\mathrm{Tr}_g(\vec{\I}_g)=\frac{1}{2}\sum_{i,j=1}^{2}g^{i,j}\vec{\I}_{i,j},
	\end{align}
	where $\vec{\I}_{i,j}=\vec{\I}(\e_i,\e_j)$, and $\vec{\I}$ is the second fundamental form of $\phi$. If $\e_k=\p{x_k}\phi$ for $k=1,2$. In particular, using $\Z_2$ notations for indices we have as $g^{i,j}=(-1)^{i,j}g_{i+1,j+1}(\det g)^{-1}$ the identities
	\begin{align*}
		g_{11}&=2\left(|\e_z|^2+\Re\s{\e_z}{\e_z}\right),\quad
		g_{12}=-2\,\Im\s{\e_z}{\e_z},\quad
		g_{22}=2\left(|\e_z|^2-\Re\s{\e_z}{\e_z}\right)\\
		\det g&=g_{11}g_{22}-g_{12}^2=4\left(|\e_z|^4-(\Re\s{\e_z}{\e_z}^2\right)-4(\Im\s{\e_z}{\e_z})^2=4\left(|\e_z|^4-|\s{\e_z}{\e_z}|^2\right).
	\end{align*}
	As $\e_1=\e_z+\e_{\z}$ and $\e_2=i(\e_z-\e_{\z})$, a trivial computation gives
	\begin{align*}
	\begin{alignedat}{2}
	    g^{1,1}&=\frac{1}{2}\frac{|\e_z|^2-\Re\s{\e_z}{\e_z}}{|\e_z|^4-|\s{\e_z}{\e_z}|^2},\qquad&&\vec{\I}(\e_1,\e_1)=2\,\Re \vec{\I}(\e_z,\e_z)+2\,\vec{\I}(\e_z,\e_{\z})\\
	    g^{1,2}&=\frac{1}{2}\frac{\Im\s{\e_z}{\e_z}}{|\e_z|^4-|\s{\e_z}{\e_z}|^2},\qquad&&\vec{\I}(\e_1,\e_2)=-2\,\Im \vec{\I}(\e_z,\e_z)\\
	    g^{2,2}&=\frac{1}{2}\frac{|\e_z|^2+\Re\s{\e_z}{\e_z}}{|\e_z|^4-|\s{\e_z}{\e_z}|^2},\qquad&& \vec{\I}(\e_2,\e_2)=-2\,\Re\vec{\I}(\e_z,\e_z)+2\,\vec{\I}(\e_z,\e_{\z}).
	\end{alignedat}
	\end{align*}
	So we have by \eqref{meandef}
	\begin{align}\label{314}
		\vec{H}&=\frac{1}{4}\left(|\e_z|^4-|\s{\e_z}{\e_z}|^2\right)^{-1}\bigg(\left(|\e_z|^2-\Re\s{\e_z}{\e_z}\right)\left(2\Re\vec{\I}(\e_z,\e_z)+2\vec{\I}(\e_z,\e_{\z})\right)\nonumber\\
		&-4\Im\s{\e_z}{\e_z}\Im\vec{\I}(\e_z,\e_z)+\left(|\e_z|^2+\Re\s{\e_z}{\e_z}\right)\left(-2\Re\vec{\I}(\e_z,\e_z)+2\vec{\I}(\e_z,\e_{\z})\right)\bigg)\nonumber\\
		&=\left(|\e_z|^4-|\s{\e_z}{\e_z}|^2\right)^{-1}\left(|\e_z|^2\vec{\I}(\e_z,\e_{\z})-\Re\s{\e_z}{\e_z}\vec{\I}(\e_{\z},\e_{\z})\right)
	\end{align}
	To apply our version of Noether's theorem, we want to write the equation as a function depending only on the derivatives of $\phi$ without taking normal components. 
	For all vector field 
	$\w$ on $\R^n$, writing
	\begin{align*}
		\w^\top&=a\e_z+b\e_{\z}
	\end{align*}
	we have
	\begin{align*}
		\begin{pmatrix}
		a\\
		b
		\end{pmatrix}&=\left(|\s{\e_z}{\e_z}|^2-\s{\e_z}{\e_{\z}}^2\right)^{-1}
		\begin{pmatrix}
		\s{\e_{\z}}{\e_{\z}}&-\s{\e_z}{\e_{\z}}\\
		-\s{\e_z}{\e_{\z}}&\s{\e_z}{\e_z}
		\end{pmatrix}
		\begin{pmatrix}
		\s{\D_XY}{\e_z}\\
		\s{\D_XY}{\e_{\z}}
		\end{pmatrix}
	\end{align*}
	so
	\begin{align*}
		\w^\top=-f(\e_z)^{-1}\left\{\left(\s{\e_{\z}}{\e_{\z}}\s{\e_z}{\w}-|\e_z|^2\s{\e_{\z}}{\w}\right)\e_z+\left(-|\e_z|^2\s{\e_z}{\w}+\s{\e_z}{\e_z}\s{\e_{\z}}{\w}\right)\e_{\z}\right\}
	\end{align*}
where $f(\zeta)=|\zeta|^4-|\s{\zeta}{\zeta}|^2$. We now set the notations
\begin{align*}
	\zeta=\e_{z},\quad \omega=\D_{\p{z}}\e_z,\quad \chi=\D_{\p{z}}\e_{\z}
\end{align*}
and
\begin{align*}
	h(\zeta,\kappa)=\left(\s{\bar{\zeta}}{\bar{\zeta}}\s{\zeta}{\kappa}-|\zeta|^2\s{\bar{\zeta}}{\kappa}\right)\zeta+\left(-|\zeta|^2\s{\zeta}{\kappa}+\s{\zeta}{\zeta}\s{\bar{\zeta}}{\kappa}\right)\bar{\zeta}.
\end{align*}
We remark that
\begin{align*}
	\bar{h(\zeta,\kappa)}=h(\zeta,\bar{\kappa})
\end{align*}
so by \eqref{314}
\begin{align}\label{expressionH}
	\H=f(\zeta)^{-1}\left(|\zeta|^2\chi+|\zeta|^2f(\zeta)^{-1}h(\zeta,\chi)-\Re\left(\s{\zeta}{\zeta}\bar{\omega}\right)-f(\zeta)^{-1}\Re\left(\s{\zeta}{\zeta}h(\zeta,\bar{\omega})\right)\right).
\end{align}
To simplify the expressions, we will take the derivative at conformal coordinates, as there will be significant amount of simplifications. We compute
\begin{align*}
    f(\e_{\z})&=|\e_z|^4=\frac{e^{4\lambda}}{4}\\
	D_{\zeta}f(\zeta)&=2(\bar{\zeta}|\zeta|^2-\zeta\s{\bar{\zeta}}{\bar{\zeta}})\\
	D_{\zeta}f(\e_{z})&=2\s{\e_z}{\e_{\z}}\,\e_{\z}=e^{2\lambda}\e_{\z}\\
	h(\e_{z},\D_{\e_{\z}}\e_{\z})&=-|\e_z|^2\s{\e_{\z}}{\D_{\p{\z}}\e_{\z}}\e_z-|\e_z|^2\s{\e_z}{\D_{\p{\z}}\e_{\z}}\e_{\z}=-\frac{e^{2\lambda}}{2}\p{\z}\left(\frac{e^{2\lambda}}{2}\right)\e_{\z}=-\frac{e^{4\lambda}}{2}(\p{\z}\lambda)\e_{\z}\\
	h(\e_z,\D_{\p{z}}\e_{\z})&=0,\;\, \text{as}\;\,\bar{\D}_{\p{z}}\e_{\z}=0\\
	D_{\zeta}h(\zeta,\kappa)&=\left(\s{\bar{\zeta}}{\bar{\zeta}}\s{\kappa}{\,\cdot\,}-\s{\bar{\zeta}}{\,\cdot\,}\s{\bar{\zeta}}{\kappa}\right)\zeta+\left(\s{\bar{\zeta}}{\bar{\zeta}}\s{\zeta}{\kappa}-|\zeta|^2\s{\bar{\zeta}}{\kappa}\right)\,\cdot\,\\
	&+\left(-\s{\bar{\zeta}}{\,\cdot\,}\s{\zeta}{\kappa}-|\zeta|^2\s{\,\cdot\,}{\kappa}+2\s{\zeta}{\,\cdot\,}\s{\bar{\zeta}}{\kappa}\right)\bar{\zeta}\\
	D_{\zeta}h(\e_z,\D_{\p{z}}\e_{\z})&=-|\e_z|^2\s{\D_{\p{z}}\e_{\z}}{\,\cdot\,}\e_{\z}=-\frac{e^{2\lambda}}{2}\s{\D_{\p{z}}\e_{\z}}{\,\cdot\,}\e_{\z}=-\frac{e^{4\lambda}}{4}\s{\H}{\,\cdot\,}\e_{\z}\\
	D_{\zeta}h(\e_z,\D_{\p{z}}\e_{z})&=-\s{\e_{\z}}{\,\cdot\,}\s{\e_{\z}}{\D_{\p{z}}\e_z}\e_{z}-|\e_z|^2\s{\e_{\z}}{\D_{\p{z}}\e_z}\,\cdot\,
	+\left(-|\e_z|^2\s{\D_{\p{z}}\e_{z}}{\,\cdot\,}+2\s{\e_{\z}}{\D_{\p{z}}\e_{z}}\s{\e_z}{\,\cdot\,}\right)\e_{\z}\\
	&=-\p{z}\left(\frac{e^{2\lambda}}{2}\right)\s{\e_{\z}}{\,\cdot\,}\e_z-\frac{e^{2\lambda}}{2}\p{z}\left(\frac{e^{2\lambda}}{2}\right)\cdot+\left(-\frac{e^{2\lambda}}{2}\s{\D_{\p{z}}\e_{z}}{\,\cdot\,}+2\p{z}\left(\frac{e^{2\lambda}}{2}\right)\s{\e_z}{\,\cdot\,}\right)\e_{\z}\\	
	D_{\zeta}h(\e_z,\D_{\p{\z}}\e_{\z})&=\left(-\s{\e_z}{\D_{\p{\z}}\e_{\z}}\s{\e_{\z}}{\,\cdot\,}-\frac{e^{2\lambda}}{2}\s{\D_{\p{\z}}\e_{\z}}{\,\cdot\,}\right)\e_{\z}\\
	&=-\left(\p{\z}\left(\frac{e^{2\lambda}}{2}\right)\s{\e_{\z}}{\,\cdot\,}+\frac{e^{2\lambda}}{2}\s{\D_{\p{\z}}\e_{\z}}{\,\cdot\,}\right)\e_{\z}\\
	D_{\kappa}h(\zeta,\kappa)&=\left(\s{\bar{\zeta}}{\bar{\zeta}}\s{\zeta}{\,\cdot\,}-|\zeta|^2\s{\bar{\zeta}}{\,\cdot\,}\right)\zeta+\left(-|\zeta|^2\s{\zeta}{\,\cdot\,}+\s{\zeta}{\zeta}\s{\bar{\zeta}}{\,\cdot\,}\right)\bar{\zeta}\\
	D_{\kappa}h(\e_z,\w)&=-\frac{e^{2\lambda}}{2}\left(\s{\e_{\z}}{\,\cdot\,}\e_z+\s{\e_{z}}{\,\cdot\,}\e_{\z}\right)=-e^{2\lambda}\Re\left(\s{\e_{\z}}{\,\cdot\,}\e_z\right)\quad \text{(if the infinitesimal symmetries are real).}
\end{align*}
Furthermore, as $\s{\e_z}{\e_z}=\s{\e_{\z}}{\e_{\z}}=0$, we have
\begin{align}\label{cl0}
	D_{\omega}\H=0.
\end{align}
 Therefore, we obtain
\begin{align}\label{cl1}
	D_{\zeta}\H&=-D_{\zeta}f(\e_z)f(\e_z)^{-2}\left(|\e_z|^2\vec{\I}(\e_z,\e_{\z})\right)+f(\e_z)^{-1}\bigg(\s{\e_{\z}}{\,\cdot\,}\vec{\I}(\e_z,\e_{\z})+|\e_z|^2f(\e_z)^{-1}D_{\zeta}h(\e_z,\D_{\e_z}\e_{\z})\nonumber\\
	&-\s{\e_{z}}{\,\cdot\,}\vec{\I}(\e_{\z},\e_{\z})\bigg)\nonumber\\
	&=-e^{2\lambda}\s{\e_{\z}}{\,\cdot\,}\left(\frac{e^{4\lambda}}{4}\right)^{-2}\frac{e^{4\lambda}}{4}(2e^{-2\lambda}\vec{\I}(\e_z,\e_{\z}))+4e^{-4\lambda}\bigg\{\frac{e^{2\lambda}}{2}\s{\e_{\z}}{\,\cdot\,}\left(2e^{-2\lambda}\vec{\I}(\e_z,\e_{\z})\right)\nonumber\\
	&+\frac{e^{2\lambda}}{2}\left(\frac{e^{4\lambda}}{4}\right)^{-1}\left(-\frac{e^{4\lambda}}{4}\s{\H}{\,\cdot\,}\e_{\z}\right)-\frac{e^{2\lambda}}{2}\s{\e_z}{\,\cdot\,}\left(2e^{-2\lambda}\vec{\I}(\e_{\z},\e_{\z})\right)\bigg\}\nonumber\\
	&=-4e^{-2\lambda}\s{\e_{\z}}{\,\cdot\,}\vec{H}+2e^{-2\lambda}\left(\s{\e_{\z}}{\,\cdot\,}\vec{H}-\s{\H}{\,\cdot\,}\e_{\z}-\s{\e_z}{\,\cdot\,}\bar{\H_0}\right)\nonumber\\
	&=-2e^{-2\lambda}\left(\s{\e_{\z}}{\,\cdot\,}\H+\s{\H}{\,\cdot\,}\e_{\z}+\s{\e_z}{\,\cdot\,}\bar{\vec{H}_0}\right).
\end{align}
The last identity is
\begin{align}\label{cl2}
	D_{\chi}\H&=4e^{-4\lambda}\left(|\e_z|^2\,\cdot\,+|\e_z|^2\left(\frac{e^{4\lambda}}{4}\right)^{-1}D_{\kappa}h(\e_z,\D_{\e_z}\e_{\z})\right)\nonumber\\
	&=4e^{-4\lambda}\left(\frac{e^{2\lambda}}{2}\,\cdot\,-\frac{e^{2\lambda}}{2}\left(\frac{e^{4\lambda}}{4}\right)^{-1}e^{2\lambda}\Re\left(\s{\e_{\z}}{\,\cdot\,}\e_z\right)\right)\nonumber\\
	&=2e^{-2\lambda}\left(\,\cdot\,-4e^{-2\lambda}\Re\left(\s{\e_{\z}}{\,\cdot\,}\e_z\right)\right).
\end{align}
Thanks to \eqref{cl0}, \eqref{cl1} and \eqref{cl2}, we obtain
\begin{align}\label{conservationH}
	\left\{\begin{alignedat}{1}
	D_{\zeta}\H&=-2e^{-2\lambda}\left(\s{\e_{\z}}{\,\cdot\,}\H+\s{\H}{\,\cdot\,}\e_{\z}+\s{\e_z}{\,\cdot\,}\bar{\vec{H}_0}\right)\\
	D_{\chi}\H&=2e^{-2\lambda}\left(\,\cdot\,-4e^{-2\lambda}\Re\left(\s{\e_{\z}}{\,\cdot\,}\e_z\right)\right)\\
	D_{\omega}\H&=0
	\end{alignedat}\right.
	.
\end{align}
Now we see that
\begin{align*}
	K_gd\vg&=(\det g)^{-1}\left(\s{\vec{\I}(\e_1,\e_1)}{\vec{\I}(\e_2,\e_2)}-|\vec{\I}(\e_1,\e_2)|^2\right)\sqrt{\det g}\,dx_1\wedge dx_2\\
	&=\frac{1}{2}\left(|\e_z|^4-|\s{\e_z}{\e_z}|^2\right)^{-\frac{1}{2}}\Big(\s{2\,\Re\vec{\I}(\e_z,\e_z)+2\,\vec{\I}(\e_z,\e_{\z})}{-2\,\Re \vec{\I}(\e_z,\e_z)+2\,\vec{\I}(\e_z,\e_{\z})}\\
	&-|2\,\Im\vec{\I}(\e_z,\e_z)|^2\Big)dx_1\wedge dx_2\\
	&=2\left(|\e_z|^4-|\s{\e_z}{\e_z}|^2\right)^{-\frac{1}{2}}\left(|\vec{\I}(\e_z,\e_{\z})|^2-|\vec{\I}(\e_z,\e_z)|^2\right)dx_1\wedge dx_2.
\end{align*}
As
\begin{align*}
	&\vec{\I}(\e_z,\e_{\z})=\D_{\p{z}}\e_{\z}+f(\e_z)^{-1}+f(\e_z)^{-1}h(\e_z,\D_{\p{z}}\e_{\z})=\chi+f(\zeta)^{-1}+f(\zeta)^{-1}h(\zeta,\chi)\\
	&\vec{\I}(\e_z,\e_z)=\omega+f(\zeta)^{-1}h(\zeta,\omega),
\end{align*}
we deduce that
\begin{align*}
	D_{\zeta}\vec{\I}(\e_z,\e_{\z})&=-D_{\zeta}f(\e_z)f(\e_z)^{-2}h(\e_z,\D_{\p{z}}\e_{\z})+f(\e_z)^{-1}D_{\zeta}h(\e_z,\D_{\p{z}}\e_{\z})\\
	&=\left(\frac{e^{4\lambda}}{4}\right)^{-1}\left(-\frac{e^{4\lambda}}{4}\s{\vec{H}}{\,\cdot\,}\e_{\z}\right)
	=-\s{\H}{\,\cdot\,}\e_{\z}\\
	D_{\zeta}|\vec{\I}(\e_z,\e_{\z})|^2&=2\s{D_{\zeta}\vec{\I}(\e_z,\e_{\z})}{\vec{\I}(\e_z,\e_{\z}}=-2\s{\H}{\,\cdot\,}\s{\e_{\z}}{\frac{e^{2\lambda}}{2}\H}=0.
\end{align*}
Therefore, we have
\begin{align}\label{conservationK}
    \left\{\begin{alignedat}{1}
	   D_{\zeta}\left(\star K_gd\vg\right)&=-2\s{\e_{\z}}{\,\cdot\,}K_g+4(\p{z}\lambda)\s{\,\cdot\,}{\bar{\H_0}}=-2K_g\e_{\z}+4(\p{z}\lambda)\bar{\H_0}\\
       D_{\chi}\left(\star K_gd\vg\right)&=4\H\\
       D_{\omega}\left(\star K_gd\vg\right)&=-2\bar{\H_0}
    \end{alignedat}\right.
    .
\end{align}
Now define
\begin{align*}
	L_0(\phi,d\phi,\D\phi)=|\H|^2 (\det g)^{\frac{1}{2}}=2|\H|^2f(\e_z)^{\frac{1}{2}}
\end{align*}
We have by \eqref{conservationH}
\begin{align}\label{cl22}
	D_{\zeta}L_0&=2e^{2\lambda}\s{D_{\zeta}\H}{\H}+2D_{\zeta}f(\e_z)f(\e_z)^{-\frac{1}{2}} |\H|^2\nonumber\\
	&=-4\s{\s{\e_{\z}}{\,\cdot\,}\H+\s{\H}{\,\cdot\,}\e_{\z}+\s{\e_z}{\,\cdot\,}\bar{\vec{H}_0}}{\H}
	+2|\H|^2\s{\e_{\z}}{\,\cdot\,}\nonumber\\
	&=-2|\H|^2\s{\e_{\z}}{\,\cdot\,}-4\s{\H}{\bar{\H_0}}\s{\e_z}{\,\cdot\,}\nonumber\\
	&=-2|\H|^2\e_{\z}-4\s{\H}{\bar{\H_0}}\e_z
\end{align}
and
\begin{align}\label{cl11}
	D_{\chi}L_0&
	=2e^{2\lambda}\s{D_{\chi}\H}{\H}=4\s{\,\cdot\,-4e^{-2\lambda}\Re\left(\s{\e_{\z}}{\,\cdot\,}\e_z\right)}{\H}=4\s{\H}{\,\cdot\,}=4\H,
\end{align}
while
\begin{align}\label{cl00}
	D_{\omega}L_0&=0.
\end{align}
Therefore by \eqref{cl22}, \eqref{cl11} and \eqref{cl00}, we have
\begin{align}\label{conservationL0}
\left\{\begin{alignedat}{1}
	D_{\zeta}L_0&=-2|\H|^2\e_{\z}-4\s{\H}{\bar{\H_0}}\e_z\\
	D_{\chi}L_0&=4\H\\
	D_{\omega}L_0&=0.
	\end{alignedat}\right.
	.
\end{align}
If $L=\star(|H|^2-K_g)d\vg=\star\left(|\H_0|^2d\vg\right)=L_0-\star \left(K_gd\vg\right)$, by \eqref{conservationK} and \eqref{conservationL0}, we have
\begin{align}\label{conservationL}
\left\{\begin{alignedat}{1}
	D_{\zeta}L&=-2|\H|^2\e_{\z}-4\s{\H}{\bar{\H_0}}\e_z+2K_g\e_{\z}-4(\p{z}\lambda)\bar{\H_0}=-2|\H_0|^2\e_{\z}-4\s{\H}{\bar{\H_0}}\e_z-4(\p{z}\lambda)\bar{\H_0}\\
D_{\omega}L&=2\bar{\H}_0\\
D_{\chi}L&=4\H-4\H=0.
\end{alignedat}\right.
\end{align}
Therefore, for any infinitesimal (real) symmetry $\vec{X}$, Noether's theorem shows that (as $D_{\chi}L=0$)
\begin{align*}
	\Re\left(\D_{\p{z}}\left(D_{\zeta}L\cdot\vec{X}-\D_{\p{z}}(\D_{\omega}L)\cdot\vec{X}+D_{\omega}L\cdot\D_{\p{z}}\vec{X}\right)\right)=0
\end{align*}
which gives
\begin{align}\label{noetherwillmore}
	\Re\left(\D_{\p{\z}}\left(\left(-2|\H_0|^2\e_z-4\s{\H}{\H_0}\e_{\z}-4(\p{z}\lambda)\H_0-2\D_{\p{\z}}\H_0\right)\cdot\vec{X}+2\H_0\cdot\D_{\p{\z}}\vec{X}\right)\right)=0.
\end{align}
As
\begin{align*}
	\D_{\p{\z}}\H_0&=\D_{\p{\z}}(e^{-2\lambda}e^{2\lambda}\H_0)=-2(\p{\z}\lambda)\H_0+e^{-2\lambda}\D_{\p{\z}}(e^{2\lambda}\H_0)\\
	&=-2(\p{\z}\lambda)\H_0+g^{-1}\otimes\bar{\partial}^N\h_0+\D_{\p{\z}}^\top\H_0.
\end{align*}
and
\begin{align*}
	\D_{\p{\z}}^\top\H_0=-|\H_0|^2\e_z-\s{\H}{\H_0}\e_{\z}
\end{align*}
we have
\begin{align*}
	|\H_0|^2\e_z+2\s{\H}{\H_0}\e_{\z}+2(\p{\z}\lambda)\H_0+\D_{\p{\z}}\H_0&=\s{\H}{\H_0}\e_{\z}+g^{-1}\otimes \bar{\partial}^N\h_0\\
	&=g^{-1}\otimes \left(\bar{\partial}^N-\bar{\partial}^\top\right)\h_0-|\h_0|_{WP}^2\, \partial\phi.
\end{align*}
which finally gives
\begin{align}\label{noetherwillmore2}
	d\,\Im \left(\left(g^{-1}\otimes \left(\bar{\partial}^N-\bar{\partial}^\top\right)\h_0-|\h_0|_{WP}^2\partial\phi\right)\cdot\vec{X}-g^{-1}\otimes\h_0\cdot\bar{\partial}\vec{X}\right)=0
\end{align}
The invariance by translation gives (taking $\vec{X}=\vec{C}\in\R^n$)
\begin{align*}
	d\,\Im\left(g^{-1}\otimes \left(\bar{\partial}^N-\bar{\partial}^\top\right)\h_0-|\h_0|_{WP}^2\partial\phi\right),
\end{align*}
while the invariance dilatation invariance corresponds to $\vec{X}=\phi$, so (as $\s{\h_0}{\p{\z}\phi}=0$)
\begin{align*}
	d\,\Im\left(\left(g^{-1}\otimes \left(\bar{\partial}^N-\bar{\partial}^\top\right)\h_0-|\h_0|_{WP}^2\partial\phi\right)\cdot\phi\right)=0.
\end{align*}
The invariance by rotation corresponds to $\vec{X}=\vec{C}\wedge\phi$ (where $\vec{C}\in\R^n$ constant), and implies that
\begin{align*}
	d\,\Im\left(\phi\wedge\left(g^{-1}\otimes \left(\bar{\partial}^N-\bar{\partial}^\top\right)\h_0+|\h_0|_{WP}^2\partial\phi\right)+g^{-1}\otimes\h_0\wedge\bar{\partial}\phi\right)=0
\end{align*}
and finally, the invariance by the composition of translations and inversions, corresponds to $\vec{X}=|\phi|^2\vec{C}-2\s{\phi}{\vec{C}}\phi$, and we obtain (as $\s{\h_0}{\p{\z}\phi}=0$)
\begin{align*}
	d\,\Im\bigg(&|\phi|^2g^{-1}\otimes \left(\bar{\partial}^N-\bar{\partial}^\top\right)\h_0-|\h_0|_{WP}^2\partial\phi-2\s{\phi}{g^{-1}\otimes \left(\bar{\partial}^N-\bar{\partial}^\top\right)\h_0-|\h_0|_{WP}^2\partial\phi}\\
	&-g^{-1}\otimes\h_0\otimes \bar{\partial}|\phi|^2
	+2\,g^{-1}\otimes\s{\h_0}{\phi}\otimes\bar{\partial}\phi\bigg)=0.
\end{align*}

In particular, the four residues are
\begin{align}\label{residues1234}
\left\{\begin{alignedat}{1}
	\vec{\gamma}_0(\phi,p)&=\frac{1}{4\pi}\,\Im\int_{\gamma}g^{-1}\otimes\left(\bar{\partial}^N-\bar{\partial}^\top\right)\h_0-|\h_0|_{WP}^2\,\partial\phi\\
    \vec{\gamma}_1(\phi,p)&=\frac{1}{4\pi}\,\Im\int_{\gamma}\phi\wedge\left( g^{-1}\otimes\left(\bar{\partial}^N-\bar{\partial}^\top\right)\h_0-|\h_0|_{WP}^2\,\partial\phi\right)+g^{-1}\otimes\h_0\wedge\bar{\partial}\phi\\
    \vec{\gamma}_2(\phi,p)&=\frac{1}{4\pi}\,\Im\int_{\gamma}\phi\cdot\left( g^{-1}\otimes\left(\bar{\partial}^N-\bar{\partial}^\top\right)\h_0-|\h_0|_{WP}^2\,\partial\phi\right)\\
    \vec{\gamma}_3(\phi,p)&=\frac{1}{4\pi}\,\Im\int_{\gamma}|\phi|^2\left(g^{-1}\otimes\left(\bar{\partial}^N-\bar{\partial}^\top\right)\h_0-|\h_0|_{WP}^2\,\partial\phi\right)\\
    &-2\s{\phi}{g^{-1}\otimes\left(\bar{\partial}^N-\bar{\partial}^\top\right)\h_0-|\h_0|_{WP}^2\,\partial\phi}\phi-g^{-1}\otimes\h_0\otimes \bar{\partial}|\phi|^2+2\,g^{-1}\otimes\s{\h_0}{\phi}\otimes\bar{\partial}\phi
\end{alignedat}\right.
\end{align}
where $p\in \Sigma^2$ and $\gamma$ is an arbitrary smooth closed curve enclosing the point $p$.

To get some intuition on these residues, it is useful to look at the simplest possible case of Willmore surfaces : the minimal surfaces in $\R^n$ with embedded ends.

\begin{prop}\label{fluxresidue}
	Let $\phi:\Sigma^2\setminus\ens{p_1,\cdots,p_m}\rightarrow\R^n$ be a complete minimal surface with embedded ends and finite total curvature. Then its first three residues are always zero, and its fourth one corresponds to 
	its flux, that is for all $j=1,\cdots,m$
	\begin{align*}
	\vec{\gamma}_3(\phi,p_j)&=\frac{1}{4\pi}\,\Im\int_{\gamma}\mathscr{I}_{\phi}\left(g^{-1}\otimes\left(\bar{\partial}^N-\bar{\partial}^\top\right)\h_0-|\h_0|_{WP}^2\,\partial\phi\right)-g^{-1}\otimes\h_0\otimes \bar{\partial}|\phi|^2+2\,g^{-1}\otimes\s{\h_0}{\phi}\otimes\bar{\partial}\phi\\
	&=\frac{1}{4\pi}\Im\,\int_{\gamma}\partial\phi.
    \end{align*}
where for all tangent $\vec{w}\in\C^n$, 
\begin{align*}
\mathscr{I}_{\phi}(\vec{w})=|\phi|^2\vec{w}-2\s{\phi}{\vec{w}}\phi.
\end{align*}
\end{prop}
\begin{proof}
    As $\H=0$, and 
    \begin{align*}
    	g^{-1}\otimes\left(\bar{\partial}^N-\partial^\top\right)\h_0-|\h_0|^2_{WP}\,\partial\phi=\partial^N\H+g^{-1}\otimes\s{\H}{\h_0}\otimes\bar{\partial}\phi=0
    \end{align*}
    the first three residues vanish, and by \eqref{residues1234} the fourth residue $\vec{\gamma}_3$ reduces to
    \begin{align}\label{reducegamma3}
    	\vec{\gamma}_3(\phi,p_j)=-\frac{1}{4\pi}\Im\int_{\gamma}g^{-1}\otimes \h_0\otimes\bar{\partial}|\phi|^2-2\,g^{-1}\otimes\s{\h_0}{\phi}\otimes\bar{\partial}\phi.
    \end{align}
    From now on, we assume $n=3$ for simplicity.
     Let us fix some $1\leq j\leq m$ and let $p_j$ be an end of $\phi$. Taking some complex chart sending $p_j$ to $0$, we can suppose that $\phi$ is parametrised by the punctured unit disk $D^2\setminus\ens{0}$. Then can suppose up to rotation that the normal $\n$ at $p$ is $\n(p_i)=(0,0,1)$, and by the Weierstrass parametrisation, this is easy to see that the embeddedness of $p_i$ implies that there exists $\alpha>0$, $\beta\in \R$ such that
    \begin{align*}
       \phi(z)=\Re\left(\int_{\ast}^z\frac{\alpha}{w^2}dw,\int_{\ast}^z\frac{i\alpha}{w^2}dw,\int_{\ast}^{z}\frac{\beta}{w^{}}dw\right)+O(1)
    \end{align*}
    for some $\alpha>0$, $\beta\in\R$. In particular, we have
    \begin{align}\label{refgamma3}
    \left\{\begin{alignedat}{1}
          |\phi(z)|^2&=\frac{\alpha^2}{|z|^{2}}+O(1),\quad \p{\z}|\phi(z)|^2=-\frac{\alpha^2 z}{|z|^{4}}\left(1+O(|z|^2)\right)\\
    e^{2\lambda}&=2|\p{z}\phi(z)|^2=\frac{\alpha^2}{|z|^{4}}\left(1+O(|z|)^2\right)\\
    \h_0(z)&=\left(0,0,\beta\frac{dz^2}{z^2}\right)+O(1)
    \end{alignedat}\right.
  \end{align}
   Indeed, we have
   \begin{align*}
   \left\{\begin{alignedat}{1}
      	\p{z}\phi&=\frac{1}{2}\left(\frac{\alpha}{z^2},\frac{i\alpha}{z^2},\frac{\beta}{z}\right)+O(1),\quad 
   \p{z}^2\phi=-\frac{1}{2}\left(\frac{2\alpha}{z^3},\frac{2i\alpha}{z^3},\frac{\beta}{z^2}\right)+O\left(\frac{1}{|z|}\right)\\
   2(\p{z}\lambda)&=e^{-2\lambda}\p{z}(e^{2\lambda})=\frac{|z|^4}{\alpha^2}\left(-2\frac{\alpha^2}{z^3\z^2}\right)+O(1)=\frac{-2}{z}+O(1)\\
   \h_0&=2\left(\p{z}^2\phi-2(\p{z}\lambda)\p{z}\phi\right)=\left(0,0,\beta\frac{dz^2}{z^2}\right)+O\left(\frac{1}{|z|}\right)
   \end{alignedat}\right.
   \end{align*}
   Therefore,  by \eqref{refgamma3}, we have
   \begin{align}\label{reducegamma32}
   	g^{-1}\otimes \bar{\partial}|\phi|^2\otimes\h_0&=\left(0,0,\frac{|z|^4}{\alpha^2}\left(-\frac{\alpha^2 z}{|z|^4}\right)\frac{\beta}{z^2}dz\right)+O(1)
   	=\left(0,0,-\beta\frac{dz}z\right)+O(1)
   \end{align}
   and as $\s{\h_0}{\phi}=O(\log|z||z|^{-2})$, we have
   \begin{align}\label{unautre}
   	g^{-1}\otimes\s{\h_0}{\phi}\otimes \bar{\partial}|\phi|=O(\log|z|)
   \end{align}
   Therefore, putting together \eqref{reducegamma3}, \eqref{reducegamma32} and \eqref{unautre} the fourth residue is equal to
   \begin{align*}
   	\vec{\gamma_3}(\phi,p_j)= \left(0,0,\frac{1}{4\pi}\Im\int_{\gamma}\beta\frac{dz}{z}\right)=\frac{1}{2}(0,0,\beta).
   \end{align*}
   which coincides exactly with the \textit{flux}, and this  shows the identity (for minimal surfaces with embedded planar ends)
   \begin{align*}
   	\Im\,\int_{\gamma}g^{-1}\otimes\h_0\otimes\bar{\partial}|\phi|^2-2\,g^{-1}\otimes\s{\h_0}{\phi}\otimes\bar{\partial}\phi=\Im\,\int_{\gamma}\partial\phi,
   \end{align*}
   concluding the proof of the proposition.
\end{proof}

    This proposition suggests that the first and fourth residues are exchanged when we apply a conformal transformation, as in the case of the inverted catenoid, we find the same residue (this is immediate from its Weierstrass parametrisation, but for a proof, see for example \cite{beriviere}).    
    
    \subsection{Correspondence between residues and conformal invariance}\label{correspondance}

    Obviously, the four residues are invariant by rotations, translations and dilatations. However, for inversions with centre \textit{inside} $\phi(\Sigma^2)$, this is not the case as the previous example showed for inversions of minimal surfaces. There is nevertheless a simple rule under which these quantities transform, which is detailed below.
    
    \begin{theorem}[Residue correspondence]\label{galois}
    	Let $\phi:\Sigma^2\rightarrow\R^n$ be a Willmore surface and let $\iota:\R^n\setminus\ens{0}\rightarrow\R^n\setminus\ens{0}$ be the inversion centred at zero. If $\vec{\Psi}=\iota\circ \phi:\Sigma^2\setminus \phi^{-1}(\ens{0})\rightarrow\R^n$, for all $p\in \Sigma^2$, we have
    	\begin{align}\label{residusnonintro}
    	\left\{\begin{alignedat}{1}
    	    		\vec{\gamma}_0(\phi,p)&=\vec{\gamma}_3(\vec{\Psi},p)\\
    	            \vec{\gamma}_1(\phi,p)&=\vec{\gamma}_1(\vec{\Psi},p)\\
    	            \vec{\gamma}_2(\phi,p)&=-\vec{\gamma}_2(\vec{\Psi},p)\\
    	            \vec{\gamma}_3(\phi,p)&=\vec{\gamma}_0(\vec{\Psi},p).
    	\end{alignedat}\right.
    	\end{align}
    	where the residues $\vec{\gamma}_0,\vec{\gamma}_1,\vec{\gamma}_2,\vec{\gamma}_3$ are given by \eqref{residues1234}.
    \end{theorem}

\begin{proof}    
    If $\vec{f}_z=\p{z}\vec{\Psi}$, and $\e_z=\p{z}\phi$, we have
    \begin{align}\label{inv00}
    \left\{\begin{alignedat}{1}
        	\vec{f}_z&=\p{z}\vec{\Psi}=|\vec{\Psi}|^2\e_z-2\s{\vec{\Psi}}{\e_z}\vec{\Psi}\\
            \D_{\p{z}}f_z&=|\vec{\Psi}|^2\left(\D_{\p{z}}\e_{z}-4\s{\e_z}{\vec{\Psi}}\e_z-2\s{\e_z}{\e_z}\vec{\Psi}\right)-2\s{\vec{\Psi}}{\D_{\p{z}}\e_z}\vec{\Psi}+8\s{\vec{\Psi}}{\e_z}^2\vec{\Psi}\\
            \D_{\p{\z}}f_z&=|\vec{\Psi}|^2\left(\D_{\p{z}}\e_{\z}-4\,\Re\left(\s{\vec{\Psi}}{\e_z}\e_{\z}\right)-2|\e_z|^2\vec{\Psi}\right)-2\s{\vec{\Psi}}{\D_{\p{z}}\e_{\z}}\vec{\Psi}+8|\s{\vec{\Psi}}{\e_z }|^2\vec{\Psi}.
    \end{alignedat}\right.
    \end{align}
    We also write 
    \begin{align}\label{confpsi}
    	e^{2\mu}=2\s{\p{z}\vec{\Psi}}{\p{\z}\vec{\Psi}}=2\s{\f_z}{\f_{\z}}
    \end{align}
    for the conformal parameter of the immersion $\vec{\Psi}:\Sigma^2\setminus\phi^{-1}(\ens{0})\rightarrow \R^n$.
    Then, the \textit{pointwise} invariance of Willmore energy implies if $L=\star\big((|\H|^2-K_g)d\vg\big)$ that
    \begin{align*}
    	L(\phi,\phi_z,\D_{\p{z}}\phi_z,\D_{\p{\z}}\phi_z)&=L(\vec{\Psi},\vec{\Psi}_z,\D_{\p{z}}\vec{\Psi}_z,\D_{\p{\z}}\vec{\Psi}_{z})\\
    	&=L(\vec{\Psi},F_1(\e_z),F_2(\e_z,\D_{\p{z}}\e_z),F_3(\e_z,\D_{\p{\z}}\e_z))
    \end{align*}
    where thanks to \eqref{inv00}
    \begin{align}\label{bigf123}
    \left\{
    \begin{alignedat}{1}
    F_1(\zeta)&=|\vec{\Psi}|^2\zeta-2\s{\vec{\Psi}}{\zeta}\vec{\Psi}\\
    F_2(\zeta,\omega)&=|\vec{\Psi}|^2\left(\omega-4\s{\zeta}{\vec{\Psi}}\zeta-2\s{\zeta}{\zeta}\vec{\Psi}\right)-2\s{\vec{\Psi}}{\omega}\vec{\Psi}+8\s{\vec{\Psi}}{\zeta}^2\vec{\Psi}\\
    F_3(\zeta,\chi)&=|\vec{\Psi}|^2\left(\chi-2\left(\s{\vec{\Psi}}{\zeta}\bar{\zeta}+\s{\vec{\Psi}}{\bar{\zeta}}\zeta\right)-2|\zeta|^2\vec{\Psi}\right)-2\s{\vec{\Psi}}{\chi}\vec{\Psi}+8|\s{\vec{\Psi}}{\zeta}|^2\vec{\Psi}.
    \end{alignedat}\right.
    \end{align}
    Therefore, writing $L(\phi)=L(\phi,\phi,\phi_z,\D_{\p{z}}\phi_z,\D_{\p{\z}}\phi_z)$ and $L(\vec{\Psi})=L(\vec{\Psi},\vec{\Psi}_z,\D_{\p{z}}\vec{\Psi}_z,\D_{\p{\z}}\vec{\Psi}_{z})$, we have
    \begin{align*}
    \left\{\begin{alignedat}{1}
    	&D_{\zeta}L(\phi)=D_{\zeta}L(\vec{\Psi})\circ D_{\zeta}F_1(\e_z)+D_{\omega}L(\vec{\Psi})\circ D_{\zeta}F_2(\e_z,\D_{\p{z}}\e_z)+D_{\chi}L(\vec{\Psi})\circ D_{\zeta}F_3(\e_z,\D_{\p{z}}\e_z)\\
        &D_{\omega}L(\phi)=D_{\omega}L(\vec{\Psi})\circ D_{\omega}F(\e_z,\D_{\p{z}}\e_z)\\
        &D_{\chi}L(\phi)=D_{\chi}L(\vec{\Psi})\circ D_{\chi}F_3(\e_z,\D_{\p{\z}}\e_z).
    \end{alignedat}\right.
    \end{align*}
    So we have
    \begin{align}\label{tc1}
    \left\{\begin{alignedat}{1}
        	D_{\zeta}F_1(\zeta)&=|\vec{\Psi}|^2\,\cdot\,-2\s{\vec{\Psi}}{\,\cdot\,}\vec{\Psi}\\
            D_{\zeta}F_2(\zeta,\omega)&=-4|\vec{\Psi}|^2\left(\s{\,\cdot\,}{\vec{\Psi}}\zeta+\s{\zeta}{\vec{\Psi}}\,\cdot\,+\s{\zeta}{\,\cdot\,}\vec{\Psi}\right)+16\s{\vec{\Psi}}{\zeta}\s{\vec{\Psi}}{\,\cdot\,}\vec{\Psi}\\
            D_{\omega}F_2(\zeta,\omega)&=|\vec{\Psi}|^2\,\cdot\,-2\s{\vec{\Psi}}{\,\cdot\,}\vec{\Psi}\\
            D_{\zeta}F_3(\zeta,\chi)&=-2|\vec{\Psi}|^2\left(\s{\vec{\Psi}}{\,\cdot\,}\bar{\zeta}+\s{\vec{\Psi}}{\bar{\zeta}}\,\cdot\,+\s{\bar{\zeta}}{\,\cdot\,}\vec{\Psi}\right)+8\s{\vec{\Psi}}{\bar{\zeta}}\s{\vec{\Psi}}{\,\cdot\,}\vec{\Psi}\\
            D_{\chi}F_3(\zeta,\chi)&=|\vec{\Psi}|^2\,\cdot\,-2\s{\vec{\Psi}}{\,\cdot\,}\vec{\Psi}.
    \end{alignedat}\right.
    \end{align}
        We recall that if $\zeta=\e_z$, $\omega=\D_{\p{z}}\e_z$, $\chi=\D_{\p{\z}}\e_z$, by \eqref{expressionH}
    \begin{align*}
    	\H_{\phi}=f(\zeta)^{-1}\left(|\zeta|^2\chi+|\zeta|^2f(\zeta)^{-1}h(\zeta,\chi)-\Re\left(\s{\zeta}{\zeta}\bar{\omega}\right)-f(\zeta)^{-1}\Re\left(\s{\zeta}{\zeta}h(\zeta,\bar{\omega})\right)\right)
    \end{align*}
    while
    \begin{align*}
    	\H_{\vec{\Psi}}&=f(F_1(\zeta))^{-1}\Big\{|F_1(\zeta)|^2F_3(\zeta,\chi)+|F_1(\zeta)|^2f(F_1(\zeta))^{-1}h(F_1(\zeta),F_3(\zeta,\chi))\\
    	&-\Re\left(\s{F_1(\zeta)}{F_1(\zeta)}h(F_1(\zeta),\bar{F_2(\zeta,\omega)})\right)-f(F_1(\zeta))^{-1}\Re\left(\s{F_1(\zeta)}{F_1(\zeta)}h(F_1(\zeta),\bar{F_2(\zeta,\omega)})\right)\Big\}.
    \end{align*}
    In the forthcoming computations, we will always make the following osculating hypothesis that after taking differentiation, on evaluates at points $\zeta$ such that
    $\s{\zeta}{\zeta}=\s{F(\zeta)}{F(\zeta)}=0$, which is legitimate as we apply conformal transformations. We have
    \begin{align*}
    	D_{\zeta}\left(f(F_1(\zeta))\right)&=D_\zeta f(F_1(\zeta))\circ D_{\zeta}F_1(\zeta)=2\left(|F_1(\zeta)|^2\s{\bar{F_1}(\zeta)}{|\vec{\Psi}|^2\,\cdot\,-2\s{\vec{\Psi}}{\,\cdot\,}\vec{\Psi}}\right)\\
    	D_{\zeta}\left(f(F_1(\e_z))\right)&=e^{2\mu}\left(|\vec{\Psi}|^2f_{\z}-2\s{\vec{\Psi}}{f_{\z}}\vec{\Psi}\right)
    \end{align*}
    so
    \begin{align*}
    	D_{\zeta}|F_1(\e_z)|^2&=\s{\bar{F_1(\e_z)}}{D_{\zeta}F_1(\e_z)}=\s{f_{\z}}{|\vec{\Psi}|^2\,\cdot\,-2\s{\vec{\Psi}}{\,\cdot\,}\vec{\Psi}}\\
    	&=|\vec{\Psi}|^2f_{\z}-2\s{\vec{\Psi}}{\f_{\z}}\vec{\Psi}.
    \end{align*}
    Therefore we define
    \begin{align*}
    	\mathscr{I}_{\vec{\Psi}}(\vec{X})=|\vec{\Psi}|^2\vec{X}-2\s{\vec{\Psi}}{\vec{X}}\vec{\Psi}
    \end{align*}
    to obtain
    \begin{align}\label{inv0}
    \left\{
    \begin{alignedat}{1}
    	D_{\zeta}(f(F_1(\e_z))^{-1})&=-D_{\zeta}(f(F_1(\e_z)))f(F(\e_z))^{-2}=-16e^{-6\mu}\mathscr{I}_{\vec{\Psi}}(\f_{\z})\\
    	D_{\zeta}|F_1(\e_z)|^2&=\mathscr{I}_{\vec{\Psi}}(\f_{\z}).
    	\end{alignedat}\right.
    \end{align}
    Also, we remark that we only need to count the normal parts of the derivatives of $\H$, as they will be multiplied by $\H$ (coming from $|\H|^2$). Therefore, as we also have $h(F_1(\e_z),F_3(\e_z))=0$, we obtain
    \begin{align}\label{s1}
    	D_{\zeta}\H_{\vec{\Psi}}&=-16e^{-6\mu}\s{\mathscr{I}_{\vec{\Psi}}(\f_{\z})}{\,\cdot\,}\left(\frac{e^{4\mu}}{4}\H_{\vec{\Psi}}\right)+4e^{-4\mu}\bigg(D_{\zeta}(|F_1(\e_z)|^2)F_3(\e_z,\D_{\p{z}}\e_{\z})+|F_1(\e_z)|^2 D_{\zeta}F_3(\e_z,\D_{\p{z}}\e_{\z})\\
    	&+|F_1(\e_z)|^2f(F_1(\e_z))^{-1}D_{\zeta}(h(F_1(\e_z),F_3(\e_z,\D_{\p{z}}\e_{\z})))\circ D_{\zeta}F_1(\e_z)\\
    	&+D_{\kappa}h(F_1(\e_z),F_3(\e_z,\D_{\p{z}}\e_{\z}))\circ D_{\chi}F_3(\e_z,\D_{\p{z}}\e_z)-\s{D_{\zeta}F_1(\e_z)}{\f_z
    	}\vec{\I}(\f_{\z},\f_{\z})\bigg)\\
        &=-4e^{-2\mu}\s{\mathscr{I}_{\vec{\Psi}}(\f_{\z})}{\,\cdot\,}\H_{\vec{\Psi}}+4e^{-4\mu}\Big\{\mathrm{(I)}+\mathrm{(II)}+\mathrm{(III)}+\mathrm{(IV)}+\mathrm{(V)}\Big\}
    \end{align}
As
    \begin{align*}
    	\left\{\begin{alignedat}{1}
    	D_{\zeta}|F_1(\e_z)|^2&=\mathscr{I}_{\vec{\Psi}}(\f_{\z})\\
    	F_3(\e_{{z}},\D_{\p{z}}{\e_{{\z}}})&=\vec{\I}(\f_z,\f_{\z})=\frac{e^{2\mu}}{2}\H_{\vec{\Psi}},
    	\end{alignedat}\right.
    \end{align*}
    we obtain
    \begin{align}\label{invI}
    	\mathrm{(I)}=D_{\zeta}(|F_1(\e_z)|^2)F_3(\e_z,\D_{\p{z}}\e_{\z})=\frac{e^{2\mu}}{2}\s{\mathscr{I}_{\vec{\Psi}}(\f_{\z})}{\,\cdot\,}\H_{\vec{\Psi}}.
    \end{align}
    Then we have by \eqref{tc1}
    \begin{align*}
    	D_{\zeta}F_3(\zeta,\chi)&=-2|\vec{\Psi}|^2\left(\s{\vec{\Psi}}{\,\cdot\,}\e_{{\z}}+\s{\vec{\Psi}}{\e_{\z}}\,\cdot\,+\s{\e_{\z}}{\,\cdot\,}\vec{\Psi}\right)+8\s{\vec{\Psi}}{\e_{\z}}\s{\vec{\Psi}}{\,\cdot\,}\vec{\Psi}
    \end{align*}
    and as
    \begin{align*}
    	|F_{1}(\e_z)|^2=|\f_z|^2=\frac{e^{2\mu}}{2},
    \end{align*}
    we obtain
    \begin{align}\label{invII}
    	\mathrm{(II)}&=|F_1(\e_z)|^2 D_{\zeta}F_3(\e_z,\D_{\p{z}}\e_{\z})=\frac{e^{2\mu}}{2}\bigg\{-2|\vec{\Psi}|^2\left(\s{\vec{\Psi}}{\,\cdot\,}\e_{{\z}}+\s{\vec{\Psi}}{\e_{\z}}\,\cdot\,+\s{\e_{\z}}{\,\cdot\,}\vec{\Psi}\right)+8\s{\vec{\Psi}}{\e_{\z}}\s{\vec{\Psi}}{\,\cdot\,}\vec{\Psi}\bigg\}\nonumber\\
    	&=e^{2\mu}\bigg\{-|\vec{\Psi}|^2\left(\s{\vec{\Psi}}{\,\cdot\,}\e_{{\z}}+\s{\vec{\Psi}}{\e_{\z}}\,\cdot\,+\s{\e_{\z}}{\,\cdot\,}\vec{\Psi}\right)+4\s{\vec{\Psi}}{\e_{\z}}\s{\vec{\Psi}}{\,\cdot\,}\vec{\Psi} \bigg\}
    \end{align}
    We see that 
    \begin{align*}
    	(D_{\zeta}h(\e_z,\D_{\p{z}}\e_z))^N=-|\f_z|^2\s{\f_{\z}}{\D_{\p{z}}\f_{\z}}(\,\cdot\,)^N=0,
    \end{align*}
    so
   \begin{align}\label{invIII}
    \mathrm{(III)}^N=0.
\end{align}
    As $(D_{\kappa}h)^N=0$,
    \begin{align}\label{invIV}
    \mathrm{(IV)}^N=0,
    \end{align}
    and
    \begin{align}\label{invV}
    	\mathrm{(V)}=-\s{\mathscr{I}_{\vec{\Psi}}(\f_z)}{\,\cdot\,}\vec{\I}(\f_{\z},\f_{\z})=-\frac{e^{2\mu}}{2}\s{\mathscr{I}_{\vec{\Psi}}(\f_z)}{\,\cdot\,}\bar{\H_{\vec{\Psi}}^0}.
    \end{align}
   Therefore, by \eqref{s1}, \eqref{tc1}, \eqref{invI}, \eqref{invII}, \eqref{invIII}, \eqref{invIV}
   \begin{align*}
   	&(D_{\zeta}\H_{\vec{\Psi}})^N=-4e^{-2\mu}\s{\mathscr{I}_{\vec{\Psi}}}{\,\cdot\,}\H_{\vec{\Psi}}+4e^{-4\mu}\bigg\{\frac{e^{2\mu}}{2}\s{\mathscr{I}_{\vec{\Psi}}}{\,\cdot\,}\H_{\vec{\Psi}}
   	+e^{2\mu}\Big(-|\vec{\Psi}|^2\left(\s{\vec{\Psi}}{\,\cdot\,}\e_{{\z}}+\s{\vec{\Psi}}{\e_{\z}}\,\cdot\,+\s{\e_{{\z}}}{\,\cdot\,}\vec{\Psi}\right)\\
   	&+4\s{\vec{\Psi}}{\e_{{\z}}}\s{\vec{\Psi}}{\,\cdot\,}\vec{\Psi}\Big) 
   	-\frac{e^{2\mu}}{2}\s{\mathscr{I}_{\vec{\Psi}}(\f_z)}{\,\cdot\,}\bar{\H_{\vec{\Psi}}^0}\bigg\}^N\\
   	&=-2e^{-2\mu}\bigg(\s{\mathscr{I}_{\vec{\Psi}}(\f_z)}{\,\cdot\,}\H_{\vec{\Psi}}+\s{\mathscr{I}_{\vec{\Psi}}(\f_z)}{\,\cdot\,}\bar{\vec{H}^0_{\vec{\Psi}}}\bigg)+4e^{-2\mu}\bigg\{-|\vec{\Psi}|^2\Big(\s{\vec{\Psi}}{\,\cdot\,}\e_{{\z}}+\s{\vec{\Psi}}{\e_{\z}}\left(\,\cdot\,\right)+\s{\e_{{\z}}}{\,\cdot\,}\vec{\Psi}\Big)\\
   	&+4\s{\vec{\Psi}}{\e_{{\z}}}\s{\vec{\Psi}}{\,\cdot\,}\vec{\Psi}\bigg\}^N
   \end{align*}
   As
   \begin{align*}
   	|\H_{\vec{\Psi}}|^2\left(\det g_{\vec{\Psi}}\right)^{\frac{1}{2}}=2f(F_1(\e_z))^{\frac{1}{2}}|\H_{\vec{\Psi}}|^2,
   \end{align*}
   we obtain the identity
   \begin{align}\label{dzetaH}
   	&D_{\zeta}\left(|\H_{\vec{\Psi}}|^2\left(\det g_{\vec{\Psi}}\right)^{\frac{1}{2}}\right)=2\frac{1}{2}e^{2\mu}\s{\mathscr{I}_{\vec{\Psi}}(\f_{\z})}{\,\cdot\,}\left(\frac{e^{4\mu}}{4}\right)^{-\frac{1}{2}}|\H_{\vec{\Psi}}|^2+4\left(\frac{e^{4\mu}}{4}\right)^{\frac{1}{2}}\s{(D_{\zeta}\H_{\vec{\Psi}})^\nperp}{\H_{\vec{\Psi}}}\nonumber\\
   	&=2|\vec{H}_{\vec{\Psi}}|^2\mathscr{I}_{\vec{\Psi}}(\f_{\z})-4\bs{\left(\s{\mathscr{I}_{\vec{\Psi}}(\f_z)}{\,\cdot\,}\H_{\vec{\Psi}}+\s{\mathscr{I}_{\vec{\Psi}}(\f_z)}{\,\cdot\,}\bar{\vec{H}^0_{\vec{\Psi}}}\right)}{\vec{H}_{\vec{\Psi}}}\nonumber\\
   	&+8\bs{\bigg\{-|\vec{\Psi}|^2\Big(\s{\vec{\Psi}}{\,\cdot\,}\e_{{\z}}+\s{\vec{\Psi}}{\e_{\z}}\left(\,\cdot\,\right)+\s{\e_{{\z}}}{\,\cdot\,}\vec{\Psi}\Big)+4\s{\vec{\Psi}}{\e_{{\z}}}\s{\vec{\Psi}}{\,\cdot\,}\vec{\Psi}\bigg\}^N}{\H_{\vec{\Psi}}}\nonumber\\
   	&=-2|\H_{\vec{\Psi}}|^2\mathscr{I}_{\vec{\Psi}}(\f_{\z})-4\s{\H_{\vec{\Psi}}}{\bar{\vec{H}_{\vec{\Psi}}^0}}\mathscr{I}_{\vec{\Psi}}(\f_{z})
   	-8|\vec{\Psi}|^2\s{\e_{\z}}{\H_{\vec{\Psi}}}\vec{\Psi}
   	-8|\vec{\Psi}|^2\s{\vec{\Psi}}{\e_{\z}}\H_{\vec{\Psi}}-8|\vec{\Psi}|^2\s{\vec{\Psi}}{\H_{\vec{\Psi}}}\e_{{\z}}\nonumber\\
   	&+32\s{\vec{\Psi}}{\e_{{\z}}}\s{\vec{\Psi}}{\H_{\vec{\Psi}}}\vec{\Psi}
   \end{align}
   As $\s{\f_z}{\f_z}=0$, we trivially obtain
   \begin{align}\label{domegaH}
   	D_{\omega}\left(|\H_{\vec{\Psi}}|^2\left(\det g_{\vec{\Psi}}\right)^{\frac{1}{2}}\right)=0
   \end{align}
   Finally, as $(D_{\kappa}h)^N=0$,
   \begin{align*}
   	(D_{\chi}\H_{\vec{\Psi}})^N=f(F_1(\e_z))^{-1}|F_1(\e_z)|^2D_{\chi}F_3(\e_z,\D_{\p{z}}\e_{\z})=2e^{-2\lambda}\mathscr{I}_{\vec{\Psi}}(\,\cdot\,).
   \end{align*}
   Therefore, we deduce that 
   \begin{align}\label{dchiH}
   	D_{\chi}\left(|\H_{\vec{\Psi}}|^2\left(\det g_{\vec{\Psi}}\right)^{\frac{1}{2}}\right)=2e^{2\mu}\s{D_{\chi}\H_{\vec{\Psi}}}{\H_{\vec{\Psi}}}=4\mathscr{I}_{\vec{\Psi}}(\H_{\vec{\Psi}}).
   \end{align}
   and putting together \eqref{dzetaH}, \eqref{domegaH} and \eqref{dchiH}, we have
   \begin{align}\label{confinv1}
    \left\{\begin{alignedat}{1}
       	D_{\zeta}\left(|\H_{\vec{\Psi}}|^2\left(\det g_{\vec{\Psi}}\right)^{\frac{1}{2}}\right)&=-2|\H_{\vec{\Psi}}|^2\mathscr{I}_{\vec{\Psi}}(\f_{\z})-4\s{\H_{\vec{\Psi}}}{\bar{\vec{H}_{\vec{\Psi}}^0}}\mathscr{I}_{\vec{\Psi}}(\f_{z})-8|\vec{\Psi}|^2\s{\e_{\z}}{\H_{\vec{\Psi}}}\vec{\Psi}\\
       	&-8|\vec{\Psi}|^2\s{\vec{\Psi}}{\e_{\z}}\H_{\vec{\Psi}}-8|\vec{\Psi}|^2\s{\vec{\Psi}}{\H_{\vec{\Psi}}}\e_{{\z}}
       	+32\s{\vec{\Psi}}{\e_{{\z}}}\s{\vec{\Psi}}{\H_{\vec{\Psi}}}\vec{\Psi}\\
    D_{\chi}\left(|\H_{\vec{\Psi}}|^2\left(\det g_{\vec{\Psi}}\right)^{\frac{1}{2}}\right)&=4\mathscr{I}_{\vec{\Psi}}(\H_{\vec{\Psi}})\\
    D_{\omega}\left(|\H_{\vec{\Psi}}|^2\left(\det g_{\vec{\Psi}}\right)^{\frac{1}{2}}\right)&=0
    \end{alignedat}\right.
   \end{align}
   Now recall the identity
   \begin{align}\label{s2}
   	\star\left( K_{g_{\vec{\Psi}}}d\mathrm{vol}_{g_{\vec{\Psi}}}\right)=2f(F_1(\e_z))^{-\frac{1}{2}}\bigg(|F_3(\e_z,\D_{\p{z}}\e_{\z})+f(F_1(\e_z))^{-1}h(F_1(\e_z),F_3(\e_z,\D_{\p{z}}\e_{\z})) |^2\\-|F_2(\e_z,\D_{\p{z}}\e_z)+f(F_1(\e_z))^{-1}h(F_1(\e_z),F_2(\e_z,\D_{\p{z}}\e_z))|^2\bigg).
   \end{align}
   We first compute thanks to \eqref{tc1}
   \begin{align*}
   	D_{\zeta}F_3(\e_z,\D_{\p{z}}\e_{\z})=-2|\vec{\Psi}|^2\left(\s{\vec{\Psi}}{\,\cdot\,}\e_{{\z}}+\s{\vec{\Psi}}{\e_{{\z}}}\,\cdot\, +\s{\e_{{\z}}}{\,\cdot\,}\vec{\Psi}\right)+8\s{\vec{\Psi}}{\e_{{\z}}}\s{\vec{\Psi}}{\,\cdot\,}\vec{\Psi},
   \end{align*}
   which directly implies that 
   \begin{align*}
   	\left(D_{\zeta}F_3(\e_z,\D_{\p{z}}\e_{\z})\right)^N=\left\{-2|\vec{\Psi}|^2\left(\s{\vec{\Psi}}{\,\cdot\,}\e_{{\z}}+\s{\vec{\Psi}}{\e_{{\z}}}\,\cdot\, +\s{\e_{{\z}}}{\,\cdot\,}\vec{\Psi}\right)+8\s{\vec{\Psi}}{\e_{{\z}}}\s{\vec{\Psi}}{\,\cdot\,}\vec{\Psi}\right\}^N.
   \end{align*}
   As $(D_k h)^N=0$, we have
   \begin{align}\label{gpart1}
   	\left(D_{\zeta}\vec{\I}(\f_z,\f_{\z})\right)^N&=\left(D_{\zeta}F_3(\e_z,\D_{\p{z}}\e_{\z})+f(F_1(\e_z))^{-1}D_{\zeta}h(\e_z,\D_{\p{z}}\e_{\z})\circ DF_1(\e_z)\right)^N\nonumber\\
   	&=\left\{-2|\vec{\Psi}|^2\left(\s{\vec{\Psi}}{\,\cdot\,}\e_{{\z}}+\s{\vec{\Psi}}{\e_{{\z}}}\,\cdot\, +\s{\e_{{\z}}}{\,\cdot\,}\vec{\Psi}\right)+8\s{\vec{\Psi}}{\e_{{\z}}}\s{\vec{\Psi}}{\,\cdot\,}\vec{\Psi}\right\}^N.
   \end{align}
   Then we have
   \begin{align}\label{somepart1}
   	D_{\zeta}F_2(\e_{z},\D_{\p{z}}\e_z)&=-4|\vec{\Psi}|^2\left(\s{\,\cdot\,}{\vec{\Psi}}\e_z+\s{\e_z}{\vec{\Psi}}\,\cdot\,+\s{\e_z}{\,\cdot\,}\vec{\Psi}\right)+16\s{\vec{\Psi}}{\e_z}\s{\vec{\Psi}}{\,\cdot\,}\vec{\Psi}.
   \end{align}
   As $D_{\kappa}(h(\f_z,\D_{\p{z}}\f_z))^N=0$, we obtain
   \begin{align}\label{somepart2}
   	\left\{D_{\zeta}\left(h(F_1(\e_z),F_2(\e_z,\D_{\p{z}}\e_z))\right)\right\}^N&=\left(D_{\zeta}h(\f_z,\D_{\p{z}\f_z})\circ D_{\zeta}F_1(\e_z)\right)^N\nonumber\\
   	&=-|\f_z|^2\s{\f_{\z}}{\D_{\p{z}}\f_z}\mathscr{I}_{\vec{\Psi}}(\,\cdot\,)^N\nonumber\\
   	&=-\frac{e^{2\mu}}{2}\p{z}\left(\frac{e^{2\mu}}{2}\right)\mathscr{I}_{\vec{\Psi}}(\,\cdot\,)^N\nonumber\\
   	&=\frac{1}{2}e^{4\mu}\left(\p{z}\mu\right)\mathscr{I}_{\vec{\Psi}}(\,\cdot\,)^N.
   \end{align}
   Therefore, by \eqref{somepart1} and \eqref{somepart2}, it follows that
   \begin{align}\label{gpart2}
   	\left(D_{\zeta}\vec{\I}(\f_z,\f_z)\right)^N&=\bigg\{-4|\vec{\Psi}|^2\left(\s{\,\cdot\,}{\vec{\Psi}}\e_z+\s{\e_z}{\vec{\Psi}}\,\cdot\,+\s{\e_z}{\,\cdot\,}\vec{\Psi}\right)+16\s{\vec{\Psi}}{\e_z}\s{\vec{\Psi}}{\,\cdot\,}\vec{\Psi}\bigg\}^N\nonumber\\
   	&+4e^{-4\mu}\left(\frac{1}{2}e^{4\mu}\left(\p{z}\mu\right)\mathscr{I}_{\vec{\Psi}}(\,\cdot\,)^N\right)\nonumber\\
   	&=\bigg\{-4|\vec{\Psi}|^2\left(\s{\,\cdot\,}{\vec{\Psi}}\e_z+\s{\e_z}{\vec{\Psi}}\,\cdot\,+\s{\e_z}{\,\cdot\,}\vec{\Psi}\right)+16\s{\vec{\Psi}}{\e_z}\s{\vec{\Psi}}{\,\cdot\,}\vec{\Psi}\bigg\}^N-2(\p{z}\mu)\mathscr{I}_{\vec{\Psi}}(\,\cdot\,)\nonumber\\
   	\left(D_{\zeta}\vec{\I}(\f_{\z},\f_{\z})\right)^N&=0.
   \end{align}
   Finally, we have by 
   \begin{align}\label{dzetaK}
   	&D_{\zeta}\star\left( K_{g_{\vec{\Psi}}}d\mathrm{vol}_{g_{\vec{\Psi}}}\right)=-D_{\zeta}(f(F_1(\e_z)))f(F_1(\e_z))^{-\frac{3}{2}}\left(\frac{e^{4\mu}}{4}\left(|\H_{\vec{\Psi}}|^2-|\H_{\vec{\Psi}}^0|^2\right)\right)\nonumber\\
   	&+4e^{-2\mu}\left(2\bs{\left(D_{\zeta}\vec{\I}(\f_z,\f_{\z})\right)^N}{\vec{\I}(\f_z,\f_{\z})}-\bs{\left(D_{\zeta}\vec{\I}(\f_z,\f_z)\right)^N}{\vec{\I}(\f_{\z},\f_{\z})}\right)\nonumber\\
   	&=-2K_{g_{\vec{\Psi}}}\mathscr{I}_{\vec{\Psi}}(\f_{\z})+4\bs{(D_{\zeta}\vec{\I}(\f_z,\f_{\z}))}{\H_{\vec{\Psi}}}-2\s{\left(\D_{\zeta}\vec{\I}(\f_z,\f_z)\right)^N}{\vec{\I}(\f_{\z},\f_{\z})}\nonumber\\
   	&=-2K_{g_{\vec{\Psi}}}\mathscr{I}_{\vec{\Psi}}(\f_{\z})-8|\vec{\Psi}|^2\s{\e_{\z}}{\H_{\vec{\Psi}}}\vec{\Psi}
   	-8|\vec{\Psi}|^2\s{\vec{\Psi}}{\e_{\z}}\H_{\vec{\Psi}}-8|\vec{\Psi}|^2\s{\vec{\Psi}}{\H_{\vec{\Psi}}}\e_{{\z}}
   	+32\s{\vec{\Psi}}{\e_{{\z}}}\s{\vec{\Psi}}{\H_{\vec{\Psi}}}\vec{\Psi}\nonumber\\
   	&+4(\p{z}\mu)\mathscr{I}_{\vec{\Psi}}\left(\bar{\H^0_{\vec{\Psi}}}\right)+8|\vec{\Psi}|^2\s{\e_z}{\bar{\H^0_{\vec{\Psi}}}}\vec{\Psi}+8|\vec{\Psi}|^2\s{\e_z}{\vec{\Psi}}\bar{\H^0_{\vec{\Psi}}}+8|\vec{\Psi}|^2\s{\vec{\Psi}}{\bar{\H^0_{\vec{\Psi}}}}\e_z-32\s{\vec{\Psi}}{\e_z}\s{\vec{\Psi}}{\bar{\H^0_{\vec{\Psi}}}}\vec{\Psi}
   \end{align}
   Then we have
   \begin{align*}
   	\left(D_{\chi}\vec{\I}(\f_z,\f_{\z})\right)^N&=\mathscr{I}_{\vec{\Psi}}(\,\cdot\,),
   \end{align*}
   which implies that
   \begin{align}\label{dchiK}
   	D_{\chi}\star\left( K_{g_{\vec{\Psi}}}d\mathrm{vol}_{g_{\vec{\Psi}}}\right)&=4e^{-2\mu}\left(2\s{\mathscr{I}_{\vec{\Psi}}(\,\cdot\,)}{\frac{e^{2\mu}}{2}\H_{\vec{\Psi}}}\right)
   	=4\mathscr{I}_{\vec{\Psi}}(\H_{\vec{\Psi}}).
   \end{align}
   As $(D_{\kappa}h)^N=0$, we obtain
   \begin{align*}
   	\left(D_{\omega}\vec{\I}(\f_z,\f_z)\right)^N=\mathscr{I}_{\vec{\Psi}}(\,\cdot\,)
   \end{align*}
   and
   \begin{align}\label{domegaK}
   	D_{\omega}\star\left( K_{g_{\vec{\Psi}}}d\mathrm{vol}_{g_{\vec{\Psi}}}\right)=-4e^{-2\mu}\bs{\mathscr{I}_{\vec{\Psi}}(\,\cdot\,)}{\frac{e^{2\mu}}{2}\bar{\H_{\vec{\Psi}}^0}}
   	=-2\mathscr{I}_{\vec{\Psi}}\left(\bar{\H_{\vec{\Psi}}^0}\right).
   \end{align}
   Finally, by \eqref{dzetaK}, \eqref{domegaK} and \eqref{domegaK}, we have
   \begin{align}
   \left\{\begin{alignedat}{1}
      		&D_{\zeta}\left(\star K_{g_{\vec{\Psi}}}d\mathrm{vol}_{g_{\vec{\Psi}}}\right)
      		=-2K_{g_{\vec{\Psi}}}\mathscr{I}_{\vec{\Psi}}(\f_{\z})-8|\vec{\Psi}|^2\s{\e_{\z}}{\H_{\vec{\Psi}}}\vec{\Psi}
      		-8|\vec{\Psi}|^2\s{\vec{\Psi}}{\e_{\z}}\H_{\vec{\Psi}}-8|\vec{\Psi}|^2\s{\vec{\Psi}}{\H_{\vec{\Psi}}}\e_{{\z}}\\
      		&+32\s{\vec{\Psi}}{\e_{{\z}}}\s{\vec{\Psi}}{\H_{\vec{\Psi}}}\vec{\Psi}
      		+4(\p{z}\mu)\mathscr{I}_{\vec{\Psi}}\left(\bar{\H^0_{\vec{\Psi}}}\right)+8|\vec{\Psi}|^2\s{\e_z}{\bar{\H^0_{\vec{\Psi}}}}\vec{\Psi}+8|\vec{\Psi}|^2\s{\e_z}{\vec{\Psi}}\bar{\H^0_{\vec{\Psi}}}\\
      		&+8|\vec{\Psi}|^2\s{\vec{\Psi}}{\bar{\H^0_{\vec{\Psi}}}}\e_z-32\s{\vec{\Psi}}{\e_z}\s{\vec{\Psi}}{\bar{\H^0_{\vec{\Psi}}}}\vec{\Psi}\\
   &D_{\chi}\left(\star K_{g_{\vec{\Psi}}}d\mathrm{vol}_{g_{\vec{\Psi}}}\right)=4\mathscr{I}_{\vec{\Psi}}\left(\H_{\vec{\Psi}}\right)\label{confinv2}\\
   &D_{\omega}\left(\star K_{g_{\vec{\Psi}}}d\mathrm{vol}_{g_{\vec{\Psi}}}\right)=-2\mathscr{I}_{\vec{\Psi}}\left(\bar{\H^0_{\vec{\Psi}}}\right).
   \end{alignedat}\right.
   \end{align}
   By \eqref{confinv1} and \eqref{confinv2}, we obtain
   \begin{align}\label{dh02}
   \left\{\begin{alignedat}{1}
      	&D_{\zeta}(\star |\H^0_{\vec{\Psi}}|^2d\mathrm{vol}_{g_{\vec{\Psi}}})=-2|\H_{\vec{\Psi}}^0|^2\mathscr{I}_{\vec{\Psi}}(\f_{\z})-4\s{\H_{\vec{\Psi}}}{\bar{\H^0_{\vec{\Psi}}}}\mathscr{I}_{\vec{\Psi}}(\f_z)-4(\p{z}\mu)\mathscr{I}_{\vec{\Psi}}\left(\bar{\H^0_{\vec{\Psi}}}\right)\\
      	&-8|\vec{\Psi}|^2\s{\e_z}{\bar{\H^0_{\vec{\Psi}}}}\vec{\Psi}-8|\vec{\Psi}|^2\s{\e_z}{\vec{\Psi}}\bar{\H^0_{\vec{\Psi}}}-8|\vec{\Psi}|^2\s{\vec{\Psi}}{\bar{\H^0_{\vec{\Psi}}}}\e_z+32\s{\vec{\Psi}}{\e_z}\s{\vec{\Psi}}{\bar{\H^0_{\vec{\Psi}}}}\vec{\Psi}\\ 
        &D_{\omega}(\star|\H^0_{\vec{\Psi}}|^2d\mathrm{vol}_{g_{\vec{\Psi}}})=2\mathscr{I}_{\vec{\Psi}}\left(\bar{\H_{\vec{\Psi}}^0}\right)\\
        &D_{\chi}\left(\star |\H^0_{\vec{\Psi}}|^2d\mathrm{vol}_{g_{\vec{\Psi}}}\right)=0
   \end{alignedat}\right.
   \end{align}
   This expression can be further simplified.
   We first observe that (recalling the definition $\f_z=\p{z}\vec{\Psi}$)
   \begin{align}\label{simpler0}
   	\s{\vec{\Psi}}{\e_z}=\bs{\vec{\Psi}}{\frac{\f_z}{|\vec{\Psi}|^2}-2\s{\vec{\Psi}}{\f_z}\frac{\vec{\Psi}}{|\vec{\Psi}|^4}}=-\frac{\s{\vec{\Psi}}{\f_z}}{|\vec{\Psi}|^2}=-\frac{1}{2}\frac{\partial_{z}|\vec{\Psi}|^2}{|\vec{\Psi}|^2}
   \end{align}
   therefore
   \begin{align}\label{simpler1}
   	-8|\vec{\Psi}|^2\s{\vec{\Psi}}{\e_z}\bar{\H_{\vec{\Psi}}^0}=4\,\partial_{z}|\vec{\Psi}|^2\bar{\H_{\vec{\Psi}}^0}.
   \end{align}
   Then we compute as $\f_z=\p{z}\vec{\Psi}$ is normal to $\H_{\vec{\Psi}}^0$ that
   \begin{align*}
   	\s{\e_z}{\bar{\H_{\vec{\Psi}}^0}}=\bs{\frac{\f_z}{|\vec{\Psi}|^2}-2\s{\vec{\Psi}}{\f_z}\frac{\vec{\Psi}}{|\vec{\Psi}|^4}}{\bar{\vec{H}_{\vec{\Psi}}^0}}=-\frac{2}{|\vec{\Psi}|^4}\s{\p{z}\vec{\Psi}}{\vec{\Psi}}\s{\vec{\Psi}}{\bar{\vec{H}_{\vec{\Psi}}^0}}=-\frac{\p{z}|\vec{\Psi}|^2}{|\vec{\Psi}|^4}\s{\vec{\Psi}}{\bar{\H_{\vec{\Psi}}^0}}
   \end{align*}
   so
   \begin{align}\label{simpler2}
   	-8|\vec{\Psi}|^2\s{\e_z}{\bar{\H_{\vec{\Psi}}^0}}\vec{\Psi}=8\frac{\p{z}|\vec{\Psi}|^2}{|\vec{\Psi}|^2}\s{\vec{\Psi}}{\bar{\H_{\vec{\Psi}}^0}}\vec{\Psi}.
   \end{align}
   The next contribution is
   \begin{align}\label{simpler4}
   	-8|\vec{\Psi}|^2\s{\vec{\Psi}}{\bar{\H_{\vec{\Psi}}^0}}\e_z=	-8|\vec{\Psi}|^2\s{\vec{\Psi}}{\bar{\H_{\vec{\Psi}}^0}}\left(\frac{\p{z}\vec{\Psi}}{|\vec{\Psi}|^2}-2\frac{\s{\p{z}\vec{\Psi}}{\vec{\Psi}}}{{|\vec{\Psi}|^4}}\vec{\Psi}\right)=-8\s{\vec{\Psi}}{\bar{\H_{\vec{\Psi}}^0}}\p{z}\vec{\Psi}+8\frac{\p{z}|\vec{\Psi}|^2}{|\vec{\Psi}|^2}\s{\vec{\Psi}}{\bar{\H_{\vec{\Psi}}^0}}\vec{\Psi}.
   \end{align}
   Now, by \eqref{simpler0}, we have
   \begin{align}\label{simpler3}
   	32\s{\vec{\Psi}}{\e_z}\s{\vec{\Psi}}{\bar{\H_{\vec{\Psi}}^0}}\vec{\Psi}=32\left(-\frac{1}{2}\frac{\p{z}|\vec{\Psi}|^2}{|\vec{\Psi}|^2}\right)\s{\vec{\Psi}}{\bar{\H_{\vec{\Psi}}^0}}\vec{\Psi}=-16\frac{\p{z}|\vec{\Psi}|^2}{|\vec{\Psi}|^2}\s{\vec{\Psi}}{\bar{\H_{\vec{\Psi}}^0}}\vec{\Psi}
   \end{align}
   Finally, thanks to \eqref{simpler1}, \eqref{simpler2}, \eqref{simpler3}, \eqref{simpler4}, we get
   \begin{align}\label{simplerend}
   	&-8|\vec{\Psi}|^2\s{\e_z}{\bar{\H^0_{\vec{\Psi}}}}\vec{\Psi}-8|\vec{\Psi}|^2\s{\e_z}{\vec{\Psi}}\bar{\H^0_{\vec{\Psi}}}-8|\vec{\Psi}|^2\s{\vec{\Psi}}{\bar{\H^0_{\vec{\Psi}}}}\e_z+32\s{\vec{\Psi}}{\e_z}\s{\vec{\Psi}}{\bar{\H^0_{\vec{\Psi}}}}\vec{\Psi}\nonumber\\
   	&=\ccancel{8\frac{\p{z}|\vec{\Psi}|^2}{|\vec{\Psi}|^2}\s{\vec{\Psi}}{\bar{\H_{\vec{\Psi}}^0}}\vec{\Psi}}+4\p{z}|\vec{\Psi}|^2\bar{\H_{\vec{\Psi}}^0}-8\s{\vec{\Psi}}{\bar{\H_{\vec{\Psi}}^0}}\p{z}\vec{\Psi}+\ccancel{8\frac{\p{z}|\vec{\Psi}|^2}{|\vec{\Psi}|^2}\s{\vec{\Psi}}{\bar{\H_{\vec{\Psi}}^0}}\vec{\Psi}}-\ccancel{16\frac{\p{z}|\vec{\Psi}|^2}{|\vec{\Psi}|^2}\s{\vec{\Psi}}{\bar{\H_{\vec{\Psi}}^0}}\vec{\Psi}}\nonumber\\
   	&=4\left(\p{z}|\vec{\Psi}|^2\bar{\H_{\vec{\Psi}}^0}-2\s{\vec{\Psi}}{\bar{\H_{\vec{\Psi}}^0}}\p{z}\vec{\Psi}\right)
   \end{align}
   and thanks to \eqref{dh02} and \eqref{simplerend}, we obtain
   \begin{align}\label{noetherinv}
   	\left\{\begin{alignedat}{1}
   	D_{\zeta}(\star |\H^0_{\vec{\Psi}}|^2d\mathrm{vol}_{g_{\vec{\Psi}}})&=-2|\H_{\vec{\Psi}}^0|^2\mathscr{I}_{\vec{\Psi}}(\f_{\z})-4\s{\H_{\vec{\Psi}}}{\bar{\H^0_{\vec{\Psi}}}}\mathscr{I}_{\vec{\Psi}}(\f_z)-4(\p{z}\mu)\mathscr{I}_{\vec{\Psi}}\left(\bar{\H^0_{\vec{\Psi}}}\right)\\
   	&+4\left(\p{z}|\vec{\Psi}|^2\bar{\H_{\vec{\Psi}}^0}-2\s{\vec{\Psi}}{\bar{\H_{\vec{\Psi}}^0}}\p{z}\vec{\Psi}\right)\\ 
   	D_{\omega}(\star|\H^0_{\vec{\Psi}}|^2d\mathrm{vol}_{g_{\vec{\Psi}}})&=2\mathscr{I}_{\vec{\Psi}}\left(\bar{\H_{\vec{\Psi}}^0}\right)\\
   	D_{\chi}\left(\star |\H^0_{\vec{\Psi}}|^2d\mathrm{vol}_{g_{\vec{\Psi}}}\right)&=0
   	\end{alignedat}\right.
   \end{align}

   Finally, we obtain the \textit{pointwise} identities (valid for arbitrary immersions, not necessarily Willmore)
   \begin{align}\label{inversionWeingarten}
   \left\{\begin{alignedat}{1}
      	&-2|\H^0_{\phi}|^2\p{\z}\phi-4\s{\H_{\phi}}{\bar{\H_{\phi}^0}}\p{z}\phi-4(\p{z}\lambda)\bar{\H_{\phi}^0}
      	=-2|\H_{\vec{\Psi}}^0|^2\mathscr{I}_{\vec{\Psi}}(\p{\z}\vec{\Psi})-4\s{\H_{\vec{\Psi}}}{\bar{\H^0_{\vec{\Psi}}}}\mathscr{I}_{\vec{\Psi}}(\p{z}\vec{\Psi})\\
      	&-4(\p{z}\mu)\mathscr{I}_{\vec{\Psi}}\left(\bar{\H^0_{\vec{\Psi}}}\right)+4\left(\,\partial_{z}|\vec{\Psi}|^2\bar{\H_{\vec{\Psi}}^0}-2\s{\vec{\Psi}}{\bar{\H_{\vec{\Psi}}^0}}\p{z}\vec{\Psi}\right)\\
   &\H_{\phi}^{0}=\mathscr{I}_{\vec{\Psi}}\left(\H_{\vec{\Psi}}^0\right)
   \end{alignedat}\right.
   \end{align}
   In particular, the second identity of \eqref{inversionWeingarten} shows the point-wise conformal invariance of the Willmore energy.
  Now, recall that Noether's theorem (Theorem \ref{noethercomplexe}) states that for all infinitesimal symmetry $\vec{X}$ of a Lagrangian $L$, we have
  \begin{align*}
  	\Re\left(\D_{\partial_{z}}\left(\frac{\partial L_0}{\partial\zeta}\cdot\vec{X}-\D_{\partial_{z}}\left(\frac{\partial L_0}{\partial \omega}\right)\cdot\vec{X}+\frac{\partial L_0}{\partial \omega}\cdot \D_{\partial_{z}}\vec{X}-\frac{1}{2}\,\D_{\p{\z}}\left(\frac{\partial L_0}{\partial \chi}\right)\cdot\vec{X}+\frac{1}{2}\frac{\partial L_0}{\partial\chi}\cdot\,\D_{\p{\z}}\vec{X}\right)\right)=0
  \end{align*}
  which gives by taking the complex conjugate if $\partial_{\chi}L=0$ the identity
  \begin{align}\label{eulerlagrange}
  	\Re\left(\D_{\partial_{\z}}\left(\bar{\frac{\partial L_0}{\partial\zeta}}\cdot\vec{X}-\D_{\partial_{\z}}\bar{\left(\frac{\partial L_0}{\partial \omega}\right)}\cdot\vec{X}+\bar{\frac{\partial L_0}{\partial \omega}}\cdot \D_{\partial_{\z}}\vec{X}\right)\right)=0.
  \end{align}
  In our case, we have 
  \begin{align*}
  	L_0=\star|\H_{\vec{\Psi}}^0|^2d\mathrm{vol}_{g_{\vec{\Psi}}}
  \end{align*}
  and we compute by \eqref{dh02}
  \begin{align}\label{respart1}
  	\D_{\p{\z}}\bar{\left(\frac{\partial L_0}{\partial\omega}\right)}&=2\,\D_{\p{\z}}\left(\mathscr{I}_{\vec{\Psi}}\left(\vec{H}_{\vec{\Psi}}^0\right)\right)=2\,\D_{\p{\z}}\left(|\vec{\Psi}|^2\vec{H}_{\vec{\Psi}}^0-2\s{\vec{H}_{\vec{\Psi}}^0}{\vec{\Psi}}\vec{\Psi}\right)\nonumber\\
  	&=2\left(|\vec{\Psi}|^2\D_{\p{\z}}(\H_{\vec{\Psi}}^0)-2\s{\vec{\Psi}}{\vec{H}_{\vec{\Psi}}^0}\vec{\Psi}+\p{\z}|\vec{\Psi}|^2\H_{\vec{\Psi}}^0-2\s{\p{\z}\vec{\Psi}}{\H_{\vec{\Psi}}^0}\vec{\Psi}-2\s{\vec{\Psi}}{\H_{\vec{\Psi}}^0}\p{\z}\vec{\Psi}\right)\nonumber\\
  	&=2\,\mathscr{I}_{\vec{\Psi}}\left(\D_{\p{\z}}\left(\H_{\vec{\Psi}}^0\right)\right)+2\,\p{\z}|\vec{\Psi}|^2\H_{\vec{\Psi}}^0-4\s{\vec{\Psi}}{\H_{\vec{\Psi}}^0}\p{\z}\vec{\Psi}.
  \end{align}
  as $\p{\z}\vec{\Psi}$ is a tangent vector and $\H_{\vec{\Psi}}^0$ is a normal vector, so $\s{\H_{\vec{\Psi}}^0}{\p{\z}\vec{\Psi}}=0$. Then, we compute
  \begin{align*}
  	\D_{\p{\z}}\H_{\vec{\Psi}}^0=\D_{\p{\z}}\left(e^{-2\mu}\h^{0}_{\vec{\Psi}}\right)=-2\left(\p{\z}\mu\right)\H_{\vec{\Psi}}^0+e^{-2\mu}\D_{\p{\z}}(\h_{\vec{\Psi}}^0)=-2\left(\p{\z}\mu\right)\H_{\vec{\Psi}}^0+g^{-1}_{\vec{\Psi}}\otimes \bar{\partial}^N\h^0_{\vec{\Psi}}+\D_{\p{\z}}^{\top}\H_{\vec{\Psi}}^0
  \end{align*}
  and by now familiar computations, we also readily obtain
  \begin{align}\label{respart2}
  	\D_{\p{\z}}^{\top}\H_{\vec{\Psi}}^0=-|\H_{\vec{\Psi}}^0|^2\f_z-\s{\H_{\vec{\Psi}}}{\H_{\vec{\Psi}}^0}\f_{\z}.
  \end{align}
  By \eqref{respart1} and \eqref{respart2}, we have
  \begin{align}\label{respartbis}
  	\D_{\p{\z}}\bar{\left(\frac{\partial L_0}{\partial\omega}\right)}&=-4\left(\p{\z}\mu\right)\mathscr{I}_{\vec{\Psi}}(\H_{\vec{\Psi}}^0)+2\,g^{-1}_{\vec{\Psi}}\otimes\,\mathscr{I}_{\vec{\Psi}}(\bar{\partial}^N\h_{\vec{\Psi}}^0)-2|\H_{\vec{\Psi}}^0|^2\mathscr{I}_{\vec{\Psi}}(\f_z)-2\s{\H_{\vec{\Psi}}}{\H_{\vec{\Psi}}^0}\mathscr{I}_{\vec{\Psi}}(\f_{\z})\nonumber\\
  	&+2\,\p{\z}|\vec{\Psi}|^2\H_{\vec{\Psi}}^0-4\s{\vec{\Psi}}{\H_{\vec{\Psi}}^0}\f_{\z}.
  \end{align}
  We now trivially have
  \begin{align}\label{respart3}
  	\bar{\frac{\partial L_0}{\partial \omega}}\cdot \D_{\p{\z}}\vec{X}=2\,\mathscr{I}_{\vec{\Psi}}(\H_{\vec{\Psi}}^0)\cdot \D_{\p{\z}}\vec{X}
  \end{align}
  Finally, we have by \eqref{dh02}
  \begin{align}\label{respart4}
  	\bar{\frac{\partial L_0}{\partial\zeta}}=-2|\H_{\vec{\Psi}}^0|^2\mathscr{I}_{\vec{\Psi}}(\f_{z})-4\s{\H_{\vec{\Psi}}}{\bar{\H^0_{\vec{\Psi}}}}\mathscr{I}_{\vec{\Psi}}(\f_{\z})-4(\p{\z}\mu)\mathscr{I}_{\vec{\Psi}}\left({\H^0_{\vec{\Psi}}}\right)+4\,\partial_{\z}|\vec{\Psi}|^2\H_{\vec{\Psi}}^0-8\s{\vec{\Psi}}{\bar{\H_{\vec{\Psi}}^0}}\f_{\z}.
  \end{align}
  In particular, by \eqref{respart1}, we have
  \begin{align}\label{respart5}
  	&\bar{\frac{\partial L_0}{\partial\zeta}}-\D_{\p{\z}}\bar{\left(\frac{\partial L_0}{\partial\omega}\right)}=-\ccancel{2|\H_{\vec{\Psi}}^0|^2\mathscr{I}_{\vec{\Psi}}(\f_{z})}-4\s{\H_{\vec{\Psi}}}{\bar{\H^0_{\vec{\Psi}}}}\mathscr{I}_{\vec{\Psi}}(\f_{\z})-\colorcancel{4(\p{\z}\mu)\mathscr{I}_{\vec{\Psi}}\left({\H^0_{\vec{\Psi}}}\right)}{blue}+4\,\partial_{\z}|\vec{\Psi}|^2\H_{\vec{\Psi}}^0-8\s{\vec{\Psi}}{\bar{\H_{\vec{\Psi}}^0}}\f_{\z}\nonumber\\
  	&-\bigg(-\colorcancel{4\left(\p{\z}\mu\right)\mathscr{I}_{\vec{\Psi}}(\H_{\vec{\Psi}}^0)}{blue}+2\,g^{-1}_{\vec{\Psi}}\otimes\,\mathscr{I}_{\vec{\Psi}}(\bar{\partial}^N\h_{\vec{\Psi}}^0)-\ccancel{2|\H_{\vec{\Psi}}^0|^2\mathscr{I}_{\vec{\Psi}}(\f_z)}-2\s{\H_{\vec{\Psi}}}{\H_{\vec{\Psi}}^0}\mathscr{I}_{\vec{\Psi}}(\f_{\z})\nonumber\\
  	&+2\,\p{\z}|\vec{\Psi}|^2\H_{\vec{\Psi}}^0-4\s{\vec{\Psi}}{\H_{\vec{\Psi}}^0}\f_{\z}.\bigg)\\
  	&=-2\,g^{-1}_{\vec{\Psi}}\otimes \mathscr{I}_{\vec{\Psi}}(\bar{\partial}^N\h_{\vec{\Psi}}^0)-2\s{\H_{\vec{\Psi}}}{\H_{\vec{\Psi}}^0}\mathscr{I}_{\vec{\Psi}}(f_{\z})+2\partial_{\z}|\vec{\Psi}|^2\H_{\vec{\Psi}}^0-4\s{\vec{\Psi}}{\H_{\vec{\Psi}}^0}\f_{\z}.
  \end{align}
  By \eqref{respart3} and \eqref{respart5}, we conclude that
  \begin{align*}
  	&\bar{\frac{\partial L_0}{\partial \zeta}}-\D_{\p{\z}}\bar{\left(\frac{\partial L_0}{\partial\omega}\right)}+\bar{\frac{\partial L_0}{\partial\omega}}=\bigg\{-2\,g^{-1}_{\vec{\Psi}}\otimes \mathscr{I}_{\vec{\Psi}}(\bar{\partial}^N\h_{\vec{\Psi}}^0)-2\s{\H_{\vec{\Psi}}}{\H_{\vec{\Psi}}^0}\mathscr{I}_{\vec{\Psi}}(f_{\z})+2\partial_{\z}|\vec{\Psi}|^2\H_{\vec{\Psi}}^0-4\s{\vec{\Psi}}{\H_{\vec{\Psi}}^0}\f_{\z}\bigg\}\cdot \vec{X}\\
  	&+2\,\mathscr{I}_{\vec{\Psi}}(\H_{\vec{\Psi}}^0)\cdot \D_{\p{\z}}\vec{X}\\
  	&=-2\bigg\{\bigg(\mathscr{I}_{\vec{\Psi}}\left(g_{\vec{\Psi}}^{-1}\otimes \bar{\partial}^N\h_{\vec{\Psi}}^0+\s{\H_{\vec{\Psi}}}{\H_{\vec{\Psi}}^0}\p{\z}\vec{\Psi}\right)-{\partial}_{\z}|\vec{\Psi}|^2\H_{\vec{\Psi}}^0+2\s{\vec{\Psi}}{\H_{\vec{\Psi}}^0}\p{\z}\vec{\Psi}\bigg)\cdot \vec{X}-\mathscr{I}_{\vec{\Psi}}(\H_{\vec{\Psi}}^0)\cdot \D_{\p{\z}}\vec{X}\bigg\}.
  \end{align*}
  Finally, thanks to \eqref{respart2}, we obtain the \emph{pointwise} identity (valid for any Willmore immersion)
   \begin{align}\label{inversionformula}
   \left\{\begin{alignedat}{1}
      	&\left(g_{\phi}^{-1}\otimes \left(\bar{\partial}^N-\bar{\partial}^\top\right)\h^0_{\phi}-|\h^0_{\phi}|^2_{WP}\,\partial\phi\right)\cdot\vec{X}-g_{\phi}^{-1}\otimes\h^0_{\phi}\cdot\bar{\partial}\vec{X}\\
        &=\mathscr{I}_{\vec{\Psi}}\left(g_{\vec{\Psi}}^{-1}\otimes\left(\bar{\partial}^N-\bar{\partial}^\top\right)\h_{\vec{\Psi}}^0-|\h_{\vec{\Psi}}^0|_{WP}^2\,\partial\vec{\Psi}\right)\cdot \vec{X}
        -g^{-1}_{\vec{\Psi}}\otimes \bigg(\bar{\partial}|\vec{\Psi}|^2\otimes \h_{\vec{\Psi}}^0-2\s{\vec{\Psi}}{\h_{\vec{\Psi}}^0}\otimes \bar{\partial}\vec{\Psi}\bigg)\cdot\vec{X}
        \\
        &-g_{\vec{\Psi}}^{-1}\otimes \mathscr{I}_{\vec{\Psi}}(\h_{\vec{\Psi}}^0)\cdot \bar{\partial}\vec{X}.
   \end{alignedat}\right.
   \end{align}
    Applying Noether's theorem with $\vec{X}=\vec{C}\in\R^n$ constant, we obtain for all $p\in\Sigma^2$ by \eqref{residues1234}
    \begin{align}\label{14}
    	\vec{\gamma}_0(\phi,p)=\vec{\gamma}_3(\vec{\Psi},p)
    \end{align}
    \textit{i.e.} the fourth residue of an inversion if equal to the first residue. As the proof is symmetric in $\phi$ and $\vec{\Psi}$, we also get
    \begin{align}\label{41}
    	\vec{\gamma}_3(\phi,p)=\vec{\gamma}_0(\vec{\Psi},p).
    \end{align}
    We now turn to th e invariance by dilatations (\textit{i.e.} zith $\vec{X}=\phi$).
    To simplify notations, let
    \begin{align*}
    \vec{\alpha}=g_{\vec{\Psi}}^{-1}\otimes\left(\bar{\partial}^N-\bar{\partial}^\top\right)\h_{\vec{\Psi}}^0-|\h_{\vec{\Psi}}^0|_{WP}^2\,\partial\vec{\Psi}.
    \end{align*}
    We have
    \begin{align}\label{gamma22}
    \mathscr{I}_{\vec{\Psi}}(\vec{\alpha})\cdot\phi=|\vec{\Psi}|^2\vec{\alpha}\cdot\phi-2\s{\vec{\Psi}}{\vec{\alpha}}\vec{\Psi}\cdot\phi=\vec{\Psi}\cdot \vec{\alpha}-2\vec{\Psi}\cdot\alpha=-\vec{\Psi}\cdot\vec{\alpha}.
    \end{align}
    On the other hand, as
    as
    \begin{align*}
    2\s{\vec{\Psi}}{\h_{\vec{\Psi}}^0}\otimes \frac{\s{\bar{\partial}\vec{\Psi}}{\vec{\Psi}}}{|\vec{\Psi}|^2}=\s{\h_{\vec{\Psi}}^0}{\vec{\Psi}}\otimes \frac{\bar{\partial}|\vec{\Psi}|^2}{|\vec{\Psi}|^2},
    \end{align*}    
    we have
    \begin{align}\label{gamma23}
    	\left(\bar{\partial}|\vec{\Psi}|^2\otimes \h_{\vec{\Psi}}^0-2\s{\vec{\Psi}}{\h_{\vec{\Psi}}^0}\bar{\partial}\otimes\vec{\Psi}\right)\cdot\phi&=\left(\bar{\partial}|\vec{\Psi}|^2\otimes \h_{\vec{\Psi}}^0-2\s{\vec{\Psi}}{\h_{\vec{\Psi}}^0}\otimes \bar{\partial}\vec{\Psi}\right)\cdot\frac{\vec{\Psi}}{|\vec{\Psi}|^2}\nonumber\\
    	&=\frac{\bar{\partial}|\vec{\Psi}|^2}{|\vec{\Psi}|^2}\otimes\s{\h_{\vec{\Psi}}^0}{\vec{\Psi}}-2\s{\vec{\Psi}}{\h_{\vec{\Psi}}^0}\frac{\s{\bar{\partial}\vec{\Psi}}{\vec{\Psi}}}{|\vec{\Psi}|^2}
    	=0.
    \end{align}
    Now, we compute
    \begin{align}\label{gamma24}
    	\mathscr{I}_{\vec{\Psi}}(\h_{\vec{\Psi}}^0)\totimes\bar{\partial}\phi&=\left(|\vec{\Psi}|^2\h_{\vec{\Psi}}^0-2\s{\vec{\Psi}}{\h_{\vec{\Psi}}^0}\vec{\Psi}\right)\totimes \left(\frac{\bar{\partial}\vec{\Psi}}{|\vec{\Psi}|^2}-\frac{\bar{\partial}|\vec{\Psi}|^2}{|\vec{\Psi}|^4}\vec{\Psi}\right)\nonumber\\
    	&=-\frac{\bar{\partial}|\vec{\Psi}|^2}{|\vec{\Psi}|^2}\otimes\s{\h_{\vec{\Psi}}^0}{\vec{\Psi}}-2\s{\vec{\Psi}}{\h_{\vec{\Psi}}}\otimes\frac{\s{\bar{\partial}\vec{\Psi}}{\vec{\Psi}}}{|\vec{\Psi}|^2}+2\s{\vec{\Psi}}{\h_{\vec{\Psi}}^0}\otimes \frac{\bar{\partial}|\vec{\Psi}|^2}{|\vec{\Psi}|^4}\s{\vec{\Psi}}{\vec{\Psi}}
    	=0
    \end{align}
    Thanks to \eqref{gamma22}, \eqref{gamma23} and \eqref{gamma24}, we obtain
     \begin{align}\label{2-2}
    \vec{\gamma}_2(\phi,p)=-\vec{\gamma}_2(\vec{\Psi},p).
    \end{align}
    Finally, as $\vec{\Psi}\wedge \vec{\Psi}=0$
    \begin{align*}
    	\mathscr{I}_{\vec{\Psi}}(\,\cdot\,)\wedge \phi=|\vec{\Psi}|^2\cdot\wedge \,\phi=\cdot\wedge \vec{\Psi}
    \end{align*}
    Therefore, we have
    \begin{align}\label{yatta0}
    	\mathscr{I}_{\vec{\Psi}}(\vec{\alpha})\wedge \phi=\vec{\alpha}\wedge \vec{\Psi}=-\vec{\Psi}\wedge \left(g_{\vec{\Psi}}^{-1}\otimes\left(\bar{\partial}^N-\bar{\partial}^\top\right)\h_{\vec{\Psi}}^0-|\h_{\vec{\Psi}}^0|_{WP}^2\,\partial\vec{\Psi}\right)
    \end{align}
    Furthermore, we have
    \begin{align}\label{yatta1}
    	\left(\bar{\partial}|\vec{\Psi}|^2\otimes \h_{\vec{\Psi}}^0-2\s{\vec{\Psi}}{\h_{\vec{\Psi}}^0}\otimes \bar{\partial}\vec{\Psi}\right)\wedge \phi&=\frac{\bar{\partial}|\vec{\Psi}|^2}{|\vec{\Psi}|^2}\otimes \h_{\vec{\Psi}}^0\wedge \vec{\Psi}-\frac{2}{|\vec{\Psi}|^2}\s{\vec{\Psi}}{\h_{\vec{\Psi}}^0}\otimes \bar{\partial}\vec{\Psi}\wedge \vec{\Psi}
    \end{align}
    and as $\vec{\Psi}\wedge \vec{\Psi}=0$, we have
    \begin{align}\label{yatta2} 
    	&\mathscr{I}_{\vec{\Psi}}(\h_{\vec{\Psi}}^0)\wedge \bar{\partial}\phi=\left(|\vec{\Psi}|^2\h_{\vec{\Psi}}^0-2\s{\vec{\Psi}}{\h_{\vec{\Psi}}^0}\vec{\Psi}\right)\wedge \left(\frac{\bar{\partial}\vec{\Psi}}{|\vec{\Psi}|^2}-\frac{\bar{\partial}|\vec{\Psi}|^2}{|\vec{\Psi}|^4}\vec{\Psi}\right)\nonumber\\
    	&=\h_{\vec{\Psi}}^0\wedge\bar{\partial}\vec{\Psi}-\frac{\bar{\partial}|\vec{\Psi}|^2}{|\vec{\Psi}|^2}\otimes \h_{\vec{\Psi}}^0\wedge\vec{\Psi}-\frac{2}{|\vec{\Psi}|^2}\s{\vec{\Psi}}{\h_{\vec{\Psi}}^0}\otimes\vec{\Psi}\wedge \bar{\partial}\vec{\Psi}
    	=\h_{\vec{\Psi}}^0\wedge\bar{\partial}\vec{\Psi}-\frac{\bar{\partial}|\vec{\Psi}|^2}{|\vec{\Psi}|^2}\otimes \h_{\vec{\Psi}}^0\wedge\vec{\Psi}+\frac{2}{|\vec{\Psi}|^2}\s{\vec{\Psi}}{\h_{\vec{\Psi}}^0}\otimes\bar{\partial}\vec{\Psi}\wedge \vec{\Psi}.
    \end{align}
    Therefore, thanks to \eqref{yatta1} and \eqref{yatta2}, we get
    \begin{align}\label{yatta3}
    	\left(\bar{\partial}|\vec{\Psi}|^2\otimes \h_{\vec{\Psi}}^0-2\s{\vec{\Psi}}{\h_{\vec{\Psi}}^0}\otimes \bar{\partial}\vec{\Psi}\right)\wedge \phi+\mathscr{I}_{\vec{\Psi}}(\h_{\vec{\Psi}}^0)\wedge \bar{\partial}\phi=\h_{\vec{\Psi}}^0\wedge\bar{\partial}\vec{\Psi}.
    \end{align}
    Finally, thanks to \eqref{inversionformula}, \eqref{yatta0}, \eqref{yatta3}, we obtain
    \begin{align*}
    	&\left(g_{\phi}^{-1}\otimes \left(\bar{\partial}^N-\bar{\partial}^\top\right)\h^0_{\phi}-|\h^0_{\phi}|^2_{WP}\,\partial\phi\right)\wedge\phi-g_{\phi}^{-1}\otimes\h^0_{\phi}\wedge \bar{\partial}{\phi}\\
    	&=-\bigg(\phi\wedge \left(g_{\phi}^{-1}\otimes \left(\bar{\partial}^N-\bar{\partial}^\top\right)\h^0_{\phi}-|\h^0_{\phi}|^2_{WP}\,\partial\phi\right)+g_{\phi}^{-1}\otimes\h^0_{\phi}\wedge \bar{\partial}{\phi}\bigg)\\
    	&=\mathscr{I}_{\vec{\Psi}}\left(g_{\vec{\Psi}}^{-1}\otimes\left(\bar{\partial}^N-\bar{\partial}^\top\right)\h_{\vec{\Psi}}^0-|\h_{\vec{\Psi}}^0|_{WP}^2\,\partial\vec{\Psi}\right)\cdot \vec{X}
    	-g^{-1}_{\vec{\Psi}}\otimes \bigg(\bar{\partial}|\vec{\Psi}|^2\otimes \h_{\vec{\Psi}}^0-2\s{\vec{\Psi}}{\h_{\vec{\Psi}}^0}\otimes \bar{\partial}\vec{\Psi}\bigg)\cdot\vec{X}
    	\\
    	&-g_{\vec{\Psi}}^{-1}\otimes \mathscr{I}_{\vec{\Psi}}(\h_{\vec{\Psi}}^0)\wedge \bar{\partial}\phi
    	=-\left(\vec{\Psi}\wedge\left(g_{\vec{\Psi}}^{-1}\otimes\left(\bar{\partial}^N-\bar{\partial}^\top\right)\h_{\vec{\Psi}}^0-|\h_{\vec{\Psi}}^0|_{WP}^2\,\partial\vec{\Psi}\right)+g^{-1}_{\vec{\Psi}}\otimes\h_{\vec{\Psi}}^0\wedge \bar{\partial}\vec{\Psi}\right)
    \end{align*}
    so
    \begin{align}\label{11}
    	\vec{\gamma}_1(\phi,p)=\vec{\gamma}_1(\vec{\Psi},p)
    \end{align}
    and by \eqref{14}, \eqref{41}, \eqref{2-2},\eqref{11}, and \eqref{residusnonintro}, this concludes the proof of the theorem.
\end{proof}

\begin{rem}
	We showed in fact a much stronger property that the correspondence of residues under conformal transformations, as we actually proved the \emph{pointwise} invariance of the four integrated tensors modulo permutations and change of sign.
\end{rem}

\begin{rem}
	    For inversions of minimal surfaces, the third residue vanish, as the integrand is
	\begin{align*}
	g^{-1}\otimes \s{\H}{\h_0}\otimes \bar{\partial}|\phi|^2=0.
	\end{align*}
	Furthermore, of $\phi$ is minimal, by the Weierstrass parametrisation, we have for some $k\in\N$
	\begin{align*}
	\phi(z)=\Re\left(\frac{\alpha_1}{z^k},\cdots,\frac{\alpha_{n}}{z^k}\right)+O\left(\frac{1}{z^{k-1}}\right)
	\end{align*}
	for some $\alpha_1,\cdots,\alpha_n\in \C\setminus\ens{0}$. Therefore 
	\begin{align*}
	e^{2\lambda}&=\frac{k^2}{2}\frac{\sum_{j=1}^{n}|\alpha_j|^2}{|z|^{2(k+1)}}\left(1+O\left(|z|^2\right)\right),\qquad
	\p{\z}\phi(z)=-\frac{k}{2}\bar{\left(\frac{{\alpha}_1}{z^{k+1}},\cdots,\frac{{\alpha}_{n}}{z^{k+1}}\right)}
	\end{align*}
	and for some $\alpha\neq 0$, $\beta_1\cdots\beta_n,\gamma_1,\cdots\,\gamma_n\in \C$, we obtain (as the first order expansion of $\p{z}^2\phi$ is a tangent vector $\h_0=O(|z|^{-(k+1)})$)
	\begin{align*}
	g^{-1}\otimes \h_0\wedge\bar{\partial}\phi&=\alpha |z|^{2(k+1)}\left(\frac{\beta_1}{z^{k+1}},\cdots\,\frac{\beta_n}{z^{k+1}}\right)\wedge \bar{\left(\frac{{\alpha}_1}{z^{k+1}},\cdots,\frac{{\alpha}_n}{z^{k+1}}\right)}dz+O(1)
	=O(1).
	\end{align*}
\end{rem}

\begin{rem}
	This correspondence can be easily anticipated as follows. First, as the square of the inversion is the identity map, we now that the inversion can only exchange residues up to a factor of $\pm 1$. Furthermore, the second residue is the only real one (the other are vectorial) so the inversion can only let it invariant or change its sign. Then, wedge products do not appear by magic, so the third residue can only change by $\pm 1$. As the first residue cannot stay invariant for the inversion of a minimal surface with non-zero flux, as the first residue of any minimal surface vanishes identically, we deduce that the first and fourth residues must be exchanged, up to a multiplication by $-1$.
\end{rem}

As we have already seen, the three first residues of a minimal surface vanish, and for minimal surfaces with embedded ends, the fourth residue is nothing else that the flux. This last fact is general.

\begin{cor}\label{inversionmin}
	Let $\phi:\Sigma^2\setminus\ens{p_1,\cdots,p_m}\rightarrow\R^n$ be a complete minimal surface with finite total curvature. Then its flux of $\phi$ is equal to  
	its fourth residue as a Willmore surface, that is for all $p\in \Sigma^2$ for all smooth curve enclosing $p_j$ ($1\leq j\leq m$) and lying inside $\Sigma^2\setminus\ens{p_1,\cdots, p_m}$, we have
	\begin{align}\label{formule}
		\Im\int_{\gamma}\partial\phi=\Im\int_{\gamma}\mathscr{I}_{\phi}\left(g^{-1}\otimes \left(\bar{\partial}^N-\bar{\partial}^\top\right)\h_0-|\h_0|^2_{WP}\,\partial\phi\right)-g^{-1}\otimes \h_0\otimes\bar{\partial}|\phi|^2
	\end{align}
	where $\mathscr{I}_{\phi}(\vec{w})=|\phi|^2\vec{w}-2\s{\phi}{\vec{w}}\phi$, for any vector $\w\in \R^n$.
\end{cor}
\begin{proof}
	By the Weierstrass parametrisation, we have near a branch point 
	\begin{align*}
		\phi(z)=\Re\left(\sum_{j=1}^{\theta_0}\frac{\alpha_j^1}{z^j}+\beta_1\log(z)+O(1),\cdots,\sum_{j=1}^{\theta_0-1}\frac{\alpha_j^n}{z^j}+\beta_n\log(z)+O(1)\right)
	\end{align*}
	so the flux is $\vec{\gamma_0}=\Im(\beta_1,\cdots,\beta_n)$.
	As for the inverted minimal surface $\vec{\Psi}=\iota\circ \phi$, we have (see \cite{beriviere} or Section \ref{4.2}) close to a branch point 
	\begin{align}
	     \vec{\Psi}(z)&=\Re\left((\vec{A}+\vec{B}z+\vec{C}z^2)\z^{\theta_0}\right)+\mu\vec{\gamma}_0|z|^{2\theta_0}\log|z|+O(|z|^{\theta_0+3})
	\end{align}
	for some $\mu>0$. As thanks to \cite{beriviere}, $\vec{\gamma_0}$ is the first residue of $\vec{\Psi}$, the correspondence shows that the fourth residue of an arbitrary minimal surface is nothing else than the flux up to a constant, which is equal to $+1$ thanks to the Proposition \ref{fluxresidue}.
\end{proof}

\begin{rem}
	We now see that Theorem \ref{A} in the introduction is the combination of Theorem \ref{correspondance} and of corollary \ref{inversionmin}.
\end{rem}

	\section{Meromorphic quartic form and Willmore surfaces in $S^n$}
	\subsection{Algebraic structure of the quartic form}
	
	  	On $\R^{n+2}$ introduce the Lorentzian metric of signature $(1,n+1)$
    	\begin{align*}
    		h=-dx_0^2+\sum_{j=1}^{n+1}dx_j^2
    	\end{align*}
    	and denote by $S^{n,1}$ the unit Lorentzian sphere, defined by
    	\begin{align*}
    		S^{n,1}=\R^{n+2}\cap\ens{x=(x_0,x_1,\cdots,x_{n+1}):|x|_h^2=-x_0^2+\sum_{j=1}^{n+1}x_j^2=1}.
    	\end{align*}
    	Let $\psi_{\phi}:\Sigma^2\rightarrow S^{n,1}\subset \R^{n+2}$ be the section defined on the normal bundle $T_{\C}^N\Sigma^2$, for all normal section $\vv{\xi}$ by
    	\begin{align*}
    	{\psi}_{\phi}(\vv{\xi})=\s{\vec{H}}{\vv{\xi}}(\vec{a}+\phi)+\vv{\xi}
    	\end{align*}
    	where $\vec{a}=(1,0,\cdots,0)\in \R^{n+2}$, and $\phi\in \R^{n+1}$ (resp. $\vv{\xi}$) is identified with $(0,\phi)\in \R^{n+2}$ (resp. $(0,\vv{\xi})$). 
    	Then for all normal section $\vv{\xi}$ such that $|\vv{\xi}|=1$, we have
    	\begin{align*}
    		\s{\psi_{\phi}}{\psi_{\phi}}_h=-\s{\H}{\vv{\xi}}^2+\s{\H}{\vv{\xi}}^2|\phi|^2+|\vv{\xi}|^2=1,
    	\end{align*}	
    	and $\psi_{\phi}:\Sigma^2\rightarrow S^{n,1}$ is called the pseudo Gauss map of $\phi$. If $n=3$, and $\n$ is the unit normal we can choose $\vv{\xi}=\n$  (the unit normal) which gives
    	\begin{align*}
    			\psi_{\phi}=(H,\phi H+\n).
    	\end{align*}
    	Then we have the following result of Bryant. 
    	\begin{theorem}
    		Let $\phi:\Sigma^2\rightarrow S^3$ be a smooth immersion of an oriented surface and endow $\Sigma^2$ with the induced conformal structure. Then $\psi_{\phi}:\Sigma^2\rightarrow S^{3,1}$ is weakly conformal, it is an immersion away from the umbilic locus of $\phi$, and if $\phi$ is a Willmore immersion, the $4$-form $\mathscr{Q}_{\phi}$ defined by
    		\begin{align*}
    		\mathscr{Q}_{\phi}=\s{\partial^2\psi_{\phi}}{\partial^2\psi_{\phi}}_h
    		\end{align*}
    			is a holomorphic quartic form. Furthermore, $\phi:\Sigma^2\rightarrow S^3$ is a Willmore surface if and only if $\psi_{\phi}:\Sigma^2\rightarrow S^{3,1}$ is harmonic.
    	\end{theorem}
    \begin{proof}
    	We first check that  $\psi_{\phi}$ is 
    	(weakly) conformal. Writing $\psi$ for $\psi_{\phi}$, we have
    	\begin{align*}
    		\p{z}\psi_{}=(\p{z}H,H\p{z}\phi+\p{z}H\phi+\p{z}\n)
    	\end{align*}
    	so
    	\begin{align*}
    		&\s{\p{z}\psi_{}}{\p{z}\psi_{}}_h=-(\p{z}H)^2+H^2\s{\p{z}\phi}{\p{z}\phi}+(\p{z}H)^2|\phi|^2+\s{\p{z}\n}{\p{z}\n}+2H\p{z}H\s{\p{z}\phi}{\phi}\\
    		&+2H\s{\p{z}\phi}{\p{z}\n}+2\p{z}H\s{\phi}{\p{z}\n}
    		=\s{\p{z}\n}{\p{z}\n}+2H\s{\p{z}\phi}{\p{z}\n}-2\p{z}H\s{\p{z}\phi}{\n}
    		=\s{\p{z}\n}{\p{z}\n}+2H\s{\p{z}\phi}{\p{z}\n}.
    	\end{align*}
    	We have
    	\begin{align*}
    		\p{z}\n=-\frac{e^{2\lambda}}{2}H\p{z}\phi-\frac{e^{2\lambda}}{2}H_0\p{\z}\phi
    	\end{align*}
    	so
    	\begin{align*}
    		\s{\p{z}\n}{\p{z}\n}&=2HH_0\s{\p{z}\phi}{\p{\z}\phi}=e^{2\lambda}\,H H_0,\qquad
    		\s{\p{z}\phi}{\p{z}\n}=-H_0\s{\p{z}\phi}{\p{\z}\phi}=-\frac{e^{2\lambda}}{2}H_0
    	\end{align*}
    	and this gives 
    	\begin{align*}
    		\s{\p{z}\psi}{\p{z}\psi}_h=0,
    	\end{align*}
    	showing the weak conformality of $\psi$. Furthermore, the pull-back of the Lorentzian metric $h$ on $\Sigma^2$ exactly gives the Willmore energy, which explains the name pseudo Gauss map, by analogy with minimal surfaces and total Gauss curvature. Indeed, one has
    	\begin{align*}
    		\s{\p{z}\psi}{\p{\z}\psi}_h=\frac{e^{2\lambda}}{4}|H_0|^2=\frac{1}{4}\left(|\H|^2-K_g+1\right)d\vg.
    	\end{align*}
    	Therefore
    	\begin{align*}
    		W(\phi)=\int_{\Sigma^2}\left(|\H|^2-K_g+1\right)d\vg=4\int_{\Sigma^2} \psi_{\phi}^\ast (d\mathrm{vol}_h).
    	\end{align*}
    	As the metric is non-positive definite, this does not imply anything on the quantization of the Willmore energy. The holomorphy of $\mathscr{Q}_{\phi}$ can be found in \cite{bryant}, theorem B, and shall be treated in general in the next theorem, once we find a pleasant expression to work with of $\mathscr{Q}_{\phi}$. Finally the last assertion can be found in a general context in \cite{ejiri}.
    \end{proof}

    	We have the following expression of the quartic form $\mathscr{Q}_{\phi}$.

    	\begin{lemme}\label{s3}
    		Let $\phi:\Sigma^2\rightarrow S^3$ be a smooth immersion of an oriented surface $\Sigma^2$. Then we have in any conformal chart
    		\begin{align}\label{Q=Q4}
    			\mathscr{Q}_{\phi}&=\s{\D_{\partial_z}^2\psi_{\phi}}{\D_{\partial_z}^2\psi_{\phi}}dz^{4}\nonumber\\
    			&=e^{2\lambda}\left(\s{\D_{\partial_z}^\nperp\D_{\partial_z}^\nperp\H}{\H_0}-\s{\D_{\partial_z}^\nperp\H}{\D_{\partial_z}^\nperp\H_0}\right)dz^{4}+\frac{e^{4\lambda}}{4}\left(1+|\H|^2\right)\s{\H_0}{\H_0}dz^4
    		\end{align}
    	\end{lemme}
     \begin{proof}
     	If $\phi:\Sigma^2\rightarrow S^n$, dropping the index $\phi$ for simplicity,  we obtain 
     	\begin{align*}
     	(\D_{\p{z}}{\psi})(\vv{\xi})&=\D_{\p{z}}\left({\psi}\left({\xi}\right)\right)-{\psi}\left(\D_{\p{z}}^N\vv{\xi}\right)\\
     	&=\left(\s{\D_{\p{z}}^N\H}{\vv{\xi}}+\s{\H}{\D_{\p{z}}^N\vv{\xi}}\right)(\vec{a}+\phi)+\s{\H}{\vv{\xi}}\e_z+\D_{\p{z}}\vv{\xi}-\left(\s{\H}{\D_{\p{z}}\vv{\xi}}(\vec{a}+\phi)+\D_{\p{z}}^N\vv{\xi}\right)\\
     	&=\s{\D_{\p{z}}^N\H}{\vv{\xi}}(\vec{a}+\phi)+\s{\H}{\vv{\xi}}\e_z+\D_{\p{z}}^\top\vv{\xi}.
     	\end{align*}
     	As
     	\begin{align*}
     	\D_{\p{z}}^\top\vv{\xi}=-\s{\H}{\vv{\xi}}\e_z-\s{\vec{H}_0}{\vv{\xi}}\e_{\z},
     	\end{align*}
     	one obtains
     	\begin{align*}
     	\left(\D_{\p{z}}{\psi}\right)(\vv{\xi})=\s{\D_{\p{z}}^N\H}{\vv{\xi}}(\vec{a}+\phi)-\s{\H_0}{\vv{\xi}}\e_{\z}.
     	\end{align*}
     	Then we have
     	\begin{align*}
     	\left(\D_{\p{z}}\D_{\p{z}}{\psi}\right)\left(\vv{\xi}\right)&=\D_{\p{z}}\left(\D_{\p{z}}{\psi}(\vv{\xi})\right)-\left(\D_{\p{z}}{\psi}\right)(\D_{\p{z}}^\nperp\vv{\xi})\\
     	&=\left(\s{\D_{\p{z}}^N\D_{\p{z}}^N\H}{\vv{\xi}}+\s{\D_{\p{z}}\vec{H}}{\D_{\p{z}}\vv{\xi}}\right)(\vec{a}+\phi)+\s{\D_{\p{z}}^N\H}{\vv{\xi}}\e_z-\left(\s{\D_{\p{z}}^\nperp\H_0}{\vv{\xi}}+\s{\H_0}{\D_{\p{z}}^\nperp\vv{\xi}}\right)\e_{\z}\\
     	&-\s{\H_0}{\vv{\xi}}\D_{\p{z}}\e_{\z}-\left(\s{\D_{\p{z}}^N\vec{H}}{\D_{\p{z}}^N\vv{\xi}}(\vec{a}+\phi)-\s{\H_0}{\D_{\p{z}}\vv{\xi}}\e_{\z}\right)\\
     	&=\s{\D_{\p{z}}^N\D_{\p{z}}^N\H}{\vv{\xi}}(\vec{a}+\phi)+\s{\D_{\p{z}}^N\H}{\vv{\xi}}\e_{z}-\s{\D_{\p{z}}^N\vec{H}_0}{\vv{\xi}}\e_{\z}-\s{\vec{H}_0}{\vv{\xi}}\D_{\p{z}}\e_{\z}.
     	\end{align*}
     	As $\s{\phi}{\D_{\p{z}}\e_{\z}}=-\dfrac{e^{2\lambda}}{2}$, one immediately obtains
     	\begin{align*}
     	&\s{\D_{\p{z}}\D_{\p{z}}{\psi}}{\D_{\p{z}}\D_{\p{z}}{\psi}}_h(\vv{\xi},\vv{\eta})=\frac{e^{2\lambda}}{2}\left(\s{\D_{\p{z}}^N\D_{\p{z}}^N\H}{\vv{\xi}}\s{\H_0}{\vv{\eta}}-\s{\D_{\p{z}}^N\vec{H}}{\vv{\xi}}\s{\D_{\p{z}}^N\H_0}{\vv{\eta}}\right)\\
     	&+\frac{e^{2\lambda}}{2}\left(\s{\D_{\p{z}}^N\D_{\p{z}}^N\H}{\vv{\eta}}\s{\H_0}{\vv{\xi}}-\s{\D_{\p{z}}^N\H}{\vv{\eta}}\s{\D_{\p{z}}^N\H_0}{\vv{\xi}}\right)
     	+\frac{e^{4\lambda}}{4}\s{\H_0}{\vv{\xi}}\s{\H_0}{\vv{\eta}}(1+|\H|^2)
     	\end{align*}
     	so for $n=3$, we have a global non-zero section $\n:\Sigma^2\rightarrow S^3$ of $\mathscr{N}$, so taking $\vv{\xi}=\vv{\eta}=\n$, we obtain the expression of the lemma.
     \end{proof}
     The next step is to show that $\mathscr{Q}_{\phi}$ admits an intrinsic expression whose principal term only depends on the Weingarten tensor $\vec{h}_0$.
    
    \begin{defi}\label{tensorproduct}
    	If $\Sigma^2$ is a closed Riemann surface, $n\geq 1$ is a fixed integer, $(M^n,h)$ is a smooth Riemannian manifold, $\s{\,\cdot\,}{\,\cdot\,}$ its scalar product, $p_1,p_2,q_1,q_2\geq 0$, and
    	\begin{align*}
    		(\vec{\alpha_1},\vec{\alpha_2})\in \Gamma(K_{\Sigma^2}^{p_1}\otimes \bar{K}_{\Sigma^2}^{\,q_1},T_{\C}M^n)\times \Gamma(K_{\Sigma^2}^{p_1}\otimes\bar{K}_{\Sigma^2}^{\,q_2},T_{\C}M^n)
    	\end{align*}
    	 are continuous sections with values in $T_{\C}M^n$, we define
    	\begin{align*}
    		\vec{\alpha}_1\totimes\vec{\alpha_2}\in \Gamma(K_{\Sigma^2}^{p_1+p_2}\otimes\bar{K}_{\Sigma^2}^{\,q_1+q_2},\C)
    	\end{align*}
    	by 
    	\begin{align*}
    		\vec{\alpha}_1\totimes\vec{\alpha}_2=\s{\vec{f}_1(z)}{\vec{f}_2(z)}dz^{p_1+p_2}\otimes d\z^{\,q_1+q_2}
    	\end{align*}
    	if in a local complex chart $z$ we have the expressions
    	\begin{align*}
    		\left\{\begin{alignedat}{1}
    		\vec{\alpha}_1&=\vec{f}_1(z)dz^{p_1}\otimes d\z^{q_1}\\
    		\vec{\alpha}_2&=\vec{f}_2(z)dz^{p_2}\otimes d\z^{q_2}.
    		\end{alignedat}\right.
    	\end{align*}
    \end{defi}

    \begin{theorem}\label{int}
    	Let $\phi:\Sigma^2\rightarrow S^3$ be a smooth immersion. Then we have
    	\begin{align}\label{intrinsèque}
    		\mathscr{Q}_{\phi}&=g^{-1}\totimes\left(\partial^N\bar{\partial}^N\vec{h}_0\totimes \vec{h}_0-\partial^N\vec{h}_0\totimes\bar{\partial}^N\vec{h}_0\right)+\frac{1}{4}(1+|\H|^2)\h_0\totimes \h_0\nonumber\\
    		&=g^{-1}\otimes\left(\partial\bar{\partial}\h_0\totimes\h_0-\partial\h_0\totimes\bar{\partial}\h_0\right)+\left(\frac{1}{4}(1+|\H|^2)+|\h_0|^2_{WP}\right)\h_0\totimes\h_0+\s{\H}{\h_0}^2
    	\end{align}
    	is a quartic differential, that is a section of $K_{\Sigma^2}^4$. Furthermore, if $\phi$ is a smooth Willmore surface, $\mathscr{Q}_{\phi}$ is holomorphic.
    \end{theorem}
    \begin{proof}
    	By \eqref{codazzi}, we have
    	\begin{align*}
    		\D_{\p{z}}^N\H=e^{-2\lambda}\D_{\p{\z}}^N\h_0
    	\end{align*}
       	Then we have, identifying by an abuse of notation $\h_0$ and $e^{2\lambda}\H_0$
       	\begin{align*}
       		\D_{\partial_z}^\nperp\vec{H}_0&=\D_{\partial_z}^\nperp\left(e^{-2\lambda}e^{2\lambda}\H_0\right)
       		=\partial_z(e^{-2\lambda})e^{2\lambda}\H_0+e^{-2\lambda}\D_{\partial_z}^\nperp(e^{2\lambda}\H_0)
       		=\partial_z(e^{-2\lambda})\vec{h}_0+e^{-2\lambda}\D_{\partial_z}^\nperp\vec{h}_0.
       	\end{align*}
       	Therefore
       	\begin{align}\label{stepQ4}
       		&\s{\D_{\partial_z}^\nperp\D_{\partial_z}^\nperp\H}{\H_0}-\s{\D_{\partial_z}^\nperp\H}{\D_{\partial_z}^\nperp\H_0}=\s{\D_{\partial_z}^\nperp\left(e^{-2\lambda}\D_{\partial_{\z}}^\nperp\vec{h}_0\right)}{\H_0}-e^{-2\lambda}\s{\D_{\partial_{\z}}^\nperp\vec{h}_0}{\partial_z(e^{-2\lambda})\vec{h}_0+e^{-2\lambda}\D_{\partial_z}^\nperp\vec{h}_0}\nonumber\\
       		&=e^{-2\lambda}\partial_z(e^{-2\lambda})\s{\D_{\partial_{\z}}^\nperp\vec{h}_0}{\vec{h}_0}+e^{-4\lambda}\s{\D_{\partial_z}^\nperp\D_{\partial_{\z}}^\nperp\vec{h}_0}{\vec{h}_0}-e^{-2\lambda}\partial_z(e^{-2\lambda})\s{\D_{\partial_z}^\nperp\vec{h}_0}{\vec{h}_0}-e^{-4\lambda}\s{\D_{\partial_{\z}}^\nperp\vec{h}_0}{\D_{\partial_z}^\nperp\vec{h_0}}\nonumber\\
       		&=e^{-4\lambda}\left(\s{\D_{\partial_z}^\nperp\D_{\partial_{\z}}^\nperp\vec{h}_0}{\vec{h}_0}-\s{\D_{\partial_z}^\nperp\vec{h}_0}{\D_{\partial_{\z}}^\nperp\vec{h}_0}\right)
       	\end{align}
       	We deduce from \eqref{Q=Q4} and \eqref{stepQ4} that
       	\begin{align*}
       		\mathscr{Q}&=e^{2\lambda}\left(\s{\D_{\partial_z}^\nperp\D_{\partial_z}^\nperp\H}{\H_0}-\s{\D_{\partial_z}^\nperp\H}{\D_{\partial_z}^\nperp\H_0}\right)dz^{4}+\frac{1}{4}(1+|\H|^2)\h_0\totimes\h_0\\
       		&=e^{-2\lambda}\left(\s{\D_{\partial_z}^\nperp\D_{\partial_{\z}}^\nperp\vec{h}_0}{\vec{h}_0}-\s{\D_{\partial_z}^\nperp\vec{h}_0}{\D_{\partial_{\z}}^\nperp\vec{h}_0}\right)dz^2+\frac{1}{4}(1+|\H|^2)\h_0\otimes \h_0\\
       		&=g^{-1}\otimes\left(\partial^N\bar{\partial}^N\h_0\totimes\h_0-\partial^N\h_0\totimes\bar{\partial}^N\h_0\right)+\frac{1}{4}(1+|\H|^2)\h_0\totimes\h_0.
       	\end{align*}
       	We see that this formula describes a well-defined tensor for any immersion. 
       	Now we note that actually we can obtain the second expression without the normal derivatives. Indeed, we have
       	\begin{align*}
       		\partial^\top\h_0&=-\s{\H}{\h_0}\otimes \partial\phi-g^{-1}\otimes (\h_0\totimes\h_0)\otimes\bar{\partial}\phi\\
       		\bar{\partial}^\top\h_0&=-|\h_0|^2_{WP}\;g\otimes \partial\phi-\s{\H}{\h_0}\otimes \bar{\partial}\phi
       	\end{align*}
       	so as $g=2\partial\phi\,\dot{\otimes}\,\bar{\partial}\phi$, we have
       	\begin{align*}
       		\partial^N\bar{\partial}^\top\h_0=-\frac{1}{2}g\otimes \left(|\h_0|_{WP}^2\;\h_0+\s{\H}{\h_0}\H\right)
       	\end{align*}
       	and
       	\begin{align*}
       		\partial^N\bar{\partial}^N\h_0\totimes\h_0=\partial^N\bar{\partial}\h_0\totimes\h_0-\partial^N\bar{\partial}^\top\h_0=\partial\bar{\partial}\h_0\totimes\h_0+\frac{1}{2}g\otimes\left(|\h_0|^2_{WP}\;\h_0\totimes\h_0+\s{\H}{\h_0}^2\right)
       	\end{align*}
       	while
       	\begin{align*}
       		\partial^N\h_0\totimes\bar{\partial}^N\h_0&=\partial\h_0\totimes\bar{\partial}\h_0-\partial^\top\h_0\totimes\bar{\partial}\h_0-\partial\h_0\totimes\bar{\partial}^\top\h_0+\partial^\top\h_0\totimes\bar{\partial}\h_0
       		=\partial\h_0\totimes\bar{\partial}\h_0-\partial^\top\h_0\totimes\bar{\partial}^\top\h_0.
       	\end{align*}
       	As $\partial\phi\totimes\partial\phi=0$ by conformality, one has
       	\begin{align*}
       		\partial^\top\h_0\totimes\bar{\partial}^\top\h_0=\frac{1}{2}g\otimes \left(|\h_0|_{WP}^2\;\h_0\totimes\h_0+\s{\H}{\h_0}^2\right),
       	\end{align*}
        so we deduce that
        \begin{align*}
        	\mathscr{Q}=g^{-1}\otimes\left(\partial\bar{\partial}\h_0\totimes\h_0-\partial\h_0\totimes\bar{\partial}\h_0\right)+\left(\frac{1}{4}(1+|\H|^2)+|\h_0|_{WP}^2\right)\h_0\totimes\h_0+\s{\H}{\h_0}^2.
        \end{align*}
       	Now suppose that $\phi:\Sigma^2\rightarrow S^3$ is a \emph{smooth} Willmore immersion.
       	To see that $\mathscr{Q}_{\phi}$ is holomorphic, the expression of the previous lemma is more useful, as the Willmore equation is more easily stated with respect to $\H$. We remark that for a stereographic projection $\pi:S^3\rightarrow \R^3$, the quartic form becomes (without changing the notations for the involved quantities)
       	\begin{align*}
       		\mathscr{Q}_{\phi}=g^{-1}\totimes \left(\partial^N\bar{\partial}^N\h_0\totimes \h_0-\partial^N\h_0\totimes \bar{\partial}^N\h_0\right)+\frac{1}{4}|\H|^2\h_0\totimes\h_0
       	\end{align*}
       	so in a conformal chart $z:D^2\rightarrow \Sigma^2$, we have by \eqref{Q=Q4}
       	\begin{align}\label{i-1}
       		\mathscr{Q}_{\phi}&=e^{2\lambda}\left(\s{\partial^2\H}{\H_0}-\s{\partial\H}{{\partial}\H_0}\right)dz^4+\frac{e^{4\lambda}}{4}|\H|^2\s{\H_0}{\H_0}dz^4\nonumber\\
       		&=\left\{e^{2\lambda}(\partial^2H\,H_0-\partial H\, \partial H_0)+\frac{e^{4\lambda}}{4}H^2H_0^2\right\}dz^4
       	\end{align}       	
       	Recall that the Willmore equation is equivalent in $\R^3$ to 
       	\begin{align*}
       		\partial\bar{\partial}H+\frac{e^{2\lambda}}{2} |H_0|^2H=0
       	\end{align*}
       	First, we have
        \begin{align}\label{i}
        	\bar{\partial}\partial^2H=\partial(\partial\bar{\partial}H)=-e^{2\lambda}(\p{z}\lambda)|H_0|^2H-\frac{e^{2\lambda}}{2}\left(\partial H_0 \bar{H}_0 H+{\partial}\bar{H_0}H_0H+|H_0|^2\partial H\right).
        \end{align}
        Then we have
        \begin{align*}
        	\partial H=e^{-2\lambda}\bar{\partial}(e^{2\lambda}H_0)=2(\p{\z}\lambda)H_0+\bar{\partial}H_0
        \end{align*}
        Finally, we obtain
        \begin{align}\label{ii}
        	\bar{\partial}\left(\partial^2H \,H_0\right)=&-e^{2\lambda}(\p{z}\lambda)|H_0|^2H H_0-\frac{e^{2\lambda}}{2}\left(\partial H_0\, |H_0|^2H+\partial \bar{H_0}\,H H_0^2+\partial H\, |H_0|^2H_0\right)\nonumber\\
        	&+\partial^2H\left(\partial H-2(\p{\z}\lambda)H_0\right)\nonumber\\
        	&=\partial^2 H(\partial H-2(\p{\z}\lambda)H_0)-\frac{e^{2\lambda}}{2}|H_0|^2(\partial (H H_0)+2(\p{\z}\lambda) H H_0)-\frac{e^{2\lambda}}{2}\partial\bar{H_0}H H_0^2.
        \end{align}
        Then we have
        \begin{align*}
        	\bar{\partial}(\partial H\,\partial H_0)&=\partial\bar{\partial}H\,\partial H_0+\partial H\partial\bar{\partial}H_0
        	=-\frac{e^{2\lambda}}{2}|H_0|^2H \partial H_0+\partial H \partial(\partial H-2(\p{\z}\lambda)H_0)\\
        	&=\partial^2 H\,\partial H-2(\p{z\z}^2\lambda)\partial H\, H_0-2(\p{\z}\lambda)\partial H\partial H_0-\frac{e^{2\lambda}}{2}|H_0|^2H\, \partial H_0
        \end{align*}
        By the Liouville equation, we have
        \begin{align*}
        	4\,\p{z\z}^2\lambda=\Delta \lambda=-e^{2\lambda}K_g,
        \end{align*}
        so we obtain as $K_g=H^2-|H_0|^2$
        \begin{align}\label{iii}
        	\bar{\partial}(\partial H\,\partial H_0)&=\partial^2 H \,\partial H+\frac{e^{2\lambda}}{2}(H^2-|H_0|^2)\partial H\,H_0-2(\p{\z}\lambda)\partial H\,\partial H_0-\frac{e^{2\lambda}}{2}|H_0|^2\partial H_0\, H\nonumber\\
        	&=\partial^2H\,\partial H+\frac{e^{2\lambda}}{2}H^2\,\partial H\,H_0-\frac{e^{2\lambda}}{2}|H_0|^2\left(\partial(H H_0)\right)-2(\p{\z}\lambda)\partial H\,\partial H_0
        \end{align}
        Therefore, by \eqref{ii} and \eqref{iii}, we have
        \begin{align*}
        	\bar{\partial}\left(\partial^2H\,H_0-\partial H\partial H_0\right)=-2(\p{\z}\lambda)(\partial^2 H\,H_0-\partial H\,\partial H_0)-\frac{e^{4\lambda}}{2}H H_0(2(\p{z}\lambda)|H_0|^2+\partial H\, H+\partial \bar{H_0}\,H_0)
        \end{align*}
        so
        \begin{align*}
        	e^{2\lambda}\bar{\partial}(\partial^2H\,H_0-\partial H\partial H_0)=-\p{\z}(e^{2\lambda})(\partial^2H\, H_0-\partial H\,\partial H_0)-\frac{e^{4\lambda}}{4}HH_0\left(2(\p{z}\lambda)|H_0|^2+\partial H\,H+\partial \bar{H}_0\,H_0\right),
        \end{align*}
        which reduces to
        \begin{align}\label{iv}
        	\bar{\partial}\left(e^{2\lambda}\left(\partial^2H\,H_0-\partial H\,\partial H_0\right)\right)&=-\frac{e^{4\lambda}}{2}H H_0(2(\p{z}\lambda)|H_0|^2+\partial H H+\partial \bar{H}_0\, H_0)\\
        	&=-\frac{e^{4\lambda}}{2}HH_0(\bar{\partial}HH_0+\partial H H).
        \end{align}
        The end is easy, as
        \begin{align}\label{v}
        	&\bar{\partial}\left(\frac{e^{4\lambda}}{4}H^2H_0^2\right)=e^{4\lambda}(\p{\z}\lambda)H^2H_0^2+\frac{e^{4\lambda}}{2}\left(\bar{\partial}H\,H H_0^2+\bar{\partial}H_0 H_0 H^2\right)\nonumber\\
        	&=e^{4\lambda}(\p{\z}\lambda)H^2H_0^2+\frac{e^{4\lambda}}{2}H_0H \left(\bar{\partial}H H_0+(\partial H-2(\p{\z}\lambda)H_0)H\right)
        	=\frac{e^{4\lambda}}{2}H H_0(\bar{\partial}H H_0+\partial H H).
        \end{align}
        Therefore, by \eqref{i-1}, \eqref{iv} and \eqref{v}, we deduce that 
        \begin{align*}
        \bar{\partial}\mathscr{Q}_{\phi}=\bar{\partial}\left(e^{2\lambda}\left(\partial H\, H_0-\partial H\,\partial H_0\right)\right)dz^4\otimes d\z+\bar{\partial}\left(\frac{e^{4\lambda}}{4}H^2H_0^2\right)dz^4\otimes d\z=0.
        \end{align*}
        Therefore $\mathscr{Q}_{\phi}$ is a holomorphic section of $K_{\Sigma^2}^4$ if $\phi:\Sigma^2\rightarrow S^3$ is a smooth Willmore surface.
       \end{proof}
   
   \begin{remimp}
   	We see that the tensor $\mathscr{Q}_{\phi}$ as defined in \eqref{intrinsèque} is well-defined for any immersion $\phi:\Sigma^2\rightarrow S^n$ for any $n\geq 3$ as the equation defining $\mathscr{Q}_{\phi}$ makes sense in any codimension, but we shall see that it is \emph{not} holomorphic in general in dimension $n\geq 4$ (see Section \ref{s4}).
   	
   	Furthermore, one might think that the Definition \ref{tensorproduct} is a bit artificial, as in codimension $1$, we have $\h_0=h_0\n$ for a scalar quadratic differential $h_0$, and as $\partial^N\n=0$, we have
   	\begin{align*}
   		\mathscr{Q}_{\phi}=g^{-1}\otimes\left(\partial\bar{\partial}h_0\otimes h_0-\partial h_0\otimes\bar{\partial}h_0\right)+\frac{1}{4}\left(1+|\H|^2\right)h_0\otimes h_0.
   	\end{align*}
   	However, not only for the generalisation in $S^4$, but already in the proof in the case of $S^3$ of the generalisation of Bryant's classification, it will be absolutely crucial to see $\mathscr{Q}_{\phi}$ as a function of the \emph{vectorial} $\h_0$ (see the proof of Theorem \ref{devh0} for more details).
   \end{remimp}

    \subsection{Asymptotic behaviour of the quartic form at branch points}\label{4.2}

    Bryant's theorem asserts that for any branched immersion $\phi:\Sigma^2\rightarrow \R^3$, if the quartic form $\mathscr{Q}_{\phi}=0$, then $\phi$ is the inversion of a complete minimal surface with finite total curvature. The partial converse is furnished by the following result.
    \begin{theorem}\label{devinv}
    	Let $\Sigma^2$ be a Riemann surface of genus $\gamma$, and $\phi:\Sigma^2\rightarrow\R^3$ be a non-completely umbilic branched Willmore surface. If $\phi$ is the inversion of a minimal surface if and only if $\mathscr{Q}_{\phi}=0$ is holomorphic. Furthermore, provided $\mathscr{Q}_{\phi}$, the dual minimal surface has zero flux if and only if $\phi$ is a \emph{true} Willmore surface.
    \end{theorem}
    \begin{proof}
    	Let $\vec{\Psi}:\Sigma^2\setminus\ens{p_1,\cdots,p_m}\rightarrow\R^3$ be a complete minimal surface with finite total curvature. Then $h_{\phi}$ is holomorphic, and $H_{\vec{\Psi}}=0$. As
    	\begin{align*}
    		\mathscr{Q}_{\vec{\Psi}}=g_{\vec{\Psi}}^{-1}\otimes\left(\partial\bar{\partial}h_{\vec{\Psi}}^0\otimes h_{\vec{\Psi}}^0-\partial h_{\vec{\Psi}}^0\otimes \bar{\partial} h_{\vec{\Psi}}^0\right)+\frac{1}{4}H_{\vec{\Psi}}^2h_{\vec{\Psi}}^2,
    	\end{align*}
    	we trivially obtain $\mathscr{Q}_{\vec{\Psi}}=0$, and as $\mathscr{Q}_{\vec{\Psi}}$ is conformally invariant, we deduce if $\phi$ is a compact inversion of $\vec{\Psi}$ that $\mathscr{Q}_{\phi}=0$. 
    	
    	Conversely, assume that $\mathscr{Q}_{\phi}=0$. Then Bryant's theorem (\cite{bryant}) implies that the dual Willmore 
    	$\vec{\Psi}:\Sigma^2\setminus\mathscr{U}_{\phi}\rightarrow\R^3$ surface is constant $\vec{\Psi}\equiv p\in S^3$, where 
    	\begin{align*}
    		\mathscr{U}_{\phi}=\Sigma^2\cap\ens{z:|\h_0(z)|^2_{WP}d\vg=0}
    	\end{align*}
    	is the umbilic locus. As the complement of $\mathscr{U}_{\phi}$ is an open dense set, and for some stereographic projection $\pi:S^3\setminus\ens{p}\rightarrow\R^3$, the composition $\pi\circ \phi:\Sigma^2\setminus\mathscr{U}_{\phi}\rightarrow\R^3$ has zero mean-curvature, we deduce that actually $\pi\circ \phi:\Sigma^2\setminus\phi^{-1}(\ens{p})\rightarrow\R^3$ has vanishing mean-curvature.
    	
    	 As there exists no compact minimal surface in $\R^3$, the set $\phi^{-1}(\ens{p})\subset \Sigma^2$ is non-empty, and the minimal surface $\pi\circ \phi:\Sigma^2\setminus\phi^{-1}(\ens{p})\rightarrow\R^3$ is a complete minimal surface with finite total curvature (by the conformal invariance of the Willmore energy). As $\phi^{-1}(\ens{p})$ can contain branch points, the ends of the dual minimal surface need not be embedded (as the inversions of any minimal surfaces with non-embedded ends show). 
    	Finally, the assertion on the residues is a direct consequence of the correspondence \ref{galois}.
    \end{proof}

    We first recall a preliminary lemma from \cite{lamm}.
    
    \begin{lemme}\label{polesorder}
    	Let $\phi:\Sigma^2\rightarrow\R^3$ be a branched Willmore surface, with branching divisor $D=\sum_{i=1}^{m}\theta_0(p_i) p_i$, where $p_1,\cdots,p_m\in\Sigma^2$ are distinct point and $\theta_0\geq 1$ are the multiplicities at the branch points, and $D_0=p_1+\cdots+p_m\in \mathrm{Div}(\Sigma^2)$. Then the meromorphic quartic form $\mathscr{Q}_{\phi}$ has poles of order at most two at each $p_i$, for $i=1,\cdots,m$, so it is a holomorphic section of the line bundle $\mathscr{L}=K_{\Sigma^2}^{4}\otimes \mathscr{O}(2D_0)$, where $K_{\Sigma^2}$ is the canonical bundle of $\Sigma^2$.
    \end{lemme}
\begin{proof}
	In the case of zero residue, by \cite{beriviere} for in some conformal chart $D^2\rightarrow\Sigma^2$, we have
	\begin{align}\label{obs}
	\phi(z)&=\Re\left(\vec{A} z^{\theta_0}+\vec{B}z^{\theta_0+1}+\vec{C}z^{\theta_0-\alpha}\z^{\theta_0}\right)+O(|z|^{\theta_0+2-\epsilon})
	\end{align}
	for some $\alpha\leq \theta_0-1$. In the worst case $\alpha=\theta_0-1$, an analogous computation as in Theorem \ref{devinv} gives for some constants $\vec{A}_0$, $\vec{A}_1\in \C^n$
	\begin{align*}
	\h_0(z)=\left(\vec{A}_0 z^{\theta_0-1}+\vec{A}_1\frac{\z^{\theta_0}}{z}\right)dz^2+O(|z|^{\theta_0-\epsilon}).
	\end{align*}
Therefore as $\H=O(|z|^{1-\theta_0})$, $\h_0=O(|z|^{\theta_0-1})$, we have
\begin{align*}
	|\H|^2\h_0\totimes\h_0=O(1),
\end{align*}
and
	\begin{align*}
	\mathscr{Q}&=g^{-1}\otimes (\partial\bar{\partial}\h_0\otimes \h_0-\partial\h_0\otimes\bar{\partial}\h_0)+\frac{1}{4}|\H|^2\h_0\totimes\h_0
	=-|z|^{2-2\theta_0}\s{\vec{A}_0}{A_1}\theta_0(\theta_0-1)z^{\theta_0-2}\frac{\z^{\theta_0-1}}{z}dz^4\\
	&+|z|^{2-2\theta_0}|\vec{A}_0|^2\bigg\{\left(-\theta_0\frac{\z^{\theta_0-1}}{z^2}\right)\left(\frac{\z^{\theta_0}}{z}\right)-\left(-\frac{\z^{\theta_0}}{z^2}\right)\left(\theta_0\frac{\z^{\theta_0}}{z}\right)\bigg\}dz^4+O\left(\frac{1}{|z|}\right)\\
	&=-\s{\vec{A}_0}{\vec{A}_1}\theta_0(\theta_0-1)\frac{dz^4}{z^2}+O\left(\frac{1}{|z|}\right)
	\end{align*}
	so the poles of $\h_0$ are of order at most $2$. For $\theta_0=1$, as we cannot neglect the residue, we also get in general a pole of order at most $2$, as the higher order of singularity of $\h_0$ is
	\begin{align*}
		\vec{\gamma}_0\frac{\z}{z}dz^2
	\end{align*}
	so the same computation applies.
\end{proof}

\begin{lemme}\label{secondresidue}
	Let $\phi:D^2\rightarrow \R^n$ be a branched Willmore disk, of branch point of order $\theta_0\geq 1$ and second residue such that $r(\phi,0)\leq \max\ens{0,\theta_0-3}$. Then 
	\begin{align*}
		|z|^{\epsilon}\mathscr{Q}_{\phi}\in L^{\infty}(D^2)\qquad\text{for all}\;\, \epsilon>0.
	\end{align*}
\end{lemme}
\begin{proof}
	First assume that $\theta_0\geq 3$. Then we have (up to renormalisation)
	\begin{align*}
	    &\phi(z)=\frac{2}{\theta_0}\,\Re\left(\vec{A}_0z^{\theta_0}\right)+O(|z|^{\theta_0+1-\epsilon}),\qquad
		e^{2\lambda}=|z|^{2\theta_0-2}\left(1+O(|z|)\right)\\
		&\H=\Re\left(\frac{\vec{C}_2}{z^{\theta_0-3}}\right)+O(|z|^{4-\theta_0-\epsilon}).
	\end{align*}
	Therefore, we have as $2\H=\Delta_g\phi$
	\begin{align*}
		\p{\z}\left(\p{z}\phi\right)=\frac{1}{2}|z|^{2\theta_0-2}\Re\left(\frac{\vec{C}_2}{z^{\theta_0-3}}\right)+O(|z|^{\theta_0+2-\epsilon})=\frac{1}{2}\Re\left(\vec{C}_2z^2\z^{\theta_0-1}\right)+O(|z|^{\theta_0+2-\epsilon})=O(|z|^{\theta_0+1-\epsilon})
	\end{align*}
	Integrating yields by Proposition \ref{integrating} (for some $\vec{A}_1,\vec{A}_2\in \C^n$)
	\begin{align*}
		\p{z}\phi=\vec{A}_0z^{\theta_0-1}+\vec{A}_1z^{\theta_0}+\vec{A}_2z^{\theta_0+1}+O(|z|^{\theta_0+2-\epsilon}).
	\end{align*}
	As  $\phi$ is conformal, we have
	\begin{align*}
		0=\s{\p{z}\phi}{\p{z}\phi}=\s{\vec{A}_0}{\vec{A}_0}z^{2\theta_0-2}+2\s{\vec{A}_0}{\vec{A}_1}z^{2\theta_0-1}+\left(\s{\vec{A}_1}{\vec{A}_1}+2\s{\vec{A}_0}{\vec{A}_2}\right)z^{2\theta_0}+O(|z|^{2\theta_0+1-\epsilon}).
	\end{align*}
	Therefore, we have
	\begin{align*}
		\s{\vec{A}_0}{\vec{A}_0}=\s{\vec{A}_0}{\vec{A}_1}=0.
	\end{align*}
	This implies that 
	\begin{align*}
		e^{2\lambda}=2|\p{z}\phi|^2=2|\vec{A}_0|^2|z|^{2\theta_0-2}+4\,\Re\left(\s{\bar{\vec{A}_0}}{\vec{A}_1}z^{\theta_0}\z^{\theta_0-1}\right)+O(|z|^{2\theta_0})=|z|^{2\theta_0-2}\left(1+2\,\Re\left(\alpha_0z\right)+O(|z|^2)\right),
	\end{align*}
	and
	\begin{align*}
		2(\p{z}\lambda)=\frac{(\theta_0-1)}{z}+\alpha_0+O(|z|).
	\end{align*}
	Therefore, we get
	\begin{align*}
		\frac{1}{2}\h_0=\p{z}^2\phi-2(\p{z}\lambda)\p{z}\phi=\left(\vec{A}_1-\alpha_0\vec{A}_0\right)z^{\theta_0-1}+\left(\vec{A}_2-\alpha_0\vec{A}_1\right)z^{\theta_0}+O(|z|^{\theta_0+1-\epsilon}).
	\end{align*}
	Therefore, we compute
	\begin{align*}
		&\partial\h_0=2(\theta_0-1)\left(\vec{A}_1-\alpha_0\vec{A}_0\right)z^{\theta_0-2}+O(|z|^{\theta_0-1})=O(|z|^{\theta_0-2})\\
		&\bar{\partial}\h_0=O(|z|^{\theta_0-\epsilon})\\
		&\partial\bar{\partial}\h_0=O(|z|^{\theta_0-1-\epsilon}).
	\end{align*}
	Therefore, we have
	\begin{align*}
		Q(\h_0)=\partial\bar{\partial}\h_0\totimes\h_0-\partial\h_0\totimes\bar{\partial}\h_0=O(|z|^{\theta_0-2-\epsilon})\times O(|z|^{\theta_0-\epsilon})-O(|z|^{\theta_0-1-\epsilon})\times O(|z|^{\theta_0-1-\epsilon})=O(|z|^{2\theta_0-2-2\epsilon})
	\end{align*}
	and as $\left(|\H|^2+|\h_0|^2_{WP}\right)\h_0\totimes\h_0=O(|z|^2)$ and $\s{\H}{\h_0}^2=O(|z|^2)$, we have
	\begin{align*}
		\mathscr{Q}_{\phi}=g^{-1}\otimes Q(\h_0)+\left(\frac{1}{4}|\H|^2+|\h_0|^2_{WP}\right)\h_0\totimes\h_0+\s{\H}{\h_0}^2=O(|z|^{-\epsilon}),
	\end{align*}
	and this concludes the proof of the Lemma (the cases $\theta_0=2$ is similar, and the case $\theta_0=1$ is trivial as $Q(\h_0)\in L^{\infty}$ whenever $\phi$ is smooth).
\end{proof}

\begin{rem}
	It is proved in \cite{blow-up} (see \cite{blow-up2} for the min-max case) that the second residue of variational Willmore (branched) immersions satisfies $r(p)\leq \max\ens{\theta_0(p)-2,0}$, and this is why some non-trivial work has to be done in the main Theorem of this article.
\end{rem}

    Let us now recall the proof of the main Theorem of \cite{lamm}.
    
    \begin{theorem}\label{easygeneralisation}
    	Let $\phi:S^2\rightarrow\R^3$ be a non-completely umbilic Willmore sphere with at most three branch points. Then $\phi$ is the inversion of a minimal surface, and furthermore, $\phi$ is a true Willmore sphere if and only if its dual minimal surface has zero-flux.
    \end{theorem}
    \begin{proof}
    	If $\phi:S^2\rightarrow \R^3$ has $m$ distinct branch points $p_1,\cdots,p_m$, $D_0=p_1+\cdots +p_m$, and $m\leq 3$, the degree of the line bundle
    	degree of $\mathscr{L}={K}^{4}_{S^2}\otimes \mathscr{O}(2D_0)$ is equal to
    	\begin{align*}
    		\mathrm{deg}(\mathscr{L})=4\,\mathrm{deg}(K_{S^2})+2m=2m-8<0
    	\end{align*}
    	and as $\mathscr{Q}_{\phi}\in H^0(S^2,\mathscr{L})$, we deduce as negative line bundle do have non-zero holomorphic sections
    	that $\mathscr{Q}_{\phi}=0$, so Bryant's theorem gives the first conclusion. As true Willmore spheres have vanishing residues, the correspondence \ref{galois} shows the last equivalence.
    	In general, by the Riemann-Roch theorem, if $m\geq 4$, as $\mathrm{deg}(\mathscr{L}^\ast\otimes K_{S^2})=6-2m<0$, we have
    	\begin{align*}
    	    \mathrm{dim}\,H^0(S^2,\mathscr{L})=1-0+\mathrm{deg}(\mathscr{L})=2m-7\geq 1,
    	\end{align*}
    	and we cannot conclude so easily.
    \end{proof}

    As there exist minimal spheres with two ends and arbitrary large (in absolute value) total curvature, there exists thereby Willmore spheres with less that two branch points of arbitrary large multiplicities at branch points. This fact suggests that the theorem shall always hold true, as the holomorphy of the quartic form only depends on the local expansion \ref{obs}.

    \subsection{Refined estimates for the Weingarten tensor}

We now state our main theorem in full generality.
  
    \begin{theorem}\label{devh0}
    	Let $\Sigma^2$ be a closed Riemann surface, $n\geq 3$ and $\phi:\Sigma^2\rightarrow S^n$ be a branched Willmore surface, with branching divisor
    	\begin{align*}
    		\theta_0(p_1)p_1+\cdots+\theta_0(p_m)p_m\in \mathrm{Div}(\Sigma^2).
    	\end{align*}
    	Suppose that for all $j\in\ens{1,\cdots,m}$ whenever $1\leq \theta_0(p_j)\leq  3$, the first residue $\vec{\gamma}_0(p_j)$ of $\phi$ vanishes and whenever $\theta_0(p_j)\geq 2$, the second residue $r_j(p_j)\in \ens{0,\cdots,\theta_0(p_j)-1}$ satisfies $r(p_j)\leq  \theta_0(p_j)-2$. Then the quartic differential $\mathscr{Q}_{\phi}$ has poles of order at most $1$ at branch point of order $\theta_0\geq 4$ and is in \emph{bounded} in a neighbourhood of branch points of order $1\leq \theta_0\leq 3$. Furthermore, suppose further  $\mathscr{Q}_{\phi}$ is meromorphic. 
    	Then 
    	\begin{align}\label{Qishol}
    		\mathscr{Q}_{\phi}\;\,\text{is holomorphic.}
    	\end{align}  
    	In particular, if $\Sigma^2$ has genus zero, then $\mathscr{Q}_{\phi}=0$, $\phi:\Sigma^2\rightarrow S^3$ is the inverse stereographic projection of a complete branched minimal surface in $\R^3$ with finite total curvature. The dual minimal surface has \emph{vanishing flux} if and only if $\phi$ is a \emph{true} Willmore sphere.
    \end{theorem}
	\begin{proof}
    \textbf{Part 1. Introduction.}
	First, we recall that
	\begin{align}\label{willmorealt1}
		&-4\,\Im\left(g^{-1}\otimes\left( \bar{\partial}^N\h_0+\s{\H}{\h_0}\otimes\bar{\partial}\phi\right)\right)=\left(\ast_g d\H-3\ast_g(d\H)^{\nperp}+\star \left(\H\wedge d\n\right)\right)
	\end{align}
	and
	\begin{align}\label{willmorealt2}
		\Im\left(g^{-1}\otimes \left(\bar{\partial}^N\h_0+\s{\H}{\h_0}\otimes \bar{\partial}\phi\right)\right)&=\Im\left(\partial\H+|\H|^2\partial\phi+2g^{-1}\otimes \s{\H}{\h_0}\otimes\bar{\partial}\phi\right).
	\end{align}
	Taking some stereographic projection whose base point is not included in $\phi(\Sigma^2)$, we can suppose by the conformal invariance of the Willmore energy that $\phi:\Sigma^2\rightarrow \R^n$.
	
	We fix some $j\in \ens{1,\cdots,m}$ and we choose some open $U\subset \Sigma^2$ such that $p_j\in U$, and a conformal chart $z:U\rightarrow D^2\subset \C$ such that $z(p_i)=0$. Therefore, we can suppose that $\phi:D^2\setminus\ens{0}\rightarrow \R^n$ is a Willmore disk, with a branch point at $0$ of order $\theta_0=\theta_0(p_j)\geq 1$. We note that in particular $\phi\in C^{\infty}(D^2\setminus\ens{0})$, so there will be no regularity issues in the application of Poincaré lemma.
	
	As the first residue $\vec{\gamma}_0=\vec{\gamma}_0(p_j)$ is defined as in \cite{beriviere}, we have
	\begin{align*}
		d\left(\ast_g d\H-3\ast_g(d\H)^N+\star (\H\wedge d\n)\right)=4\pi\vec{\gamma}_0\,\delta_0
	\end{align*}
	where $\delta_0$ is the Dirac mass at $0\in D^2$, 
	we have by \eqref{willmorealt1} and \eqref{willmorealt2}
	\begin{align*}
		d\,\Im\left(g^{-1}\otimes\left( \bar{\partial}^N\h_0+\s{\H}{\h_0}\otimes\bar{\partial}\phi\right)\right)=-\pi\vec{\gamma}_0\delta_0
	\end{align*}
	\begin{rem}
		If we take instead $\vec{\gamma}_0$ as in our definition in \eqref{residues1234}, it is changed by a $-4$ factor.
	\end{rem}
    In particular, the $1$-form
	\begin{align*}
		\Im\left(\partial\H+|\H|^2\partial\phi+2g^{-1}\otimes \s{\H}{\h_0}\otimes\bar{\partial}\phi+\vec{\gamma}_0\,\partial\log|z|\right)
	\end{align*}
	is closed on $D^2\setminus\ens{0}$ and has zero winding number (around $0$), so it is exact and by Poincaré lemma, and there exists a smooth function $\vec{L}:D^2\setminus\ens{0}\rightarrow\R^n$ such that we have
	\begin{align*}
		d\vec{L}=
		\Im\left(\partial\H+|\H|^2\partial\phi+2g^{-1}\otimes \s{\H}{\h_0}\otimes\bar{\partial}\phi+\vec{\gamma}_0\,\partial\log|z|\right).
	\end{align*}
	The canonical complex structure induced from $\C$ on $D^2_{\ast}=D^2\setminus\ens{0}$ yields a direct sum decomposition of the $\C$-vector space $\Omega^1(D^2\setminus\ens{0},\C^n)$ of $1$-differential forms with values in $\C^n$ as
	\begin{align}\label{directsum}
		\Omega^1(D^2_{\ast},\C^n)=\Omega^{(1,0)}(D^2_{\ast},\C^n)\oplus \Omega^{(0,1)}(D^2_{\ast},\C^n).
	\end{align}
	In other word, if $z$ is the global complex coordinate, then
	\begin{align*}
		\Omega^{(1,0)}(D^2_{\ast},\C^n)&=\Omega^1(D^2_{\ast},\C^n)\cap\ens{\omega: \omega=\vec{F}\,dz,\; \text{for some}\;\, \vec{F}\in C^{\infty}(D^2_{\ast},\C^n)},\\
		\Omega^{(0,1)}(D^2_{\ast},\C^n)&=\Omega^1(D^2_{\ast},\C^n)\cap\ens{\omega: \omega=\vec{F}\,d\z,\; \text{for some}\;\, \vec{F}\in C^{\infty}(D^2_{\ast},\C^n)}.
	\end{align*} 
	As
	\begin{align*}
		\vec{\alpha}=\partial\H+|\H|^2\partial\phi+2g^{-1}\otimes \s{\H}{\h_0}\otimes\bar{\partial}\phi+\vec{\gamma}_0\,\partial\log|z|\in \Omega^{(1,0)}(D^2\setminus\ens{0}),
	\end{align*}
	and
	\begin{align*}
		\bar{\vec{\alpha}}\in \Omega^{(0,1)}(D^2\setminus\ens{0}),
	\end{align*}
	thanks to the decomposition
	\begin{align*}
		d\vec{L}=\partial\vec{L}+\bar{\partial}\vec{L}\in \Omega^{(1,0)}(D^2_{\ast},\C^n)\oplus \Omega^{(0,1)}(D^2_{\ast},\C^n)
	\end{align*}
	we must have by the direct sum decomposition of $\Omega^1(D^2_{\ast},\C^n)$ in \eqref{directsum}
	\begin{align*}
		2i\partial\vec{L}=\partial\H+|\H|^2\partial\phi+2g^{-1}\otimes \s{\H}{\h_0}\otimes\bar{\partial}\phi+\vec{\gamma}_0\,\partial\log|z|
	\end{align*}
	and rearranging this expression, we obtain
	\begin{align}\label{globaldisk}
		\partial\left(\H-2i\vec{L}+\vec{\gamma}_0\log|z|\right)=-|\H|^2\partial\phi-2g^{-1}\otimes \s{\H}{\h_0}\otimes\bar{\partial}\phi.
	\end{align}
	
	We now describe the strategy of the proof. In the following proof, we will first see that as variational Willmore surfaces have second residue $r(0)\leq \max\ens{\theta_0-2,0}$ that the quartic form admits poles of order at most $1$. Then, by using the meromorphy of $\mathscr{Q}_{\phi}$ and the extensive computer-assisted computations of \cite{sagepaper}, we will derive some special cancellations which will make the poles of order $1$ vanish.

	\textbf{Part 2. Cancellation and conservation laws.}

	\textbf{Step 1. First order expansion when $\theta_0\geq 3$}.
	
	As $r(0)\leq \theta_0-2$, there exists $\vec{C}_1\in \C^n$ such that (by \cite{beriviere})
	\begin{align}
		\H=\Re\left(\frac{\vec{C}_1}{z^{\theta_0-2}}\right)+O(|z|^{3-\theta_0-\epsilon}).
	\end{align}

	By Theorem \ref{dev1storder}, there exists $\vec{A}_0\in \C^n$, which we normalise to verify
	\begin{align*}
		|\vec{A}_0|^2=\frac{1}{2},
	\end{align*}
	such that
	\begin{align}\label{1stdevelopment}
	\left\{
	\begin{alignedat}{1}
			\partial_z\phi&=\vec{A}_0\,z^{\theta_0-1}+O(|z|^{\theta_0}),\\
	g&=|z|^{2(\theta_0-1)}\left(1+O(|z|)\right)|dz|^2,\\
	\H&=\Re\left(\frac{\vec{C}_1}{z^{\theta_0-2}}\right)+O(|z|^{2-\theta_0-\epsilon})\\
	\h_0&=O(|z|^{\theta_0-1}).
	\end{alignedat}\right.
	\end{align}
	The last estimate on $\h_0$ comes from the fact that $e^{-\lambda}\h_0\in L^{\infty}(D^2)$ by \cite{beriviere}.
	Therefore, one has by \eqref{1stdevelopment}
	\begin{align}\label{0stequation}
		|\H|^2\p{z}\phi+2\,g^{-1}\otimes\s{\H}{\h_0}\otimes \bar{\partial}\phi &=
		O(|z|^{2-\theta_0-\epsilon}).
	\end{align}
	As a result, we obtain by \eqref{1stdevelopment} and \eqref{globaldisk}
	\begin{align}\label{firstdev}
		&\partial\left(\H-2i\vec{L}+\vec{\gamma}_0\log|z|\right)=-|\H|^2\partial\phi-2\,g^{-1}\otimes\s{\H}{\h_0}\otimes\bar{\partial}\phi\nonumber=O(|z|^{2-\theta_0-\epsilon}).
	\end{align}
    Here we see that we must suppose $\theta_0\geq 3$ to carry on the general computation. Then by Proposition \ref{integrating} there exists
	 $\vec{Q}\in C^{\infty}(D^2\setminus\ens{0},\C^n)\cap L^2(D^2,|z|^{\theta_0-1}|dz|^2)$ such that
	\begin{align*}
	\partial \vec{Q}=-|\H|^2\partial\phi-2\,g^{-1}\otimes \s{\H}{\h_0}\otimes \bar{\partial}\phi
	\end{align*}
	and
	\begin{align*}
		\vec{Q}&=O(|z|^{3-\theta_0-\epsilon}).
	\end{align*}
	Therefore, we obtain
	\begin{align}
		\partial\left(\vec{H}-2i\partial\vec{L}+\vec{\gamma}_0\log|z|-\vec{Q}\right)=0,\quad \text{on}\;\, D^2\setminus\ens{0}
	\end{align}
	and there exists $\vec{C}_1\in \C^n$ such that (as $r(0)\leq \theta_0-2$)
	\begin{align*}
		&\H-2i\vec{L}+\vec{\gamma}_0\log|z|=\frac{\bar{\vec{C}_1}}{\z^{\theta_0-2}}+\vec{Q}+O(|z|^{3-\theta_0})=\frac{\bar{\vec{D}_1}}{\z^{\theta_0-2}}
		+O(|z|^{3-\theta_0-\epsilon})
	\end{align*}
	As $\vec{\H}$ and $\vec{L}$ are \textit{real}, one has
	\begin{align}\label{0orderH}
		\H+\vec{\gamma}_0\log|z|&=\Re\left(\frac{\vec{D}_1}{z^{\theta_0-2}}\right)+O(|z|^{3-\theta_0-\epsilon}).
	\end{align}
	and the equation \eqref{0orderH} reduces to
	\begin{align}\label{2ndH}
		&\H+\vec{\gamma}_0\log|z|=\Re\bigg(\frac{\vec{C}_1}{z^{\theta_0-2}}\bigg)+O(|z|^{3-\theta_0-\epsilon}).
	\end{align}
	We shall keep in mind that the only important constants are $\vec{A}_j,\vec{B}_j,\vec{C}_j$ for $0\leq j\leq 2$, and that the other are simply artefacts of the integrations, but do not play any role. This will become transparent when we will obtain the expansion of $\h_0$ with respect to $\{\vec{A}_j,\vec{C}_j\}_{0\leq j\leq 2}$ (we will actually show that $\vec{B}_0$ vanish).
	
	We recall that by definition of the mean curvature,
	\begin{align*}
		\Delta\phi=4\p{z}\p{\z}\phi=2e^{2\lambda}\H.
	\end{align*}
	and an easy computation shows that for some $\alpha_0\in\C$, we have
	\begin{align}\label{liouville}
		e^{2\lambda}=|z|^{2\theta_0-2}\left(1+2\,\Re\left(\alpha_0 z\right)+O(|z|^{2-\epsilon})\right).
	\end{align}
	Let us check this fact.
	The Liouville equation shows that 
	\begin{align}\label{liouville-infty}
		-\Delta\lambda=e^{2\lambda}K_g+2\pi(\theta_0-1)\delta_0,
	\end{align}
	where $\delta_0$ is the Dirac mass at $0\in D^2$.
	Therefore, the function $u:D^2\rightarrow\R$ defined by 
	\begin{align*}
		e^{2u}=|z|^{2-2\theta_0}e^{2\lambda}
	\end{align*}
	satisfies the following Liouville equation
	\begin{align}\label{liouville-infty+1}
		-\Delta u=e^{2\lambda}K_g,
	\end{align}
	and as $\phi\in W^{2,p}(D^2)$ for all $p<\infty$, we have
	\begin{align*}
	|e^{2\lambda}K_g|\leq \frac{1}{2}|\vec{\I}|_g^2\in \bigcap_{p<\infty}L^p(D^2)
	\end{align*}
	so by a classical Calder\'{o}n-Zygmund estimate, we have
	\begin{align*}
		u\in \bigcap_{p<\infty}W^{2,p}(D^2)\subset \bigcap_{\alpha<1}C^{1,\alpha}(D^2).
	\end{align*}
	In particular, we have
	\begin{align*}
		e^{2u}\in \bigcap_{\alpha<1}C^{1,\alpha}(D^2).
	\end{align*}
    and the expansion \eqref{liouville} simply corresponds to the first order Taylor expansion of $e^{2u}$, as we know that $e^{2u(0)}=1$ by the normalisation we made. Furthermore, as $\theta_0\geq 3$, the logarithm term is an error in \eqref{2ndH}, so we have
	\begin{align*}
		\p{\z}\left(\p{z}\phi\right)&=\frac{1}{4}\left(\vec{C}_1z\z^{\theta_0-1}+\bar{\vec{C}_1}z^{\theta_0-1}\z\right)+O(|z|^{\theta_0+1-\epsilon}).
	\end{align*}
	Therefore, for some $\vec{A}_1,\vec{A}_2\in \C^n$ (as $\vec{A}_0$ has already been defined in \eqref{1stdevelopment}), one obtains
	\begin{align}\label{devordre2}
		\p{z}\phi&=\vec{A}_0z^{\theta_0-1}+\vec{A}_1z^{\theta_0}+\vec{A}_2z^{\theta_0+1}+\frac{1}{4\theta_0}\vec{C}_1z\z^{\theta_0}+\frac{1}{8}\bar{\vec{C}_1}\z^{2}z^{\theta_0-1}+O(|z|^{\theta_0+2-\epsilon}).
	\end{align}
	Here is the first crucial step of the proof. As $\phi$ is conformal, we have
	\begin{align*}
		\s{\p{z}\phi}{\p{z}\phi}=0,
	\end{align*}
	and in the product, we see that we must neglect all term of order more than $|z|^{2\theta_0+1}$. This yields
	\begin{align}\label{cancel0}
		\s{\p{z}\phi}{\p{z}\phi}&=\s{\vec{A}_0}{\vec{A}_0}z^{2\theta_0-2}+2\s{\vec{A}_0}{\vec{A}_1}z^{2\theta_0-1}+\left(\s{\vec{A}_1}{\vec{A}_1}+2\s{\vec{A}_0}{\vec{A}_2}\right)z^{2\theta_0}\nonumber\\
		&+\frac{1}{2\theta_0}\s{\vec{A}_0}{\vec{C}_1}|z|^{2\theta_0}+\frac{1}{4}\s{\vec{A}_0}{\bar{\vec{C}_1}}z^{2\theta_0-2}\z^2 
		+O(|z|^{2\theta_0+1-\epsilon}).
	\end{align}
	Therefore, we have
	\begin{align}\label{cancellations}
		\left\{\begin{alignedat}{1}
		&\s{\vec{A}_0}{\vec{A}_0}=0,\qquad \s{\vec{A}_0}{\vec{A}_1}=0,\qquad \s{\vec{A}_1}{\vec{A}_1}+2\s{\vec{A}_0}{\vec{A}_2}=0\\
		&\s{\vec{A}_0}{\vec{C}_1}=\s{\bar{\vec{A}_0}}{\vec{C}_1}=0.
		\end{alignedat}\right.
	\end{align}
	Summing up, we have the following expansions
	\begin{align}\label{3rddev}
	\left\{\begin{alignedat}{1}
			\p{z}\phi&=\vec{A}_0z^{\theta_0-1}+\vec{A}_1z^{\theta_0}+\vec{A}_2z^{\theta_0+1}+\frac{1}{4\theta_0}\vec{C}_1z\,\z^{\theta_0}+\frac{1}{8}\bar{\vec{C}_1}z^{\theta_0-1}\z^2+O(|z|^{\theta_0+2-\epsilon})\\
	        \H&=\Re\left(\frac{\vec{C}_1}{z^{\theta_0-2}}\right)+O(|z|^{3-\theta_0-\epsilon}).
	\end{alignedat}\right.
	\end{align}
	We check that these expansions are consistent, as
	\begin{align*}
		\H&=\frac{1}{2}\Delta_g\phi=2e^{-2\lambda}\p{z\z}^2\phi=2z^{1-\theta_0}\z^{1-\theta_0}\left(\frac{1}{4}\vec{C}_1z\z^{\theta_0-1}+\frac{1}{4}\bar{\vec{C}_1}z^{\theta_0-1}\z\right)+O(|z|^{3-\theta_0-\epsilon})\\
		&=\frac{1}{2}\left(\vec{C}_1z^{2-\theta_0}+\bar{\vec{C}_1}\z^{2-\theta_0}\right)+O(|z|^{3-\theta_0-\epsilon})=\Re\left(\frac{\vec{C}_1}{z^{\theta_0-2}}\right)+O(|z|^{3-\theta_0-\epsilon}).
	\end{align*}	
	\textbf{In particular, by Proposition \ref{parenthesis}, we have $\n\in W^{2,\infty}(D^2)$}. We will see how this improvement of regularity shows that the poles of the quartic form are of order at most $1$.
	
	\textbf{Step 3. Removability of poles of order $2$ of the quartic form $\mathscr{Q}_{\phi}$}.
	
	We have $e^{2\lambda}=2\s{\p{z}\phi}{\p{\z}\phi}$ so by \eqref{cancellations}
	\begin{align}\label{generalmetric}
		e^{2\lambda}&=|z|^{2\theta_0-2}+4\,\Re\left(\s{\bar{\vec{A}_0}}{\vec{A}_1}z^{\theta_0}\z^{\theta_0-1}\right)+4\,\Re\left(\s{\bar{\vec{A}_0}}{\vec{A}_2}z^{\theta_0+1}\z^{\theta_0-1}\right)+2|\vec{A}_1|^2|z|^{2\theta_0}+P(|z|^{2\theta_0})\nonumber\\
		&=|z|^{2\theta_0-2}\left(1+2\,\Re\left(\alpha_0 z+\alpha_1 z^2\right)+2|\vec{A}_1|^2|z|^2+O(|z|^{3-\epsilon})\right),
	\end{align}
	where we defined
	\begin{align}\label{defalpha01}
		\left\{\begin{alignedat}{1}
		\alpha_0&=\s{\bar{\vec{A}_0}}{\vec{A}_1}\\
		\alpha_1&=\s{\bar{\vec{A}_0}}{\vec{A}_2}
		\end{alignedat}\right.
	\end{align}
    Therefore, we obtain
	\begin{align*}
		\p{z}(e^{2\lambda})&=(\theta_0-1)z^{\theta_0-2}\z^{\theta_0-1}+\theta_0\,\alpha_0|z|^{2\theta_0-2}+(\theta_0-1)\bar{\alpha_0}z^{\theta_0-2}\z^{\theta_0}+(\theta_0+1)\alpha_1 z^{\theta_0}\z^{\theta_0-1}\\
		&+(\theta_0-1)\bar{\alpha_1}z^{\theta_0-2}\z^{\theta_0+1}+2\theta_0|\vec{A}_1|^2z^{\theta_0-1}\z^{\theta_0}+O(|z|^{2\theta_0-\epsilon})\\
		&=|z|^{2\theta_0-2}\left(\frac{(\theta_0-1)}{z}+\theta_0\alpha_0+(\theta_0-1)\bar{\alpha_0}\frac{\z}{z}+(\theta_0+1)\alpha_1z+(\theta_0-1)\bar{\alpha_1}\frac{\z^2}{z}+2\theta_0|\vec{A}_1|^2\z+O(|z|^{2-\epsilon})\right)
	\end{align*}
	and
	\begin{align}
	e^{-2\lambda}
	&=|z|^{2-2\theta_0}\left(1-\alpha_0z-\bar{\alpha_0}\z+\left(\alpha_0^2-\alpha_1\right)z^2+\left(\bar{\alpha_0}^2-\bar{\alpha_1}\right)\z^2-2\left(|\vec{A}_1|^2-|\alpha_0|^2\right)|z|^2+O(|z|^{3-\epsilon})\right)
	\end{align}
	as
	\begin{align*}
		4\,\Re\left(\alpha_0 z\right)^2=(\alpha_0z+\bar{\alpha_0}\z)^2=\alpha_0^2z^2+\bar{\alpha_0}^2\z^2+2|\alpha_0|^2|z|^2=2\,\Re\left(\alpha_1^2z^2\right)+2|\alpha_0|^2|z|^2
	\end{align*}
	Therefore, we obtain 
	\begin{align*}
		&2(\p{z}\lambda)=e^{-2\lambda}\p{z}(e^{2\lambda})=\left(1-\alpha_0z-\bar{\alpha_0}\z+\left(\alpha_0^2-\alpha_1\right)z^2+\left(\bar{\alpha_0}^2-\bar{\alpha_1}\right)\z^2-2\left(|\vec{A}_1|^2-|\alpha_0|^2\right)|z|^2+O(|z|^{3-\epsilon})\right)\times \\
		&\left(\frac{(\theta_0-1)}{z}+\theta_0\alpha_0+(\theta_0-1)\bar{\alpha_0}\frac{\z}{z}+(\theta_0+1)\alpha_1z+(\theta_0-1)\bar{\alpha_1}\z+2\theta_0|\vec{A}_1|^2\z+O(|z|^{2-\epsilon})\right)\\
		&=\frac{(\theta_0-1)}{z}+\alpha_0+\left(2\alpha_1-\alpha_0^2\right)z+\left(2|\vec{A}_1|^2-|\alpha_0|^2\right)\z+O(|z|^{2-\epsilon}).
	\end{align*}
	so
	\begin{align}\label{dzlambda}
		2(\p{z}\lambda)	&=\frac{(\theta_0-1)}{z}+\alpha_0+\left(2\alpha_1-\alpha_0^2\right)z+\left(2|\vec{A}_1|^2-|\alpha_0|^2\right)\z+O(|z|^{2-\epsilon}).
	\end{align}
	We finally come to the expansion of the Weingarten tensor
	First, we have
	\begin{align}\label{gendevphi0}
		\p{z}\phi=\vec{A}_0z^{\theta_0-1}+\vec{A}_1z^{\theta_0}+\vec{A}_2z^{\theta_0+1}+\frac{1}{4\theta_0}\vec{C}_1z\z^{\theta_0}+\frac{1}{8}\bar{\vec{C}_1}z^{\theta_0-1}\z^2+O(|z|^{\theta_0+2-\epsilon})
	\end{align}
	so
    \begin{align*}
    \p{z}^2\phi&=(\theta_0-1)\vec{A}_0z^{\theta_0-2}+\theta_0\vec{A}_1z^{\theta_0-1}+(\theta_0+1)\vec{A}_2z^{\theta_0}+\frac{1}{4\theta_0}\vec{C}_1\z^{\theta_0}+\frac{(\theta_0-1)}{8}\bar{\vec{C}_1}z^{\theta_0-2}\z^2+O(|z|^{\theta_0+1-\epsilon}).
     \end{align*}
     Then we have
     \begin{align*}
     	&2\left(\p{z}\lambda\right)\p{z}\phi=\left(\frac{(\theta_0-1)}{z}+\alpha_0+\left(2\alpha_1-\alpha_0^2\right)z+\left(2|\vec{A}_1|^2-|\alpha_0|^2\right)\z+O(|z|^{2-\epsilon})\right)\times \\
     	&\left(\vec{A}_0z^{\theta_0-1}+\vec{A}_1z^{\theta_0}+\vec{A}_2z^{\theta_0+1}+\frac{1}{4\theta_0}\vec{C}_1z\z^{\theta_0}+\frac{1}{8}\bar{\vec{C}_1}z^{\theta_0-1}\z^2+O(|z|^{\theta_0+2-\epsilon})\right)\\
     	&=(\theta_0-1)\vec{A}_0z^{\theta_0-2}+(\theta_0-1)\vec{A}_1z^{\theta_0-1}+(\theta_0-1)\vec{A}_2z^{\theta_0}+\frac{(\theta_0-1)}{4\theta_0}\vec{C}_1\z^{\theta_0}+\frac{(\theta_0-1)}{8}\bar{\vec{C}_1}z^{\theta_0-2}\z^2\\
     	&+\alpha_0\vec{A}_0z^{\theta_0-1}+\alpha_0\vec{A}_1z^{\theta_0}+(2\alpha_1-\alpha_0^2)\vec{A}_0z^{\theta_0}+\left(2|\vec{A}_1|^2-|\alpha_0|^2\right)\vec{A}_0z^{\theta_0-1}\z+O(|z|^{\theta_0+1-\epsilon}).
     \end{align*}
     Therefore, we deduce that
     \begin{align*}
     	&\p{z}^2\phi-2(\p{z}\lambda)\p{z}\phi=\colorcancel{(\theta_0-1)\vec{A}_0z^{\theta_0-2}}{red}+\theta_0\vec{A}_1z^{\theta_0-1}+(\theta_0+1)\vec{A}_2z^{\theta_0}+\frac{1}{4\theta_0}\vec{C}_1\z^{\theta_0}+\colorcancel{\frac{(\theta_0-1)}{8}\bar{\vec{C}_1}z^{\theta_0-2}\z^2}{blue}\\
     	&-\bigg\{\colorcancel{(\theta_0-1)\vec{A}_0z^{\theta_0-2}}{red}+(\theta_0-1)\vec{A}_1z^{\theta_0-1}+(\theta_0-1)\vec{A}_2z^{\theta_0}+\frac{(\theta_0-1)}{4\theta_0}\vec{C}_1\z^{\theta_0}+\colorcancel{\frac{(\theta_0-1)}{8}\bar{\vec{C}_1}z^{\theta_0-2}\z^2}{blue}\\
     	&+\alpha_0\vec{A}_0z^{\theta_0-1}+\alpha_0\vec{A}_1z^{\theta_0}+(2\alpha_1-\alpha_0^2)\vec{A}_0z^{\theta_0}+\left(2|\vec{A}_1|^2-|\alpha_0|^2\right)\vec{A}_0z^{\theta_0-1}\z\bigg\}+O(|z|^{\theta_0+1-\epsilon})\\
     	&=\left(\vec{A}_1-\Big(2|\vec{A}_1|^2-|\alpha_0|^2\Big)\vec{A}_0\z\right)z^{\theta_0-1}+\left(2\vec{A}_2-\alpha_0\vec{A}_1-(2\alpha_1-\alpha_0^2)\vec{A}_0\right)z^{\theta_0}-\frac{(\theta_0-2)}{4\theta_0}\vec{C}_1\z^{\theta_0}+O(|z|^{\theta_0+1-\epsilon})
     \end{align*}
	Finally, we have
	\begin{align}\label{2ndh0}
		&\h_0(z)=2e^{2\lambda}\p{z}(e^{-2\lambda}\p{z}\phi)dz^2=\left(\p{z}^2\phi-2(\p{z}\lambda)\p{z}\phi\right)dz^2\nonumber\\
		&=2\left(\vec{A}_1-\alpha_0\vec{A}_0-\Big(2|\vec{A}_1|^2-|\alpha_0|^2\Big)\vec{A}_0\z\right)z^{\theta_0-1}+2\left(2\vec{A}_2-\alpha_0\vec{A}_1-(2\alpha_1-\alpha_0^2)\vec{A}_0\right)z^{\theta_0}-\frac{(\theta_0-2)}{2\theta_0}\vec{C}_1\z^{\theta_0}\nonumber\\
		&+O(|z|^{\theta_0+1-\epsilon})
	\end{align}
	We recall that the only (possibly) singular part of Bryant's quartic form $\mathscr{Q}_{\phi}$ when $\theta_0\geq 2$, is
	\begin{align*}
		Q(\h_0)=g^{-1}\otimes\left(\partial\bar{\partial}\h_0\totimes\h_0-\partial\h_0\totimes\bar{\partial}\h_0\right).
	\end{align*}
	Using
	\begin{align*}
		\s{\vec{A}_0}{\vec{A}_0}=\s{\vec{A}_0}{\vec{A}_1}=\s{\vec{A}_0}{\vec{C}_1}=0,
	\end{align*}
	and the fact (already used in several places) that for any quadratic differential 
	\begin{align*}
	\vec{\alpha}\in \Gamma(K_{D^2_{\ast}}^2\,,\C^n)
	\end{align*}
	such $\vec{\alpha}=\vec{\Lambda}f_1(z)f_2(\z)dz^2$, where $\vec{\Lambda}\in\C^n$ is fixed and $f_1,f_2:D^2_{\ast}\rightarrow\C$ are holomorphic, we have
	\begin{align*}
		Q(\vec{\alpha})=\s{\vec{\Lambda}}{\vec{\Lambda}}\,g^{-1}\otimes\left(f_1'\bar{f_2'}\cdot f_1f_2-f_1'\bar{f_2}\cdot f_1\bar{f_2'}\right)=0,
	\end{align*}
	we obtain
	\begin{align}\label{singularityQ}
		Q(\h_0)&={(\theta_0-1)(\theta_0-2)}\s{\vec{A}_1}{\vec{C}_1}\frac{1}{z}+O(|z|^{-\epsilon}).
	\end{align}
	so the poles of $\mathscr{Q}_{\phi}$ are of order at most $1$, and this extends Bryant's theorem for \emph{variational} branched Willmore spheres with less than $7$ branch points by Riemann-Roch theorem.
	
	\begin{rem}
		Notice that the quartic form would also be holomorphic with if $\vec{A}_1=\vec{A}_2=0$ in the expansion \eqref{gendevphi0} of $\phi$. In this case, no assumption is needed on the first or second residue (take any inversion of a minimal surface as in the proof of Theorem \ref{devinv}). However, there are no analytic way to have access to this harmonic part coming from integration of .
	\end{rem}
	
	The end of the proof will be devoted to the derivation of the cancellation of $\s{\vec{A}_1}{\vec{C}_1}=0$. We will see that this fact is a direct consequence of the conservation laws.

	\begin{rem}\label{importantremarkerros}
		One can wonder why we only obtain power functions, as $\phi$ is \emph{not} smooth through the branch point. However, the bootstrap procedure we have implemented in the first steps of shows we will have only power functions in the expansion of $\H$ until we get to 
		\begin{align*}
			\partial\left(\H-2i\vec{L}+\vec{\gamma}_0\log|z|\right)=\cdots+\frac{\vec{E}}{z}+O(|z|^{-\epsilon}).
		\end{align*}
		for some $\vec{E}\in\mathbb{C}^n$,
		which will make a logarithm term appear, and gives
		\begin{align*}
			\H=\cdots+\left(2\,\Re(\vec{E})-\vec{\gamma}_0\right)\log|z|+O(|z|^{1-\epsilon}).
		\end{align*}
		In the next expansions, as we only make products, integration of derivation of tensors, we see that the only possible components in the Taylor expansion of $\phi$ are
		\begin{align*}
			z^a\z^b\log^p|z|\quad a,b\in\Z,\quad p\in \N.
		\end{align*}
		In particular, no fractional powers of the type $|z|^{\alpha}$ for some $\alpha\in (0,\infty)\setminus\N$ may occur in the Taylor expansion of $\phi$, although the branched immersion $\phi$ is in general \emph{not} smooth. As $\phi$ is continuous on $D^2$, terms of the type
		\begin{align*}
			\Re(z^{\alpha}\z^{\beta})
		\end{align*}
		were excluded from the beginning, if $\alpha,\beta\in \R$, $\alpha+\beta\in (0,\infty)$ and $\alpha\notin\Z$ or $\beta\notin\Z$, as the angle function is not a well-defined continuous function on $D^2$.
		
		In particular, all errors of the type
		\begin{align*}
			O(|z|^{a-\epsilon})
		\end{align*}
		for some $a\in\Z$ could be replaced by 
		\begin{align*}
			O(|z|^{a}\log^p|z|)
		\end{align*}
		for some $p\in \N$ sufficiently large enough, and more importantly, errors can be differentiated (and integrated by Proposition \ref{integrating}) as polynomials, in the following sense : for all $\vec{F}\in \ens{\phi,\p{z}\phi,\p{\z}\phi,\H,\h_0}$, if
		\begin{align*}
			\vec{F}=\vec{F}_0+O(|z|^{a-\epsilon})
		\end{align*}
		for some $a\in \Z$ and some function $\vec{F}_0$, rational in $z,\z$, and polynomial in $\log|z|$, we have for all $\alpha,\beta\in \N$
		\begin{align*}
			\p{z}^{\alpha}\p{\z}^{\beta}\vec{F}=\p{z}^\alpha\p{\z}^\beta\vec{F}_0+O(|z|^{a-\alpha-\beta-\epsilon}).
		\end{align*}
	\end{rem}

 \textbf{Step 4. Conservation and cancellation laws for     $\theta_0\geq 4$ : invariance by inversions.}
 
 We stress out the following remark.
 
 \textbf{\emph{ From this point, we will need to use the computer-assisted proof from \cite{sagepaper}.}}
 
 Let $\vec{F}\in C^{\infty}(D^2\setminus\ens{0},\C^n)$ such that
 \begin{align*}
 	\vec{\beta}= \mathscr{I}_{{\phi}}\left(\partial\H+|\H|^2\partial\phi+2\,g^{-1}\otimes\s{\H}{\h_0}\totimes\bar{\partial}\phi\right)-g^{-1}\otimes\left(\bar{\partial}|\phi|^2\otimes\h_0-2\s{\phi}{\h_0}\otimes\bar{\partial}\phi\right)=\vec{F}(z)dz
 \end{align*}
 where for all vector $\vec{X}\in \C^n$, we have
 \begin{align*}
 	\mathscr{I}_{\phi}(\vec{X})=|\phi|^2\vec{X}-2\s{\phi}{\vec{X}}\phi.
 \end{align*}
 The conservation law associated to the invariance by inversions of the Willmore energy shows that $\Im(\vec{\beta})$ is closed. Furthermore, as $\vec{\beta}$ is a $\C^n$-valued $1$-form of type $(1,0)$, there exists a smooth function $\vec{F}\in C^{\infty}(D^2\setminus\ens{0},\R^n)$ such that
 \begin{align}\label{defF}
 	\vec{\beta}=\vec{F}(z)dz.
 \end{align}
 In particular, we have
 \begin{align}\label{d2=0}
 	d\,\vec{\beta}=\partial\vec{\beta}+\bar{\partial}\vec{\beta}=\p{z}\vec{F}(z)dz\wedge dz+\p{\z}\vec{F}(z)d\z\wedge dz=\p{\z}\vec{F}(z)\,d\z \wedge dz.
 \end{align}
 Therefore, we have by \eqref{d2=0}
 \begin{align}\label{key}
 	0=d\,\Im(\vec{\beta})=\Im(d\vec{\beta})=\Im(\p{\z}\vec{F}(z)\,d\z\wedge dz)=-2\,\Re\left(\p{\z}\vec{F}(z)\right)\,dx_1\wedge dx_2
 \end{align}
 as
 \begin{align*}
 	d\z\wedge dz=\left(dx_1-idx_2\right)\wedge (dx_1+idx_2)=2i\, dx_1\wedge dx_2.
 \end{align*}
 Finally, we deduce from \eqref{defF} and \eqref{key} that the $1$-form $\Im(\vec{\beta})$ is closed if and only if
 \begin{align}\label{eqinv}
 	\Re\left(\p{\z}\vec{F}(z)\right)=0,
 \end{align} 
 Thanks to \cite{sagepaper},
 the coefficient $\vec{\Omega}\in \C^n$ in
 \begin{align}
 	\frac{\z^{\theta_0+2}}{z}
 \end{align}
 in the Taylor expansion of 
 \[
 \Re(\p{\z}\vec{F}(z))=0,
 \] is given by 
 \begin{align}\label{yang}
 \vec{\Omega}=\frac{4 \, {\left(\theta_{0}^{2} \overline{\alpha_{0}}^{3} + 2 \, \theta_{0} \overline{\alpha_{0}}^{3} - 3 \, \overline{\alpha_{0}}^{3}\right)} \overline{\vec{A}_{0}}\cdot \bar{\vec{A}_0}}{\theta_{0}}\,\vec{A}_0-\frac{4 \, {\left(\theta_{0}^{2} \overline{\alpha_{0}}^{3} + 2 \, \theta_{0} \overline{\alpha_{0}}^{3} - 3 \, \overline{\alpha_{0}}^{3}\right)} \vec{A}_{0}\cdot  \overline{\vec{A}_{0}}}{\theta_{0}}\bar{\vec{A}_0}
\end{align}
 As $\s{\vec{A}_0}{\vec{A}_0}=0$ while $|\vec{A}_0|^2=\dfrac{1}{2}$, we obtain by \eqref{yang}
 \begin{align}\label{eqinv2}
 	\vec{\Omega}&=-\frac{4 \, {\left(\theta_{0}^{2} \overline{\alpha_{0}}^{3} + 2 \, \theta_{0} \overline{\alpha_{0}}^{3} - 3 \, \overline{\alpha_{0}}^{3}\right)} \vec{A}_{0}\cdot  \overline{\vec{A}_{0}}}{\theta_{0}}\bar{\vec{A}_0}=-\frac{2}{\theta_0}(\theta_0^2+2\theta_0-3)\bar{\alpha_0}^3\bar{\vec{A}_0}\\
 	&=-\frac{2}{\theta_0}(\theta_0+3)(\theta_0-1)\bar{\alpha_0}^3\bar{\vec{A}_0}=0.
 \end{align}
 Now, as $\vec{A}_0\neq 0$ by the very definition of a branch point of order $\theta_0\geq 1$, and $\theta_0\geq 2$ (as $\theta_0\geq 4$ in this step), we thereby deduce that
 \begin{align*}
 	\alpha_0=2\s{\bar{\vec{A}_0}}{\vec{A}_1}=0
 \end{align*}
 so we recover the previous result.

    \textbf{Step 5. Cancellation laws for $\theta_0\geq 5$.}
    
    We find in \cite{sagepaper} that the \emph{fourth} order  expansion of the quartic form is for $\theta_0\geq 5$
    \begin{align}
    &	\mathscr{Q}_{\phi}=(\theta_0-1)(\theta_0-2)\s{\vec{A}_1}{\vec{C}_1}\frac{dz^4}{z}+\bigg\{(\theta_0-2)(\theta_0-3)\s{\vec{A}_1}{\vec{C}_2}+2\theta_0(\theta_0-2)\s{\vec{A}_2}{\vec{C}_1}\bigg\}dz^4\nonumber\\
    &+\bigg\{(\theta_0-3)(\theta_0-4)\s{\vec{A}_1}{\vec{C}_3}+3(\theta_0+1)(\theta_0-2)\s{\vec{A}_3}{\vec{C}_1}+6\s{\bar{\vec{A}_0}}{\vec{A}_2}\s{\vec{A}_1}{\vec{C}_1}+2(\theta_0-1)(\theta_0-3)\s{\vec{A}_2}{\vec{C}_2}\bigg\}z\,dz^4\nonumber\\
    &-6(\theta_0-2)\left(|\vec{A}_1|^2\s{\vec{A}_1}{\vec{C}_1}-\s{\bar{\vec{A}_1}}{\vec{C}_1}\s{\vec{A}_1}{\vec{A}_1}\right)\z\, dz^4\nonumber\\
    &-\frac{3(\theta_0-2)}{2\theta_0}\left(|\vec{C}_1|^2\s{\vec{A}_1}{\vec{A}_1}-\s{\vec{A}_1}{\vec{C}_1}\s{\vec{A}_1}{\bar{\vec{C}_1}}\right)z^{\theta_0}\z^{2-\theta_0}dz^4+O(|z|^3)
    \end{align}
    \normalsize
    If we suppose that $\mathscr{Q}_{\phi}$ is meromorphic, then we obtain
    \begin{align}\label{system}
    \left\{\begin{alignedat}{1}
    |\vec{A}_1|^2\s{\vec{A}_1}{\vec{C}_1}&=\s{\bar{\vec{A}_1}}{\vec{C}_1}\s{\vec{A}_1}{\vec{A}_1}\\
    |\vec{C}_1|^2\s{\vec{A}_1}{\vec{A}_1}&=\s{\vec{A}_1}{\bar{\vec{C}_1}}\s{\vec{A}_1}{\vec{C}_1},
    \end{alignedat}\right.
    \end{align}
    Remarking that is a linear system in $(\s{\vec{A}_1}{\vec{C}_1},\s{\vec{A}_1}{\vec{A}_1})$, we can recast \eqref{system} as
    \begin{align}\label{salvationmatrix}
    \begin{dmatrix}
    |\vec{A}_1|^2 & -\s{\bar{\vec{A}_1}}{\vec{C}_1}\\
    -\s{\vec{A}_1}{\bar{\vec{C}_1}} & |\vec{C}_1|^2
    \end{dmatrix}
    \begin{dmatrix}
    \s{\vec{A}_1}{\vec{C}_1}\\
    \s{\vec{A}_1}{\vec{A}_1}
    \end{dmatrix}=0.
    \end{align}
    Thanks to Cauchy-Schwarz inequality, we obtain
    \begin{align}\label{determinant}
    \det 	\begin{dmatrix}
    |\vec{A}_1|^2 & -\s{\bar{\vec{A}_1}}{\vec{C}_1}\\
    -\s{\vec{A}_1}{\bar{\vec{C}_1}} & |\vec{C}_1|^2
    \end{dmatrix}=|\vec{A}_1|^2|\vec{C}_1|^2-|\s{\vec{A}_1}{\bar{\vec{C}_1}}|^2\geq 0.
    \end{align}
    Therefore, if the determinant is positive, we obtain
    \begin{align*}
    \s{\vec{A}_1}{\vec{C}_1}=0,
    \end{align*}
    and the holomorphy of the quartic form, and if the determinant vanishes,we obtain
    \begin{align*}
    \vec{A}_1\;\,\text{and}\;\, \vec{C}_1\;\,\text{are proportional}.
    \end{align*}
    But in general this is not enough to conclude.

    \textbf{Step 6. Conclusion for $\theta_0\geq 5$.}
    
    From now on, we suppose thanks to \eqref{system} that
    \begin{align}\label{step5}
    |\vec{A}_1|^2\s{\vec{A}_1}{\vec{C}_1}=\s{\bar{\vec{A}_1}}{\vec{C}_1}\s{\vec{A}_1}{\vec{A}_1}.
    \end{align}
    By adopting the notations of step $3$, we compute in \cite{sagepaper} that the coefficient in
    \begin{align*}
    	z^{\theta_0+3}
    \end{align*}
    in
    \begin{align*}
    	\Re\left(\p{\z}\vec{F}(z)\right)=0
    \end{align*}
    defined in \eqref{defF} is equal, for some $\lambda_1,\lambda_2,\lambda_3\in \C$ to
    \begin{align}\label{endend}
    	-\frac{2(\theta_0-4)}{\theta_0^2(\theta_0-3)}|\vec{A}_1|^2\s{\vec{A}_1}{\vec{C}_1}\vec{A}_0+\lambda_1\bar{\vec{A}_0}+\lambda_2\vec{A}_1+\lambda_3\bar{\vec{A}_1}=0
    \end{align}
    As
    \begin{align*}
    	\s{\vec{A}_0}{\vec{A}_0}=\s{\vec{A}_0}{\vec{A}_1}=\s{\vec{A}_0}{\bar{\vec{A}_1}}=0,
    \end{align*}
    the vector $\vec{A}_0\in \C^n\setminus\ens{0}$ (as $\phi$ has a branch point of multiplicity $\theta_0\geq 0$, the vector $\vec{A}_0$ is non-zero by definition) is linearly independent with $\bar{\vec{A}_0},\vec{A}_1$ and $\bar{\vec{A}_1}$, so \eqref{endend} implies that
    \begin{align}\label{endendend}
    	-\frac{2(\theta_0-4)}{\theta_0^2(\theta_0-3)}|\vec{A}_1|^2\s{\vec{A}_1}{\vec{C}_1}=0.
    \end{align}
    As $\theta_0\geq 5$, we deduce that
    \begin{align}\label{alphaomega}
    	|\vec{A}_1|^2\s{\vec{A}_1}{\vec{C}_1}=0.
    \end{align}
    Therefore, either $\vec{A}_1=0$, or $\s{\vec{A}_1}{\vec{C}_1}=0$. As both alternatives show that
    \begin{align}\label{fin}
    	\s{\vec{A}_1}{\vec{C}_1}=0,
    \end{align}
    we are done as the quartic form admits the following Taylor expansion
    \begin{align*}
    	\mathscr{Q}_{\phi}=(\theta_0-1)(\theta_0-2)\s{\vec{A}_1}{\vec{C}_1}\frac{dz^4}{z}+O(1).
    \end{align*} 
    This concludes the proof of the case $\theta_0\geq 5$.

    \textbf{Step 7. Case $\theta_0=4$.}
    
    In this case, we can show that the fifth order expansion of the quartic form is the following
    	\begin{align*}
    	\mathscr{Q}_{\phi}&=6\s{\vec{A}_1}{\vec{C}_1}\frac{dz^4}{z}-12\left(|\vec{A}_1|^2\s{\vec{A}_1}{\vec{C}_1}-\s{\bar{\vec{A}_1}}{\vec{C}_1}\s{\vec{A}_1}{\vec{A}_1}\right)\z\, dz^4\\
    	&-\frac{3}{4}\left(|\vec{C}_1|^2\s{\vec{A}_1}{\vec{A}_1}-\s{\vec{A}_1}{\bar{\vec{C}_1}}\s{\vec{A}_1}{\vec{C}_1}\right)z^{\theta_0}\,\z^{2-\theta_0}\,dz^4-\frac{3}{8}\s{\vec{A}_1}{\vec{C}_1}\bar{\s{\vec{C}_1}{\vec{C}_1}}\frac{\z^4}{z}\log|z|+O(|z|^4).
    	\end{align*}
    	Therefore, we obtain the additional relation
    	\begin{align}\label{add4}
    		\s{\vec{A}_1}{\vec{C}_1}\bar{\s{\vec{C}_1}{\vec{C}_1}}=0
    	\end{align}
    	and thanks to \eqref{determinant} and the following discussion, we have either
    	\begin{align*}
    		|\vec{A}_1|^2|\vec{C}_1|^2-|\s{\vec{A}_1}{\bar{\vec{C}_1}}|^2>0
    	\end{align*}
    	and
    	\begin{align*}
    		\s{\vec{A}_1}{\vec{C}_1}=\s{\vec{A}_1}{\vec{A}_1}=0
    	\end{align*}
        or
        \begin{align*}
        	|\vec{A}_1|^2|\vec{C}_1|^2-|\s{\vec{A}_1}{\bar{\vec{C}_1}}|^2=0.
        \end{align*}
        Then if $\vec{A}_1=0$ or $\vec{C}_1=0$, we are done as $\s{\vec{A}_1}{\vec{C}_1}=0$, and otherwise, there exists $\lambda\in \C\setminus\ens{0}$ such that
        \begin{align*}
        	\vec{C}_1=\lambda\,\vec{A}_1.
        \end{align*}
        But this implies by \eqref{add4} that
        \begin{align*}
        	0=\s{\vec{A}_1}{\vec{C}_1}\bar{\s{\vec{C}_1}{\vec{C}_1}}=\bar{\lambda}|\s{\vec{A}_1}{\vec{C}_1}|^2
        \end{align*}
        and as $\lambda\neq 0$, we obtain
        \begin{align*}
        	\s{\vec{A}_1}{\vec{C}_1}=0
        \end{align*}
        and this concluded the proof of the case $\theta_0=4$.
    
    
    \textbf{Step 8. Case $\theta_0=3$.} In this case, we will check directly the holomorphy of the quartic form for \emph{true} Willmore disks.
    
    Recall that the expansion \eqref{gendevphi0} is valid for all $\theta_0\geq 3$ and yields
    	\begin{align*}
    \p{z}\phi=\vec{A}_0z^{\theta_0-1}+\vec{A}_1z^{\theta_0}+\vec{A}_2z^{\theta_0+1}+\frac{1}{4\theta_0}\vec{C}_1z\z^{\theta_0}+\frac{1}{8}\bar{\vec{C}_1}z^{\theta_0-1}\z^2+O(|z|^{\theta_0+2-\epsilon})
    \end{align*}
    so taking $\theta_0=3$ in this equation, we find
    \begin{align}\label{3dzphi0}
    	\p{z}\phi=\vec{A}_0z^2+\vec{A}_1z^3+\vec{A}_2z^4+\frac{1}{12}\vec{C}_1z\z^3+\frac{1}{8}\bar{\vec{C}_1}|z|^4+O(|z|^{5-\epsilon})
    \end{align}
    
    First, as for all $\theta_0\geq 3$, we have $\H=O(|z|^{2-\theta_0})$, and $\p{z}\phi=O(|z|^{\theta_0-1})$, we deduce that
    \begin{align}\label{osc}
    |\H|^2\p{z}\phi=O(|z|^{3-\theta_0}).
    \end{align}
    Furthermore, we have
    \begin{align*}
    \h_0&=2\left(\vec{A}_1-\alpha_0\vec{A}_0\right)z^{\theta_0-1}dz^2+O(|z|^{\theta_0}),\qquad
    \H=\frac{1}{2}\frac{\vec{C}_1}{z^{\theta_0-2}}+\frac{1}{2}\frac{\bar{\vec{C}_1}}{\z^{\theta_0-2}}+O(|z|^{3-\theta_0-\epsilon})
    \end{align*}
    so (as $\s{\vec{A}_0}{\vec{C}_1}=\s{\bar{\vec{A}_0}}{\vec{C}_1}=0$)
    \begin{align}\label{genpart1}
    \s{\H}{\h_0}=\s{\vec{A}_1}{\vec{C}_1}z\,dz+\s{\vec{A}_1}{\bar{\vec{C}_1}}z^{\theta_0-1}\z^{2-\theta_0}\, dz+O(|z|^{2})
    \end{align}
    Now, as
    \begin{align*}
    \p{z}\phi=\vec{A}_0z^{\theta_0-1}+O(|z|^{\theta_0}),\quad e^{2\lambda}=|z|^{2\theta_0-2}+O(|z|^{2\theta_0-1})
    \end{align*}
    we trivially have
    \begin{align}\label{genpart2}
    e^{-2\lambda}\p{\z}\phi=\bar{\vec{A}_0}z^{1-\theta_0}+O(|z|^{2-\theta_0-\epsilon}).
    \end{align}
    Finally, by \eqref{genpart1} and \eqref{genpart2}, we have
    \begin{align*}
    g^{-1}\otimes\s{\H}{\h_0}\otimes\bar{\partial}\phi=\s{\vec{A}_1}{\vec{C}_1}\bar{\vec{A}_0}z^{2-\theta_0}\,dz+\s{\vec{A}_1}{\bar{\vec{C}_1}}\bar{\vec{A}_0}\z^{2-\theta_0}\,dz+O(|z|^{3-\theta_0})
    \end{align*}
    so we obtain by \eqref{osc} the equation
    \begin{align}\label{3boot2}
    \partial\left(\vec{H}-2i\vec{L}+\vec{\gamma}_0\log|z|\right)&=-|\H|^2\partial\phi-2\,g^{-1}\otimes\s{\H}{\h_0}\otimes\bar{\partial}\phi\\
    &=-2\s{\vec{A}_1}{\vec{C}_1}z^{2-\theta_0}-2\s{\vec{A}_1}{\bar{\vec{C}_1}}\bar{\vec{A}_0}\z^{2-\theta_0}+O(|z|^{3-\theta_0-\epsilon}).
    \end{align}
    Taking $\theta_0=3$ in \eqref{3boot2} yields
    \begin{align*}
    \partial\left(\vec{H}-2i\vec{L}+\vec{\gamma}_0\log|z|\right)=-2\s{\vec{A}_1}{\vec{C}_1}\bar{\vec{A}_0}\frac{dz}{z}-2\s{\vec{A}_1}{\bar{\vec{C}_1}}\frac{dz}{\z}+O(|z|^{-\epsilon}).
    \end{align*}
    so for some $\bar{\vec{D}_2}\in\mathbb{C}^n$, we have
    \begin{align}\label{3h1}
    \H-2i\vec{L}+\vec{\gamma}_0\log|z|=\bar{\vec{D}_2}-4\s{\vec{A}_1}{\vec{C}_1}\bar{\vec{A}_0}\log|z|-2\s{\vec{A}_1}{\bar{\vec{C}_1}}\bar{\vec{A}_0}\frac{z}{\z}+O(|z|^{1-\epsilon})
    \end{align}
    Therefore, if we define
    \begin{align}\label{3defconts1}
    \left\{
    \begin{alignedat}{1}
    \vec{C}_2&=\Re\left(\vec{D}_2\right)\in\R^n\\
    \vec{B}_1&=-2\s{\bar{\vec{A}_1}}{\vec{C}_1}\vec{A}_0\in\mathbb{C}^n\\
    \vec{\gamma}_1&=-\vec{\gamma}_0-4\,\Re\left(\s{\vec{A}_1}{\vec{C}_1}\bar{\vec{A}_0}\right)\in\mathbb{R}^n
    \end{alignedat}\right.
    \end{align}
    we obtain by \eqref{3h1}
    \begin{align}\label{3h2}
    \H=\Re\left(\frac{\vec{C}_1}{z}+\vec{B}_1\frac{\z}{z}\right)+\vec{C}_2+\vec{\gamma}_1\log|z|+O(|z|^{1-\epsilon}).
    \end{align}
    Now, we have
    \begin{align*}
    	e^{2\lambda}&=|z|^{2\theta_0-2}\left(1+2\,\Re\left(\alpha_0z+\alpha_1z^2\right)+2|\vec{A}_1|^2|z|^2+O(|z|^{3-\epsilon}\right)\\
    	&=|z|^{4}+\alpha_0z^3\z+\bar{\alpha_0}z^2\z^3+\alpha_1z^4\z^2+\bar{\alpha_1}z^2\z^4+2|\vec{A}_1|^2|z|^6+O(|z|^{7-\epsilon}).
    \end{align*}
    and recall the equation
    \begin{align}\label{nandoumo}
    	\Delta\phi=\frac{e^{2\lambda}}{2}\H.
    \end{align}
    To obtain a second order expansion of the right-hand side of \eqref{nandoumo}, we only need to develop $e^{2\lambda}$ up to order $2$, and we compute directly
	\begin{align}\label{dev3}
		&\frac{e^{2\lambda}}{2}\H=\frac{1}{2}\Big(|z|^4+\alpha_0z^3\z^2+\bar{\alpha_0}z^2\z^3+O(|z|^6)\Big)\cdot \left(\frac{1}{2}\frac{\vec{C_1}}{z}+\frac{1}{2}\frac{\bar{\vec{C}_1}}{\z}+\frac{1}{2}\vec{B}_1\frac{\z}{z}+\frac{1}{2}\bar{\vec{B}_1}\frac{z}{\z}+\vec{C}_2+\vec{\gamma}_1\log|z|+O(|z|^{1-\epsilon})\right)\nonumber\\
		&=\frac{1}{2}\bigg(\frac{1}{2}\vec{C}_1z\z^2+\frac{1}{2}\bar{\vec{C}_1}z^2\z+\frac{1}{2}\vec{B}_1z\z^3+\frac{1}{2}\bar{\vec{B}_1}z^3\z+\vec{C}_2|z|^4+\vec{\gamma}_1|z|^4\log|z|\nonumber\\
		&+\frac{\alpha_0}{2}\vec{C}_1|z|^4+\frac{1}{2}\alpha_0\bar{\vec{C}_1}z^3\z+\frac{1}{2}\bar{\alpha_0}\vec{C}_1z\z^3+\frac{1}{2}\bar{\alpha_0}\bar{\vec{C}_1}|z|^4+O(|z|^{5-\epsilon})\bigg)\nonumber\\
		&=\frac{1}{2}\bigg(\frac{1}{2}\vec{C}_1z\z^2+\frac{1}{2}\bar{\vec{C}_1}z^2\z+\frac{1}{2}\left(\vec{B}_1+{\alpha_0}\vec{C}_1\right)z\z^3+\frac{1}{2}\left(\bar{\vec{B}_1}+\alpha_0\bar{\vec{C}_1}\right)z^3\z+\left(\vec{C}_2+\Re)\alpha_0\vec{C}_1\right)|z|^4\nonumber\\
		&+\vec{\gamma_1}|z|^4\log|z|+O(|z|^{5-\epsilon})\bigg)
	\end{align}
	Therefore, by \eqref{3dzphi0}, \eqref{nandoumo} and \eqref{dev3}, and Proposition \ref{integrating}, there exists $\vec{A}_3\in\C^n$ such that
	\begin{align}\label{3end}
		\p{z}\phi&=\vec{A}_0z^{2}+\vec{A}_1z^3+\vec{A}_2z^4+\vec{A}_3z^5+\frac{1}{2}\bigg(\frac{1}{6}\vec{C}_1z\z^3+\frac{1}{4}\bar{\vec{C}_1}|z|^4+\frac{1}{8}\left(\vec{B}_1+\bar{\alpha_0}\vec{C}_1\right)z\z^4+\frac{1}{4}\left(\bar{\vec{B}_1}+\alpha_0\bar{\vec{C}_1}\right)z^3\z^2\nonumber\\
		&+\frac{1}{3}\left(\vec{C}_2+\Re(\alpha_0\vec{C}_1)\right)z^2\z^3+\frac{\vec{\gamma}_1}{3}z^2\z^3\left(\log|z|-\frac{1}{6}\right)\bigg)+O(|z|^{6-\epsilon})\nonumber\\
		&=\vec{A}_0z^2+\vec{A}_1z^3+\vec{A}_2z^4+\vec{A}_3z^5+\frac{1}{12}\vec{C}_1z\z^3+\frac{1}{8}\bar{\vec{C}_1}|z|^4+\frac{1}{16}\left(\vec{B}_1+\bar{\alpha_0}\vec{C}_1\right)z\z^4+\frac{1}{8}\left(\bar{\vec{B}_1}+\alpha_0\bar{\vec{C}_1}\right)\nonumber\\
		&+\frac{1}{6}\left(\vec{C}_2+\Re(\alpha_0\vec{C}_1)-\frac{\vec{\gamma}_1}{6}\right)z^2\z^3+\frac{\vec{\gamma}_1}{6}z^2\z^3\log|z|+O(|z|^{6-\epsilon}).
	\end{align}
	As $\phi$ is conformal, we have $\s{\p{z}\phi}{\p{z}\phi}=0$ and we compute easily by \eqref{3end} that 
	\begin{align*}
		0=\s{\p{z}\phi}{\p{z}\phi}=\s{\vec{A}_0}{\vec{A}_0}z^{4}+\cdots+\frac{1}{3}\s{\vec{A}_0}{\vec{\gamma}_1}z^4\z^3\log|z|+O(|z|^{8-\epsilon})
	\end{align*}
	so
	\begin{align}\label{3cancel}
		\s{\vec{A}_0}{\vec{A}_0}=\s{\vec{A}_0}{\vec{\gamma}_1}=0.
	\end{align}
	Now,  by \eqref{3defconts1}, and \eqref{3cancel}, we have as $|\vec{A}_0|^2=\dfrac{1}{2}$
	\begin{align}\label{3true}
		0=\s{\vec{A}_0}{\vec{\gamma}_1}=-\s{\vec{A}_0}{\vec{\gamma}_0+2\s{\vec{A}_1}{\vec{C}_1}\bar{\vec{A}_0}+2\bar{\s{\vec{A}_1}{\vec{A}_1}}\vec{A}_0}=-\left(\s{\vec{A}_0}{\vec{\gamma}_0}+\s{\vec{A}_1}{\vec{C}_1}\right)
	\end{align}	
	sp for a \emph{true} Willmore disk, we have $\vec{\gamma}_0=0$,a nd we deduce by \eqref{3true} that
	\begin{align}
		\s{\vec{A}_1}{\vec{C}_1}=0,
	\end{align}
	proving the holomorphy of $\mathscr{Q}_{\phi}$ at a \emph{true} branch point of multiplicity $\theta_0=3$, as
	\begin{align*}
		\mathscr{Q}_{\phi}=2\s{\vec{A}_1}{\vec{C}_1}\frac{dz^4}{z}+O(1).
	\end{align*}
	\begin{rem}
		It does not seem possible to remove this pole in general for branch points of multiplicity $\theta_0=3$ and non-zero residue.
	\end{rem}
	
	\textbf{Step 9. Case $\theta_0=1,2$.} Then both residues vanish, so $\phi$ is smooth and $\mathscr{Q}_{\phi}$ is holomorphic (see Lemma \cite{secondresidue} for a more general proof of this fact).
    
    \textbf{Part 3. Conclusion.}

    Now we suppose that $\mathscr{Q}_{\phi}=0$ and $n=3$. We remark that this is always the case if $\phi$ is a \emph{true} Willmore sphere as $K_{S^2}^4$ is a negative holomorphic bundle (see Proposition \ref{bundletrivial} for instance). Then by Bryant's theorem, some stereographic projection $\pi:S^3\setminus\ens{p}\rightarrow \R^3$ makes the mean curvature of $\pi\circ \phi:\Sigma^2\setminus\phi^{-1}(\ens{p})$ vanish identically.

    As there does not exist compact minimal surfaces in $\R^3$, $\phi^{-1}(\ens{p})$ is not empty, and the same reasoning as in \cite{bryant} shows that the minimal surface $\pi\circ\phi:\Sigma^2\setminus\phi^{-1}(\ens{p})$ is complete. The proof is almost trivial, as a divergent sequence $\ens{q_k}_{k\in \N}$ in $\Sigma^2\setminus \phi^{-1}(\ens{p})$ must converge to some point of $\phi^{-1}(\ens{p})$, but this implies as $\phi$ is continuous that $\phi(q_k)\rightarrow p\in S^3$ as $k\rightarrow\infty$, so $\pi\circ \phi(q_k)\rightarrow \infty$ in $\R^3$.
    
    By the conformal invariance of the Willmore energy, $\pi\circ\phi$ has finite total curvature, but it can have interior branch points.
    \end{proof}

Finally, we expand on Remark \ref{importantremarkerros} to stress out that although branched Willmore spheres are not in general smooth through their branch points, they nevertheless always admit in their Taylor expansion only \emph{integer} powers of $z,\z$ and $\log|z|$. 

\begin{cor}
	Let $n\geq 3$, and $\phi\in C^{0}(D^2,\R^n)\cap C^{\infty}(D^2\setminus\ens{0},\R^n)$ be a Willmore disk, with a unique branch point located at $0$ of multiplicity $\theta_0\geq 1$. Then there exists $\vec{A}_0\in \C^n\setminus\ens{0}$ such that
	\begin{align*}
		\phi(z)=\Re\left(\vec{A}_0z^{\theta_0}\right)+O\left(|z|^{\theta_0+1}\log|z|\right)
	\end{align*}
	and for all $m\geq \theta_0+1$, there exists
	\begin{align*}
		\ens{\vec{A}_{k,l,p}: k,l\in \Z,\;\, \theta_0+1\leq k+l\leq m,\;\, p\in \N}\subset (\mathbb{C}^n)^{\,\Z\times \Z\times \N}
	\end{align*}
	and $p_m\in\N$ such that
	\begin{align}\label{taylorgeneral}
		\phi(z)=\Re\left(\vec{A}_0z^{\theta_0}+\sum_{k,l,p}^{}\vec{A}_{k,l,p}z^k\z^l\log^p|z|\right)+O\left(|z|^{m+1}\log^{p_m}|z|\right),
	\end{align} 
	where the $\vec{A}_{k,l,p}\in \C^n$ are almost all zero, that is, all but finitely many. 
\end{cor}

\begin{rem}
	The proof of the main Theorem \ref{devh0} gives in particular an algorithm to compute all the coefficients in the Taylor expansion of a branched Willmore surface, which was implemented in \cite{sagepaper}. 
\end{rem}

\section{Willmore spheres in $S^4$}\label{s4}
\subsection{Removability of the poles of the meromorphic differentials}

We fix a closed Riemann surface $\Sigma^2$. We recall that we defined in Section \ref{4} for immersions $\phi:\Sigma^2\rightarrow S^4$ on $\C^6$ the $\C$-extension of the Lorentzian metric $h$ on $\R^6$ defined by
\begin{align*}
	h=-dx_0^2+dx_1^2+dx_2^2+dx_3^2+dx_4^2+dx_5^2,
\end{align*}
which permitted to define the section $\psi_{\phi}\in \Gamma((T_{\C}^\nperp \Sigma^2)^\ast\otimes \C^{6})$ defined by
\begin{align*}
	\psi_{\phi}(\vv{\xi})=\s{\H}{\vv{\xi}}(\vec{a}+\phi)+\vv{\xi}
\end{align*}
for all $\vv{\xi}\in \Gamma(T_{\C}^\nperp\Sigma^2)$. As $T_{\C}^\nperp \Sigma^2$ decomposes as 
\begin{align*}
	T_{\C}^\nperp \Sigma^2=(T_{\C}^\nperp \Sigma^2)^{(1,0)}\oplus (T_{\C}^\nperp \Sigma^2)^{(0,1)}=\mathscr{N}\oplus \bar{\mathscr{N}}
\end{align*}
according to the eigenspaces of the almost complex structure $J$ corresponding to the eigenvalues $i$ and $-i$, this permits to identify the holomorphic line bundle structure on $T_{\C}^\nperp \Sigma^2$ with $\mathscr{N}=(T_{\C}^\nperp\Sigma^2)^{(1,0)}$.  In particular, we also have a decomposition
\begin{align*}
	\psi_{\phi}=\psi_{\phi}^{(1,0)}+\psi_{\phi}^{(0,1)}=\psi_{\phi}^{(1,0)}+\bar{\psi_{\phi}^{(1,0)}}
 \end{align*}
where $\psi_{\phi}^{(1,0)}\in \Gamma(\mathscr{N}^\ast\otimes\C^6)$ (resp. $\psi_{\phi}^{(0,1)}\in \Gamma(\bar{\mathscr{N}}^\ast\otimes\C^6)$), which simply means that $\psi_{\phi}^{(1,0)}$ vanishes on $\bar{\mathscr{N}}$ (resp. on $\mathscr{N}$), so defines a section of $\mathscr{N}\otimes \C^6$ (resp. $\bar{\mathscr{N}}\otimes\C^6$). For notational convenience, we shall write $\Psi=\psi_{\phi}^{(1,0)}$.  The pseudo Gauss map $\mathscr{G}:\Sigma^2\rightarrow \mathbb{P}^{4,1}$, is then defined as $\mathscr{G}=[\Psi]$, where $\mathbb{P}^{4,1}$ is the indefinite complex projective plane, defined by
\begin{align*}
	\mathbb{P}^{4,1}=\mathbb{P}^5\cap\ens{[Z]: \s{Z}{\bar{Z}}_h>0}.
\end{align*}
We remark that the indefinite Hermitian product $\s{\,\cdot\,}{\bar{\,\cdot\,}}_h:\mathscr{N}^\ast\otimes \bar{\mathscr{N}}^\ast\rightarrow\C$ furnishes a non-vanishing section of $\mathscr{N}^\ast\otimes \bar{\mathscr{N}}^\ast$, which makes this line bundle holomorphically trivial. In \cite{montiels4}, the three following sections are introduced
\begin{align}\label{defabc}
\left\{\begin{alignedat}{1}
	\mathscr{T}_{\phi}&=\s{\partial^2\Psi}{\partial\bar{\Psi}}_h\in \Gamma(K_{\Sigma^2}^3\otimes \mathscr{N}^\ast\otimes \bar{\mathscr{N}}^\ast)\\
    \mathscr{Q}_{\phi}&=2\s{\partial^2\Psi}{\partial^2\bar{\Psi}}_h\in \Gamma(K_{\Sigma^2}^4\otimes \mathscr{N}^\ast\otimes \bar{\mathscr{N}}^\ast)\\
    \mathscr{O}_{\phi}&=\s{\partial^2\Psi}{\partial^2\Psi}_h\otimes\s{\partial^2\bar{\Psi}}{\partial^2\bar{\Psi}}\in \Gamma({K}_{\Sigma^2}^8\otimes \mathscr{N}^\ast\otimes \bar{\mathscr{N}}^\ast\otimes \mathscr{N}^\ast\otimes \bar{\mathscr{N}}^\ast).
\end{alignedat}\right.
\end{align}

where we noted for simplicity of notation $\partial=\partial^N$ and $\bar{\partial}=\bar{\partial}^N$ the operators defined in \eqref{delbar} (this shall not imply any confusion, as we will only deal with normal sections in this section). Let us recall a useful lemma from \cite{montiels4}.
\begin{lemme}(Montiel)
	Let $\phi:\Sigma^2\rightarrow S^4$ be a Willmore surface. Then we have
\begin{align}\label{partialdel}
\left\{\begin{alignedat}{1}
&\bar{\partial}\mathscr{T}_{\phi}=0\\
&\bar{\partial}\mathscr{Q}_{\phi}={K^\nperp}g\otimes \mathscr{T}_{\phi}\\
&\bar{\partial}\mathscr{O}_{\phi}=2\mathscr{D}\otimes \mathscr{T}_{\phi}
\end{alignedat}\right.
\end{align}
where $K^N=R^N(\e_z,\e_{\z},\vv{\xi},\bar{\vv{\xi}})$ is the normal curvature (where $\vv{\xi}\in \Gamma(\mathscr{N})$ is any section such that $|\vv{\xi}|=1$) and where $\mathscr{D}\in \Gamma(K_{\Sigma^2}^5\otimes \bar{K}_{\Sigma^2}\otimes\mathscr{N}^\ast\otimes \bar{\mathscr{N}}^\ast)$ is a non-zero section.
\end{lemme}
In particular, we see that $\mathscr{Q}_{\phi}$ and $\mathscr{O}_{\phi}$ are not holomorphic in general if the genus of $\Sigma^2$ is not zero. 

Suppose one moment that $\Sigma^2$ has genus $0$, and that the immersion $\phi:S^2\rightarrow S^4$ is smooth, as $\mathscr{T}_{\phi}$ is holomorphic, and $\mathscr{N}^\ast \otimes \bar{\mathscr{N}}^\ast$ is holomorphically trivial, we have
\begin{align*}
	\mathscr{A}_{\phi}\in H^0(K_{S^2}^3\otimes \mathscr{N}^\ast\otimes \bar{\mathscr{N}}^\ast)\simeq H^0(K_{S^2}^3)
\end{align*}
so as $K_{S^2}^3$ is a negative bundle, we deduce that $\mathscr{T}_{\phi}=0$, so by \ref{partialdel} the sections $\mathscr{Q}_{\phi}$ and $\mathscr{O}_{\phi}$ are holomorphic, so they also vanish by the same remark on $\mathscr{N}^\ast\otimes \bar{\mathscr{N}}^\ast$. 

We can easily compute (see \cite{montiels4}, Remark 5)
\begin{align*}
	\mathscr{T}_{\phi}=g^{-1}\otimes \left(\bar{\partial}^N\h_0\totimes J\h_0\right)=g^{-1}\otimes \left(\bar{\partial}\h_0\totimes J\h_0\right),
\end{align*}
where $J$ is the almost complex structure defined in section $\ref{4}$. As at a branch point $p\in S^2$ of multiplicity $\theta_0\geq 1$, for some complex coordinate $z:D^2\rightarrow S^2$ sending $0$ to $p$, we have \text{a priori} the estimates
\begin{align*}
	\h_0=O(|z|^{\theta_0-1}),\quad e^{2\lambda}=|z|^{2-2\theta_0}\left(1+O(|z|)\right),
\end{align*}
which implies that
\begin{align*}
	\mathscr{T}_{\phi}=O(|z|^{2-2\theta_0}|z|^{\theta_0-2}|z|^{\theta_0-1})=O(|z|^{-1}).
\end{align*}
Therefore, $\mathscr{T}_{\phi}$ has poles order of at most $1$. In particular, if $D=\sum_{j=1}^{m} \theta_0(p_j)p_j$ is the branching divisor of $\phi$, and $D_0=p_1+\cdots+p_m$, then
\begin{align*}
\mathscr{A}_{\phi}\in H^0(K_{S^2}^3\otimes \mathscr{O}(D_0)\otimes \mathscr{N}^\ast\otimes\bar{\mathscr{N} }^\ast)\simeq H^0(K_{S^2}^3\otimes \mathscr{O}(D_0))
\end{align*}
 and as
\begin{align*}
	\mathrm{deg}(K_{S^2}^3\otimes \mathscr{O}(D_0))=m-6<0
\end{align*}
for $m\leq 5$, we deduce that $\mathscr{T}_{\phi}$, is equal to zero for branched Willmore spheres with less than $5$ branch points, so $\mathscr{Q}_{\phi}$ and $\mathscr{O}_{\phi}$ are meromorphic.

We now come back to the general case where $\Sigma^2$ is an arbitrary closed Riemann surface. By \eqref{partialdel}, we only know that $\mathscr{T}_{\phi}$ is meromorphic, so $\mathscr{Q}_{\phi}$ and $\mathscr{O}_{\phi}$ are not even meromorphic, and we cannot get a partial result on the classification. 

Now assume that $\phi$ is variational.

Taking a stereographic projection $S^4\rightarrow\R^4$ of $\phi$ of centre outside of $\phi(\Sigma^2)\subset S^4$, by conformal invariance of Willmore energy, we can see $\phi$ as a Willmore immersion $\Sigma^2\rightarrow\R^4$. If $p\in \Sigma^2$ is a branch point of $\phi$ of multiplicity $\theta_0\geq 3$, there exists by the proof of Theorem \ref{devh0}. a complex coordinate $z:D^2\rightarrow S^2$ sending $0$ to $p$ such that for some $\vec{A}_1\in \C^n$, we have
\begin{align}\label{3form1}
	\h_0=\vec{A}_1z^{\theta_0-1}+O(|z|^{\theta_0-\epsilon})
\end{align}
for all $\epsilon>0$ (we only need the first upper regularity at branch points for the holomorphy of $\mathscr{T}_{\phi}$). In particular, we deduce that
\begin{align}\label{3form2}
	\bar{\partial}\h_0=O(|z|^{\theta_0-1-\epsilon})
\end{align}
and as by definition of a branch point of multiplicity $\theta_0\geq 2$, there exists $\lambda>0$ such that
\begin{align*}
	g=e^{2\lambda}|dz|^2=\lambda|z|^{2\theta_0-2}\left(1+O(|z|)\right)|dz|^2
\end{align*}
we deduce by by \eqref{3form1} and \eqref{3form2} that
\begin{align}\label{Tholomorphic}
	\mathscr{T}_{\phi}&=g^{-1}\otimes\left(\bar{\partial}\h_0\totimes J\h_0\right)=O(|z|^{-\epsilon})\qquad \text{for all}\;\, \epsilon>0.
\end{align}
$\mathscr{T}_{\phi}$ is holomorphic everywhere on $z(D^2)$ by a classical singularity removability result.


Therefore, we have established the following.

\begin{prop}
	Let $\Sigma^2$ be a closed Riemann surface, and $\phi:\Sigma^2\rightarrow S^4$ be a branched Willmore surface. The $3$-form $\mathscr{T}_{\phi}$ defined by
	\begin{align}
		\mathscr{T}_{\phi}=g^{-1}\otimes\left(\bar{\partial}\h_0\totimes J\h_0\right)
	\end{align}
	is a holomorphic section of $K_{\Sigma^2}^3$. In particular, if $\Sigma^2$ has genus $0$, then $\mathscr{T}_{\phi}$ vanishes and the respectively $4$-forms and $8$-forms $\mathscr{Q}_{\phi}$ and $\mathscr{O}_{\phi}$ defined in  are meromorphic. 
\end{prop}

This is not by chance that to denote Montiel's quartic form, we used the same notations as Bryant's quartic form, as the object of the next proposition is to show that they virtually coincide, an that the form  $\mathscr{O}_{\phi}$ of degree $8$ enjoys a similar null structure.

\begin{theorem}\label{threedev}
	Let $\phi:\Sigma^2\rightarrow S^4$ be a smooth immersion. Then we have
	\small
	\begin{align}\label{48formula}
		&\mathscr{Q}_{\phi}=\,g^{-1}\otimes \left(\partial^N\bar{\partial}^N\h_0\,\dot{\otimes}\, \h_0-\partial^N\h_0\otimes\bar{\partial}^N\h_0
		\right)+\frac{1}{4}(1+|\H|^2)\h_0\,\dot{\otimes}\, \h_0\nonumber\\
		&\mathscr{O}_{\phi}=g^{-2}\otimes\bigg\{\frac{1}{4}(\partial^N\bar{\partial}^N\h_0\,\dot{\otimes}\,\partial^N\bar{\partial}^N\h_0)\otimes (\h_0\,\dot{\otimes}\,\h_0)+\frac{1}{4}(\partial^N\h_0\,\dot{\otimes}\, \partial^N\h_0)\otimes (\bar{\partial}^N\h_0\,\dot{\otimes}\, \bar{\partial}^N\h_0)\nonumber\\
		&-\frac{1}{2}(\partial^N\bar{\partial}^N\h_0\totimes\partial^N\h_0)\otimes(\bar{\partial}^N\h_0\totimes\h_0)-\frac{1}{2}(\partial^N\bar{\partial}^N\h_0\totimes\bar{\partial}^N\h_0)\otimes(\partial^N\h_0\totimes\h_0)+\frac{1}{2}(\partial^N\bar{\partial}^N\h_0\totimes\h_0)\otimes(\partial^N\h_0\totimes\bar{\partial}^N\h_0)\bigg\}\nonumber\\
		&+\frac{1}{4}(1+|\H|^2)\,g^{-1}\otimes\left\{\frac{1}{2}(\partial^N\bar{\partial}^N\h_0\totimes \h_0)\otimes (\h_0\totimes\h_0)-(\partial^N\h_0\totimes\h_0)\otimes (\bar{\partial}^N\h_0\totimes\h_0)+\frac{1}{2}(\partial^N\h_0\totimes\bar{\partial}^N\h_0)\otimes (\h_0\totimes\h_0)\right\}\nonumber\\
		&+\frac{1}{64}\left(1+|\H|^2\right)^2\,\left(\h_0\totimes\h_0\right)^2.
	\end{align}
	\normalsize
\end{theorem}
\begin{proof} 
	\textbf{For the sake of simplicity of notations, we will write $\partial$ (resp. $\bar{\partial}$) instead of $\partial^N$ (resp. $\bar{\partial}^N$)}
	We take some conformal chart $z$ such that we have a local orthonormal frame $(\n_1,\n_2)$ of the normal bundle. If $J$ is the almost complex structure introduced in Section \ref{4}, we recall that $J\n_1=-\n_2$. In particular, defining 
	\begin{align*}
		\e_1=\frac{1}{\sqrt{2}}(\n_1+i\,\n_2),\quad \e_2=\frac{1}{\sqrt{2}}(\n_1-i\,\n_2),
	\end{align*}
	then as $T_{\C}^\nperp\Sigma^2$ splits in
	\begin{align*}
		T_{\C}^\nperp\Sigma^2=\mathscr{N}\oplus \bar{\mathscr{N}},
	\end{align*}
	where $\mathscr{N}$ (resp. $\bar{\mathscr{N}}$) is the eigenspace of $J$ associated to the eigenvalue $i$ (resp. $-i$), and the eigenvector vector $\e_1$ (resp. $\e_2$) is a local trivialisation of $\mathscr{N}$ (resp. $\bar{\mathscr{N}}$), and for all section $\vec{F}\in \Gamma(T_{\C}^\nperp\Sigma^2)$, we shall adopt the notational convention
	\begin{align*}
		\vec{F}={F}^1\,\e_1+{F}^2\,\e_2.
	\end{align*}
	Note that $(\e_1,\e_2)$ is an orthonormal basis of $T_{\C}^\nperp\Sigma^2$ for the Hermitian product $\s{\,\cdot\,}{\bar{\,\cdot\,}}$, which implies that
	\begin{align*}
		\s{\e_1}{\e_1}=\s{\e_2}{\e_2}=0,\quad \s{\e_1}{\e_2}=1
	\end{align*}
	so in particular, we have (if $\vec{G}=G_1\e_1+G_2\e_2$ is a normal section)
	\begin{align}\label{twisted}
	    \s{\vec{F}}{\vec{F}}=2F^1F^2,\quad 
		\s{\vec{F}}{\vec{G}}=F^1G^2+F^2G^1.
	\end{align}
	We write
	\begin{align*}
		\h_0&=h^1\,\e_1+h^2\,\e_2,\quad
		\partial\h_0=h^1_{z}\,\e_1+h^2_{z}\,\e_2,\quad
		\bar{\partial}\h_0=h^1_{\z}\,\e_1+h^2_{\z}\,\e_2,\quad \partial\bar{\partial}\h_0=h^1_{z\z}\,\e_1+h^2_{z\z}\,\e_2.
	\end{align*}
	Then we recall that for all $\vv{\xi},\vv{\eta}\in \Gamma(T_{\C}^\nperp\Sigma^2)$,
	\begin{align*}
	&\s{\D_{\p{z}}\D_{\p{z}}{\psi}}{\D_{\p{z}}\D_{\p{z}}{\psi}}_h(\vv{\xi},\vv{\eta})=\frac{e^{2\lambda}}{2}\left(\s{\D_{\p{z}}^N\D_{\p{z}}^N\H}{\vv{\xi}}\s{\H_0}{\vv{\eta}}-\s{\D_{\p{z}}^N\vec{H}}{\vv{\xi}}\s{\D_{\p{z}}^N\H_0}{\vv{\eta}}\right)\\
	&+\frac{e^{2\lambda}}{2}\left(\s{\D_{\p{z}}^N\D_{\p{z}}^N\H}{\vv{\eta}}\s{\H_0}{\vv{\xi}}-\s{\D_{\p{z}}^N\H}{\vv{\eta}}\s{\D_{\p{z}}^N\H_0}{\vv{\xi}}\right)
	+\frac{e^{4\lambda}}{4}\s{\H_0}{\H_0}(1+|\H|^2)\\
	&=\frac{1}{2}\,g^{-1}\otimes\bigg\{\s{\partial\bar{\partial}\h_0}{\vv{\xi}}\otimes\s{\h_0}{\vv{\eta}}+\s{\partial\bar{\partial}\h_0}{\vv{\eta}}\otimes\s{\h_0}{\vv{\xi}}-\s{\partial\h_0}{\vv{\xi}}\otimes\s{\bar{\partial}\h_0}{\vv{\eta}}-\s{\partial\h_0}{\vv{\eta}}\otimes\s{\partial\h_0}{\vv{\xi}}\bigg\}\\
	&+\frac{1}{4}(1+|\H|^2)\s{\h_0}{\vv{\xi}}\otimes\s{\h_0}{\vv{\eta}}.
	\end{align*}
	Furthermore, we note that
	\begin{align*}
		\s{\vec{F}}{\e_1}\s{\vec{G}}{\e_2}+\s{\vec{F}}{\e_2}\s{\vec{G}}{\e_1}&=F^2G^1+F^1G^2=\s{\vec{F}}{\vec{G}}.
	\end{align*}
	Therefore we deduce that (as $\psi_{\phi}={\Psi}+\bar{\Psi}$)
	\begin{align*}
		\mathscr{Q}_{\phi}&=2\s{\partial^2{\Psi}}{\partial^2\bar{\Psi}}=2\s{\partial^2\psi}{\partial^2\psi}_h(\e_1,\e_2)\\
		&=\,g^{-1}\otimes\left(\partial\bar{\partial}\h_0\totimes\h_0-\partial\h_0\totimes\bar{\partial}\h_0\right)+\frac{1}{4}\left(1+|\H|^2\right)\h_0\totimes\h_0,
	\end{align*}
	so this justifies the introduction of the factor $2$ in the definition of $\mathscr{Q}_{\phi}$, as we recover the same expression of Bryant's quartic form, virtually extended to immersions in $S^4$.
	Then we have
	\begin{align}\label{O1}
		\mathscr{O}_{\phi}&=\s{\partial^2\Psi}{\partial^2\Psi}\otimes\s{\partial^2\bar{\Psi}}{\partial^2\bar{\Psi}}=\s{\partial^2\psi}{\partial^2\psi}(\e_1,\e_1)\otimes \s{\partial^2\psi}{\partial^2\psi}(\e_2,\e_2)\nonumber\\
		&=\left(e^{-2\lambda}\left(h^1_{z\z}h^1-h^1_zh^1_{\z}\right)+\frac{1}{4}(1+|\H|^2)(h^1)^2\right)\left(e^{-2\lambda}\left(h_{z\z}^2h^2-h^2_zh^2_{\z}\right)+\frac{1}{4}(1+|\H|^2)(h^2)^2\right)dz^8\nonumber\\
		&=e^{-4\lambda}\left(h_{z\z}h_{z\z}^2h^1h^2+h_z^1h_{\z}^1h_{z}^2h_{\z}^2-h_{z\z}^1h^1h_z^2h_{z}^2-h_{z\z}^2h^2h_z^1h_{\z}^1\right)\nonumber\\
		&+\frac{1}{4}(1+|\H|^2)e^{-2\lambda}\bigg\{(h^1)^2\left(h_{z\z}^2h^2-h_z^2h_{\z}^2\right)+(h^2)^2\left(h_{z\z}^1h^1-h_{z}^1h_{\z}^1\right)\bigg\}
		+\frac{1}{16}(1+|\H|^2)^2(h^1h^2)^2\nonumber\\
		&=e^{-4\lambda}\mathrm{(I)}+\frac{1}{4}(1+|\H|^2)e^{-2\lambda}\mathrm{(II)}+\frac{1}{64}(1+|\H|^2)^2(\h_0\totimes\h_0)\otimes (\h_0\totimes\h_0)
	\end{align}
	with evident definitions of $\mathrm{(I)}$ and $\mathrm{(II)}$, as $h^1h^2=\dfrac{1}{2}\h_0\totimes \h_0$ by \eqref{twisted}.
	We compute
	\begin{align*}
		(h^1)^2(h_{z\z}^2h^2-h^2_zh^2_{\z})&=h^1h^2((h^1h^2_{z\z}+h^2h^1_{z\z})-h^2h^1_{z\z})-h^1h^2_zh^1h_{\z}^2\\
		&=\frac{1}{2}(\h_0\totimes\h_0)\otimes (\partial\bar{\partial}\h_0\totimes\h_0)-(h^2)^2h_{z\z}^1h^1-h^1h^2_zh^1h^2_{\z}.
	\end{align*}
	so
	\begin{align*}
		(h^1)^2(h_{z\z}^2h^2-h^2_zh^2_{\z})+(h^2)^2(h^1_{z\z}h^1-h^1_zh^1_{\z})=\frac{1}{2}(\h_0\totimes\h_0)\otimes (\partial\bar{\partial}\h_0\totimes\h_0)-h^1h_z^2h^1h_{\z}^2-h^2h_z^1h^2h_{\z}^1
	\end{align*}
	and
	\begin{align*}
		h^1h_z^2h^1h_{\z}^2+h^2h_z^1h^2h_{\z}^1&=((h^1h_{\z}^2+h^2h_{\z}^1)-h^2h_{\z}^1)h^1h_{z}^2+((h^2h_{\z}^1+h^1h_{\z}^2)-h^1h_{\z}^2)h^2h_{z}^1\\
		&=(\h_0\totimes\bar{\partial}\h_0)(h^1h_z^2+h^2h_z^1)-\frac{1}{2}(\h_0\totimes\h_0)(h^1_{\z}h^2_{z}+h^2_{\z}h^1_{z})\\
		&=(\partial\h_0\totimes \h_0)\otimes(\bar{\partial}\h_0\otimes\h_0)-\frac{1}{2}(\partial\h_0\totimes\bar{\partial}\h_0)\otimes (\h_0\totimes\h_0).
	\end{align*}
	We deduce that
	\begin{align}\label{O2}
		\mathrm{(II)}=\frac{1}{2}(\partial\bar{\partial}\h_0\totimes\h_0)\otimes(\h_0\totimes\h_0)+(\partial\h_0\totimes\h_0)\otimes(\bar{\partial}\h_0\otimes\h_0)-\frac{1}{2}(\partial\h_0\totimes\bar{\partial}\h_0)\otimes(\h_0\totimes\h_0).
	\end{align}
	The idea here is to make circular permutations to obtain non circular computations.  The first two terms already have the good algebraic structure as
	\begin{align}
		h^1_{z\z}h_{z\z}^2h^1h^2+h^1_zh^2_zh^1_{\z}h_{\z}^2=\frac{1}{4}(\partial\bar{\partial}\h_0\totimes\partial\bar{\partial}\h_0)\otimes (\h_0\totimes\h_0)+\frac{1}{4}(\partial\h_0\totimes\partial\h_0)\otimes (\bar{\partial}\h_0\totimes\bar{\partial}\h_0).
	\end{align}
	Then we have
	\begin{align*}
		h^1_{z\z}h^1h_z^2h_{\z}^2&=((h_{z\z}^1h_z^2+h_{z\z}^2h_z^1)-h_{z\z}^2h^1_z)h^1h_{\z}=(\partial\bar{\partial}\h_0\totimes\partial\h_0)h^1h^2_{\z}-h^2_{z\z}h^1_zh^1h^2_{\z}\\
		h^2_{z\z}h^2h^1_zh^1_{\z}&=(\partial\bar{\partial}\h_0\totimes\partial\h_0)h^2h^1_{\z}-h^1_{z\z}h^2_zh^2h^1_{\z}
	\end{align*}
	therefore
	\begin{align}\label{step1}
		h^1_{z\z}h^1h_z^2h_{\z}^2+h^2_{z\z}h^2h^1_zh^1_{\z}=(\partial\bar{\partial}\h_0\totimes\partial\h_0)(\bar{\partial}\h_0\totimes\h_0)-h^2_{z\z}h^1_zh^1h_{\z}^2-h^1_{z\z}h^2_zh^2h^1_{\z}.
	\end{align}
	Then
	\begin{align*}
		h^2_{z\z}h_z^1h^1h_{\z}^2&=((h_{z\z}^2h^1+h^1_{z\z}h^2)-h^1_{z\z}h^2)h^1_zh_{\z}^2=(\partial\bar{\partial}\h_0\totimes\h_0)h^1_zh^2_{\z}-h^1_{z\z}h^2h^1_zh_{\z}^2\\
		h^1_{z\z}h^2_zh^2h^1_{\z}&=(\partial\bar{\partial}\h_0\totimes\h_0)h^2_zh^1_{\z}-h^2_{z\z}h^1h^2_zh^1_{\z}
	\end{align*}
	so
	\begin{align}\label{step2}
		h^2_{z\z}h^1_zh^1h_{\z}^2+h^1_{z\z}h^2_zh^2h^1_{\z}=(\partial\bar{\partial}\h_0\totimes\h_0)\otimes(\partial\h_0\totimes\bar{\partial}\h_0)-h^1_{z\z}h^2h^1_zh^2_{\z}-h^2_{z\z}h^1h^2_zh^1_{\z}.
	\end{align}
	We are almost done, as
	\begin{align*}
		h^1_{z\z}h^2h_z^1h_{\z}^2&=(\partial\bar{\partial}\h_0\totimes\bar{\partial}\h_0)h^2h^1_z-h^2_{z\z}h^1_{\z}h^2h^1_z\\
		h_{z\z}^2h^1h_z^2h_{\z}^1&=(\partial\bar{\partial}\h_0\totimes\bar{\partial}\h_0)h^1h^2_z-h^1_{z\z}h_{\z}^2h^1h_{\z}^2,
	\end{align*}
	so
	\begin{align}\label{step3}
		h^1_{z\z}h^2h^1_zh^2_{\z}+h^2_{z\z}h^1h^2_zh^1_{\z}=(\partial\bar{\partial}\h_0\totimes\bar{\partial}\h_0)\otimes (\partial\h_0\totimes\h_0)-(h^1_{z\z}h^1h^2_zh_{\z}^2+h^2_{z\z}h^2h^1_zh^1_{\z})
	\end{align}
	and we recognize the left-hand side of \eqref{step1}. Taking the signs in account, we have
	\begin{align}\label{step4}
		h^1_{z\z}h^1h^2_zh_{\z}^2+h^2_{z\z}h^2h^1_zh^1_{\z}&=\frac{1}{2}(\partial\bar{\partial}\h_0\totimes\partial\h_0)\otimes(\bar{\partial}\h_0\totimes\h_0)+\frac{1}{2}(\partial\bar{\partial}\h_0\totimes\bar{\partial}\h_0)\otimes ({\partial}\h_0\otimes\h_0)\nonumber\\
		&-\frac{1}{2}(\partial\bar{\partial}\h_0\totimes\h_0)\otimes(\partial\h_0\totimes\bar{\partial}\h_0).
	\end{align}
	Therefore, we have
	\begin{align}\label{O3}
		\mathrm{(I)}&=h^1_{z\z}h_{z\z}^2h^1h^2+h^1_zh^2_zh^1_{\z}h_{\z}^2-(h^1_{z\z}h^1h^2_zh_{\z}^2+h^2_{z\z}h^2h^1_zh^1_{\z})\nonumber\\
		&=\frac{1}{4}(\partial\bar{\partial}\h_0\totimes\partial\bar{\partial}\h_0)\otimes (\h_0\otimes\h_0)+\frac{1}{4}(\partial\h_0\totimes\partial\h_0)\otimes (\bar{\partial}\h_0\totimes\bar{\partial}\h_0)\nonumber\\
		&-\frac{1}{2}(\partial\bar{\partial}\h_0\totimes\partial\h_0)\otimes(\bar{\partial}\h_0\totimes\h_0)-\frac{1}{2}(\partial\bar{\partial}\h_0\totimes\bar{\partial}\h_0)\otimes ({\partial}\h_0\otimes\h_0)\nonumber\nonumber\\
		&+\frac{1}{2}(\partial\bar{\partial}\h_0\totimes\h_0)\otimes(\partial\h_0\totimes\bar{\partial}\h_0)
	\end{align}
	so putting together \eqref{O1}, \eqref{O2}, \eqref{O3}, we obtain the expression announced in the proposition.
\end{proof}

\begin{rem}
	Without our analysis from the previous section, the $8$-differential $\mathscr{O}_{\phi}$ would have poles of order at most $4$, and by Riemann-Roch theorem, this would allow a generalisation of Montiel's theorem for branched Willmore spheres with less than $3$ branch points, while the very first step of the proof where we show the additional regularity at branch points would raise this number of branch points to $7$.
\end{rem}


\begin{theorem}\label{s4vanish}
	Let $\Sigma^2$ be a closed Riemann surface, $\phi:S^2\rightarrow S^4$ be a  branched Willmore surface such that for all $p\in \Sigma^2$ the first and second residue $\vec{\gamma}_0(p)$ and $r(p)$ satisfy
	\begin{align*}
	\left\{
	\begin{alignedat}{2}
		&\vec{\gamma}_0(p)=0\qquad \qquad&&\text{if}\;\, 1\leq \theta_0(p)\leq 3\\
		&r(p)\leq \theta_0(p)-2\qquad \qquad&&\text{if}\;\, \theta_0(p)\geq 2.
		\end{alignedat}\right.
	\end{align*}
	 If the cubic form $\mathscr{T}_{\phi}$ vanishes, the respectively quartic and octic forms $\mathscr{Q}_{\phi}$ and $\mathscr{O}_{\phi}$ are holomorphic. In particular, if $\Sigma^2$ has genus $0$, then the respectively cubic, quartic and octic \emph{holomorphic} differentials $\mathscr{T}_{\phi}$, $\mathscr{Q}_{\phi}$, and $\mathscr{O}_{\phi}$ vanish identically.
\end{theorem}
\begin{proof}
	If $\mathscr{T}_{\phi}=0$, then $\mathscr{Q}_{\phi}$ and $\mathscr{O}_{\phi}$ are meromorphic. Then, Theorem \ref{devh0} applies and shows that $\mathscr{Q}_{\phi}$ is holomorphic.
	
	To see that $\mathscr{O}_{\phi}$ is holomorphic is a bit more delicate and is the object of Chapter $4$ in \cite{sagepaper}. Notice also that this octic differential is holomorphic once $\mathscr{Q}_{\phi}$ and $\mathscr{Q}_{\phi}$ are holomorphic.
 \end{proof}
We now recall one of Montiel's main theorem of \cite{montiels4}.
\begin{theorem}[Montiel]\label{montiel1}
	Let $\phi:\Sigma^2\rightarrow S^4$ be a branched Willmore sphere, and $\mathscr{G}:\Sigma^2\rightarrow\C\mathbb{P}^{4,1}$ be its pseudo Gauss map. Then $\mathscr{G}$ is meromorphic, of anti-holomorphic, or lies in a null totally geodesic complex hypersurface of the null quadric $Q^{3,1}\subset \C\mathbb{P}^{4,1}$, defined by
	\begin{align}
		Q^{3,1}=\C\mathbb{P}^{4,1}\cap\ens{[Z]: \s{Z}{\bar{Z}}_h=0}.
	\end{align} 
\end{theorem}
In the third case, the condition is equivalent to the following assertion : there exists a null vector $p\in \R^6$ such that $\s{\psi_{\phi}^{(1,0)}}{q}=0$. Up to scaling, we have $q=-(a+p)$ for some $p\in S^4$, and $a=(1,0,\cdots,0)\in \R^5$ and this is equivalent to
\begin{align*}
	\s{\H}{\vv{\xi}}\left(1-\s{\phi}{p}\right)-\s{\phi}{p}=0,
\end{align*}
for all $\vv{\xi}\in T_{\C}^\nperp \Sigma^2$. Therefore, we have
\begin{align*}
	\H=\frac{p^N}{1-\s{\phi}{p}},
\end{align*}
but this exactly means that the mean curvature of $\pi_p\circ \phi:S^2\setminus\phi^{-1}(\ens{p})\rightarrow\R^4$ (where $\pi_p:S^2\setminus\ens{p}\rightarrow\R^4$ is the stereographic projection based in $p$) vanishes identically. In particular the dual minimal surface is complete and has finite total curvature by the conformal invariance of the Willmore energy, and furthermore, has zero flux if and only $\phi$ is a true Willmore sphere by Theorem \ref{galois}. However, the number of ends of the dual minimal surface is not given easily thanks to the more complicated relationship between the order of branch points of minimal surfaces and the multiplicities appearing in the Jorge-Meeks formula. Nevertheless, the Willmore energy is still quantized by $4\pi$ for Willmore spheres in these class. We shall see shortly that his phenomenon is valid for all Willmore spheres.

\subsection{Twistor constructions}

We refer to \cite{bryant2} for references on the material introduced here. Let $\mathbb{H}$ be the real division algebra of quaternions. A convenient notation is to write every quaternion as $q=z_0+jz_1$, where $z_0,z_1\in \C$, and $j\in \mathbb{H}$ is such that
\begin{align*}
	j^2=-1,\quad zj=j\z
\end{align*}
for all $z\in \C$. For all $\vec{v}\in \mathbb{H}^2\setminus\ens{0}$, let $\vec{v}\,\C$ and $\vec{v}\, \mathbb{H}$ the complex line and quaternion line associated to $\vec{v}$. As the preceding definition of $\mathbb{H}$ makes it a $\C$-vector space, where $\C$ acts on $\H$ by right multiplication, we can view $\vec{v}\,\C\subset \vec{v}\,\mathbb{H}$. Identifying $\mathbb{H}^2$ with $\C^4$ thanks to the map $\varphi :\C^4\rightarrow \mathbb{H}^2$, such that for all $z=(z_0,z_1,z_2,z_3)\in \C^4$
\begin{align*}
	\varphi(z)=(z_0+jz_1,z_2+jz_3),
\end{align*}
the map
\begin{align}\label{fibration}
	\begin{aligned}
	 \mathbb{H}^2\setminus\ens{0}&\rightarrow \mathbb{H}\mathbb{P}^1=\mathbb{P}(\mathbb{H})\\
	 \vec{v}\,\C&\mapsto \vec{v}\,\mathbb{H}
	\end{aligned}
\end{align}
induced a well-defined map $T:\mathbb{C}\mathbb{P}^3\rightarrow \mathbb{H}\mathbb{P}^1$, which is nothing else than the Penrose fibration. As $T^{-1}(\vec{v}\,\mathbb{H})$ is equal to the complex lines of $\vec{v}\,\mathbb{H}\simeq \C^2$, the fibres are bi-holomorphic to $\C\mathbb{P}^1$, and it is proved in \cite{bryant2} that this map is a surjective submersion, so we obtain a fibration
\[
\begin{tikzcd}
\mathbb{C}\mathbb{P}^1 \arrow{r}{\iota} & \mathbb{C}\mathbb{P}^3 \arrow{d}{T}\\
& \mathbb{H}\mathbb{P}^1.
\end{tikzcd}
\]
Then for all smooth immersion $\phi:\Sigma^2\rightarrow S^4$, we can define a section
\begin{align*}
	\bar{\partial}\phi\wedge \xi\in \Gamma(\mathscr{N}^\ast\otimes \bar{K}_{\Sigma^2}\otimes \wedge^2\C^5)
\end{align*}
the class of this section in $\mathbb{C}\mathbb{P}^9$ is the Penrose lifting $\vec{\Gamma}_{\phi}:\Sigma^2\rightarrow\mathbb{C}\mathbb{P}^9$. Actually, by the Veronese embedding $\mathbb{C}\mathbb{P}^3=\mathbb{C}\mathbb{P}^9=\mathbb{P}(\wedge^2\C^5)$, one can check that we obtain a map $\vec{\Gamma}_{\phi}:\Sigma^2\rightarrow\mathbb{C}\mathbb{P}^3$, as the special expansion of $\phi$ at branch points first proved in \cite{beriviere} shows that $\widetilde{\mathscr{G}}$ is well-defined at branch points. This phenomenon is very similar to minimal surfaces in Euclidean spaces. We recall the following theorem of Montiel.

\begin{theorem}
	The holomorphic locus (resp. anti-holomorphic) of the pseudo Gauss map $\vec{\mathscr{G}}_{\phi}:\Sigma^2\rightarrow\mathbb{C}\mathbb{P}^{4,1}$ and of Penrose lifting $\vec{\Gamma}_{\phi}:\mathbb{C}\mathbb{P}^1\rightarrow \mathbb{C}\mathbb{P}^3$ of a conformal immersion $\phi:S^2\rightarrow S^4$ are equal.
\end{theorem}

Therefore, by Theorem \ref{montiel1}, we can assume up to replacing $\phi$ by $-\phi$, that $\mathscr{G}:\Sigma^2\rightarrow\mathbb{C}\mathbb{P}^{4,1}$ is holomorphic. To be able to conclude, we need to prove that whenever the Penrose lifting ${\Gamma}_{\phi}:\Sigma^2\rightarrow\C\mathbb{P}^3$ of a branched Willmore sphere $\phi:\Sigma^2\rightarrow \mathbb{H}\mathbb{P}^1$ is holomorphic (a condition equivalent to the holomorphy of the pseudo Gauss map), then the following diagram commutes
\[
\begin{tikzcd}
-\Sigma^2\arrow[swap]{rd}{\phi} \arrow{r}{\vec{\Gamma}_{\phi}} & \mathbb{C}\mathbb{P}^3 \arrow{d}{T}\\%
& \mathbb{H}\mathbb{P}^1
\end{tikzcd}
\]
where we identified $S^4$ with $\mathbb{H}\mathbb{P}^1$. Indeed, the long exact sequence of homotopy derived from \eqref{fibration} show that $\mathbb{H}\mathbb{P}^1$ is simply connected, while it is proved in \cite{bryant2} that $\mathbb{H}\mathbb{P}^1$ can be equipped with a metric of constant sectional curvature $1$, so is isometric to $S^4$ by a classical theorem of Riemannian geometry. As the commutativity of this diagram is also proved in the aforementioned paper, we are done.

As $\vec{\Gamma}_{\phi}:\Sigma^2\rightarrow\mathbb{C}\mathbb{P}^3$ is holomorphic, by a theorem of Chow (\cite{chow}), its image is an algebraic curve, and its projection in $S^4$ through the Penrose fibration is an algebraic surface in $S^4$ which coincides with the original Willmore sphere $\phi:S^2\rightarrow S^4$. Therefore we have proved the following.

\begin{theorem}\label{s4generalfinal}
	Let $\Sigma^2$ be a closed Riemann surface and $\phi:\Sigma^2\rightarrow S^4$ be a branched Willmore surface such that 
	$p\in \Sigma^2$ the first and second residue $\vec{\gamma}_0(p)$ and $r(p)$ satisfy
	\begin{align*}
	\left\{
	\begin{alignedat}{2}
	&\vec{\gamma}_0(p)=0\qquad \qquad&&\text{if}\;\, 1\leq \theta_0(p)\leq 3\\
	&r(p)\leq \theta_0(p)-2\qquad \qquad&&\text{if}\;\, \theta_0(p)\geq 2.
	\end{alignedat}\right.
	\end{align*}
	Then $\mathscr{T}_{\phi}$ is holomorphic, and if $\mathscr{T}_{\phi}=0$, then the meromorphic $4$ and $8$-forms $\mathscr{Q}_{\phi}$ and $\mathscr{O}_{\phi}$ are holomorphic. If 
	$\mathscr{T}_{\phi}=\mathscr{Q}_{\phi}=\mathscr{O}_{\phi}=0$, the pseudo Gauss map $\mathscr{G}:\Sigma^2\rightarrow \mathbb{C}\mathbb{P}^{4,1}$  of $\phi$ is either holomorphic or anti-holomorphic, or lies in a null totally geodesic hypersurface of the null quadric $Q^{3,1}\subset \mathbb{C}\mathbb{P}^{4,1}$. In the first case, $\phi$ is the image by the Penrose twistor fibration of a (singular) algebraic curve $C\subset \mathbb{C}\mathbb{P}^3$, and in the other case, $\phi$ is the inverse stereographic projection of a complete (branched)  minimal surface with finite total curvature in $\R^4$ and zero flux. Furthermore, the two possibilities coincide if and only if the algebraic curve $C\subset \mathbb{C}\mathbb{P}^3$ lies in some hypersurface $H\simeq  \mathbb{C}\mathbb{P}^2\subset\mathbb{C}\mathbb{P}^3$. In particular, the hypothesis are always satisfied for a Willmore sphere.
\end{theorem}

Furthermore, let us note that for a Willmore sphere $\phi:S^2\rightarrow S^4$ which is the Penrose twistor projection of an algebraic curve of $\mathbb{C}\mathbb{P}^3$ of degree $d$, we have 
\begin{align*}
	W(\phi)=\int_{S^2}(1+|\H|^2)d\vg=4\pi d,
\end{align*}
while for inverse stereographic projections of minimal surfaces, the energy is also quantized by $4\pi$ thanks to the Jorge-Meeks formula (see \cite{jorge} and the preceding section, as this formula is valid in any codimension). For a more detailed discussion on the minimizers of the Willmore energy for spheres in $S^4$ with respect to the regular homotopy class, we refer to the paper of Montiel \cite{montiels4}, and for a formula relating the degree of the dual algebraic curve with geometric invariants, we refer to the Plücker formula presented in the book of Griffiths and Harris (\cite{griffiths}).

\section{Appendix}

\subsection{Line bundles and complex operators}\label{line}

For the definitions of line bundles and divisors, we refer to \cite{bost}. If $\mathscr{L}$ is a holomorphic line bundle on a compact Riemann surface $\Sigma^2$, we note $\Gamma(\Sigma^2,\mathscr{L})$ the space of its smooth sections. All subsequently considered line bundle are always holomorphic. We simply recall that a divisor on $\Sigma^2$ is an element of the free abelian group $\mathrm{Div}(\Sigma^2)$ generated by the points of $\Sigma^2$, that is a formal sum
\begin{align*}
D=\sum_{p\in \Sigma^2}^{}n_p(D)\,p
\end{align*}
where $n_p(D)\in \Z$ for all $p\in \Sigma^2$ and $n_p(D)$ is equal to zero for almost all $p\in \Sigma^2$, that is, for all but finitely many. 

Defining an equivalence relation on the set $\mathscr{E}$ of pairs $(\mathscr{L},s)$, where $\mathscr{L}$ is a line bundle and $s$ is a non-zero meromorphic section, by $(\mathscr{L},s)\sim (\mathscr{L}',s')$ if $s^{-1}\otimes s'$ is a holomorphic section of $\mathscr{L}^\ast\otimes \mathscr{L}'$, we have an isomorphism
\begin{align*}
\mathscr{E}/\sim&\rightarrow\mathrm{Div}(\Sigma)\\
[(\mathscr{L},s)]&\mapsto \mathrm{div}(s).
\end{align*}
As $[\mathscr{L},s]\otimes [\mathscr{L}',s']=[\mathscr{L}\otimes \mathscr{L}',s\otimes s']$, $\mathrm{div}$ is a group homomorphism. Furthermore, the injectivity is easy to see, as $(\mathscr{L},s)\sim (\mathscr{L}',s')$ if and only if $s^{-1}\otimes s'$ is holomorphic, that is
\begin{align*}
0=\mathrm{div}(s^{-1}\otimes s')=\mathrm{div}(s^{-1})+\mathrm{div}(s')=-\mathrm{div}(s)+\mathrm{div}(s').
\end{align*}
In the surjectivity is constructed for any divisor $D\in \mathrm{Div}(\Sigma)$ a canonical line bundle $(\mathscr{O}(D),1_{\mathscr{O}(D)})$, and the injectivity shows that for all non-zero meromorphic section $s$ of a line bundle $\mathscr{L}$, we have
\begin{align}\label{trivialbundle}
\mathscr{L}\simeq \mathscr{O}(\mathrm{div}(s)).
\end{align}

We now come to the degree of a line bundle. If $\Sigma^2$ is a compact connected Riemann surface, we will denote $K_{\Sigma^2}=T^{\ast}\Sigma^2$ its canonical line bundle, which is simply the dual tangent bundle. We define the degree of an holomorphic line bundle $\mathscr{L}$ as its first Chern class, \textit{i.e.} $\mathrm{deg}(\mathscr{L})=c_1(\mathscr{L})\in H^2(\Sigma^2,\Z)\simeq \Z$. As $c_1$ is a morphism, we have in particular for all line bundles $\mathscr{L},\mathscr{L}'$
\begin{align}\label{morphism}
\mathrm{deg}(\mathscr{L}\otimes \mathscr{L}')=\mathrm{deg}(\mathscr{L})+\mathrm{deg}(\mathscr{L}'),\quad \mathrm{deg}(\mathscr{L}^\ast)=-\mathrm{deg}(\mathscr{L}).
\end{align}
We have for example thanks to Gauss-Bonnet and Poincaré-Hopf theorems and the definition of the degree as given in \cite{bost}
\begin{align*}
\mathrm{deg}(K_{\Sigma^2})=2g-2
\end{align*}
where $g$ is the genus of $\Sigma^2$.
Finally, this is not hard to check that for any divisor $D$, we have
\begin{align}\label{degree}
\mathrm{deg}(\mathscr{O}(D))=\sum_{p\in\Sigma}^{}n_p(D),
\end{align}
Finally, we note the following easy consequence of \eqref{trivialbundle} and \eqref{degree}.
\begin{prop}\label{bundletrivial}
	A holomorphic line bundle on a compact connected Riemann surface of negative degree has no non-zero holomorphic section.
\end{prop}
\begin{proof}
	By the isomorphism \eqref{trivialbundle}, if $\mathscr{L}$ is a line bundle with negative degree, for all non-zero meromorphic section $s$ of $\mathscr{L}$ we have $\mathscr{L}\simeq\mathscr{O}(\mathrm{div}s)$ and this implies that 
	\begin{align*}
	\mathrm{deg}(\mathrm{div}(s))=\sum \mathrm{zeroes}(s)-\sum \mathrm{poles}(s) <0
	\end{align*}
	therefore non-zero meromorphic sections are never holomorphic.
\end{proof}
The existence of holomorphic sections of holomorphic bundles is the subject of the next theorem.

Let $\mathscr{L}$ be a holomorphic line bundle on a compact Riemann surface $\Sigma^2$, and $\bar{\partial}_{\mathscr{L}}$ the complex elliptic operator
\begin{align*}
\bar{\partial}_{\mathscr{L}}:\Gamma(\Sigma^2,\mathscr{L})\rightarrow\Gamma(\Sigma^2,\mathscr{L}\otimes \bar{K}_{\Sigma^2})
\end{align*}
defined as follows : for all \textit{smooth} section $s$ of $\mathscr{L}$, for all local trivialisation $(U,s_U)$ of $\mathscr{L}$, as $s_U$ never vanishes on $U$, there exists a unique smooth function $f$ such that
\begin{align*}
s=fs_U.
\end{align*}
and on $U$, we define
\begin{align*}
\bar{\partial}_{\mathscr{L}}s=\bar{\partial}f\otimes s\in \mathscr{L}\otimes \bar{K}_{\Sigma^2}.
\end{align*}
This definition does not depend on the local trivialisation, as the transition functions are holomorphic. In particular, the kernel of $\bar{\partial}_{\mathscr{L}}$ is $H^0(\Sigma,\mathscr{L})$, and its cokernel,
\begin{align*}
\Gamma(\Sigma^2,\mathscr{L}\otimes \bar{K}_{\Sigma^2})/\bar{\partial}_{\mathscr{L}}\left(\Gamma(\Sigma^2,\mathscr{L})\right)
\end{align*}
is called the first Dolbeault cohomology group of $\mathscr{L}$ and denoted by $H^1(\Sigma,\mathscr{L})$. 

A theorem of Dolbeault asserts that $H^0(\Sigma^2,\mathscr{L})$ and $H^1(\Sigma^2,\mathscr{L})$ are finite dimensional complex vector spaces.
Actually, this can be easily verified, as  the ellipticity of $\bar{\partial}_{\mathscr{L}}$ (coupled with the fact that $\bar{\partial}_{\mathscr{L}}$ acts on a \emph{compact} manifold) ensures that $\bar{\partial}_{\mathscr{L}}$ is a Fredholm operator, that is
\begin{align*}
\mathrm{dim}\,\mathrm{Ker}(\bar{\partial}_{\mathscr{L}})<\infty,\quad \mathrm{dim}\,\mathrm{Coker}(\bar{\partial}_{\mathscr{L}})<\infty.
\end{align*}
The index of a Fredholm operator is defined by
\begin{align*}
\mathrm{Ind}(\bar{\partial}_{\mathscr{L}})&=\mathrm{dim}\,\mathrm{Ker}(\bar{\partial}_{\mathscr{L}})-\mathrm{dim}\,\mathrm{Coker}(\bar{\partial}_{\mathscr{L}})
=\mathrm{dim}\,H^0(\Sigma^2,\mathscr{L})-\mathrm{dim}\,H^1(\Sigma^2,\mathscr{L}).
\end{align*}
Computing the index of this operator is nothing else than the classical Riemann-Roch theorem.
\begin{theorem}
	Let $\Sigma^2$ be a compact connected Riemann surface of genus $g$ and $\mathscr{L}$ be a holomorphic line bundle on $\Sigma^2$. Then
	\begin{align}
	\mathrm{dim}\,H^0(\Sigma^2,\mathscr{L})-\mathrm{dim}\,H^1(\Sigma^2,\mathscr{L})=1-g+\mathrm{deg}(\mathscr{L}).
	\end{align}
\end{theorem}

Furthermore, a theorem of duality of Serre shows that
\begin{align*}
H^1(\Sigma^2,\mathscr{L})\simeq H^0(\Sigma^2,\mathscr{L}^\ast\otimes K_{\Sigma^2}).
\end{align*}
In particular, if $\mathrm{deg}(\mathscr{L}^\ast\otimes K_{\Sigma^2})<0$, \textit{i.e.}
\begin{align*}
\mathrm{deg}(\mathscr{L})>2g-2,
\end{align*}
then the line bundle $\mathscr{L}^\ast\otimes K_{\Sigma^2}$ is holomorphically trivial and
\begin{align*}
\mathrm{dim}\, H^0(\Sigma^2,\mathscr{L})=1-g+\mathrm{deg}(\mathscr{L}).
\end{align*}

\subsection{Almost-harmonic equation and approximate parametrix of $\bar{\partial}$ operator}

\begin{lemme}\label{pseudoDelta}
	Let $\Sigma^2$ be closed Riemann surface, $n\geq 3$, and $\phi:\Sigma^2\rightarrow \R^n$ be a smooth immersion. Then its Gauss map $\n:\Sigma^2\rightarrow \mathscr{G}_{n-2}(\R^n)$ satisfies the following almost-harmonic equation
	\begin{align}\label{eq1}
	\Delta_g\n+|d\n|_g^2\n=8\,g^{-1}\otimes \Im\left(\star\left(\bar{\partial}\H\wedge \partial\phi\right)\right)+2i\star g^{-2}\otimes \left(\bar{\h}_0\wedge \h_0\right).
	\end{align}
\end{lemme}
\begin{proof}
		As $\n=2ie^{-2\lambda}\ast(\e_{\z}\wedge\e_z)$ we have
	\begin{align}\label{lap0}
	\D_{\p{z}}\n&=2i\p{z}(e^{-2\lambda})\star(\e_{\z}\wedge\e_z)+i\star\left(\H\wedge \e_{z}+\e_{\z}\wedge \H_0\right)+2ie^{-2\lambda}\star\left(\e_{\z}\wedge \D^{\top}_{\p{z}}\e_{\z}\right) \nonumber\\
	&=e^{2\lambda}\p{z}(e^{-2\lambda})\n+i\star\left(\H\wedge \e_z+\e_{\z}\wedge\H_0\right)+e^{-2\lambda}\p{z}(e^{2\lambda})\n\nonumber\\
	&=i\star\left(\H\wedge \e_z+\e_{\z}\wedge\H_0\right)\nonumber\\
	\D_{\p{\z}}\D_{\p{z}}\n&=i\star(\D_{\p{\z}}\H\wedge \e_z+\vec{H}\wedge \D_{\p{\z}}\e_z+\D_{\p{\z}}\e_{\z}\wedge \H_0+\e_{\z}\wedge \D_{\p{\z}}\H_0)
	\end{align}
	Then we compute
	\begin{align*}
	\D_{\p{\z}}\H&=\bar{\partial}\H+\D^{\top}_{\p{\z}}\H\\
	&=\bar{\partial}\H-\s{\H}{\bar{\H_0}}\e_z-|\H|^2\e_{\z}
	\end{align*}
	therefore
	\begin{align}\label{lap1}
	\D_{\p{\z}}\H\wedge \e_z=\bar{\partial}\H\wedge \e_z-|\H|^2\e_{\z}\wedge\e_z.
	\end{align}
	Then as $\D_{\p{\z}}\e_{z}=\dfrac{e^{2\lambda}}{2}\H$, we have
	\begin{align}\label{lap2}
	\H\wedge \D_{\p{\z}}\e_z=0
	\end{align}
	Now we obtain
	\begin{align}\label{lap3}
	\D_{\p{\z}}\e_{\z}\wedge \H_0=\frac{e^{2\lambda}}{2}\bar{\H}_0\wedge \H_0+e^{-2\lambda}\p{\z}(e^{2\lambda})\e_{\z}\wedge \H_0
	\end{align}
	Then
	\begin{align}\label{lap4}
	\e_{\z}\wedge\D_{\p{\z}}\H_0&=e^{2\lambda}\p{\z}(e^{-2\lambda})\e_{\z}\wedge \H_0+2e^{-2\lambda}\e_{\z}\wedge \D_{\p{\z}}^\nperp(\vec{\I}(\e_z,\e_z))+\e_{\z}\wedge \D^{\top}_{\p{\z}}\H_0\nonumber\\
	&=e^{2\lambda}\p{\z}(e^{-2\lambda})\e_{\z}\wedge \H_0+\e_{\z}\wedge {\partial}\H-|\H_0|^2\e_{\z}\wedge \e_z
	\end{align}
	Finally, we obtain by \eqref{lap0}, \eqref{lap1}, \eqref{lap2}, \eqref{lap3}, \eqref{lap4}
	\begin{align*}
	\D_{\p{\z}}\D_{\p{z}}\n&=i\star\bigg(\e_{\z}\wedge\bar{\partial}\H-|\H|^2\e_{\z}\wedge \e_z+\frac{e^{2\lambda}}{2}\bar{\H}_0\wedge \H_0+e^{-2\lambda}\p{\z}(e^{2\lambda})\e_{\z}\wedge \H_0\\
	&+e^{2\lambda}\p{\z}(e^{-2\lambda})\e_{\z}\wedge\H_0+\e_{\z}\wedge \partial \H-|\H_0|^2\e_{\z}\wedge \e_z\bigg)\\
	&=-\frac{e^{2\lambda}}{2}(|\H|^2+|\H_0|^2)\n+2\,\Im \left(\star\left(\bar{\partial}\H\wedge \e_z\right)\right)+\frac{e^{2\lambda}}{2}i\star(\bar{\H}_0\wedge \H_0)
	\end{align*}
	as 
	\begin{align*}
	|\H_0|^2=|\H|^2-K_g+K_h
	\end{align*}
	and 
	\begin{align*}
	|\vec{\I}|^2_g=4|\H|^2-2K_g+2K_h
	\end{align*}
	we obtain
	\begin{align*}
	|\H|^2+|\H_0|^2=2|\H|^2-K_g+K_h=\frac{1}{2}|\vec{\I}|_g^2=\frac{1}{2}|d\n|_g^2
	\end{align*}
	therefore as
	\begin{align*}
	\Delta_g\n+|d\n|_g^2\n=8\,e^{-2\lambda}\Im\left(\star\left(\bar{\partial}\H\wedge \e_z\right)\right)+2i\star (\bar{\H}_0\wedge \H_0).
	\end{align*}
	which is the expected almost harmonic equation. In particular, we see that for $n=3$, $\phi$ has constant mean curvature if and only if $\h_0$ is holomorphic, and by \eqref{eq1} this is equivalent to the harmonicity of $\n:\Sigma^2\rightarrow S^2$. Finally, we note that the equation is indeed real, as for any complex vector $\w$
	\begin{align*}
	i\bar{\w}\wedge \w=i(\Re\w-i\Im\w)\wedge (\Re\w+i\Im \w)=-2\Re\w\wedge \Im\w
	\end{align*}
	and this concludes the proof.
\end{proof}

\begin{prop}\label{parenthesis}
	Let $n\geq 3$, $\phi\in C^{\infty}(D^2\setminus\ens{0},\R^n)$ be a branched Willmore disk with a unique branch point at zero of multiplicity $\theta_0\geq 3$. If we have for some $\vec{C}_1\in \C^n\setminus\ens{0}$ and some $\alpha\leq \theta_0-2$
	\begin{align*}
		\H=\Re\left(\frac{\vec{C}_1}{z^{\alpha}}\right)+O(|z|^{1-\alpha})
	\end{align*}
	for all $\epsilon>0$, if $\n$ is the unit normal of $\phi$, we have
	\begin{align}
		\n\in C^{1,1}(D^2,\mathscr{G}_{n-2}(\R^n)).
	\end{align}
\end{prop}
	\begin{proof}	
	As the regularity can only increase as $\alpha$ decreases, we suppose that $\alpha=\theta_0-2$. Therefore, there exists $\vec{C}_1\in \C^n$ such that
	\begin{align}\label{2ndres}
	\H=\Re\left(\frac{\vec{C}_1}{z^{\theta_0-2}}\right)+O(|z|^{3-\theta_0-\epsilon}).
	\end{align}
    By the almost-harmonic equation satisfied by the unit normal $\n$ in Lemma \ref{pseudoDelta} of the appendix, we obtain
	\begin{align}\label{pseudoDelta0}
	\Delta\n+|\D\n|^2\n&=8\,\Im\left(\star\left(\bar{\partial}\H\wedge \partial\phi\right)\right)+2i\star(e^{\lambda}\bar{\H}_0\wedge e^{\lambda}\vec{H}_0).
	\end{align}
	Now, by Codazzi's identity, we have 
	\begin{align}\label{cod1}
	\bar{\partial}^N\h_0=g\otimes \partial^N\H=g\otimes \partial\H+|\H|^2\,g\otimes \partial\phi+\s{\H}{\h_0}\otimes\bar{\partial}\phi.
	\end{align}
	Furthermore, we easily compute that
	\begin{align}
	\bar{\partial}^\top \h_0&=-|\h_0|^2_{WP}\,g\otimes\partial\phi-\s{\H}{\h_0}\otimes\bar{\partial}\phi
	=-\left(|\H|^2-K_g\right)\,g\otimes\partial\phi-\s{\H}{\h_0}\otimes\bar{\partial}\phi
	\end{align}
	so
	\begin{align}\label{cod2}
	\bar{\partial}^N\h_0=\bar{\partial}\h_0-\bar{\partial}^\top\h_0=\bar{\partial}\h_0+\left(|\H|^2-K_g\right)\,g\otimes\partial\phi+\s{\H}{\h_0}\otimes\bar{\partial}\phi.
	\end{align}
	Putting together \eqref{cod1} and \eqref{cod2}, we get
	\begin{align}\label{cod3}
	\bar{\partial}\h_0=g\otimes \partial\H+(K_g)\,g\otimes \partial\phi.
	\end{align}
	Recalling that for $e^{2u}=|z|^{2-2\theta_0}e^{2\lambda}$, we have 
	\begin{align*}
	-\Delta u=e^{2\lambda}K_g\in L^{\infty}(D^2)
	\end{align*}
	we deduce that
	\begin{align*}
	(K_g)\,g\otimes \partial\phi=O(|z|^{\theta_0-1})
	\end{align*}
	while by 
	\begin{align*}
	\partial\H=-\frac{(\theta_0-2)}{2}\frac{dz}{z^{\theta_0-1}}+O(|z|^{2-\theta_0}).
	\end{align*}
	Therefore, we deduce as $e^{2\lambda}=|z|^{2\theta_0-2}(1+O(|z|))$ that
	\begin{align}
		\bar{\partial}\h_0=O(|z|^{\theta_0-1})
	\end{align}
	so by Proposition \ref{integrating}, there exists $\vec{D}_1,\vec{A}_1\in\C^n$ such that
	\begin{align}\label{almostharmh0}
		\h_0=\vec{D}_2z^{\theta_0-2}dz^2+\vec{A}_1z^{\theta_0-1}dz^2+O(|z|^{\theta_0}).
	\end{align}
	However, as we saw in the beginning of the proof of Theorem \ref{devh0} that $\h_0=O(|z|^{\theta_0-1})$, so
	\begin{align}
		\vec{D}_0=0.
	\end{align}
	Indeed, by the definition  of branch points we have for some $\vec{A}_0\in \mathbb{C}^n\setminus\ens{0}$ the expansions
	\begin{align*}
	\left\{\begin{alignedat}{1}
	\phi(z)&=\Re\left(\vec{A}_0z^{\theta_0}\right)+O(|z|^{\theta_0+1})\\
	2(\p{z}\lambda)&=\frac{(\theta_0-1)}{z}+O(1)
	\end{alignedat}\right.
	\end{align*}
	so
	\begin{align*}
		\h_0&=2\left(\p{z}^2\phi-2(\p{z}\lambda)\p{z}\phi\right)dz^2\\
		&=2\left(\frac{\theta_0(\theta_0-1)}{2}z^{\theta_0-2}-\left(\frac{(\theta_0-1)}{z}+O(1)\right)\left(\frac{\theta_0}{2}z^{\theta_0-2}+O(|z|^{\theta_0-1})\right)\right)dz^2+O(|z|^{\theta_0-1})\\
		&=O(|z|^{\theta_0-1}).
	\end{align*}
	and by \eqref{almostharmh0}, we obtain the expansion
	\begin{align}\label{normaldevh0}
		\h_0=\vec{A}_1z^{\theta_0-1}+O(|z|^{\theta_0})
	\end{align}
	As $\h_0=O(|z|^{\theta_0-1})$, we have $e^{\lambda}\H_0\in L^{\infty}(D^2)$, and 
	\begin{align*}
	\partial\H=O(|z|^{1-\theta_0}),
	\end{align*}
	so
	\begin{align}\label{c21estimate0}
	\bar{\partial}\H\wedge \partial\phi\in L^{\infty}(D^2),
	\end{align}
	while $\D\n\in L^p(D^2)$ for all $p<\infty$ (as $\phi\in W^{2,p}(D^2)$ for all $p<\infty$), we have
	\begin{align*}
	\Delta\n\in \bigcap_{p<\infty}L^p(D^2)
	\end{align*}
	and by standard Calder\'{o}n-Zygmund estimates, one has
	\begin{align*}
	\n\in \bigcap_{p<\infty}W^{2,p}(D^2).
	\end{align*}
	In particular, $\D\n\in L^{\infty}(D^2)$ (this was already proved in \cite{beriviere}), so reinserting this information in \eqref{pseudoDelta0}, we obtain
	\begin{align}\label{deltainfty}
	\Delta\n\in L^{\infty}(D^2),
	\end{align}
	and 
	\begin{align}\label{bmo}
	\D^2\n\in BMO(D^2).
	\end{align}
	Finally, we deduce immediately from \eqref{bmo} that
	\begin{align}\label{regup}
	\phi\in \bigcap_{p<\infty}W^{3,p}(D^2)\hookrightarrow \bigcap_{\alpha<1}C^{2,\alpha}(D^2).
	\end{align}
	We will now prove the extra regularity 
	\begin{align*}
	\n\in C^{1,1}(D^2).
	\end{align*}
	Indeed, if $\n:D^2\rightarrow \wedge^{n-2}\R^n$ is the Gauss map of $\phi$, then by the Lemma \ref{pseudoDelta}, we deduce that
	\begin{align*}
	\p{z}\n=i\star\left(\H\wedge \p{z}\phi+\p{\z}\phi\wedge \H_0\right)
	\end{align*}
	so
	\begin{align}\label{2pseudoharmonic}
	\p{z}^2\n=i\star\left(\p{z}\H\wedge\p{z}\phi+\H\wedge\p{z}^2\phi+\p{z\z}^2\phi\wedge \H_0+\p{\z}\phi\wedge\p{z}\H_0\right).
	\end{align}
	Firstly, by \eqref{c21estimate0}, we have
	\begin{align*}
	\partial_z\H\wedge \partial_z\phi\in L^{\infty}(D^2),
	\end{align*}
	and quite trivially $\p{z}^2\phi=O(|z|^{\theta_0-2})$, but $\alpha\leq \theta_0-2$ shows that $\H=O(|z|^{2-\theta_0})$, we have
	\begin{align}\label{higherreg1}
	\H\wedge \p{z}^2\phi\in L^{\infty}(D^2).
	\end{align}
	Now, using $e^{2\lambda}=|z|^{2\theta_0-2}\left(1+O(|z|^2)\right)$ and \eqref{normaldevh0}, we deduce that \Big(recall that $\h_0=\left(e^{2\lambda}\H_0\right)\,dz^2$\Big)
	\begin{align}\label{devH0}
	\H_0=\frac{\vec{A}_1}{\z^{\theta_0-1}}+O(|z|^{2-\theta_0}),
	\end{align}
	and this implies that
	$
	\partial_{z}\H_0=O(|z|^{1-\theta_0}),
	$
	so
	\begin{align}\label{higherreg2}
	\p{\z}\phi\wedge \partial_z\H_0\in L^{\infty}(D^2).
	\end{align}
	The trivial estimate
	$
	\p{z}^2\phi=O(|z|^{\theta_0-2}),
	$
	implies
	\begin{align}\label{higherreg3}
	\H\wedge\p{z}^2\phi\in L^{\infty}(D^2)
	\end{align}
	while as
	\begin{align*}
	\Delta\phi=2e^{2\lambda}\H=O(|z|^{\theta_0}),
	\end{align*}
	we obtain by \eqref{devH0}
    $
	\p{z\z}^2\phi\wedge \H_0=O(|z|),
	$
	and
	\begin{align}\label{higherreg4}
		\p{z\z}^2\phi\wedge\H_0\in L^{\infty}(D^2)
	\end{align}
	and by \eqref{c21estimate0}. Therefore, putting together \eqref{higherreg1}, \eqref{higherreg2}, \eqref{higherreg3}, and \eqref{higherreg4}, and looking at \eqref{2pseudoharmonic}, we finally have
	\begin{align}
		\p{z}^2\n\in L^{\infty}(D^2)
	\end{align}
	and by \eqref{deltainfty}, $\p{z\z}^2\n\in L^{\infty}(D^2)$, so
	\begin{align*}
		\p{z}\n\in W^{1,\infty}(D^2)
	\end{align*}
	and as $\n$ is \emph{real}, we have
	\begin{align}
		\n\in W^{2,\infty}(D^2)=C^{1,1}(D^2)
	\end{align}
	which concludes the proof of the proposition.
\end{proof}

We now come to the proposition allowing one to integrate solutions of the $\bar{\partial}$ equation to obtain a Taylor expansion at singular points (see the appendix of \cite{beriviere}). 
We first recall the boundedness of the maximal operator and an easy lemma.
\begin{theorem}\label{endlemma1}
	Let $1<p<\infty$. There exists a constant $C=C(p)$ independent of $n$ such that for all $f\in L^p(\R^n)$,
	\begin{align*}
	\np{{M}f}{p}{\R^n}\leq C\np{f}{p}{\R^n}
	\end{align*}
	where ${M}$ is the centred maximal function for Euclidean balls.
\end{theorem}
\begin{lemme}\label{endlemma2}
	Let $0<\alpha<n$ and $r>0$. Then for any $f\in L^1_{\mathrm{loc}}(\R^n)$, for all $x\in \R^n$, we ave
	\begin{align*}
	\int_{B_r(x)}\frac{f(y)}{|x-y|^{n-\alpha}}dy\leq \frac{2^n\alpha(n)}{2^\alpha-1}r^{\alpha}Mf(x)
	\end{align*}
\end{lemme}
\begin{proof}
	For $k\in\N$, let $B_k=B_{2^{-k}r}(x)$. We have
	\begin{align*}
	\int_{B_r(x)}\frac{f(y)}{|x-y|^{n-\alpha}}dy&=\sum_{k\in\N}^{}\int_{B_k\setminus B_{k+1}}\frac{f(y)}{|x-y|^{n-\alpha}}\leq\sum_{k\in\N}^{}\left(\frac{r}{2^{k+1}}\right)^\alpha\frac{1}{(2^{-(k+1)}r)^n}\int_{B_{2^{-k}r}(x)}f(y)dy\\
	&=\sum_{k\in\N}^{}2^n\alpha(n)\left(\frac{r}{2^{k+1}}\right)^\alpha \dashint{B_{2^{-k}r}(x)}f(y)dy\leq \frac{2^n\alpha(n)r^\alpha}{2^\alpha-1}Mf(x).
	\end{align*}
	This computation concludes the proof of the lemma.
\end{proof}
\begin{prop}\label{integrating}
    Let $u\in C^1(\bar{D^2}\setminus\ens{0})\cap L^2(D^2)$ be such that
    \begin{align*}
    	{\partial}_{\z}\,u(z)=\mu(z)f(z),\quad z\in D^2\setminus\ens{0} 
    \end{align*}
    where $f\in L^p(D^2)$ for some $2<p\leq \infty$, and $|\mu(z)|\simeq |z|^a\log^b|z|$ at $0$ for some $a\in \N$, and $b\geq 0$. Then
    \begin{align*}
    	u(z)=P(z)+|\mu(z)|T(z)
    \end{align*}
    for some \emph{polynomial $P$} of degree less than $a$, and a function $T$ such that
    \begin{align*}
    	T(z)=O(|z|^{1-\frac{2}{p}}\log^{\frac{2}{p'}}|z|).
    \end{align*}
     In particular, if $f\in L^{\infty}(D^2)$, we have
    \begin{align*}
    	u(z)=P(z)+O(|z|^{a+1}\log^{b+2}|z|).
    \end{align*}
\end{prop}
\begin{proof}
	By the general Cauchy formula (see \cite{hormandercomplex}), for all $z\in D^2\setminus\ens{0}$,
	\begin{align}\label{cauchy}
	u(z)&=\frac{1}{2\pi i}\left\{\int_{S^1}\frac{u(\zeta)}{\zeta-z}d\zeta+\int_{D^2}\frac{\p{\z}u(\zeta)}{\zeta-z}d\zeta\wedge d\bar{\zeta}\right\}=\frac{1}{2\pi i}\left\{\int_{S^1}\frac{u(\zeta)}{\zeta-z}d\zeta+\int_{D^2}\frac{\mu(\zeta)f(\zeta)}{\zeta-z}d\zeta\wedge d\bar{\zeta}\right\}\nonumber\\
	&=\frac{1}{2\pi i}\left(u_1(z)+u_2(z)\right).
	\end{align}
	In particular, $u$ is analytic on $D^2\setminus \ens{0}$. We now fix a constant $C>0$ such that
	\begin{align*}
	|\mu(|z|)|\leq C|z|^{a}(1+\log^b|z|)\quad \text{for all }\;\, z\in D^2.
	\end{align*}
	Now developing
	\begin{align*}
	\frac{1}{\zeta-z}=\sum_{n=0}^{\infty}z^n\zeta^{-(n+1)}
	\end{align*}
	we obtain for $|z|<1$
	\begin{align*}
	u_1(z)=\frac{1}{2\pi i}\int_{S^1}\frac{u(\zeta)}{\zeta-z}d\zeta=\sum_{n=0}^{\infty}\left(\frac{1}{2\pi i}\int_{S^1}u(\zeta)\zeta^{-(n+1)}d\zeta\right)z^n=\sum_{n\in\N}^{}c_n z^n
	\end{align*}
	As $u\in C^0(\bar{D^2}\setminus\ens{0})$, we deduce that $|c_n|\leq \Vert  u\Vert_{L^{\infty}(S^1)}$, and as $u\in C^1(\bar{D^2}\setminus\ens{0})$, we have $n|c_n|=O(1)$, so $\ens{c_n}_{n\in\N}\in l^2(\N)$, and the formula is valid in $L^2$ on the boundary $S^1$ too. In particular, $u_1$ is analytic in $D^2$, so we can write
	\begin{align}\label{analytic0}
	u_1(z)=\sum_{n=0}^{a}c_nz^n+\varphi_1(z)
	\end{align}
	where $\varphi_1(z)=O(|z|^{a+1})$ is analytic. Then we decompose
	\begin{align}\label{decomp2}
	u_2(z)=\int_{D(2|z|)}+\int_{D\setminus D(2|z|)}=u_2^1(z)+u_2^2(z)
	\end{align}
	Then by Lemma \ref{endlemma2} with $n=2$, $\alpha=1$, we have
	\begin{align}
	|u_2^1(z)|&\leq C2^a|z|^{a}(1+\log^b|z|)\int_{D(0,2|z|)}\frac{|f(\zeta)|}{|\zeta-z|}|d\zeta|^2\leq C 2^a|z|^a(1+\log^b|z|)\int_{D(z,3|z|)}\frac{|f(\zeta)|}{|\zeta-z|}|d\zeta|^2\nonumber\\
	&\leq C2^{a+3}\pi\,|z|^{a+1}(1+\log^b|z|){M}f(z)\leq C_1\np{f}{p}{D^2}|z|^{1-\frac{2}{p}}|\mu(z)|.
	\end{align}
	Then we have
	\begin{align*}
	u_2^2(z)=\int_{D\setminus D(0,2|z|)}\frac{\mu(\zeta)f(\zeta)}{\zeta-z}d\zeta\wedge d\bar{\zeta}&=\sum_{n\in\N}^{}\int_{D\setminus D(0,2|z|)}\left(\frac{\mu(\zeta)f(\zeta)}{\zeta^{n+1}}d\zeta\wedge d\bar{\zeta}\right)z^n\\
	&=\sum_{n\in\N}^{}d_n(z)\,z^n
	\end{align*}
	and for all $n\leq a$, one has by H\"{o}lder's inequality
	\begin{align}\label{coefficientborne}
		\left|\int_{D^2}^{}\frac{\mu(\zeta)f(\zeta)}{\zeta^{n+1}}d\zeta\wedge d\bar{\zeta}\right|\leq C\int_{D^2}\frac{|f(\zeta)|}{|\zeta|}|d\zeta|^2\leq 2\left(\frac{2\pi}{2-p'}\right)^{\frac{1}{p'}}C\np{f}{p}{D^2}.
	\end{align}
	We will also need this further decomposition
	\begin{align*}
		u_2(z)&=\sum_{n=0}^a\left( \int_{D^2}\frac{\mu(\zeta)f(\zeta)}{\zeta^{n+1}}d\zeta\wedge d\bar{\zeta}\right)z^n-\sum_{n=0}^{a}\left(\int_{D(0,2|z|)}\frac{\mu(\zeta)f(\zeta)}{\zeta^{n+1}}d\zeta\wedge d\bar{\zeta}\right)z^n\\
		&+\sum_{n=a+1}^{\infty}\left(\int_{D\setminus D(0,2|z|)}\frac{\mu(\zeta)f(\zeta)}{\zeta^{n+1}}d\zeta\wedge d\bar{\zeta}\right)z^n.
	\end{align*}
	By \eqref{coefficientborne}, the first term is a polynomial of degree at most $a$, and for $0\leq n\leq a$,
	\begin{align*}
		\left|\int_{D(0,2|z|)}\frac{\mu(\zeta)f(\zeta)}{\zeta^{n+1}}d\zeta\wedge d\bar{\zeta}\right|&\leq C|z|^{a-n}(1+\log^b|z|)\int_{D(0,2|z|)}\frac{|f(\zeta)|}{|\zeta|^{}}|d\zeta|^2\\
		&\leq 2^{\frac{2}{p'}}\left(\frac{2\pi}{2-p'}\right)^{\frac{1}{p'}}C\np{f}{p}{D^2}|z|^{a-n+1-\frac{2}{p}}(1+\log^b|z|).
	\end{align*}
	For $0\leq n\leq a$, one has
	\begin{align}\label{2eterme}
		\left|\int_{D(0,2|z|)}\frac{\mu(\zeta)f(\zeta)}{\zeta^{n+1}}d\zeta\wedge d\bar{\zeta}\right|\leq C'\np{f}{p}{D^2}|z|^{1-\frac{2}{p}}|\mu(z)|.
	\end{align}
	Finally,
\begin{align*}
		&\left|\int_{D\setminus D(0,2|z|)}\frac{\mu(\zeta)f(\zeta)}{\zeta^{n+1}}d\zeta\wedge d\bar{\zeta}\right|\leq C(1+\log^b(|z|)\int_{D\setminus D(0,2|z|)}|\zeta|^{a+1-n-\frac{2}{p}}|\zeta|^{-\frac{2}{p'}}|f(\zeta)||d\zeta|^2\\
			&\leq C2^{a+1-n-\frac{2}{p}}|z|^{a+1-n-\frac{2}{p}}(1+\log^b|z|)\log^{\frac{2}{p'}}|z|\left(\int_{D^2}\frac{|d\zeta|^2}{|\zeta|^2\log^2\left(\frac{|\zeta|}{2}\right)}\right)^{\frac{1}{p'}}\np{f}{p}{D^2}\\
			&\leq \frac{C'}{2^n}|z|^{a+1-n-\frac{2}{p}}(1+\log^{b+\frac{2}{p'}}|z|)\np{f}{p}{D^2}.
\end{align*}
Therefore,
\begin{align}\label{3eterme}
	\left|\sum_{n=a+1}^{\infty}\left(\int_{D\setminus D(0,2|z|)}\frac{\mu(\zeta)f(\zeta)}{\zeta^{n+1}}d\zeta\wedge d\bar{\zeta}\right)z^n\right|&\leq 2C'|z|^{a+1-\frac{2}{p}}(1+\log^{b+\frac{2}{p'}}|z|)\np{f}{p}{D^2}\\
	&\leq C''|\mu(z)||z|^{1-\frac{2}{p}}(1+\log^{\frac{2}{p'}}|z|)\np{f}{p}{D^2}
\end{align}
and putting together \eqref{cauchy}, \eqref{analytic0} \eqref{decomp2}, \eqref{2eterme}, \eqref{3eterme}, we can write
\begin{align*}
	u(z)=P(z)+|\mu(z)|T(z)
\end{align*}
where $T(z)=O(|z|^{1-\frac{2}{p}}\log^{\frac{2}{p'}}|z|)$, and this concludes the proof.
\end{proof}

\begin{rem}
    If $\phi:S^2\rightarrow\R^3$ is the inverted catenoid, we easily get
	\begin{align*}
	\h_0(z)&=\left(\z,-i\z,\frac{1}{2}\frac{\z}{z}\right)dz^2+O(|z|^2\log^2|z|)
	=\vec{\gamma}_0 \frac{\z}{z}dz^2+\vec{A}\z dz^2+O(|z|^2\log^2|z|)
	\end{align*}
	therefore the error term is essentially optimal, as it cannot be better than $O(|z|^2\log^2|z|)$ for a Willmore sphere at a multiplicity $1$ branch point. In particular, the estimate of Theorem \ref{devh0} is optimal.
\end{rem}

    \nocite{}
\bibliographystyle{plain}
\bibliography{biblio_classification}

\end{document}